%% file: main.tex
\documentclass[12pt]{amsart}
\input{command}

\usepackage{parskip}
\usepackage{tikzit}
 \usepackage{relsize}
\input{tikz-setup}
\usepackage[utf8]{inputenc}
\usepackage{setspace}
\usepackage{amssymb}
\usepackage{tabu}
\usepackage[strict]{changepage}
\usepackage[english]{babel}
\usepackage{adjustbox}
\usepackage{enumerate}
\usepackage{breqn}
\usepackage{thmtools}
\usepackage[shortlabels]{enumitem}
\usepackage[T1]{fontenc}
\usepackage{tikz}
\usetikzlibrary{shapes.geometric,intersections}
\usetikzlibrary{calc,intersections,through,backgrounds}
\usetikzlibrary{patterns}
\usepackage{mathdots}
\usepackage{fancyhdr}

\usepackage{array}
\newcolumntype{Q}{>{\centering}p{\friezelen}<{}}

\usepackage{comment}
\usepackage{mathtools}
\usepackage{verbatim}
\usetikzlibrary{arrows}
\tikzset {->-/.style={decoration={markings, mark=at position .5 with {\arrow{latex}}}, postaction={decorate}}}

\usepackage{graphicx}
\usepackage{color}
\usepackage{etoolbox}
\usepackage[margin=1in]{geometry}
\usepackage{amsmath,amsthm,amssymb,graphicx,tikz,tikz-cd}
\usepackage{hyperref}
\hypersetup{
    colorlinks = true,
    citecolor = orange,
}
\usepackage[noabbrev]{cleveref}
\usepackage{subcaption}

\makeatletter
\def\@tocline#1#2#3#4#5#6#7{\relax
  \ifnum #1>\c@tocdepth 
  \else
    \par \addpenalty\@secpenalty\addvspace{#2}%
    \begingroup \hyphenpenalty\@M
    \@ifempty{#4}{%
      \@tempdima\csname r@tocindent\number#1\endcsname\relax
    }{%
      \@tempdima#4\relax
    }%
    \parindent\z@ \leftskip#3\relax \advance\leftskip\@tempdima\relax
    \rightskip\@pnumwidth plus4em \parfillskip-\@pnumwidth
    #5\leavevmode\hskip-\@tempdima
      \ifcase #1
       \or\or \hskip 1em \or \hskip 2em \else \hskip 3em \fi%
      #6\nobreak\relax
    \hfill\hbox to\@pnumwidth{\@tocpagenum{#7}}\par
    \nobreak
    \endgroup
  \fi}
\makeatother

\makeatletter
\patchcmd{\@settitle}{\uppercasenonmath\@title}{}{}{}
\patchcmd{\@setauthors}{\MakeUppercase}{}{}{}
\patchcmd{\section}{\scshape}{}{}{}
\makeatother
\makeatletter
\patchcmd{\@maketitle}
  {\ifx\@empty\@dedicatory}
  {\ifx\@empty\@date \else {\vskip2ex 
  \centering\footnotesize\@date\par\vskip1ex}\fi
   \ifx\@empty\@dedicatory}
  {}{}
\patchcmd{\@adminfootnotes}
  {\ifx\@empty\@date\else \@footnotetext{\@setdate}\fi}
  {}{}{}
\makeatother

\theoremstyle{plain}
\newtheorem{thm}{Theorem}[section]

\newtheorem{lemma}[thm]{Lemma}
\newtheorem{prop}[thm]{Proposition}

\newtheorem{cor}[thm]{Corollary}

\theoremstyle{definition}
\newtheorem{remark}[thm]{Remark}
\newtheorem{ex}[thm]{Example}
\newtheorem{defn}[thm]{Definition}

\newcommand{\opt}[0]{\textup{opt}}
\newcommand{\fix}[0]{\textup{fix}}
{\newcommand{\Jm}{J_{\mathbf{m}}}

\newcommand{\cal}[1]{\mathcal{#1}}

\setcounter{MaxMatrixCols}{11}

\setlist[enumerate]{leftmargin=*}

\usepackage[backend=bibtex,maxbibnames=99]{biblatex}
\addbibresource{main.bib}

\title{Snake Graphs for Graph LP Algebras}

\author[Esther Banaian]{Esther Banaian$^{\clubsuit}$ }
\author[Sunita Chepuri]{Sunita Chepuri$^{\diamondsuit}$}
\author[Elizabeth Kelley]{Elizabeth Kelley$^{\heartsuit}$}
\author[Sylvester Zhang]{Sylvester W. Zhang$^{\spadesuit}$}
\thanks{$^{\clubsuit}$Aarhus University. \href{mailto:banaian@math.au.dk}{banaian@math.au.dk}}
\thanks{$^{\diamondsuit}$University of Illinois  Urbana Champaign. \href{mailto:kelleye@illinois.edu}{kelleye@illinois.edu}}
\thanks{$^{\heartsuit}$University of Puget Sound. 
\href{mailto:chepu003@umn.edu}{schepuri@pugetsound.edu}}
\thanks{$^{\spadesuit}$University of Minnesota. \href{mailto:swzhang@umn.edu}{swzhang@umn.edu}}
\date{\today}

\begin{document}

\maketitle
\begin{abstract}
Graph LP algebras are a generalization of cluster algebras introduced by Lam and Pylyavskyy.  We provide a combinatorial proof of positivity for certain cluster variables in these algebras. This proof uses a hypergraph generalization of snake graphs, a class of planar graphs which were used by Musiker, Schiffler, and Williams to prove positivity for cluster algebras from surfaces. These results extend those given in our previous paper, where we used a related combinatorial object known as a $T$-path. 
\end{abstract}

\tableofcontents

\section{Introduction}

Cluster algebras, introduced by Fomin and Zelvinsky \cite{FZ1} in their study of total positivity and dual canonical bases, are a particular type of commutative ring that possess additional combinatorial structure. Loosely speaking, a cluster algebra is a commutative ring with a distinguished set of generators called \emph{cluster variables}, which are grouped into subsets of equal cardinality called \emph{clusters}. Two clusters which differ at only one entry are connected by a \emph{mutation}
\[\{x_1,\cdots,x_i,\cdots,x_n\}\leftrightarrow\{x_1,\cdots,x_i',\cdots,x_n\}\]
with exchange relation $x_i'x_i = F(x_1,\cdots,x_n)$ where $F$ is a binomial in $x_1, \dots, x_n$.

The two most important features of cluster algebras are:
\begin{enumerate}[(1)]
    \item (\emph{Laurent phenomenon}) For any choice of cluster $\calC=(x_1,\dots,x_n)$, every cluster variable can be written as a Laurent polynomial in $x_1,\dots,x_n$.
    \item (\emph{Positivity}) The Laurent polynomial in (1) has positive coefficients.
\end{enumerate}

In \cite{LP-12}, Lam and Pylyavskyy generalized cluster algebras to a broader family of commutative algebras which enjoy the Laurent phenomenon, hence called Laurent Phenomenon (LP) algebras. In an LP algebra, the exchange relations are relaxed to be arbitrary Laurent polynomials satisfying certain irreducibility conditions. In \cite{LP-12}, Lam and Pylyavskyy proved that LP algebras satisfy the Laurent Phenomenon and conjectured that positivity holds as well.

Graph LP algebras \cite{lam2016linear} are a special class of LP algebras whose structure is encoded by a graph. In particular, cluster variables are in bijection with connected subsets of the vertex set and clusters are in bijection with \emph{maximal nested collections} on the vertex set (Theorem \ref{thm:ClustersInGraphLPAlgebra}). The exchange relations can also be phrased in terms of the graph (Lemmas 4.7 and 4.11 from \cite{lam2016linear}). In \cite{rooted}, the current authors proved that in a graph LP algebra from a tree graph, positivity holds when we express a cluster variable in terms of a special kind of cluster called a \emph{rooted cluster}. The proof in \cite{rooted} is based on a combinatorial expansion formula using \emph{hyper $T$-paths} that is in spirit of Schiffler's $T$-path formula for type $A$ cluster algebras \cite{schiffler2006cluster}. 

In cluster combinatorics, $T$-paths are known to be in bijection with perfect matchings on certain planar bipartite graphs called \emph{snake graphs} \cite{ms10}. Here, we provide a graph-theoretic analogue to our hyper $T$-paths by building hypergraphs.  By considering this alternate framework, we are also able to significantly generalize our positivity result from \cite{rooted}.

When $\Gamma$ is a path, a large subalgebra of $\mathcal{A}_\Gamma$ coincides with a type $A$ cluster algebra (Corollary 6.2 \cite{lam2016linear}), in which case our construction specializes to that of Musiker-Schiffler \cite{ms10}. For a general tree $\Gamma$, our graphs are built up from Musiker--Schiffler snake graphs, which we can obtain by considering path graphs which sit as subgraphs of $\Gamma$. In order to capture the combinatorics of several Musiker--Schiffler snake graphs glued together, we introduce \emph{admissible matchings}, which will be mixed-dimer covers of these graphs with additional requirements. When we prove our main result for cluster variables associated to sets of size 1, we also provide results which compares our admissible matchings with perfect matchings of Musiker--Schiffler snake graphs (Lemmas \ref{lem:r1admissible} and \ref{lem:HowDoGvRestrictToHi}).

Our main result can be summarized as follows. These statements are made more precise in Theorems \ref{thm:main} and \ref{thm:off-in-cluster}.

\textbf{Main Result:} Let $\Gamma$ be a tree and $\cal A_\Gamma$ the graph LP algebra associated to $\Gamma$. Let $ \{Y_{S_1},\cdots,Y_{S_n}\}$ be a cluster of $\cal A_\Gamma$, that is, let $\mathcal I:=\{S_1,\cdots,S_n\}$ be a maximal nested collection on the vertices of $\Gamma$. Then, given a weakly rooted set $S$ with respect to $\mathcal I$, $Y_S$ can be written as a generating function of admissible matchings of a certain graph $\mathcal{G}_S$. As a result, we write $Y_S$ as a Laurent polynomial with denominator consisting of variables from the given cluster. With certain conditions on $S$, the numerator of this expression also consists entirely of monomials in the cluster.

The plan of the paper is as follows. In \Cref{sec:Background}, we review background on graph LP algebras and snake graphs for Type $A$ cluster algebras. In \Cref{sec:construction} we describe the construction of the hypergraph $\cal G_S$ and discuss some of its basic properties. \Cref{sec:main_formula} is devoted to our main result, which is the expansion formula for $Y_S$ using admissible matchings of $\cal G_S$. Unlike the classical case, our combinatorial formula does not imply positivity for all cluster variables $Y_S$, so we give a classification in \Cref{sec:positivity}. In Section \ref{sec:TPath}, we show how our expansion formula vian admissible matchings of graphs can be rephrased in terms of hyper $T$-paths, as in our previous work \cite{rooted}. The next two sections are devoted to the proof of our expansion formula (Theorem \ref{thm:main}). We prove our formula for singleton sets in Section \ref{sec:SingeltonProof} and then prove the formula for general weakly rooted in Section \ref{sec:ProofGeneral}. Finally, in Section \ref{sec:Counting} we provide a determinantal formula (Theorem \ref{thm:Counting}) which gives the number of admissible matchings of the graph $\mathcal{G}_S$ for any weakly rooted set $S$.

\section{Background}\label{sec:Background}

\subsection{Graph LP Algebras}

A \emph{cluster algebra} is defined by a \emph{seed} consisting of a  collection of \emph{cluster variables} $\{x_1,\ldots,x_n\}$, called a \emph{cluster}, along with exchange polynomials $F_i(x_1,\ldots,x_n)$ associated to each cluster variable that satisfy certain properties. In particular, each $F_i$ is a binomial. When we \emph{mutate} a seed at some cluster variable, that cluster variable changes according the the exchange polynomials and all of the exchange polynomials also mutate.  This forms a new seed for the cluster algebra. See \cite{FZ1} for more details. An LP algebra has this same structure, but the conditions on the exchange polynomials $F_i$ are weakened. For example, the $F_i$ no longer are required to be binomial. We direct an interested reader to \cite{LP-12} for more details about general LP algebras.

In this paper, we will focus on graph LP algebras. Here, we review their definition and also introduce some notation which will be useful throughout the article. The original definition of graph LP algebras in \cite{lam2016linear} was in the more general setting of directed graphs, but we work in the context of undirected graphs.

\begin{defn}\label{def:GraphLPAlgebra}
Let $\Gamma = (V,E)$ be an undirected graph on $[n] := \{ 1, \dots, n \}$ and let $R = \mathbb{Z}[A_1, \dots, A_n]$ for algebraically independent variables $A_1,\ldots,A_n$. The \emph{graph LP algebra} $\mathcal{A}_{\Gamma}$ is the LP algebra generated by the initial seed with cluster $\{ X_1,\ldots,X_n \}$ and exchange polynomial $F_i = A_i + \sum_{(i,j) \in E} X_j$ associated to $X_i$ for each $1 \leq i \leq n$.
\end{defn}

For each $S\subseteq V(\Gamma)$, we can define a variable $Y_S$ as an expression in the variables $X_1,\dots,X_n$ (see Section 1.3 of~\cite{lam2016linear}). In particular, $Y_\emptyset=1$ and if $S$  has connected components $S_1,\dots,S_k$ then $Y_S=Y_{S_1}\dots Y_{S_k}$. The seeds of a graph LP algebra can then be described in terms of certain sets of vertex subsets called \emph{nested collections}.

\begin{defn}\label{def:nested}
Let $\Gamma$ be an undirected graph on $[n]$. A family of connected subsets $\mathcal{I} = \{ S_1, \dots, S_k \}$ of $[n]$ is called a \emph{nested collection} if
\begin{enumerate}
    \item for any $i,j \leq k$, either $S_i \subseteq S_j$, $S_j \subseteq S_i$, or $S_i \cap S_j = \emptyset$, and
    \item if $S_{i_1}, \dots, S_{i_\ell}$ are pairwise disjoint, then the connected components of $\cup_{j=1}^{\ell} S_{i_j}$ are exactly $S_{i_1}, \dots, S_{i_{\ell}}$.
\end{enumerate} 

We say that $\mathcal{I}$ is a \emph{maximal nested collection on $S \subseteq [n]$} if $\cup_{i=1}^k S_i = S$ and there is no $S' \subseteq S$ such that $\{ S_1, \dots, S_k, S' \}$ is a nested collection.  We will call maximal nested collections on $[n]$ simply \emph{maximal nested collections}.  See Figures~\ref{fig:nested} and~\ref{fig:cluster} for both examples and non-examples. 
\end{defn}

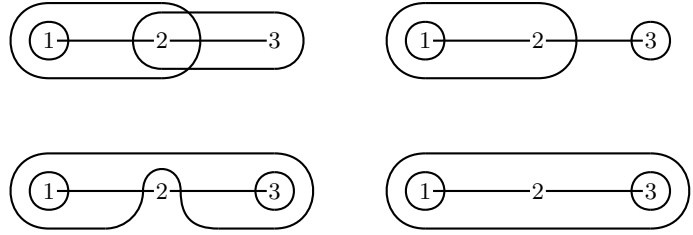
\begin{figure}
\centering
\input{figures/ex1}
\caption{The upper left collection is not a nested collection because the set $\{1,2\}$ and the set $\{2,3\}$ have nonempty intersection, yet neither is contained in the other. The upper right collection is not a nested collection because the set $\{1,2\}$ and the set $\{3\}$ are disjoint but $\{1,2\}\cup\{3\}=\{1,2,3\}$ which has only one connected component.  The lower left collection is not a nested collection because the set $\{1,3\}$ is not connected.  The lower right collection is a nested collection.}
\label{fig:nested}
\end{figure}

\begin{figure}
    \centering
    \input{figures/graph_fig}
    \caption{A graph $\Gamma$ with a maximal nested collection $\mathcal I$.  This graph and nested collection will be used as a running example throughout the paper. }
    \label{fig:cluster}
\end{figure}
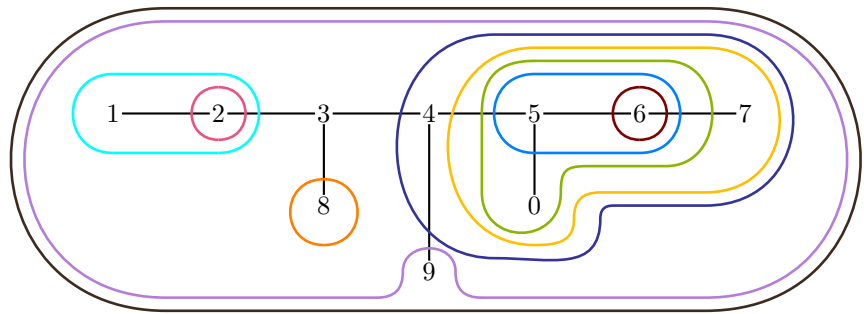

Lam and Pylyavskyy prove that the clusters of a graph LP algebra have the following structure. 

\begin{thm}[Theorem 1.1 \cite{lam2016linear}]\label{thm:ClustersInGraphLPAlgebra}
Let $\Gamma$ be an undirected graph on $[n]$, and let $\mathcal{A}_{\Gamma}$ be the associated graph LP algebra. Then, the clusters of $\mathcal{A}_\Gamma$ are of the form $\{X_{i_1}, \dots, X_{i_k} \} \cup \{ Y_S : S \in \mathcal{I} \}$ where $\mathcal{I}$ is a maximal nested collection on $[n] \backslash \{ i_1, \dots, i_k \}$.
\end{thm}

If two subsets, $I,J$, of $ V(\Gamma)$ are such that $\{I,J\}$ is a nested collection, we say they are \emph{compatible}. Otherwise, they are \emph{incompatible}.  In other words, sets $I$ and $J$ are compatible if and only if $Y_I$ and $Y_J$ appear together in a cluster of $\mathcal{A}_\Gamma$.

We will focus our work on clusters of the form $\{ Y_S : S \in \mathcal{I} \}$ where $\mathcal{I}$ is a maximal nested collection.  Therefore, our work will be focused on understanding variables of the form $Y_S$.
To simplify notation, we  often write $Y_{s_1 \cdots s_r}$ as shorthand for $Y_{\{ s_1, \dots, s_r \}}$ and $Ij$ for $I\cup\{j\}$ with $I\subseteq V(\Gamma)$, $j\in V(\Gamma)$ and $j\not\in I$.  We also refer to vertex subsets of size one as \emph{singletons}.

Following \cite{lam2016linear}, we define a partial order $ <_{\mathcal{I}}$ on $V(\Gamma)$ where $i <_{\mathcal{I}} j$ whenever $i \in I_j$. As an example, the maximal nested collection $\mathcal{I}$ in Figure \ref{fig:cluster}, we have $8 <_{\mathcal{I}} 3$.
Every maximal nested collection $\mathcal{I}$ on $\Gamma$ induces a bijection between the sets in $\mathcal{I}$ and $V(\Gamma)$ given by sending $v \in V(\Gamma)$ to $I_v$, the smallest set in $\mathcal{I}$ containing $v$. Going the other direction, let $m_{\mathcal{I}}(S)$ denote the largest element of a connected set $S$ with respect to $\mathcal{I}$. That is, $m_{\mathcal{I}}(I_v) = v$.   For example, $I_3 = \{2,3,4,5,7,8\}$, and $m_{\mathcal{I}}(\{2,5\}) = 2$ for the maximal nested collection $\mathcal{I}$ in Figure \ref{fig:cluster}. The bijection these induce between $V(\Gamma)$ and a maximal nested collection $\mathcal{I}$ is also described in Proposition 7.6 in \cite{postnikov2009permutohedra} and Lemma 2.4 in \cite{barnard2021lattices}.

For the remainder of this article, we will restrict our attention to the case where $\Gamma$ is a tree. Given a tree $\Gamma$, we build a larger tree $\Gamma'$ as follows. For each leaf $\ell \in \Gamma$, we add a vertex $\ell'$ and an edge $(\ell',\ell)$. Thus, every vertex from the original graph $\Gamma$ has degree at least two and we retain the same number of leaves. We will use $\deg_\Gamma$ to denote degree in $\Gamma$, and $\deg_{\Gamma'}$ to refer to degree in $\Gamma'$.  We extend the partial order from $\mathcal{I}$ on $V(\Gamma)$ to one on $V(\Gamma')$ by putting $m_{\mathcal I}(V(\Gamma)) \lessdot_{\mathcal{I}} v'$ for all leaves $v$ of $\Gamma$.

\begin{ex}
Consider the graph $\Gamma$ with maximal nested collection shown in \Cref{fig:cluster}.  The graph $\Gamma'$ is as follows.
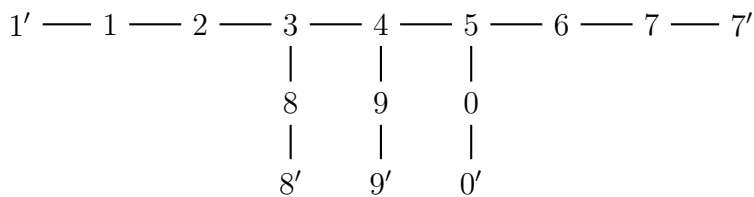
\begin{center}
    \input{figures/gammaprime}
\end{center}

The partial order on vertices of $\Gamma$, is given by the Hasse diagram below on the left.  The extension of this partial order to the vertices of $\Gamma'$ is given by the Hasse diagram on the right.
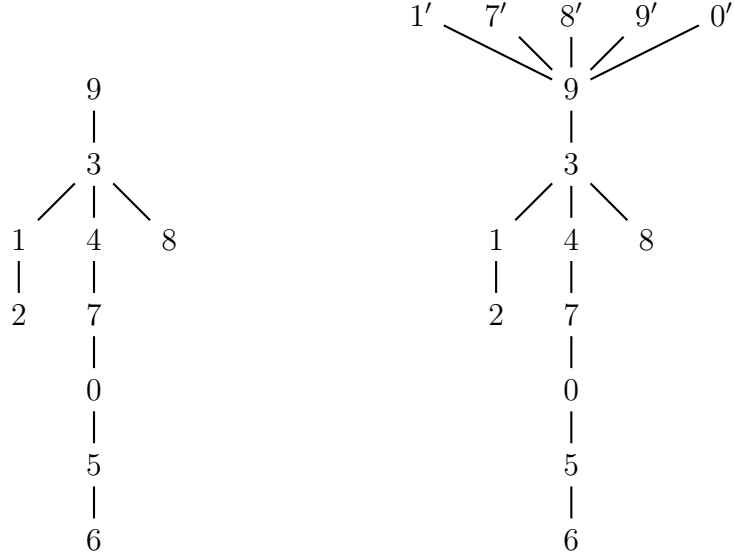
\begin{center}
\begin{tikzpicture}
\node (1) at (-1,4) {$1$};
\node (2) at (-1,3) {$2$};
\node (3) at (0,5) {$3$};
\node (4) at (0,4) {$4$};
\node (5) at (0,1) {$5$};
\node (6) at (0,0) {$6$};
\node (7) at (0,3) {$7$};
\node (8) at (1,4) {$8$};
\node (9) at (0,6) {$9$};
\node (0) at (0,2) {$0$};
\draw (9) to (3);
\draw (3) to (1);
\draw (3) to (8);
\draw (3) to (4);
\draw (1) to (2);
\draw (4) to (7);
\draw (7) to (0);
\draw (0) to (5);
\draw (5) to (6);
\end{tikzpicture}
\hspace{1in}
\begin{tikzpicture}
\node (1) at (-1,4) {$1$};
\node (2) at (-1,3) {$2$};
\node (3) at (0,5) {$3$};
\node (4) at (0,4) {$4$};
\node (5) at (0,1) {$5$};
\node (6) at (0,0) {$6$};
\node (7) at (0,3) {$7$};
\node (8) at (1,4) {$8$};
\node (9) at (0,6) {$9$};
\node (0) at (0,2) {$0$};
\node (1') at (-2,7) {$1'$};
\node (8') at (0,7) {$8'$};
\node (9') at (1,7) {$9'$};
\node (0') at (2,7) {$0'$};
\node (7') at (-1,7) {$7'$};
\draw (9) to (3);
\draw (3) to (1);
\draw (3) to (8);
\draw (3) to (4);
\draw (1) to (2);
\draw (4) to (7);
\draw (7) to (0);
\draw (0) to (5);
\draw (5) to (6);
\draw (9) to (1');
\draw (9) to (8');
\draw (9) to (9');
\draw (9) to (0');
\draw (9) to (7');
\end{tikzpicture}
\end{center}
\end{ex}

We will use the following notation throughout the rest of the paper.
\begin{itemize}
\item Define $\Gamma^{\mathcal{I}}_{>v}$ as the subset of vertices of $V(\Gamma)$ that are larger than $v$ and define $\Gamma^{\mathcal{I}}_{< v}$ as the subset of vertices of $V(\Gamma)$ that are smaller than $v$. We also write $(\Gamma')^{\mathcal{I}}_{*}$ to denote these concepts for the partial order extended to $\Gamma'$.
\item Define $\mathcal{C}_v$ as the subset of vertices of $V(\Gamma)$ that are covered by $v$.
\item Define $N_\Gamma(v) := \{w \in \Gamma: (v,w) \in E(\Gamma)\}$ and define $N_{\Gamma'}(v)$ similarly on $\Gamma'$. If $S \subseteq V(\Gamma)$, we define $N_\Gamma(S) := (\cup_{v \in S} N_{\Gamma}(v)) \backslash S$.
\item Given $x,y \in V(\Gamma')$, let $[x,y]$ be the set of vertices on the unique path between $x$ and $y$, including $x$ and $y$. This path exists and is unique since $\Gamma'$ is a tree. 
\end{itemize}

\begin{ex}
Continuing with the graph and nested collection from \Cref{fig:cluster}, we have:
\begin{itemize}
\item $\Gamma^{\mathcal{I}}_{>3}=\{9\}$, $(\Gamma')^{\mathcal{I}}_{>3}=\{9,1',7',8',9',0'\}$ and $\Gamma^{\mathcal{I}}_{<3}=(\Gamma')^{\mathcal{I}}_{<3}=\{1,4,8,2,7,0,5,6\}$
\item $\mathcal{C}_3=\{1,4,8\}$
\item $N_{\Gamma}(3)=\{2,4,8\}$ and $N_{\Gamma}(\{3,4\})=\{2,5,8,9\}$
\item $[3,9]=\{3,4,9\}$ and $[3,9']=\{3,4,9,9'\}$
\end{itemize}
\end{ex}

\subsection{Snake Graphs for Polygon Cluster Algebras}\label{subsec:snake-graphs}

We now review the construction of snake graphs for surface cluster algebras from a marked disk (i.e. a polygon), i.e. type-$A$ cluster algebras.  In this case, each arc of the polygon corresponds to a cluster variable, and each triangulation corresponds to a cluster.

To each arc in a triangulated polygon, we will associate a particular snake graph whose edges are weighted by certain cluster variables. 

\begin{defn} \label{def:snake}
Let $T$ be a triangulation of an inscribed polygon and $\gamma$ an arc between vertices $a$ and $b$ of the polygon such that $a$ and $b$ are not neighboring vertices and $\gamma$ is not already in $T$. Orient $\gamma$ to travel from $a$ to $b$.
Let $\tau_1,\tau_2,\cdots,\tau_k$ denote, in order, the set of arcs in $T$ crossed by $\gamma$. For each $\tau_i$, we draw a square \emph{tile} $S_{\tau_i}$, which we refer to as the tile labeled by $\tau_i$, that corresponds to gluing the two triangles in $T$ which are bordered by $\tau_i$. We include a dashed edge labeled $\tau_i$ in the tile $S_{\tau_i}$  which we call a \emph{diagonal}.
Each tile has two possible planar embeddings. For odd $i$, we choose the embedding of $S_{\tau_i}$ with orientation that matches $T$. For even $i$, we choose the embedding of $S_{\tau_i}$ with orientation opposite to $T$. We draw these tiles so the diagonals are always between northwest and southeast corners of the tiles.  For each pair of adjacent arcs $\tau_i$ and $\tau_{i+1}$, the corresponding tiles $S_{\tau_i},S_{\tau_{i+1}}$ share a common edge. 
We \emph{glue} $S_{\tau_i}$ and $S_{\tau_{i+1}}$ by identifying this common edge. To produce the \emph{snake graph} $\mathcal{G}_{\gamma}$ of $\gamma$, we repeat this gluing operation for all consecutive pairs of arcs crossed by $\gamma$.
\end{defn}

\begin {ex}
Figure \ref{fig:snake} shows an example of the snake graph associated to an arc in an octagon. The vertices of the snake graph are labeled by the corresponding vertices in the polygon. Although edge labels are not shown in the figure, assume each edge is weighted with a cluster variable $x_{ij}$.
\begin{figure}[ht]
    \centering
   \begin {center}
    \begin {tikzpicture}[scale=1.6]
        \coordinate (v7) at (1,0);
        \coordinate (v8) at ({cos(1*2*pi/8 r)}, {sin(1*2*pi/8 r)});
        \coordinate (v1) at ({cos(2*2*pi/8 r)}, {sin(2*2*pi/8 r)});
        \coordinate (v2) at ({cos(3*2*pi/8 r)}, {sin(3*2*pi/8 r)});
        \coordinate (v3) at ({cos(4*2*pi/8 r)}, {sin(4*2*pi/8 r)});
        \coordinate (v4) at ({cos(5*2*pi/8 r)}, {sin(5*2*pi/8 r)});
        \coordinate (v5) at ({cos(6*2*pi/8 r)}, {sin(6*2*pi/8 r)});
        \coordinate (v6) at ({cos(7*2*pi/8 r)}, {sin(7*2*pi/8 r)});
        
        \node  at (v1) [above]  {$1$};
        \node  at (v2) [above left]  {$2$};
        \node  at (v3) [left]  {$3$};
        \node  at (v4) [below left]  {$4$};
        \node  at (v5) [below]  {$5$};
        \node  at (v6) [below right]  {$6$};
        \node  at (v7) [right]  {$7$};
        \node  at (v8) [above right]  {$8$};

        \draw (v1) -- (v7);       
        \draw (v2) -- (v7);
        \draw (v7) -- (v3);
        \draw (v7) -- (v4);
        \draw (v4) -- (v6);
        \draw[dashed] (0,-1) -- (v8);

        \draw (v1) -- (v2) -- (v3) -- (v4) -- (v5) -- (v6) -- (v7) -- (v8) -- cycle;

        \begin {scope}[shift={(2.5,-0.5)}, scale=0.6]
                \draw [thick] (0,0) node [below left] {$5$} -- (3,0) node [below right] {$7$} -- (3,1) node [below right] {$2$} -- (4,1) node [below right] {$7$} -- (4,2) node [above right] {$8$} -- (2,2) node [above left] {$7 $} -- (2,1) node [above left] {$3$} -- (0,1) node [above left] {$4$} -- cycle;
                \draw [thick](1,0) -- (1,1);
                \draw [thick](2,0) -- (2,1) -- (3,1);
                \draw [thick](3,1) -- (3,2);

                \draw (1,0) node[below] {6};
                \draw (1,1) node[above] {7};
                \draw (2,0) node[below] {4};
                \draw (3,2) node[above] {1};
\foreach \i in {0,1,2}{
                \draw [dashed] (\i+1,0) -- (\i,1);}
\foreach \i in {2,3}{
                \draw [dashed] (\i+1,1) -- (\i,2);}

        \end {scope}

    \end {tikzpicture}
    \end {center}

    \caption{The snake graph $\mathcal{G}_\gamma$ for $\gamma=(5,8)$.
    }
    \label{fig:snake}
\end{figure}
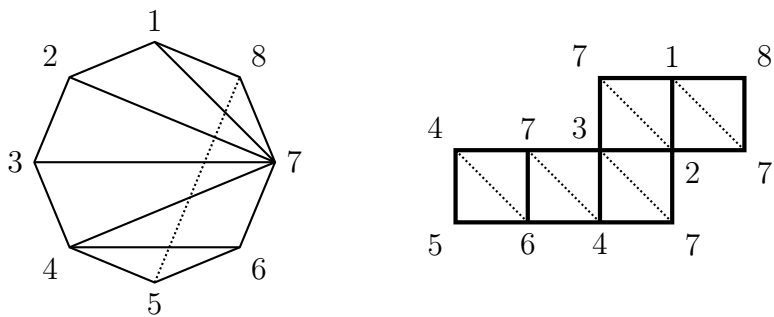
\end {ex}

    A \emph{perfect matching} of $\mathcal{G}$, also called a \emph{dimer cover}, is a subset $M$ of the edges of $\mathcal{G}$, such that each vertex is incident to exactly one edge of $M$.  For a perfect matching $M$, $\mathrm{wt}(M)$ is defined as the product of the weights of the edges in $M$. Note that we do not consider the diagonals in our snake graph to be edges of the graph, so they cannot be used in matchings. 

    \begin {thm}[\cite{ms10} Theorem 3.1] \label{thm:dimer}
    Let $T$ be a triangulation of a polygon, and let $\gamma$ be an arc not in $T$. Then the cluster variable expansion $x_\gamma$ is given by
    \[ x_\gamma = \frac{1}{x_{i_1}x_{i_2} \cdots x_{i_k}} \sum_{M \in \mathcal{P}_\gamma} \mathrm{wt}(M) \]
    where $\mathcal{P}_\gamma$ is the set of perfect matchings of $\mathcal{G}_\gamma$ and $\tau_{i_1},\ldots,\tau_{i_k}$ are the arcs in T which label diagonals in $\mathcal{G}_{\gamma}$.
\end {thm}

The following definition from snake graphs for polygon cluster algebras will be useful later.  We say a snake graph $\mathcal{G}_\gamma$ has a \emph{zig-zag} shape if there are no consecutive tiles $S_{\tau_{i-1}},S_{\tau_i},S_{\tau_{i+1}}$ such that $S_{\tau_{i-1}}$ and $S_{\tau_{i+1}}$ share parallel edges of $S_{\tau_i}$. For example, the snake graph in Figure \ref{fig:snake} does not have a zig-zag shape, but it would if we removed the leftmost tile. 

\section{Construction}\label{sec:construction}

This section outlines the construction of snake graphs for a class of sets we call \emph{weakly rooted} with respect to a maximal nested collection $\mathcal I$ on a tree $\Gamma$.  The motivation of our construction comes from the simplest type of graph LP algebra: the case where $\Gamma$ is a path.  In this case, the subalgebra of graph LP algebra spanned by the variables $\{Y_S\}_{S\subseteq V(\Gamma)}$ is in fact a type A cluster algebra.  We can recover the snake graph construction discussed in Section~\ref{subsec:snake-graphs} for a graph LP algebra from a path graph $\Gamma$ on $[n]$ by:
\begin{enumerate}
\item Constructing $\Gamma'$
\item Creating a polygon from $\Gamma'$ by adding an edge weighted $Y_{[n]}$ between the two leaves
\item Adding an internal arc weighted $Y_I$  in the polygon between the two elements of $N_{\Gamma'}(I)$ for each $I \in \mathcal{I}$
\item Proceeding to make snake graphs as usual from this triangulation, where the snake graph corresponding to the arc between vertices $x$ and $y$ in the polygon corresponds to the set $[x,y]\setminus\{x,y\}$ in $\Gamma'$.  (Note that even though we construct this set in $\Gamma'$ it will only contain elements that are vertices of $\Gamma$.)
\end{enumerate}

It follows from the correspondence between the subalgebra of the graph LP algebra and a type $A$ cluster algebra that the above process will give a snake graph which encodes the expansion of $Y_{[x,y]\setminus\{x,y\}}$.

For a graph LP algebra from a tree $\Gamma$, we can draw inspiration from the above method.  If we delete everything in $\Gamma$ except for the path between two leaves, we can make a snake graph for a set of vertices on this path in a similar way.  This is less straightforward than in the previous setting because there may be sets in $\mathcal I$ that differ only on the part of $\Gamma$ that has been deleted, leading to multiple internal arcs in our triangulation between the same two vertices.  We are able to handle this in the case where we are making the snake graph for a singleton set $\{v\}$ using the following heuristic: if $v$ is not a neighbor of any such set, we only consider the arc corresponding to the smallest set, and otherwise we make copies of the vertex shared by the two arcs to allow each internal arc to be between different pairs of vertices.  Once we have made all the snake graphs for a singleton set arising from different paths in $\Gamma$, we can identify edges and vertices that the different graphs have in common to get a single snake graph for the set.  For a larger weakly rooted set $S$, we get a snake graph for $Y_S$ by combining snake graphs for the individual vertices and then possibly deleting a portion of the graph.

For the rest of this paper, let $\Gamma$ be a tree with maximal nested collection $\mathcal I$ and $\Gamma'$ be the extended graph described Section \ref{sec:Background}. We will use the same symbols to denote a vertex in $\Gamma'$ and in a snake graph as it will be clear which graph contains the vertex. 

\subsection{Component Snake Graphs for Singleton Sets}
\label{subsec:SingletonConstruction1}

In \Cref{subsec:SingletonConstruction2}, we will describe how to construct a snake graph for a singleton set $\{v\}$ for some $v\in V(\Gamma)$.  However, in the case where $\{v\}\not\in \mathcal I$ we first need some building blocks that we call \emph{component snake graphs}.  In this section we describe the construction of the component snake graphs.  We will have two running examples, both using the graph and maximal nested collection from \Cref{fig:cluster}.  Our first running example will be when $v=4$ and the second will be when $v=3$.

Let $N_{\Gamma'}(v) = \{a_1,\ldots,a_p\}$. Note that $p = \deg_{\Gamma'}(v) \geq 2$ because every vertex from $\Gamma$ has degree at least 2 in $\Gamma'$. Assume that the elements of $N_{\Gamma'}(v)$ are labeled such that $\{a_1,\ldots,a_r\} = N_{\Gamma'}(v) \cap (\Gamma')^{\mathcal{I}}_{<v}$ (and necessarily, $\{a_{r+1},\ldots,a_p\} =  N_{\Gamma'}(v) \cap (\Gamma')^{\mathcal{I}}_{>v}$). Since $\{v\}\notin\mathcal I$, it must be the case that $N_{\Gamma'}(v) \cap (\Gamma')^{\mathcal{I}}_{<v}$ is nonempty, and therefore $r\geq 1$.  For $1\leq i\leq p$, let $c_i = \vert \{I \in \mathcal{I} : a_i\in I, v\not\in I\} \vert$. Index the sets counted by $c_i$ as $I^i_j$ such that $I_{a_i} = I^i_{c_i} \subseteq I^i_{c_i-1} \subseteq \cdots \subseteq I^i_1$. Note that $c_i = 0$ for all $r+1 \leq i \leq p$. For each $1 \leq i \leq r, 1 \leq j \leq c_i$, let $a_j^i = m_{\mathcal{I}}(I_j^i)$. When $v$ is ambiguous, we will write $a_i(v),\ c_i(v),\ I_j^i(v)$, and $a_j^i(v)$.

\begin{ex}
If $v = 4$, we have $p=3$ and $r=1$.  Since 5 is the only neighbor of 4 less than 4, $a_1$ must be 5.  We can have $a_2$ and $a_3$ as 3 and 9 in either order; let's choose $a_2=3$, and $a_3=9$.
Since there are 3 sets that contain 5 but not 4, $c_1=3$.  These sets are $\{5,6\}\subseteq \{5,6,0\}\subseteq \{5,6,7,0\}$, so we have $I^1_3=\{5,6\}$, $I^1_2=\{5,6,0\}$, and $I^1_1=\{5,6,7,0\}$.  We also have $a^1_3= m_{\mathcal{I}}(\{5,6\})=5$, $a^1_2=m_{\mathcal{I}}(\{5,6,0\})=0$, and $a^1_1 = m_{\mathcal{I}}(\{5,6,7,0\})=7$.
Since there are no sets that contain 3 but not 4 or contain 9 but not 4, $c_2=c_3=0$.
\end{ex}

\begin{ex}
If $v=3$, we have $p=3$ and $r=3$.  Since 8, 2, and 4 are all neighbors of 3 that are less than 3, we can choose which to denote as $a_1, a_2$, and $a_3$.  Let's choose $a_1=8$, $a_2=2$, and $a_3=4$.
Since there is 1 set that contains 8 but not 3, $c_1=1$.  This set is $\{8\}=I_8$, so we have $I^1_1=\{8\}$ and $a^1_1=8$.  Since there are 2 sets that contain 2 but not 3, $c_1=2$.  These sets are $\{2\}\subseteq \{1,2\}$, so we have $I^2_2=\{2\}$ and $I^2_1=\{1,2\}$.  We also get that $a^2_2=2$ as $\{2\}=I_2$ and $a^2_1=1$ as $\{1,2\}=I_1$.
Since there is 1 set that contain 4 but not 3, $c_3=1$.  This set is $\{4,5,6,7,0\}=I_4$, so we have $I^3_1=\{4,5,6,7,0\}$ and $a^3_1=4$.
\end{ex}

Let $A = \{\ell \in \Gamma' : \deg_{\Gamma'}(\ell) = 1 \text{ and } a_1 \in [v,\ell]\}$ be the set of leaves of $\Gamma'$ such that the path between each leaf in $A$ and $v$ contains $a_1$. Let $B = \{\ell \in \Gamma' : \deg_{\Gamma'}(\ell) = 1 \text{ and } a_1 \notin [v,\ell]\}$ be the complement of $A$ in the set of leaves of $\Gamma'$.  
Fix $\ell \in A, k \in B$, and let $a_i$ be the unique vertex in $N_{\Gamma'}(v) \cap [v,k]$.
Create a $(c_1+c_i+3)$-gon $Q_{\ell,k}$. 
For $1 \leq j \leq c_1$, let $b_{1,j}^{\ell,k}$ denote the vertex other than $v$ in $N_{\Gamma'}(I_j^1) \cap [\ell,k]$, and for $1 \leq j \leq c_i$, let $b_{i,j}^{\ell,k}$ denote the vertex other than $v$ in  $N_{\Gamma'}(I_j^i) \cap [\ell,k]$. 
Label the vertices of $Q_{\ell,k}$, in cyclic order, as $b_{i,1}^{\ell,k},\cdots,b_{i,c_i-1}^{\ell,k},b_{i,c_i}^{\ell,k},a_i,v, a_1, b_{1,c_1}, b_{1,c_1-1}^{\ell,k},\cdots,b_{1,1}^{\ell,k}$.

\begin{ex}
For $v=4$, we have $A=\{7',0'\}$ and $B=\{1',8',9'\}$.  The following are the polygons $Q_{\ell,k}$ for each pair $(\ell,k)\in A\times B$.
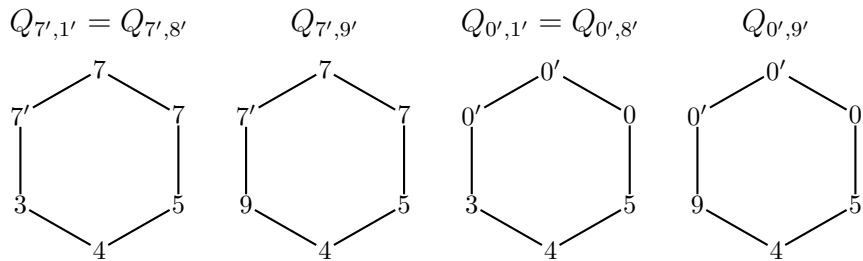
\begin{center}
\input{figures/G4Q}
\end{center}
\end{ex}

\begin{ex}
For $v=3$ with our choice of $a_1=8$, we have $A=\{8'\}$ and $B=\{1',7',9',0'\}$.  The following are the polygons $Q_{\ell,k}$ for each pair $(\ell,k)\in A\times B$.
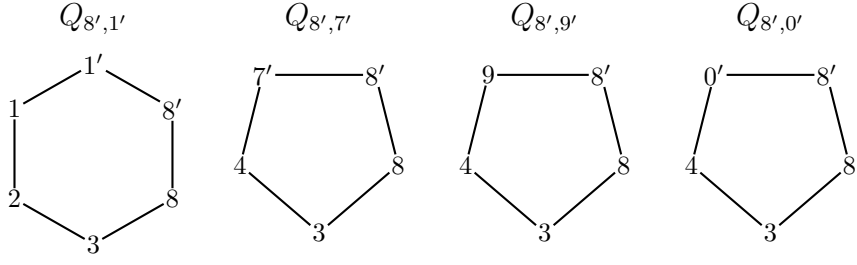
\begin{center}
\input{figures/G3Q}
\end{center}
\end{ex}

We place a \emph{weighted} triangulation $T_{\ell,k}$ on $Q_{\ell,k}$; that is, every arc in $T$ is weighted with an element of $\mathcal{A}_\Gamma$ . For the purposes of this construction, we consider $T_{\ell,k}$ to include both a maximal set of pairwise non-crossing internal arcs and the complete set of boundary arcs. The internal arcs of $T_{\ell,k}$ are all the arcs of the form  $(v,b^{\ell,k}_{1,j})$ or $(v,b^{\ell,k}_{i,j})$. The arc $(v,b^{\ell,k}_{1,j})$ has weight $Y_{I_j^1}$ and the arc $(v,b^{\ell,k}_{i,j})$ has weight $Y_{I_{j}^i}$.

We also need to describe the weights of the boundary arcs.  Generally, if $(x,y) \in E(\Gamma')$ then we will weight this edge with 1, if $(x,y) \not\in E(\Gamma')$ and $x\neq y$ we will weight the edge with $Y_I$ for some $I\in\mathcal I$ where $x,y \in N_{\Gamma'}(I)$, and if $(x,y) \not\in E(\Gamma')$ and $x=y$ we will weight the edge with $Y_I^2$ for some $I$ that may or may not be in $\mathcal I$ with $x,y \in N_{\Gamma'}(I)$.  More explicitly, we have the following cases:
\begin{enumerate}
\item We weight $(a_i,v)$ and $(a_1,v)$ with 1.
\item If $(a_i,b_{i,c_i}^{\ell,k}) \in E(\Gamma')$, we again weight this edge with 1. Otherwise, $b_{i,c_i}^{\ell,k} \in N_\Gamma(I_z)$ for some unique $z \in \mathcal{C}_{a_i}$, and we weight this edge with $Y_{I_z}$.  We use the same process to weight $(a_1,b_{1,c_1}^{\ell,k})$.  Note that the edge $(a_i,b_{i,c_i}^{\ell,k})$ is a boundary edge only if $i\leq r$.
\item If $i \leq r$, we weight the edge $(b_{1,1}^{\ell,k}, b_{i,1}^{\ell,k})$ with $I_v$, and if $i > r$, we weight the edge $(b_{1,1}^{\ell,k}, a_i)$ with $I_v$. 
\item There are three possibilities for a boundary arc of the form $(b_{i,j}^{\ell,k}, b_{i,j+1}^{\ell,k})$ with $j < c_i$: $b_{i,j}^{\ell,k} \in N_{\Gamma'}(a_j^i)$,  $b_{i,j}^{\ell,k} \in N_{\Gamma'}(I_z)$ for $z \in \mathcal{C}_{a_j^i} \backslash \{a_{j+1}^i\}$, or $b_{i,j}^{\ell,k} \in N_{\Gamma'}(I_{j+1}^i)$.
\begin{enumerate}
\item If $b_{i,j}^{\ell,k} \in N_{\Gamma'}(a_j^i)$, we know that $b_{i,j+1}^{\ell,k} = a_j^i$ since $a_j^i \in N_{\Gamma'}(I_{j+1})$. Since $(b_{i,j}^{\ell,k}, b_{i,j+1}^{\ell,k}) \in E(\Gamma')$, we weight this edge with 1.
\item If $b_{i,j}^{\ell,k} \in N_{\Gamma'}(I_z)$ for $z \in \mathcal{C}_{a_j^i} \backslash \{a_{j+1}^i\}$, we also know that $b_{i,j+1}^{\ell,k} = a_j^i$, and we weight the edge $(b_{i,j}^{\ell,k}, b_{i,j+1}^{\ell,k})$ with $Y_{I_z}$.
\item If $b_{i,j}^{\ell,k} \in N_{\Gamma'}(I_{j+1}^i)$, then $b_{i,j}^{\ell,k} = b_{i,j+1}^{\ell,k}$. If $b_{i,j}^{\ell,k}$ is adjacent to a vertex on the path $[v,a_j^i]$, we weight this edge with 1. Otherwise, we weight this edge with $Y^2_{C_{i,j}^{\ell,k}}$ where  $C_{i,j}^{\ell,k}$ is the connected component of $I_{j+1}^i \backslash[v,a_j^i]$ which contains a vertex adjacent to $b_{i,j}^{\ell,k}$.
\end{enumerate}
The process detailed in step (4) is identical, up to changing indices, for weighting an edge $(b_{1,j}^{\ell,k},b_{1,j+1}^{\ell,k})$. Note that the sets $C_{i,j}^{\ell,k},C_{1,j}^{\ell,k}$ are not necessarily in $\mathcal{I}$.
\end{enumerate}

\newpage
\begin{ex}
For $v=4$, the following are the triangulations $T_{\ell,k}$ for each pair $(\ell,k)\in A\times B$.
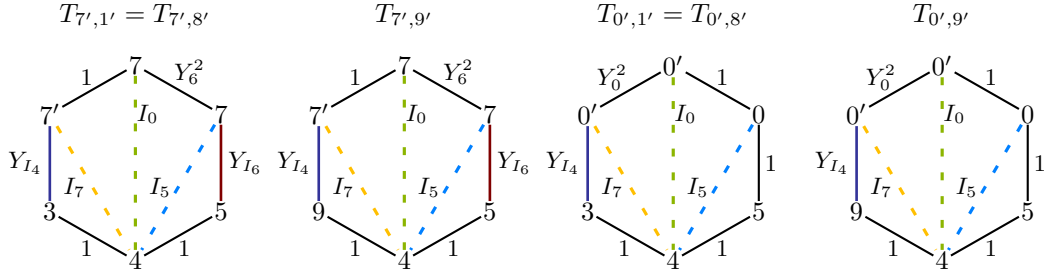
\begin{center}
\input{figures/G4T}
\end{center}
\end{ex}

\begin{ex}
For $v=3$, the following are the triangulations $T_{\ell,k}$ for each pair $(\ell,k)\in A\times B$.

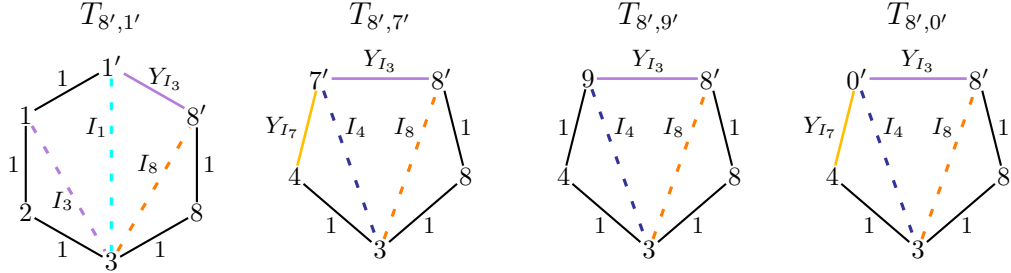
\begin{center}
\input{figures/G3T}
\end{center}
\end{ex}

Let $\gamma_{\ell,k}$ be the arc in $Q_{\ell,k}$ between the vertices $a_1$ and $a_i$. Denote the snake graph constructed from $\gamma_{\ell,k}$ with respect to the triangulation $T_{\ell,k}$ as $H_{\ell,k}$. We call $H_{\ell,k}$ a \emph{component snake graph}.

We weight the edges of $H_{\ell,k}$ by the weight of the corresponding arc in $T_{\ell,k}$.  The diagonals of $H_{\ell,k}$ correspond to internal arcs of $T_{\ell,k}$.  We \emph{label} the diagonal associated to the arc weighted $Y_{I^i_j}$ by the set $I^i_j$.

\begin{ex}
For $v=4$, the following are the component snake graphs $H_{\ell,k}$ for each pair $(\ell,k)\in A\times B$. From this point forward, for readability, we will suppress the label for all edges with weight 1.
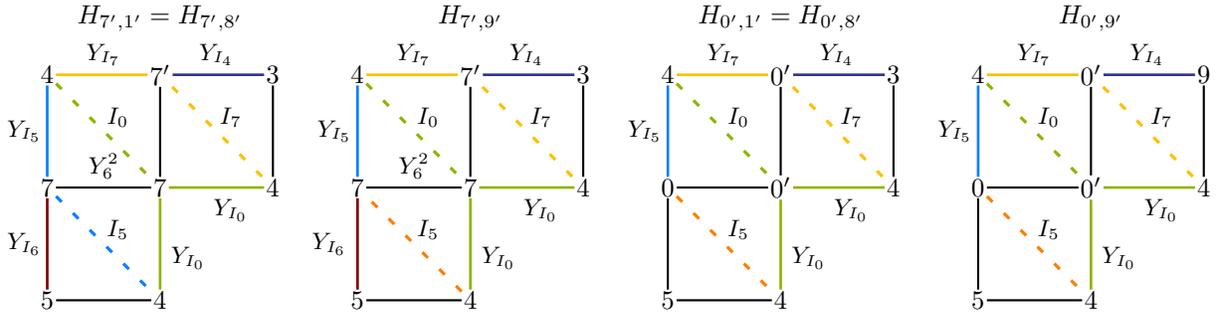
\begin{center}
\input{figures/G4H}
\end{center}
\end{ex}

\newpage
\begin{ex}
For $v=3$, the following are the component snake graphs $H_{\ell,k}$ for each pair $(\ell,k)\in A\times B$.
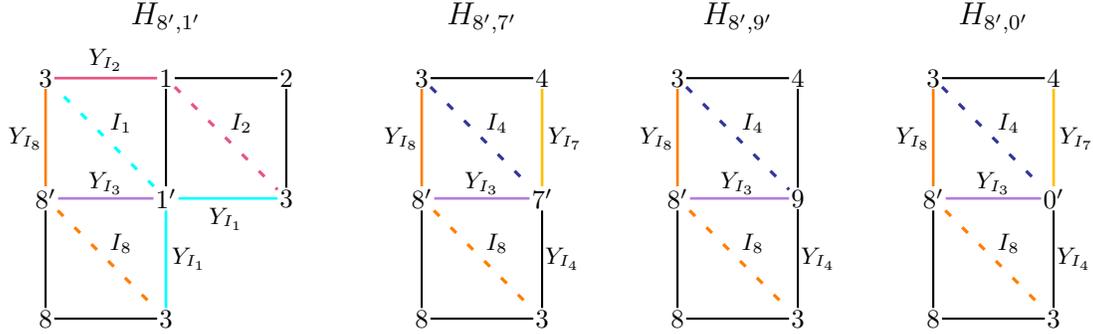
\begin{center}
\input{figures/G3H}
\end{center}
\end{ex}

\subsection{Snake Graphs for Singleton Sets}
\label{subsec:SingletonConstruction2}

We are now ready to build the snake graph $\mathcal{G}_v:= \mathcal{G}_{\{v\}}$.  If $\{v\} \in \mathcal{I}$, then the snake graph consists of a single hyperedge (an edge with possibly more than two endpoints) weighted $Y_S$ that connects all vertices in $N_{\Gamma'}(v)$.  If $\{v\} \not\in \mathcal{I}$, first build the component snake graphs $H_{\ell,k}$ for all $(\ell,k)\in A\times B$.  We will glue these together to make $\mathcal G_v$.

To describe the gluing process, we need to better understand the structure of the component snake graphs.

\begin{prop}\label{prop:ConsistentSize}
Given a pair of leaves $\ell, k$ of $\Gamma$ such that $a_1 \in [\ell,v]$ and $a_i \in [k,v]$ for $1<i\leq p$, the component snake graph $H_{\ell,k}$ satisfies the following properties.
\begin{enumerate}
    \item $H_{\ell,k}$ has a zig-zag shape.
    \item The tiles of $H_{\ell,k}$ have diagonals labeled by $I_{c_i}^i,I_{c_i-1}^i,\ldots,I_1^i, I_1^1,\ldots,I_{c_1}^1$, ordered such that consecutive sets in this list correspond to tiles that share an edge. 
\end{enumerate}
\end{prop}

\begin{proof}
It follows from our construction that $\gamma_{\ell,k}$ crosses an edge with weight $Y_{I_j^i}$ for all $1 \leq j \leq c_i$ and an edge with weight $Y_{I_j^1}$ for all $1 \leq j \leq c_1$. Moreover, $\gamma_{\ell,k}$ crosses these in the given order based on the way we order the vertices of the polygon $Q_{\ell,k}$.
This proves statement (2).
Furthermore, the fact that $v$ is an endpoint of every internal arc in $T_{\ell,k}$ ensures that $H_{\ell,k}$ has a zig-zag shape, which proves statement (1).
\end{proof}

Since each component snake graph is a surface snake graph, we can also make use of the following terminology and well-known fact.

\begin{defn}\label{def:BoundarySurfaceSnakeGraph}
An edge in a surface snake graph that borders exactly one tile is called a \emph{boundary edge}. An edge which borders two tiles is called an \emph{internal edge}.   
\end{defn}

The following comes immediately from the process of constructing a surface snake graph. 

\begin{lemma}\label{lem:LabelBoundaryArcSurfaceSnake}
Let $e$ be a boundary edge in a surface snake graph which is incident to two diagonals labeled $d_1$ and $d_2$. If $e$ lies on the tile with the diagonal labeled $d_1$, then $e$ has weight $Y_{d_2}$.
\end{lemma}

The general form of component snake graphs, as given by Proposition~\ref{prop:ConsistentSize} and Lemma~\ref{lem:LabelBoundaryArcSurfaceSnake}, is shown in Figures~\ref{fig:singleton-zigzag} and~\ref{fig:singleton-zigzag-partial}. The former shows $H_{\ell,k}$ where the neighbor of $v$ which lies on $[v,k]$ is $a_i$ for some $1<i \leq r$ (that is, $a_i<_{\mathcal{I}}v$) and the latter shows in $H_{\ell,k}$ where the neighbor of $v$ which lies on $[v,k]$ is $a_i$ for some $r<i \leq p$ (that is, $a_i>_{\mathcal{I}}v$). 

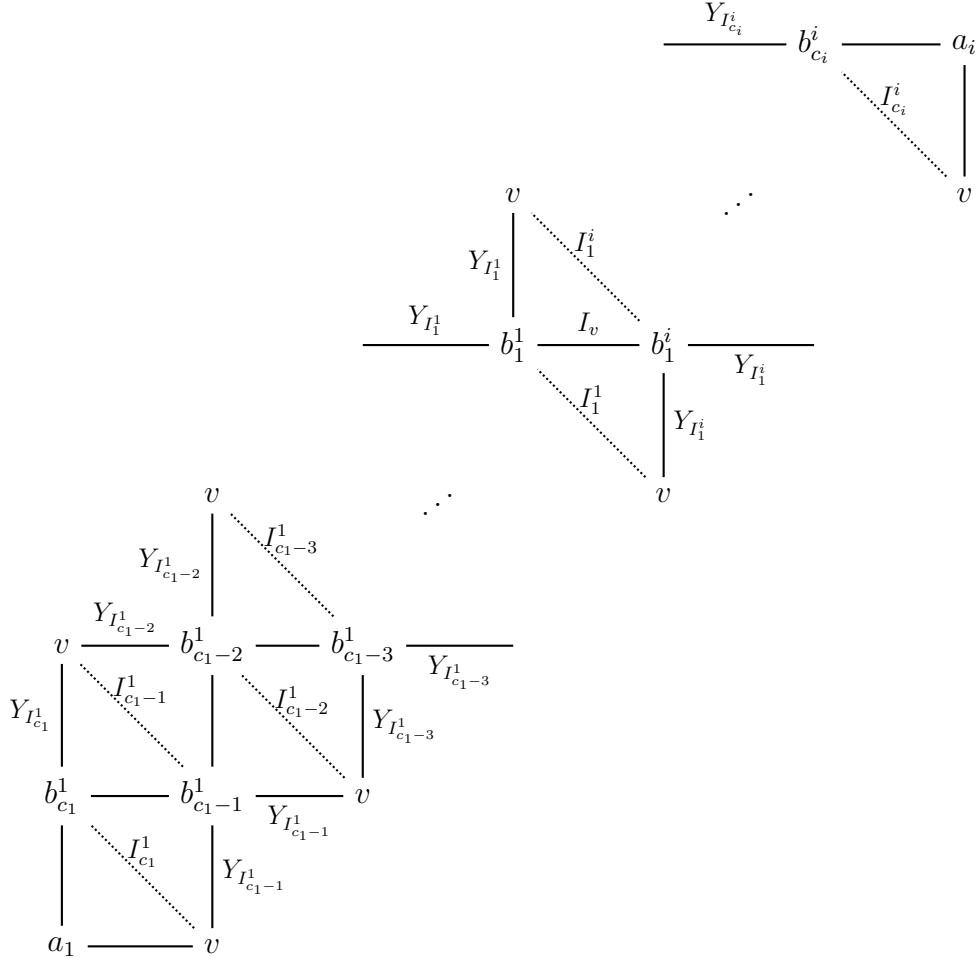
\begin{figure}[ht]
\begin{tikzpicture}[scale=2]
\node (11) at (1,0) {$v$};
\node (12) at (0,2) {$v$};
\node (13) at (2,1) {$v$};
\node (14) at (1,3) {$v$};
\node (15) at (4,3) {$v$};
\node (16) at (3,5) {$v$};
\node (17) at (6,5) {$v$};
\node (2) at (0,0) {$a_1$};
\node (a) at (0,1) {$b_{c_1}^1$};
\node (b) at (1,1) {$b_{c_1-1}^1$};
\node (c) at (1,2) {$b_{c_1-2}^1$};
\node (d) at (2,2) {$b_{c_1-3}^1$};
\node (i) at (3,4) {$b_1^1$};
\node (j) at (4,4) {$b_1^i$};
\node (r) at (5,6) {$b_{c_i}^i$};
\node (s) at (6,6) {$a_i$};
\node at (2.5,3) {$\reflectbox{$\ddots$}$};
\node at (4.5,5) {$\reflectbox{$\ddots$}$};
\draw (11) to (2);
\draw (2) to (a);
\draw (a) to (b);
\draw (b) to (c);
\draw (c) to (d);
\draw (i) to node[scale=0.85,above]{$I_v$} (j);
\draw (r) to (s);
\draw (s) to (17);
\draw (a) to node[scale=0.85,left]{$Y_{I_{c_1}^1}$} (12);
\draw (b) to node[scale=0.85,right]{$Y_{I_{c_1-1}^1}$} (11);
\draw (b) to node[scale=0.85,below]{$Y_{I_{c_1-1}^1}$} (13);
\draw (c) to node[scale=0.85,above]{$Y_{I_{c_1-2}^1}$} (12);
\draw (c) to node[scale=0.85,left]{$Y_{I_{c_1-2}^1}$} (14);
\draw (d) to node[scale=0.85,right]{$Y_{I_{c_1-3}^1}$} (13);
\draw (d) to node[scale=0.85,below]{$Y_{I_{c_1-3}^1}$} (3,2);
\draw (i) to node[scale=0.85,above]{$Y_{I_1^1}$} (2,4);
\draw (i) to node[scale=0.85,left]{$Y_{I_1^1}$} (16);
\draw (j) to node[scale=0.85,right]{$Y_{I_1^i}$} (15);
\draw (j) to node[scale=0.85,below]{$Y_{I_1^i}$} (5,4);
\draw (r) to node[scale=0.85,above]{$Y_{I_{c_i}^i}$} (4,6);
\draw[dashed] (11) to node[scale=0.85,above]{$I_{c_1}^1$} (a);
\draw[dashed] (b) to node[scale=0.85,above]{$\ \ I_{c_1-1}^1$} (12);
\draw[dashed] (13) to node[scale=0.85,above]{$\ \ I_{c_1-2}^1$} (c);
\draw[dashed] (d) to node[scale=0.85,above]{$\ \ I_{c_1-3}^1$} (14);
\draw[dashed] (15) to node[scale=0.85,above]{$I_1^1$} (i);
\draw[dashed] (j) to node[scale=0.85,above]{$I_1^i$} (16);
\draw[dashed] (17) to node[scale=0.85,above]{$I_{c_i}^i$} (r);
\end{tikzpicture}
\caption{A component snake graph $H_{\ell,k}$ where $a_i\in [v,k]\cap N_{\Gamma'}(v)$ is less than $v$. For shorthand, let $b_j^1$ denote $b_{1,j}^{\ell,k}$ and similarly for $b_j^i$.} 
\label{fig:singleton-zigzag}
\end{figure}

\begin{figure}[ht]
\begin{tikzpicture}[scale=2]
\node (11) at (1,0) {$v$};
\node (12) at (0,2) {$v$};
\node (13) at (2,1) {$v$};
\node (14) at (1,3) {$v$};
\node (15) at (4,3) {$v$};
\node (2) at (0,0) {$a_1$};
\node (a) at (0,1) {$b_{c_1}^1$};
\node (b) at (1,1) {$b_{c_1-1}^1$};
\node (c) at (1,2) {$b_{c_1-2}^1$};
\node (d) at (2,2) {$b_{c_1-3}^1$};
\node (i) at (3,4) {$b_1^1$};
\node (j) at (4,4) {$a_i$};
\node at (2.5,3) {\reflectbox{$\ddots$}};
\draw (2) to (a);
\draw (a) to (b);
\draw (b) to (c);
\draw (c) to (d);
\draw (i) to node[scale=0.85,above]{$Y_{I_v}$} (j);
\draw (2) to (11);
\draw (a) to node[scale=0.85,left]{$Y_{I_{c_1}^1}$} (12);
\draw (b) to node[scale=0.85,right]{$Y_{I_{c_1-1}^1}$} (11);
\draw (b) to node[scale=0.85,below]{$Y_{I_{c_1-1}^1}$} (13);
\draw (c) to node[scale=0.85,above]{$Y_{I_{c_1-2}^1}$} (12);
\draw (c) to node[scale=0.85,left]{$Y_{I_{c_1-2}^1}$} (14);
\draw (d) to node[scale=0.85,right]{$Y_{I_{c_1-3}^1}$} (13);
\draw (d) to node[scale=0.85,below]{$Y_{I_{c_1-3}^1}$} (3,2);
\draw (i) to node[scale=0.85,above]{$Y_{I_1^1}$} (2,4);
\draw (j) to node[scale=0.85,right]{} (15);
\draw[dashed] (11) to node[scale=0.85,above]{$I_{c_1}^1$} (a);
\draw[dashed] (b) to node[scale=0.85,above]{$\ \ I_{c_1-1}^1$} (12);
\draw[dashed] (13) to node[scale=0.85,above]{$\ \ I_{c_1-2}^1$} (c);
\draw[dashed] (d) to node[scale=0.85,above]{$\ \ I_{c_1-3}^1$} (14);
\draw[dashed] (15) to node[scale=0.85,above]{$I_1^1$} (i);
\end{tikzpicture}
\caption{A component snake graph $H_{\ell,k}$ where $a_i\in [v,k]\cap N_{\Gamma'}(v)$ is greater than $v$. For shorthand, let $b_j^1$ denote $b_{1,j}^{\ell,k}$ and similarly for $b_j^i$.} 
\label{fig:singleton-zigzag-partial}
\end{figure} 

We build the snake graph $\mathcal{G}_v$ with $\{v\}\not\in\mathcal I$ by gluing together the component snake graphs. Assuming that all of the component snake graphs are oriented as in Figures~\ref{fig:singleton-zigzag} and~\ref{fig:singleton-zigzag-partial} (i.e., such that the tile with diagonal labeled by $I_{c_1}^1$ is in the southwest corner and the vertex of that tile labeled $v$ is in the southeast corner of the tile), we start by identifying all vertices across component snake graphs that have the same label and are in the same location on tiles with the same label.  For example, every component snake graph has a vertex $v$ in the southeast corner of a tile labeled $I_{c_1}^1$, so all of these vertices will be identified.  Notice we now have a graph with only one connected component.  We also identify the vertex $v$ on the diagonals $I_1^i$ from all component snake graphs $H_{\ell,k}$ where $a_i\in [v,k]\cap N_{\Gamma'}(v)$ for some $k \in B$ is less than $v$.  We now combine all diagonals with the same label, possibly resulting in diagonals with multiple endpoints.  Finally, we combine any edges that have the same weight, share at least one endpoint, and are adjacent to the same diagonals.  This process may create hyperedges.
At the end of this process, every edge with weight $Y_S$ is incident to exactly one vertex $w$ for each $w \in N_{\Gamma'}(S)$ and similarly for diagonals labeled $S$, and every internal edge with weight $Y_C^2$, which is on tiles labeled $I_j^i$ and $I_{j+1}^i$ is incident to two vertices $w$ for each $w \in N_{\Gamma'}(C) \backslash [v,a_j^i]$.  The resulting graph is $\mathcal{G}_v$.

We give special names to two of the vertices in $\mathcal{G}_v$ for $\{v\}\not\in \mathcal I$.
\begin{itemize}
    \item Every component snake graph has a tile labeled $I_1^1$.  They also all have a vertex $v$ in the same location on this tile.  These vertices were all identified in the process of creating $\mathcal{G}_v$. We denote the resulting vertex in $\mathcal{G}_v$ as $v_\opt$.
    \item If $r>1$, then there is some collection of component snake graphs $H_{\ell,k}$ where $a_i\in [v,k]\cap N_{\Gamma'}(v)$ is less than $v$.  Then the vertex $v$ on the diagonals $I_1^i$ from all of these graphs got identified in the process of creating $\mathcal{G}_v$; we denote this vertex in $\mathcal{G}_v$ as $v_\fix$.  Notice that while other vertices are incident to 0 or 1 diagonal, $v_\fix$ is incident to $r-1$ diagonals.
\end{itemize}  

Each vertex in $\mathcal{G}_v$ is assigned a \emph{valence} $a \oplus b$.  If $\{v\}\in\mathcal{I}$, all vertices of $\mathcal{G}_v$ are assigned valence $1 \oplus 0$.  Otherwise, for each vertex $x\neq v_\fix$ in $\mathcal{G}_v$, we look at a component snake graph that has a vertex $x$ incident to the same diagonals.  If our vertex $x$ in $\mathcal{G}_v$ is incident to $b$ more edges than our vertex $x$ in the component snake graph, then the valence of vertex $x$ in $\mathcal{G}_v$ is $1\oplus b$.  Note that this is well-defined by \Cref{prop:ConsistentSize}.  The vertex $v_\fix$, if it exists in $\mathcal{G}_v$, is assigned valence $(r-1) \oplus 0$.

We denote the valence of a vertex $x$ as $\textrm{val}(x)$.  In our figures, we will include the $\textrm{val}(x)$ next to vertex $x$ whenever $\textrm{val}(x)$ is not $1\oplus0$.  The valence of $x$ determines the minimum and maximum number of edges that can be incident to $x$ in the set of matchings of $\mathcal{G}_v$ that we will consider in our expansion formula. See \Cref{sec:main_formula} for details.

\begin{ex}\label{ex:G4}
The following is the snake graph $\mathcal G_4$. 
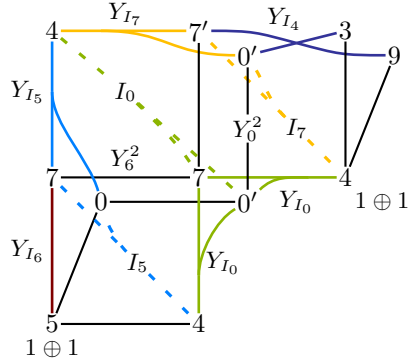
\begin{center}
\input{figures/G4}
\end{center}
\end{ex}

\begin{ex}\label{ex:G3}
The following is the snake graph $\mathcal G_3$. 
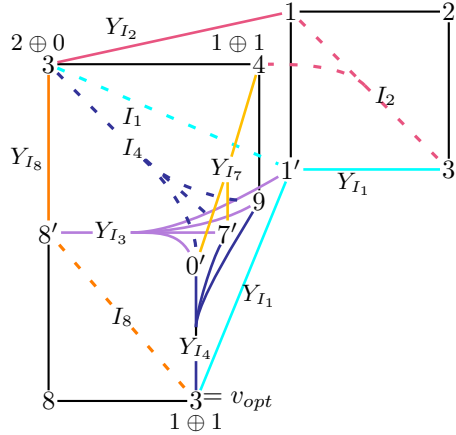
\begin{center}
\input{figures/G3}
\end{center}
\end{ex}

\begin{remark}
For a vertex $v$ with $r>1$, our construction of $\mathcal{G}_v$ is not unique.  In particular, it depends on the choice of $a_1$.  However, this ambiguity will turn out to be irrelevant to us as our main theorem will hold for any snake graph $\mathcal{G}_v$ obtained by following the process detailed above.
\end{remark}

\subsection{Snake Graphs for Weakly Rooted Sets}\label{sec:Gluing}

Next, we describe how to construct snake graphs for a special class of larger vertex subsets called weakly rooted sets.

\begin{defn}\label{def:RootedAndWeaklyRooted}
Let $S$ be a connected vertex subset of a graph $\Gamma$, $\mathcal{I}$ be a maximal nested collection on $\Gamma$ and define $\overline{S}:=\{ v\in S|I_v\not \subseteq S\}.$
\begin{enumerate}
    \item $S$ is \emph{rooted} with respect to $\mathcal{I}$ if for all pairs $i,j \in S$, $i <_{\mathcal{I}} j$ if and only if $j \in [i,m_{\mathcal{I}}(S)]$.
    \item $S$ is \emph{weakly rooted} with respect to $\mathcal{I}$ if $\overline{S}$ is a connected rooted set.
\end{enumerate}
 When $S$ is weakly rooted, we refer to $\overline{S}$ as the \emph{rooted portion} of $S$.
\end{defn}

\begin{ex}\label{ex:weakly-rooted-set}
Given the maximal nested collection $\mathcal I$ shown in \Cref{fig:cluster}, the set $S_1 = \{3,4,5,6,8\}$ is rooted with respect to $\mathcal I$ because starting at any vertex in the set and the walking towards $3=m_{\mathcal I}(S_1)$, we go through bigger and bigger vertices with respect to $<_{\mathcal I}$.  The set $S_2 = \{3,4,5,6,8,0\}$ is not rooted because the path $[0,3] = [0,m_{\mathcal{I}}(S_2)]$ contains $5$ and $5<_{\mathcal I}0$.  However, $\{3,4,5,6,8,0\}$ is weakly rooted.  Since $\{5,6,0\},\{8\}\in\mathcal I$, the rooted portion of $\{3,4,5,6,8,0\}$ is $\{3,4\}$, which is rooted.
\end{ex}

Note that any element of $\mathcal{I}$ and any set of size 1 or 2 is weakly rooted with respect to $\mathcal{I}$.   In addition, if $\mathcal{I}$ is a rooted cluster, as in \cite{rooted}, then all connected vertex subsets are weakly rooted with respect to $\mathcal{I}$. Note that in this latter case, $c_i(v)=1$ for all $v\in\overline{S}$ and for all $1\leq i\leq r$. 

In order to build $\mathcal{G}_S$ we will build singleton snake graphs for each vertex in $\bar{S}$ and then remove certain subgraphs and adjust the valences. In order to justify that this process is well-defined, we first provide a few results concerning properties of singleton snake graphs.

\begin{prop}\label{prop:ij-edge}
Given any $a_i \in N_{\Gamma'}(v)$, there is a unique edge of the form $(v,a_i)$ in  the snake graph $\mathcal{G}_v$ where the vertex $a_i$ does not lie on any diagonals. If $a_i <_{\mathcal{I}} v$, then the vertex labeled $v$ lies on the diagonal labeled $I_{a_i} = I_{c_i}^i$. If $a_i >_{\mathcal{I}} v$, then this edge is $(v_\opt,a_i)$. 
\end{prop}

\begin{proof}
If $a_i \in N_{\Gamma'}(v)$, then it follows from the construction that in any component snake graph $H_{\ell,k}$ which $a_i$ appears in, there is a boundary edge $(v,a_i)$ such that the vertex $v$ lies on the diagonal $I_{a_i}$. Since this edge is consistent throughout the relevant component snake graphs, we have the same edge when we glue these together.
If $a_i >_{\mathcal{I}} v$, then the edge $(v,a_i)$ shares the vertex $v$ with the diagonal labeled $I_1^1$, so this is $v_\opt$ by definition. 

The fact that this edge is unique in each case follows from the fact that, in a tree, it is impossible for $a_i \in N_\Gamma(I)$ for some $I \in \mathcal{I}$ which is incompatible with $\{v\}$. 
\end{proof}

The second result will be necessary to guarantee that our process of trimming portions of a snake graph $\mathcal{G}_S$ for sets $I \subset S, I \in \mathcal{I}$ is well-defined. 

\begin{prop}\label{prop:APlaceToCut}
For every diagonal labeled $I_j^i$ in a snake graph $\mathcal{G}_v$, there is a vertex labeled $a_j^i$ which is not on this diagonal but such that every edge incident to this vertex is also incident to the diagonal $I_j^i$, $I_{j+1}^i$, or $I_{j+2}^i$.
\end{prop}

\begin{proof}
Recall from \Cref{subsec:SingletonConstruction1} that  $a_{j}^{i} = m_{\mathcal{I}}(I_{j}^{i})$ and that $b_{i,j}^{\ell,k}$ is the unique vertex of $\Gamma$ besides $v$ in $[\ell,k] \cap N_{\Gamma'}(I_j^i)$ for $1\leq i\leq r$ and $1\leq j\leq c_i$.

First consider the case where $j<c_i$. Since $a_j^i \in N_{\Gamma'}(I_{j+1}^i)$, there must be at least one pair $\ell,k$ such that $b_{i,j+1}^{\ell,k} = a_j^i$. In any component snake graph $H_{\ell,k}$  for such a pair, there will be a vertex labeled $a_j^i = b_{i,j+1}^{\ell,k}$.  All edges incident to this vertex will be on a tile $I_j^i$, $I_{j+1}^i$, or $I_{j+2}^i$ if it exists, so the claim for the full snake graph $\mathcal{G}_v$ will follow.

Now consider the case where $j=c_i$.  In this case $a_j^i=a_i$.  In any component snake graph $H_{\ell,k}$ where $a_i\in [\ell,k]$, there is a vertex labeled $a_i$.  This vertex is not on any diagonals and all edges incident to this vertex is incident to the diagonal $I_{c_i}^i$.  Thus, the claim for $\mathcal{G}_v$ follows. 
\end{proof}

We are now ready to construct the snake graph $\mathcal G_S$ for a weakly rooted set $S$ of arbitrary size.  If $S\in\mathcal I$, then $\mathcal G_S$ is a single hyperedge weighted $Y_S$ that connects all vertices in $N_{\Gamma'}(S)$. Otherwise, we construct the snake graph $\mathcal G_S$ via the following procedure.

First, for each $v\in\overline{S}$ we do the following:
\begin{enumerate}[(1)]
\item Construct $\mathcal G_v$.
\item Let $\{a_{i_1},\dots,a_{i_k}\}=\{a_1(v),\dots,a_r(v)\}\,\cap\,(S\setminus \overline{S})$ for $r = \vert N_{\Gamma}(v)\,\cap\,\Gamma^{\mathcal{I}}_{<v} \vert$. That is, $\{a_{i_1},\dots,a_{i_k}\}$ is the set of neighbors of $v$ that are less than $v$ and in $S\setminus\overline{S}$. For $1\leq \ell\leq k$, let $m_\ell$ be minimal such that $1\leq m_\ell\leq c_{i_\ell}(v)$ and $I^{i_\ell}_{m_\ell}(v)\subsetneq S$.  Note that such a $m_\ell$ always exists because $S$ must contain $I_{a_{i_\ell}}=I^{i_\ell}_{c_{i_\ell}}(v)$ for $a_{i_\ell}$ to be in $S\setminus\overline{S}$ by the definition of weakly rooted.
\item For $1\leq\ell\leq k$, $\mathcal{G}_v$ contains a diagonal labeled $I^{i_\ell}_{m_\ell}$. We delete any edge in $\mathcal{G}_v$ which is incident to a diagonal $I$ such that $I \subsetneq I^{i_\ell}_{m_\ell}$. If no such $I$ exists, then we delete all edges which are incident to the vertex $a^{i_\ell}_{m_\ell}$; note that this vertex is guaranteed to exist by Proposition \ref{prop:APlaceToCut}. Next, we delete all vertices that are not incident to any edges. Then, we set the valence of the vertex $v$ which is incident to the diagonal $I^{i_\ell}_{m_\ell}$ to be one less in the first coordinate than it was previously. Finally, we delete all diagonals $I \subseteq I^{i_\ell}_{m_\ell}$ and we delete all vertices with valence $0 \oplus 0$ and their incident edges.

We have now constructed $\mathcal{G}_{\{v\}\cup I^{i_1}_{m_1}\cup\cdots\cup I^{i_k}_{m_k}}$.  We will call this graph $\mathcal G_v^S$ for the rest of the steps in this construction.
\end{enumerate}
If $\overline{S}$ is a singleton, we are now done.  Otherwise, we need to combine the  graphs we made in steps (1) - (3).  Note that if $v,w\in\overline{S}$ are adjacent in $\Gamma$ then the above process does not remove the edge $(v,w)$ in $\mathcal G_v$ or $\mathcal G_w$ described in \Cref{prop:ij-edge}. Thus $\mathcal G_v^S$ and $\mathcal G_w^S$ both contain an edge $(v,w)$.  Let $L=v_1,v_2,\dots,v_{|\overline{S}|}$ be a list of the elements in $\overline{S}$ such that $\{v_1,\dots,v_i\}$ is connected for all $1\leq i\leq|\overline{S}|$.  We will finish constructing $\mathcal G_S$ with the following steps.
\begin{enumerate}
\item[(4)] Start with $\mathcal G_{v_1}^S$ and $\mathcal G_{v_2}^S$.  Since $v_1$ and $v_2$ are adjacent in $\overline{S}$, there is an edge $(v_1,v_2)$ in each.  Form $\mathcal{G}_{v_1,v_2}^S$ by identifying these edges. For the vertices $v_1$ and $v_2$ that were identified, set their valences in $\mathcal{G}_{v_1,v_2}^S$ to be the coordinate-wise sum of their valences in $\mathcal{G}_{v_1}^S$ and $\mathcal{G}_{v_2}^S$ with 1 subtracted from the first coordinate.
\item[(5)] Since $\{v_1,v_2,v_3\}$ is  connected, $v_3$ must be adjacent to $v_1$ or $v_2$ in $\overline{S}$.  Without loss of generality, assume it's adjacent to $v_1$.  Then $\mathcal{G}_{v_1,v_2}^S$ has an edge $(v_1,v_3)$ inherited from $\mathcal{G}_{v_1}^S$ and $\mathcal{G}_{v_3}^S$ also has an edge $(v_1,v_3)$.  To obtain $\mathcal{G}_{v_1,v_2,v_3}^S$, glue these graphs together along these edges as we glued $\mathcal{G}_{v_1}^S$ and $\mathcal{G}_{v_2}^S$ along $(v_1,v_2)$.
\item[(6)] Continue the process in step (5) for $v_4,v_5$, and so on through $v_{|\overline{S}|}$.  We define $\mathcal G_S:=\mathcal{G}_{v_1,v_2,\dots,v_{|\overline{S}|}}^S$.
\end{enumerate}

\begin{ex}\label{ex:weakly-rooted-graph}
Consider the maximal nested collection $\mathcal I$ from \Cref{fig:cluster} and the weakly rooted set $S=\{3,4,5,6,8,0\}$.  From \Cref{ex:weakly-rooted-set}, we know that $\overline{S}=\{3,4\}$.

\newpage
We'll start by following steps (1) - (3) for $v=4$.  We constructed $\mathcal{G}_4$ in \Cref{ex:G4}: 
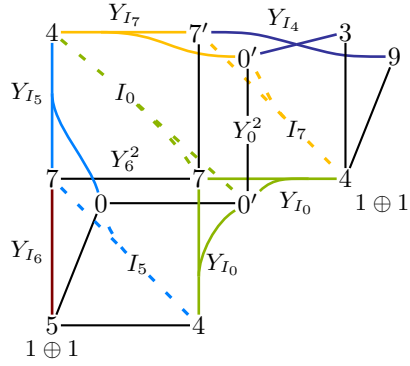
\begin{center}

\input{figures/G4}
\end{center}
The only neighbor of 4 that is in $S\setminus\overline{S}$ is 5, which we labeled as $a_1$ when creating $\mathcal{G}_4$.  This means $i_1=1$ in this case.  Because $I^1_2=\{5,6,0\}=I_0$ is contained in $S$ but $I^1_1=\{5,6,7,0\}=I_7$ is not contained in $S$, we have $m_1=2$.  In step 3, we delete everything incident to the diagonal $I_5$ since $I_5\subseteq I_0$. 
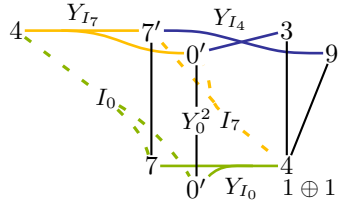
\begin{center}

\begin{tikzpicture}[scale=0.6]
	\begin{pgfonlayer}{nodelayer}
		\node [style=label] (0) at (-7, 2) {$4$};
		\node [style=label] (1) at (-4, 2) {$7'$};
		\node [style=label] (2) at (-3, 1.5) {$0'$};
		\node [style=label] (5) at (-4, -1) {$7$};
		\node [style=label] (6) at (-3, -1.5) {$0'$};
		\node [style=label] (9) at (-1, -1) {$4$};
		\node [style=label] (10) at (-1, 2) {$3$};
		\node [style=label] (11) at (0, 1.5) {$9$};
		\node [style=none] (12) at (-6, 2) {};
		\node [style=none] (13) at (-2, 1.75) {};
		\node [style=none] (14) at (-1.75, -1) {};
		\node [style=label-s] (18) at (-5, 0.5) {$I_0$};
		\node [style=label-s] (19) at (-2.25, 0) {$I_7$};
		\node [style=label-s] (20) at (-5.5, 2.35) {$Y_{I_7}$};
		\node [style=label-s] (21) at (-2.25, 2.25) {$Y_{I_4}$};
		\node [style=label-s] (22) at (-3, 0) {$Y_0^2$};
		\node [style=label-s] (23) at (-0.5, -1.5) {$1\oplus 1$};
		\node [style=label-s] (24) at (-2, -1.5) {$Y_{I_0}$};
	\end{pgfonlayer}
	\begin{pgfonlayer}{edgelayer}
		\draw [style=e4, in=165, out=0, looseness=1.25] (1) to (13.center);
		\draw [style=e4, in=180, out=-15, looseness=1.25] (13.center) to (11);
		\draw [style=e4] (2) to (13.center);
		\draw [style=e4] (13.center) to (10);
		\draw [style=e0] (5) to (14.center);
		\draw [style=e0] (14.center) to (9);
		\draw [style=e0, in=45, out=180] (14.center) to (6);
		\draw [style=e7] (0) to (12.center);
		\draw [style=e7, in=180, out=0, looseness=1.25] (12.center) to (2);
		\draw [style=e7] (12.center) to (1);
		\draw (1) to (5);
		\draw (10) to (9);
		\draw (9) to (11);
		\draw [style=d7] (1) to (19);
		\draw [style=d7, in=300, out=135] (19) to (2);
		\draw [style=d7] (19) to (9);
		\draw [style=d0] (0) to (18);
		\draw [style=d0, in=105, out=-45] (18) to (5);
		\draw [style=d0, in=120, out=-30, looseness=0.75] (18) to (6);
		\draw (2) to (22);
		\draw (22) to (6);
	\end{pgfonlayer}
\end{tikzpicture}

\end{center}
We then reduce the valence of vertex 4 adjacent to diagonal $I_0$.
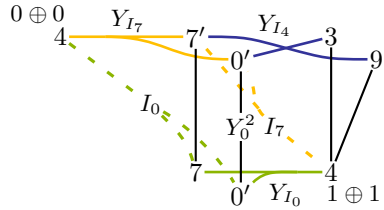
\begin{center}
\begin{tikzpicture}[scale=0.6]
	\begin{pgfonlayer}{nodelayer}
		\node [style=label] (0) at (-7, 2) {$4$};
		\node [style=label] (1) at (-4, 2) {$7'$};
		\node [style=label] (2) at (-3, 1.5) {$0'$};
		\node [style=label] (5) at (-4, -1) {$7$};
		\node [style=label] (6) at (-3, -1.5) {$0'$};
		\node [style=label] (9) at (-1, -1) {$4$};
		\node [style=label] (10) at (-1, 2) {$3$};
		\node [style=label] (11) at (0, 1.5) {$9$};
		\node [style=none] (12) at (-6, 2) {};
		\node [style=none] (13) at (-2, 1.75) {};
		\node [style=none] (14) at (-1.75, -1) {};
		\node [style=label-s] (18) at (-5, 0.5) {$I_0$};
		\node [style=label-s] (19) at (-2.25, 0) {$I_7$};
		\node [style=label-s] (20) at (-5.5, 2.35) {$Y_{I_7}$};
		\node [style=label-s] (21) at (-2.25, 2.25) {$Y_{I_4}$};
		\node [style=label-s] (22) at (-3, 0) {$Y_0^2$};
		\node [style=label-s] (23) at (-0.5, -1.5) {$1\oplus 1$};
		\node [style=label-s] (24) at (-2, -1.5) {$Y_{I_0}$};
		\node [style=label-s] (25) at (-7.5, 2.5) {$0\oplus 0$};
	\end{pgfonlayer}
	\begin{pgfonlayer}{edgelayer}
		\draw [style=e4, in=165, out=0, looseness=1.25] (1) to (13.center);
		\draw [style=e4, in=180, out=-15, looseness=1.25] (13.center) to (11);
		\draw [style=e4] (2) to (13.center);
		\draw [style=e4] (13.center) to (10);
		\draw [style=e0] (5) to (14.center);
		\draw [style=e0] (14.center) to (9);
		\draw [style=e0, in=45, out=180] (14.center) to (6);
		\draw [style=e7] (0) to (12.center);
		\draw [style=e7, in=180, out=0, looseness=1.25] (12.center) to (2);
		\draw [style=e7] (12.center) to (1);
		\draw (1) to (5);
		\draw (10) to (9);
		\draw (9) to (11);
		\draw [style=d7] (1) to (19);
		\draw [style=d7, in=300, out=135] (19) to (2);
		\draw [style=d7] (19) to (9);
		\draw [style=d0] (0) to (18);
		\draw [style=d0, in=105, out=-45] (18) to (5);
		\draw [style=d0, in=120, out=-30, looseness=0.75] (18) to (6);
		\draw (2) to (22);
		\draw (22) to (6);
	\end{pgfonlayer}
\end{tikzpicture}

\end{center}
Lastly, we delete the diagonal $I_0$, the vertex with valence $0\oplus0$, and all edges incident to the vertex with valence $0\oplus0$.
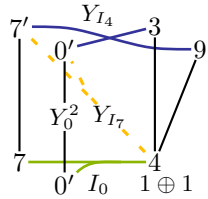
\begin{center}

\begin{tikzpicture}[scale = 0.6]
	\begin{pgfonlayer}{nodelayer}
		\node [style=label] (1) at (-4, 2) {$7'$};
		\node [style=label] (2) at (-3, 1.5) {$0'$};
		\node [style=label] (5) at (-4, -1) {$7$};
		\node [style=label] (6) at (-3, -1.5) {$0'$};
		\node [style=label] (9) at (-1, -1) {$4$};
		\node [style=label] (10) at (-1, 2) {$3$};
		\node [style=label] (11) at (0, 1.5) {$9$};
		\node [style=none] (13) at (-2, 1.75) {};
		\node [style=none] (14) at (-1.75, -1) {};
		\node [style=label-s] (19) at (-2, 0) {$Y_{I_7}$};
		\node [style=label-s] (21) at (-2.25, 2.25) {$Y_{I_4}$};
		\node [style=label-s] (22) at (-3, 0) {$Y_0^2$};
		\node [style=label-s] (23) at (-0.75, -1.5) {$1\oplus 1$};
		\node [style=label-s] (24) at (-2.25, -1.5) {$I_0$};
	\end{pgfonlayer}
	\begin{pgfonlayer}{edgelayer}
		\draw [style=e4, in=165, out=0, looseness=1.25] (1) to (13.center);
		\draw [style=e4, in=180, out=-15, looseness=1.25] (13.center) to (11);
		\draw [style=e4] (2) to (13.center);
		\draw [style=e4] (13.center) to (10);
		\draw [style=e0] (5) to (14.center);
		\draw [style=e0] (14.center) to (9);
		\draw [style=e0, in=45, out=180] (14.center) to (6);
		\draw (1) to (5);
		\draw (10) to (9);
		\draw (9) to (11);
		\draw [style=d7] (1) to (19);
		\draw [style=d7, in=300, out=135] (19) to (2);
		\draw [style=d7] (19) to (9);
		\draw (2) to (22);
		\draw (22) to (6);
	\end{pgfonlayer}
\end{tikzpicture}

\end{center}
This is $\mathcal{G}_{4560}=\mathcal{G}_4^S$. 

\newpage
Now we'll follow steps (1) - (3) for $v=3$. We constructed $\mathcal{G}_3$ in \Cref{ex:G3}:
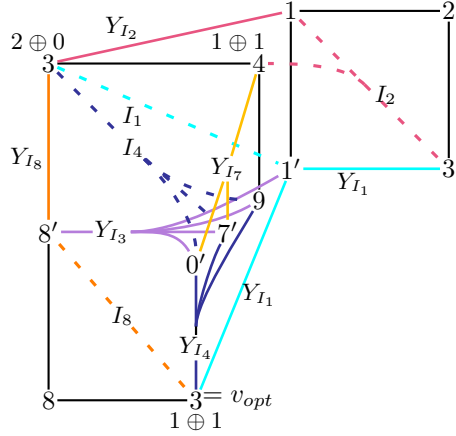
\begin{center}
\input{figures/G3}
\end{center}
The only neighbor of 3 that is in $S\setminus\overline{S}$ is 8, which we labeled as $a_1$ when creating $\mathcal{G}_4$.  This means $i_1=1$ in this case.  Because $I^1_1=\{8\}=I_8$ is contained in $S$, we have $m_1=1$.  In step 3, we delete everything incident to the vertex 8 since there no diagonals labeled with sets properly contained in $I_8$. 
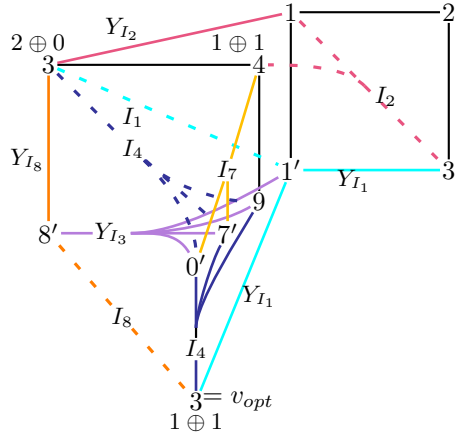
\begin{center}
\input{figures/G3x}
\end{center}
We then reduce the valence of vertex 3 adjacent to diagonal $I_8$ and delete the diagonal $I_8$.  
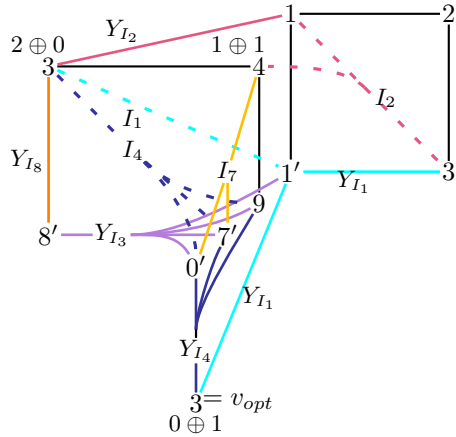
\begin{center}
\input{figures/G38}
\end{center}
This is $\mathcal{G}_{38}=\mathcal{G}_3^S$.

Gluing $\mathcal{G}_4^S$ and $\mathcal{G}_3^S$ along the edge (3,4), we get the snake graph $\mathcal{G}_{\{345680\}}=\mathcal{G}_S$.  The vertices 3 we are identifying have valence $1\oplus 0$ and $2\oplus 0$, so the identified vertex 3 has valence $(1+2-1)\oplus(0+0)=2\oplus0$.  The vertices 4 we are identifying have valence $1\oplus 1$ and $1\oplus 1$, so the identified vertex 4 has valence $(1+1-1)\oplus(1+1)=1\oplus2$. 
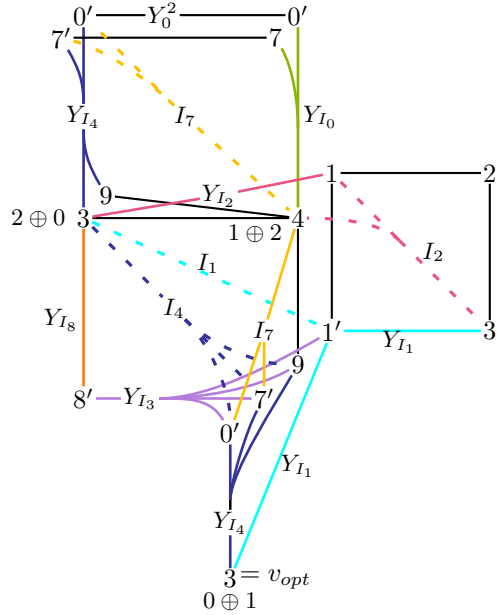
\begin{center}
\input{figures/GS}
\end{center}
\end{ex}

\begin{remark}\label{rem:gluing-trimming-commutes}
Our algorithm for constructing $\mathcal{G}_S$ first ``trims" $\mathcal{G}_v$ for each $v\in\overline{S}$ and then ``glues" the resulting graphs together. Heuristically, these two procedures commute, so we could instead combine trimming and gluing each $\mathcal G_v$ in any order.
\end{remark}

\begin{remark}
As with singleton sets, our construction of $\mathcal{G}_S$ for a weakly rooted set $S$ is not always unique.  However, as with singleton sets, our main theorem will hold for any snake graph $\mathcal{G}_S$ obtained by following the construction in this section.
\end{remark}

The following is an observation about the labels of diagonals in the snake graphs encountered here.

\begin{prop}\label{prop:NoDiagonalTwice}
Let $I \in \mathcal{I}$, and let $S \subset V(\Gamma)$ be a weakly rooted set. Then, $\mathcal{G}_{S}$ has exactly one diagonal labeled $I$ if $I$ and $S$ are incompatible and otherwise $\mathcal{G}_S$ has no diagonal labeled $I$.
\end{prop}

\begin{proof}
This claim is immediately true for singleton snake graphs by \Cref{prop:ConsistentSize}.

Let's consider the case of a  rooted set $S$. Recall $\mathcal{G}_S$ is the result of gluing the $\mathcal{G}_v$ for all $v \in S$. Since the claim is true for each $\mathcal{G}_v$, we know that there will be at least one diagonal for each $I$ incompatible with an element of $S$ and no other labels. Moreover, since $S$ is rooted, there is no  $I \in \mathcal{I}$ which is contained in $S$.   Thus the collection of sets incompatible with some $v \in S$ and the collection of sets incompatible with $S$ will be the same.  Thus, there will be at least one diagonal for each $I$ incompatible with $S$ and no other labels. 

Now, suppose for sake of contradiction that there exist two distinct vertices $v,w \in S$ and $I \in \mathcal{I}$ such that $\mathcal{G}_{\{v\}}$ and $\mathcal{G}_{\{w\}}$ both have a diagonal labeled $Y_I$.  Since $S$ is connected and $\Gamma$ is a tree, we know that $[v,w] \subseteq S$. In order for $I$ to be incompatible with both $v$ and $w$, there must be at least one element of $I$ adjacent to $v$ and at least one element adjacent to $w$. Thus, since $\Gamma$ is a tree and $I$ is connected, we need $[v,w] \subseteq I$; otherwise there would be two distinct paths between $v$ and $w$. Moreover, we must have at least one vertex $u \neq v,w$ with $u \in [v,w]$ for this to be possible. But, in this case $S$ would not be rooted since $I_u \subseteq I$ would be smaller than $I_v$ and $I_w$. 

For a weakly rooted set $S$, the labels of diagonals in $\mathcal{G}_S$ are the subset of the labels of diagonals for the snake graph of the rooted portion $\mathcal{G}_{\overline{S}}$ which are not contained in $\mathcal{S}$. Thus, since the claim is true for the snake graph for the rooted portion of $S$, it's true for $S$.
\end{proof}

\section{Cluster Expansions from Snake Graphs}\label{sec:main_formula}

In this section, we give an expansion formula for cluster variables using \emph{admissible matchings} of the snake graphs constructed in Section~\ref{sec:construction}. We begin by defining this restricted class of matchings.

\subsection{Admissible Matchings of Snake Graphs}\label{subsec:AdmissibleMatchings}

Given a snake graph $\mathcal{G}$, we say a subset $P$ of the edge set is a \emph{matching} of $\mathcal{G}$ if, for every vertex $v \in \mathcal{G}$ with $\textrm{val}(v) = a \oplus b$, the number of edges in $P$ incident to $v$ is at least $a$ and no more than $a + b$. If every vertex has valence $1 \oplus 0$, then a matching would be a \emph{perfect matching}. Perfect matchings of snake graphs from surfaces give combinatorial expansion formulas for snake graphs from surfaces in \cite{musiker2011positivity}; matchings of snake graphs with several different valences have also been used to give formula for cluster algebras of type $D$ in \cite{wright2020mixed,farrell2022mixed} and for super cluster algebras \cite{musiker2022double,musiker2023higher}.

We will be interested in particular matchings of our snake graphs called \emph{admissible matchings}.  In order to state our admissibility conditions, we need to define boundary and internal edges for our snake graphs. However, since these are hypergraphs, the tile structure is less clear than for surface snake graphs so our definition will take more care than Definition \ref{def:BoundarySurfaceSnakeGraph}. Given a diagonal $d$ of a snake graph (or component snake graph), let $W(d)$ denote the label of this diagonal. Given an edge $e$ of a snake graph (or component snake graph) with weight $Y_C$ or $Y_C^2$, let $W(e)=C$. If $e$ has weight $1$, let $W(e) = \emptyset$.

\begin{defn}\label{def:BoundaryHyperSnakeGraph}
Consider a snake graph $\mathcal{G}_S$ for a weakly rooted set $S$. 
\begin{enumerate}
    \item If an edge is of the form $(v,w)$ where $v$ and $w$ are both in $\bar{S}$ and $(v,w) \in E(\Gamma')$, then $(v,w)$ is an \emph{internal edge}.
    \item Given any other edge $e$ in $\mathcal{G}_S$, there is a unique $v \in \bar{S}$ such that $e$ appeared in the smaller snake graph $\mathcal{G}_v^S = \mathcal{G}_{v \cup I_{m_1}^{i_1} \cup \cdots \cup I_{m_k}^{i_k}}$ where $I_{m_1}^{i_1},\ldots,I_{m_k}^{i_k} \in \mathcal{I}$ are as in step (2) of the construction of $\mathcal{G}_S$. Set $J_{\mathbf{m}}:= \{v\} \cup I_{m_1}^{i_1} \cup \cdots \cup I_{m_k}^{i_k}$. If either $e$ shares vertices with only one diagonal in $\mathcal{G}_{J_{\mathbf{m}}}$ or $e$ shares vertices with a diagonal $d$ in $\mathcal{G}_{J_v}$ such that $W(e) = W(d)$, then we say that $e$ is a \emph{boundary edge} in $\mathcal{G}_S$. Otherwise, we say $e$ is an \emph{internal edge} in $\mathcal{G}_S$.
\end{enumerate}    
\end{defn}

Note that our definitions of boundary and internal edges agrees with Definition \ref{def:BoundarySurfaceSnakeGraph} when the snake graph is a surface snake graph. Moving forward, when we refer to boundary and internal edges, we will always mean in the sense of Definition \ref{def:BoundaryHyperSnakeGraph}.

\begin{ex}
Consider the snake graph $\mathcal{G}_{S}$ for $S = \{3,4,5,6,8,0\}$ from Example~\ref{ex:weakly-rooted-graph}. Here, $\overline{S} = \{ 3, 4 \}$, so the edge $(3,4)$ is internal by part (1) of Definition \ref{def:BoundaryHyperSnakeGraph}, even though this edge was a boundary edge form the two graphs $\mathcal{G}_3^S = \mathcal{G}_{38}$ and $\mathcal{G}_4^S = \mathcal{G}_{4560}$ which we glued to form $\mathcal{G}_{345680}$. The remaining edges are evaluated by part (2) of Definition \ref{def:BoundaryHyperSnakeGraph}. For example, both edges $Y_{I_1}$ come from $\mathcal{G}_3^S$, and here they are boundary edges as they share vertices with the diagonal $I_1$, therefore they are also boundary edges in $\mathcal{G}_S$.
\end{ex}

\begin{defn}\label{def:AdmissibleMatching}
A matching $P$ of a snake graph $\mathcal{G}$ is \emph{admissible} if the following are satisfied.
\begin{enumerate}
    \item \label{admissibilityCondition1} There does not exist a subset of $P$ of the form $\{e_1, e_2\}$ such that $e_1$ and $e_2$ share at least one vertex and $e_1$ and $e_2$ both share vertices with a diagonal $d$ satisfying $W(e_1) \subsetneq W(d) \subsetneq W(e_2)$ (with $W(e_1)$ possibly $\emptyset$). 
    \item \label{admissibilityCondition2} For every pair of boundary edges $e_1$ and $e_2$ which both share vertices with a diagonal $d$ and which do not share vertices with each other such that $W(e_1), W(e_2)\neq\emptyset$ and $W(e_1) \subsetneq W(d) \subsetneq W(e_2)$, either both $e_1$ and $e_2$ are in $P$ or neither are in $P$.
    \item \label{admissibilityCondition3} There does not exist a subset of $P$ of the form $\{e_1,e_2\}$ such that $e_1$ and $e_2$ are both boundary edges which do not share vertices with each other but which share vertices with the same pair of diagonals $d_1$ and $d_2$ where $W(e_1) = W(d_1)$, $W(e_2) = W(d_2)$, and $W(d_1) \subseteq W(d_2)$. 
    \item \label{admissibilityCondition4} Let $e_1$ and $e_2$ be two edges that share a vertex such that there are vertices $v_1, v_2$  where $v_1$ is incident to $e_1$ but not $e_2$ and $v_2$ is incident to $e_2$ but not $e_1$.  If $v_1$ and $v_2$ are both incident to diagonal $d$ where $W(e_1),W(e_2)\subseteq W(d)$, and $e_1, e_2$ are both internal edges, or both boundary edges incident to internal edges, then $e_1$ and $e_2$ are either both in $P$ or neither are in $P$.
\end{enumerate}
\end{defn}

Note that these admissibility conditions are automatically satisfied for all perfect matchings of surface snake graphs.

In the following series of examples, we show subsets of edges which respect all valence conditions but violate exactly one of the conditions in Definition \ref{def:AdmissibleMatching} to show the necessity of each. 

\begin{ex}
\label{example:nonadmissibleMatching-Condition1}

In Figure \ref{fig:nonadmissible-Condition1}, we provide a graph on six vertices with a maximal nested collection as well as the snake graph $\mathcal{G}_{25}$. We highlight a set of edges such that every vertex with valence $a \oplus b$ is incident to between $a$ and $a+b$ edges. However, the labeled edges $e_1$ and $e_2$, which have labels $W(e_1) = \emptyset$ and $W(e_2) = \{1,2,3\}$, share vertices with each other as well as with the diagonal with label $\{1\}$. By Condition~(\ref{admissibilityCondition1}), we cannot have both $e_1$ and $e_2$ in an admissible matching, so this set of edges is not an admissible matching.

\begin{figure}
    \centering
    \input{figures/fig6-1}
    \input{ReduceOpacity}
    \input{figures/fig6-2}
    \input{ResetOpacity}
    \caption{On the left is a graph with a maximal nested collection and on the right is a subset of edges of the snake graph $\mathcal{G}_{25}$ which violates Condition~(\ref{admissibilityCondition1}) of Definition \ref{def:AdmissibleMatching}.  We suppress the edge weights and diagonal labels for clarity.}
    \label{fig:nonadmissible-Condition1}
\end{figure}
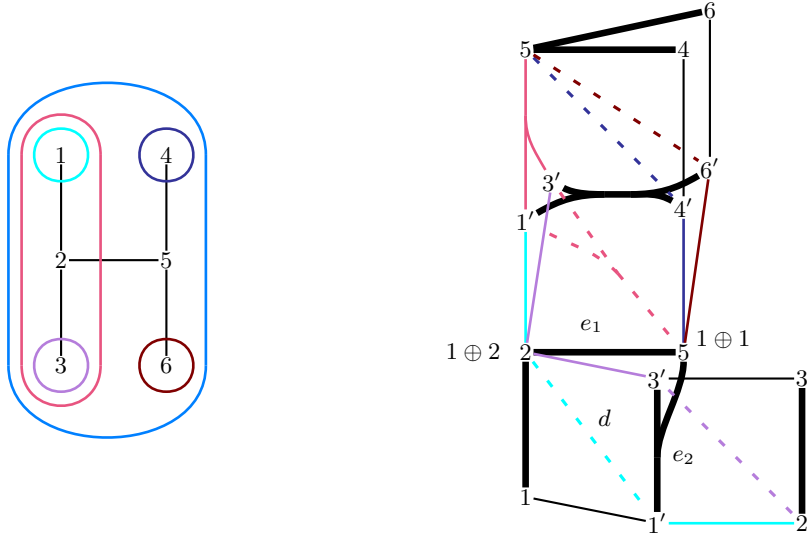

\end{ex}

\begin{ex}\label{example:nonadmissibleMatching-Condition2}

In Figure \ref{fig:nonadmissible-Condition2}, we consider the same graph as in Example \ref{example:nonadmissibleMatching-Condition1} but with a slightly different maximal nested collection. We provide the snake graph $\mathcal{G}_{25}$ with respect to this maximal nested collection. The highlighted set of edges violates Condition~(\ref{admissibilityCondition2}) because it includes the edge $e_1$ and not the edge $e_2$, even though $e_1$ and $e_2$ do not share vertices with each  other but do share vertices with a diagonal labeled $\{1,2\}$, which is a superset of $W(e_1)=\{1\}$ and a subset of $W(e_2) = \{1,2,3\}$. 

\begin{figure}
\centering
\hspace{6.75em}
   \input{figures/fig7-1}%
   \input{ReduceOpacity}%
\input{figures/fig7-2}%
\input{ResetOpacity}%

    \caption{On the left is a graph with a maximal nested collection and on the right is a subset of edges of the snake graph $\mathcal{G}_{25}$ which violates Condition~(\ref{admissibilityCondition2}) of Definition \ref{def:AdmissibleMatching}. We suppress the edge weights and diagonal labels for clarity.}\label{fig:nonadmissible-Condition2}
\end{figure}
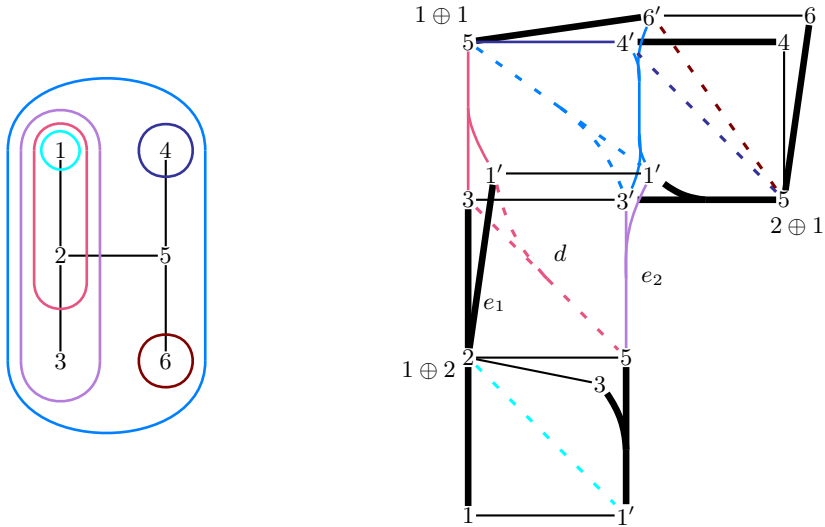

\end{ex}

\begin{ex}\label{example:nonadmissibleMatching-Condition3}

Beginning with the graph $\Gamma$ and maximal nested collection $\mathcal{I}$ from Figure \ref{fig:cluster}, consider the snake graph $\mathcal{G}_{345680}$ as was constructed in Section \ref{sec:Gluing}. In Figure \ref{fig:nonadmissible-Condition3}, we provide a subset of edges which violates Condition~(\ref{admissibilityCondition3}) by including the edges $Y_{I_4}$ and $Y_{I_7}$, which are each incident to diagonals $I_4$ and $I_7$. 

\begin{figure}
\centering
\hspace{1.7em}
\input{ReduceOpacity}
\input{figures/BadMatching1}
\input{ResetOpacity}
\caption{A subset of edges of the snake graph $\mathcal{G}_{345680}$ for the graph with maximal nested collection in Figure \ref{fig:cluster} which violate Condition~(\ref{admissibilityCondition3}). We suppress the edge weights and diagonal labels for clarity.}
\label{fig:nonadmissible-Condition3}
\end{figure}
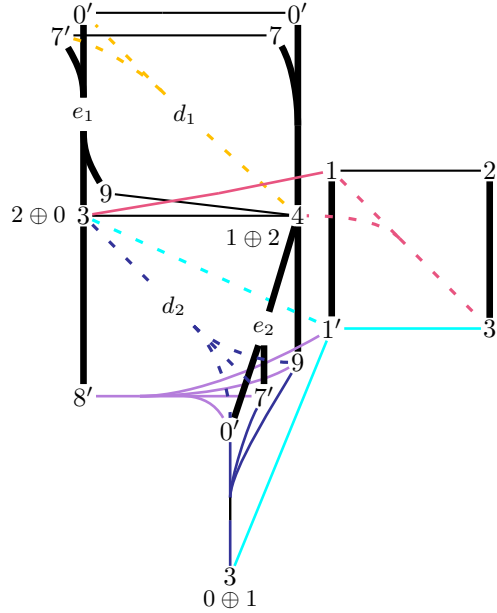

\end{ex}

\begin{ex}\label{example:nonadmissibleMatching-Condition4}

Now, we consider the graph with maximal nested collection in Figure \ref{fig:nonadmissible-Condition4} and the snake graph $\mathcal{G}_0$ also in this figure. This matching includes the internal edge $(2,4)$ but not the internal edge $(2,3')$. This violates Condition~(\ref{admissibilityCondition4}) since these two internal edges share vertex 2 and since $3'$ and $4$ are incident to the same diagonal. 

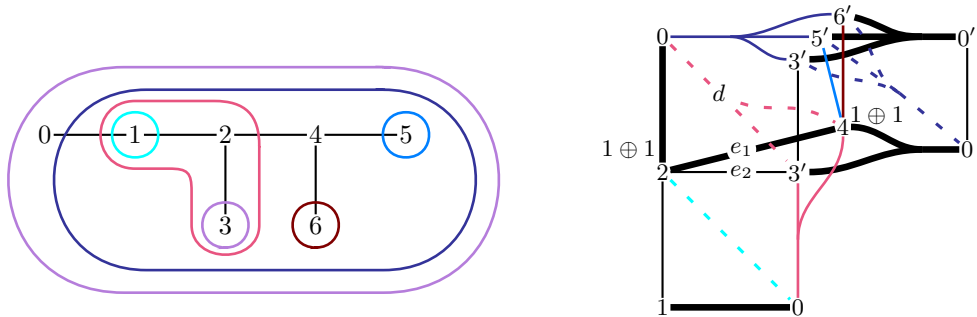
\begin{figure}
\centering
\input{figures/BadMatching2}
\caption{On the left is a graph with a maximal nested collection and on the right is a subset of edges of the snake graph $\mathcal{G}_0$ which violates Condition~(\ref{admissibilityCondition4}) of Definition \ref{def:AdmissibleMatching}.  We suppress the edge weights and diagonal labels for clarity.}
 \label{fig:nonadmissible-Condition4}
\end{figure}
\end{ex}

\subsection{Statement of Main Theorem}
\label{subsec:main_theorem}

Given a snake graph $\mathcal{G}$, let $\ell(\mathcal{G}):=\prod
_{d\in D} Y_{W(d)}$ where $D$ is the set of diagonals of $\mathcal{G}$.  Let $\chi(\mathcal{G}) := \frac{1}{\ell(G)} \sum_{P \in \mathcal{P}} \wt(P)$ where $\mathcal{P}$ is the set of admissible matchings of $\mathcal{G}$ and $\wt(P):=\prod_{e\in P}\wt(e)$ is the product of the weights of the edges in $P$. 

\begin{thm}\label{thm:main}
For any cluster $\mathcal{I}$ and any weakly rooted set $S$ with respect to $\mathcal I$, \[
\chi(\mathcal{G}_S) = Y_S.
\]
\end{thm}

\begin{ex}\label{ex:main-thm}
Below is the set of admissible matchings, with their weights, of the graph $\mathcal G_{345680}$ constructed in \Cref{ex:weakly-rooted-graph}. Recall in Example \ref{example:nonadmissibleMatching-Condition3} we demonstrated another matching of this graph which respected valence of all vertices but which was not admissible.

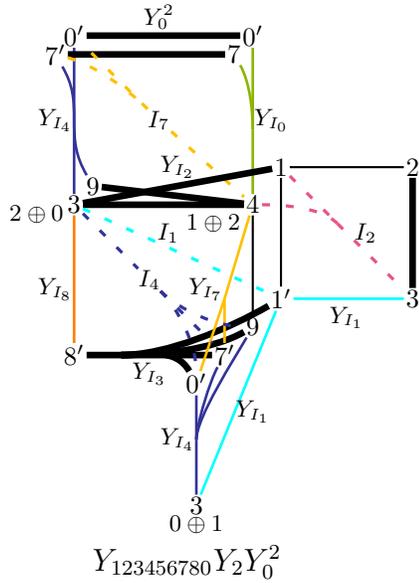
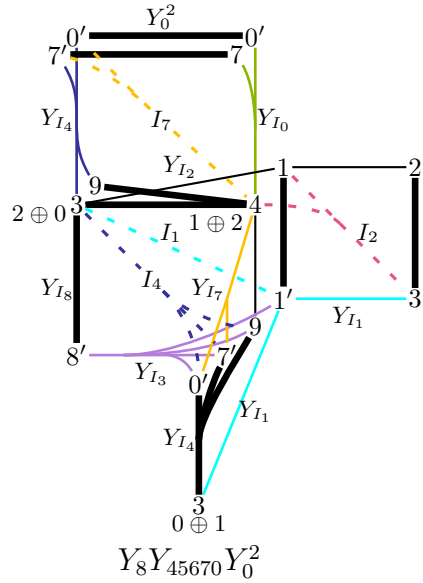
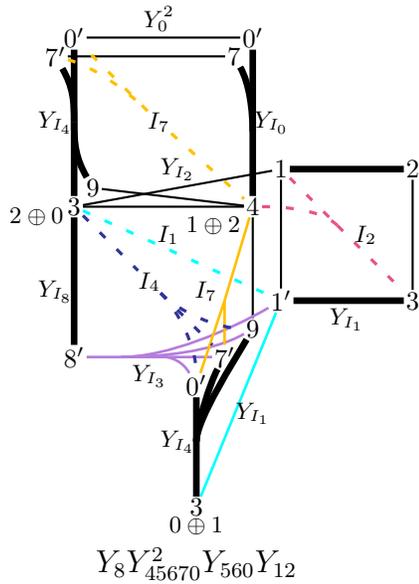
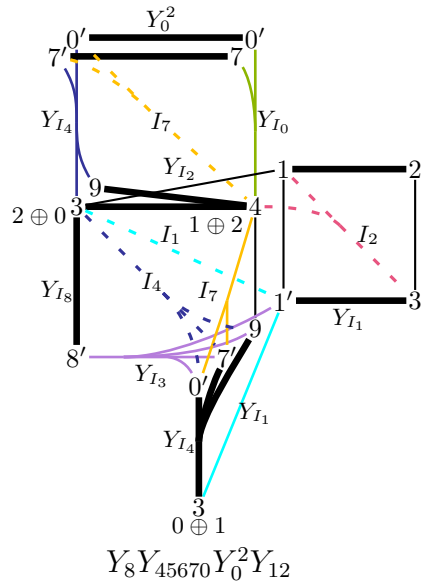
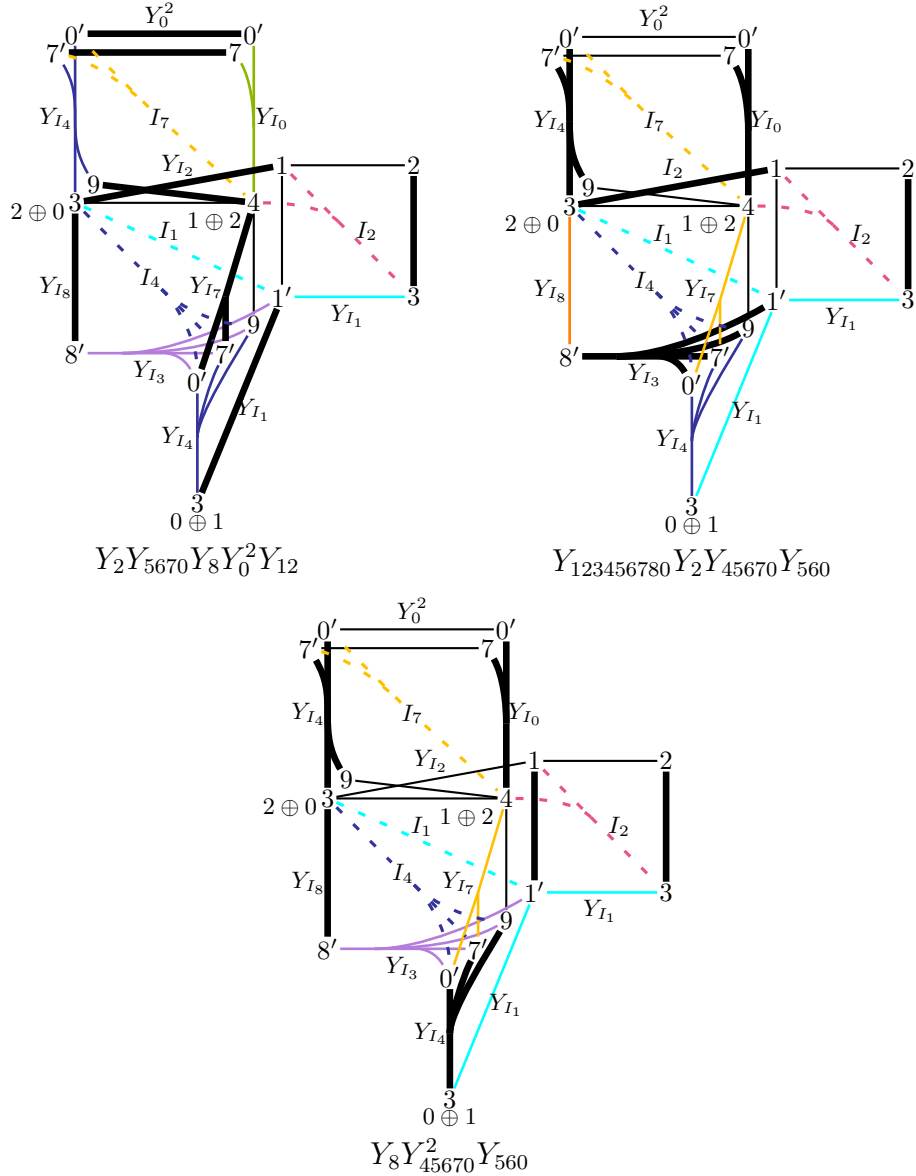
\begin{center}
\input{ReduceOpacity}

\input{figures/matching1}
\qquad
\input{figures/matching2}

\input{figures/matching3}
\qquad
\input{figures/matching4}

\input{figures/matching5}
\qquad
\input{figures/matching6}

\input{figures/matching7}

\input{ResetOpacity}
\end{center}

We have $\ell(\mathcal G_{345680})=Y_{2}Y_{12}Y_{5670}Y_{45670}$, so \Cref{thm:main} tells us that \begin{align*}
Y_{345680}=&\ \frac{1}{Y_{2}Y_{12}Y_{5670}Y_{45670}}\left(Y_{123456780}Y_2Y_{0}^2+Y_8Y_{45670}Y_{0}^2+Y_8Y_{45670}^2Y_{560}Y_{12}+Y_8Y_{45670}Y_{0}^2Y_{12}+\right.\\
&\ \left.Y_2Y_{5670}Y_8Y_{0}^2Y_{12}+Y_{123456780}Y_2Y_{45670}Y_{560}+Y_8Y_{45670}^2Y_{560}\right)
\end{align*}
We can check this expression by starting with the cluster given by the maximal nested collection in \Cref{fig:cluster} and mutating at $Y_{5670}$, the $Y_{12}$, then $Y_{2}$, then $Y_{45670}$ with the mutation equations given in Lemma 4.11 of~\cite{lam2016linear}.
\end{ex}

\begin{remark}
\begin{enumerate}
    \item Note that \Cref{thm:main} does not guarantee an expression for $Y_S$ in terms of the cluster variables $\{Y_I\ |\ I\in\mathcal I\}$.  In particular, the expression in the previous example contains $Y_0$ and $\{0\}$ is not a set in the original nested collection in \Cref{fig:cluster}.
    \item The denominator of the expansion we obtain from \Cref{thm:main} is guaranteed to be square-free by Proposition \ref{prop:NoDiagonalTwice}.  However, if we expand any  variables that  are not in the original  cluster, we no longer necessarily have a square-free denominator.  For example, if we expand $Y_0$ in terms of the cluster in \Cref{ex:main-thm}, the denominator of the expression is no longer square-free.
\end{enumerate}
\end{remark}

We defer the proof of Theorem~\ref{thm:main} to Sections~\ref{sec:SingeltonProof} and \ref{sec:ProofGeneral}, as it is quite technical.

\section{Positivity}\label{sec:positivity}

As noted above, Theorem~\ref{thm:main} does not necessarily given an expansion for cluster variable $Y_S$ in terms of the nested collection $\mathcal I$ because the edges in the snake graph may have weights that are not expressed in terms of $\mathcal I$.  The following theorem characterizes when our formula gives an expansion purely in terms of the cluster from the maximal nested collection. 

\begin{thm}\label{thm:off-in-cluster}
If $S$ is a weakly rooted set with respect to $\mathcal I$, then the snake graph $\mathcal{G}_S$ has edges weighted only by products of variables in $\{Y_I\ |\ I\in\mathcal I\}$ if and only if for every vertex $v\in \overline{S}$, the connected components of the set $I_{j+1}^i(v)\setminus[v,a_j^i(v)]$ are elements of $\mathcal I$ for all $1\leq i\leq r$ and $1\leq j < c_i(v)$ where $I_j^i(v)\nsubseteq S$. 
\end{thm}

\begin{proof}
First suppose that for every vertex $v\in \overline{S}$, the set $I_{j+1}^i(v)\setminus[v,a_j^i(v)]$ consists of connected components which are elements of $\mathcal I$ for all $1\leq j < c_i(v)$ where $I_j^i(v)\nsubseteq S$. The snake graph $\mathcal{G}_{\overline{S}}$ is a gluing of $\mathcal{G}_{v}^S$ for all $v\in\overline {S}$.  From the list of edge weights in \Cref{subsec:SingletonConstruction1}, the only edges in $\mathcal{G}_{v}$ with weights that are not 1 or in the set $\{Y_I\ |\ I\in\mathcal I\}$ have weight $Y_C^2$ where $C$ is a connected component of $I_{j+1}^i(v)\setminus[v,a_j^i(v)]$.  Thus, the only edges in $\mathcal{G}_{\overline{S}}$ with weights that are not 1 or in the set $\{Y_I\ |\ I\in\mathcal I\}$ have weight $Y_C^2$ where $C$ is a connected component of $I_{j+1}^i(v)\setminus[v,a_j^i(v)]$ for some $v\in\overline{S}$.  Such an edge is adjacent to diagonals $I_j^i(v)$ and $I_{j+1}^i(v)$ in $\mathcal{G}_v$ and thus is adjacent to diagonals $I_j^i(v)$ and $I_{j+1}^i(v)$ in $\mathcal{G}_{\overline{S}}$.  If $I_j^i(v)\subseteq S$, then this edge gets deleted in creating $\mathcal{G}_v^S$ ,which is a subgraph of $\mathcal{G}_{\overline{S}}$.  Otherwise, the weight of the edge is a product of variables in $\{Y_I\ |\ I\in\mathcal I\}$ by our assumption.

For the other direction, suppose there is a vertex $v\in \overline{S}$ such that the set $I_{j+1}^i(v)\setminus[v,a_j^i(v)]$ has a connected component $C\not\in\mathcal I$ for some $1\leq j < c_i(v)$ where $I_j^i(v)\nsubseteq S$.  Let $u$ be a leaf in $\Gamma'$ such that $[u,v]$ has nonempty intersection with $C$.  Consider some pair $\ell,k$ such that $u=\ell$ if $i=1$ and  $u=k$ if $i>1$.  By definition, $I^i_{j+1}(v)$ is a connected component of $I^i_j(v)\setminus\{a_j^i(v)\}$.  Thus the neighbor of $I^i_{j+1}(v)$ that is not $v$ in $[\ell,k]$ and the neighbor of $I^i_j(v)$ that is not $v$ in $[\ell,k]$ are the same.  In other words, $b_{i,j}^{\ell,k}=b_{i,j+1}^{\ell,k}$.  The vertex $b_{i,j+1}^{\ell,k}$ is not adjacent to a vertex on the path $[v,a_j^i(v)]$ as this would force $[u,v]\cap C$ to be empty.  Thus, $H_{\ell,k}$ has an edge weighted $Y_C^2$ adjacent to diagonals $I^i_j(v)$ and $I^i_{j+1}(v)$.  This means $\mathcal{G}_v$ has an edge weighted $Y_C^2$ adjacent to diagonals $I^i_j(v)$ and $I^i_{j+1}(v)$ and therefore $\mathcal{G}_{\overline{S}}$ has an edge weighted $Y_C^2$ adjacent to diagonals $I^i_j(v)$ and $I^i_{j+1}(v)$.  Since $I_j^i(v)\nsubseteq S$, this edge is not removed when creating $\mathcal{G}_v^S$, so it remains in $\mathcal{G}_{S}$. We can conclude that there is an edge weighted $Y_C^2$ in $\mathcal{G}_S$.  Our proofs in Sections \ref{subsec:ris1}, \ref{subsec:risnot1}, \ref{subsec:SingeltonProofWeakSingleton}, and \ref{sec:GlueProof} show that this edge gets used in an admissible matching. Since $Y_C^2$ is not a product of variables in $\{Y_I\ |\ I\in\mathcal I\}$, the statement holds.
\end{proof}

\begin{ex}
As mentioned in the previous section, the expansion for $Y_{34}$ in \Cref{ex:main-thm} contains cluster variables that are not in the original cluster $\{Y_I\ |\ \in\mathcal{I}\}$.  Notice that vertex 4 is in $\overline{S}$ and $I^1_2(4)\setminus[4,a^1_1(4)]=\{5,6,0\}\setminus[4,7]=\{0\}$.  Since $\{0\}$ is not in $\mathcal{I}$, the fact that the snake graph does not give an expansion in terms of the variables $\{Y_I\ |\ \in\mathcal{I}\}$ is consistent with \Cref{thm:off-in-cluster}.
\end{ex}

Together, Theorem~\ref{thm:main} and Theorem~\ref{thm:off-in-cluster} produce the following corollary:

\begin{cor}
If  $S$ is a weakly rooted set with respect to $\mathcal I$ and for every vertex $v\in \overline{S}$, the set $I_{j+1}^i(v)\setminus[v,a_j^i(v)]$ consists of connected components which are elements of $\mathcal I$ for all $1\leq i\leq r$ and $1\leq j < c_i(v)$ where $I_j^i(v)\nsubseteq S$, then $Y_S$ can be written as a Laurent polynomial with positive coefficients in the cluster variables $\{Y_I\ |\ I\in\mathcal I\}$.
\end{cor}

\section{Hyper $T$-paths}\label{sec:TPath}

In \cite{rooted}, we constructed \emph{hyper $T$-paths} for any set $S$ with respect to a rooted cluster and gave an expansion formula for the variable $Y_S$ that was indexed by the set of all complete hyper $T$-paths for $S$. This setting was more restrictive than our current setting of weakly rooted sets with respect to an arbitrary cluster. One of the main distinctions is that, in terms of present notation, in \cite{rooted} all singleton sets $\{v\}$ had $c_i(v) = 1$ for all $1 \leq i \leq r$.

In the setting of ordinary cluster algebras, there is a correspondence between $T$-paths and perfect matchings of snake graphs, described in \cite{musiker2011positivity}, where the vertices of the snake graph appear as \emph{nodes} in the $T$-path and the edges used in the perfect matching and diagonals appear as \emph{connections} between the nodes. For ordinary cluster algebras, $T$-paths are genuine path graphs. The spirit of this correspondence extends to the graph LP algebra setting, where we can now use our snake graph construction to describe a new, more general, version of hyper $T$-paths via the analogous correspondence with admissible matchings of $\mathcal{G}_S$, for any weakly rooted set $S$. Now, vertices of $\mathcal{G}_S$ continue to correspond to nodes, but the edges of an admissible matching and diagonals correspond to \emph{hyperedges} and the resulting hyper $T$-path is a \emph{hypergraph}. 

To describe the resulting hyper $T$-paths, we modify Definition 4.3 of \cite{rooted} to allow for our more general setting. Given a weakly rooted vertex subset $S$ of a tree $\Gamma$ with snake graph $\mathcal{G}_S$, let $\mathcal{B}$ denote the set of vertices of $\mathcal{G}_S$ with valence $0 \oplus b_k$. Note that vertices in $\mathcal{B}$ necessarily have $b_k > 0$ since vertices with valence $0 \oplus 0$ are deleted during the construction of $\mathcal{G}_S$. For any vertex subset $S$, recall that $N_{\Gamma'}(S)$ denotes the neighbors of $S$ in $\Gamma'$. 

To state our modified definition, we also need some relevant terminology for hypergraphs. In an undirected hypergraph with vertex set $X$ and hyperedge set $E$, each $e \in E$ is a subset of $X$. Elements of $e$ are referred to as \emph{endpoints} of $e$. A \emph{path} on a hypergraph is a sequence $(v_0, e_1, v_1, \dots, v_{m-1}, e_m, v_{m})$ of distinct vertices and hyperedges such that $\{v_{i-1}, v_i \} \subseteq e_i$. Paths on a hypergraph correspond to paths, defined in the usual way, on its incidence graph. A \emph{cycle} on a hypergraph is a path where $v_{0} = v_{m}$.

\begin{defn}
    \label{def:hyperTpath}
    Let $S$ be a weakly rooted vertex subset of $\Gamma$ with snake graph $\mathcal{G}_S$, constructed as in Section~\ref{sec:construction}. A \emph{complete hyper $T$-path} for $S$ with respect to $\mathcal{I}$ is a set of nodes, which correspond to the vertices of $\mathcal{G}_S$, that are joined by connections, which are hyperedges that correspond to edges and diagonals of $\mathcal{G}_S$. Connections that correspond to edges of $\mathcal{G}_S$ are called \emph{odd} and connections that correspond to diagonals of $\mathcal{G}_S$ are called \emph{even}. Each odd connection has an associated weight that is the same as the weight of the corresponding edge in $\mathcal{G}_S$.  Each even connection has an associated label that is the same as the label of the corresponding diagonal in $\mathcal{G}_S$.  For an odd connection $c$ corresponding to edge $e$, we define $W(c)=W(e)$ and for an even connection $c$ corresponding to diagonal $d$, we define $W(c)=W(d)$. Each node has the same associated valence $a_k \oplus b_k$ as the vertex it corresponds to in $\mathcal{G}_S$. This collection of nodes and connections forms a connected diagram, such that:
    \begin{enumerate}
        \item If a connection is weighted by $Y_I$, it joins nodes labeled by all of the vertices in $I$ with multiplicity one. If a connection is weighted by $Y_{I}^2$, then it joins a collection of nodes labeled by all vertices of $N_{\Gamma'}(I)$ except for the neighbor that lies on the path between an  element $v\in I$ and an  element of $u\in S$, with multiplicity two.  Note that this vertex does not depend on the choice of $v$ and $u$.
        \item There is a distinguished set of \emph{boundary nodes} labeled by elements of $N_{\Gamma'}(S)\cup\mathcal{B}$?. The set of boundary nodes must contain exactly one node labeled by each vertex in $N_{\Gamma'}(S)$ and at most one node labeled by each vertex in $\mathcal{B}$. All other nodes are called \emph{internal nodes}.
        \item A boundary node labeled by $v_k$ is adjacent to between $a_k$ and $a_k+b_k$ odd connections and no even connections.
        \item \label{TpathInternalNodeValence} An internal node labeled by $v_k$ is adjacent to exactly $a_k$ even connections and between $a_k$ and $a_k + b_k$ odd connections.
        \item If $x,y  \in N_{\Gamma'}(S)$ are such that there exist $v,w\in S$ with $x<_{\mathcal{I}} v$ and $y <_{\mathcal{I}} w$ (where possibly $v = w$), any path in any complete hyper $T$-path from boundary node $x$ to boundary node $y$ uses even connections labeled, in order, by $I_{x}$, $I_{u_\ell}$, $I_{u_{\ell-1}}, \dots, I_{u_1}, I_{w_1}, I_{w_2}, \dots, I_{w_k}, I_{y}$ where the path from $x$ to $y$ in $\Gamma'$ is $x, u_{\ell}, u_{\ell-1}, \dots, u_{1}, x \vee y, w_1, w_2, \dots, w_{k}, y$ for $k,\ell \geq 0$.
        \item If $x,y \in N_\Gamma(S)$ such that there exists $v \in S$ with $x <_{\mathcal{I}} v$ and $y >_{\mathcal{I}} m_{\mathcal{I}}(S)$, any path in any complete hyper $T$-path from boundary node $x$ to boundary node $y$ uses even connections that are labeled, in order, by $I_{x}, I_{u_\ell}, \dots, I_{u_2}$, where the path from $x$ to $y$ in $\Gamma'$ is $x, u_{\ell}, u_{\ell-1}, \dots, u_{1}, y$ for $\ell \geq 1$. If $\ell = 1$, then a path from $x$ to $y$ uses only the even connection $I_{x}$.
        \item If $x,y$ are boundary nodes and the path from $x$ to $y$ in $\Gamma'$ is $x, u_{\ell}, \dots, u_{k}, \dots, u_1,$ ${x\vee y},$ $ w_{1}, \dots, w_{g}, \dots, w_{h}, y$ where $u_{\ell}, \dots, u_{k-1} \in I \in \mathcal{I}$ such that $I \subseteq S$ and $w_{g+1}, \dots, w_{h} \in J \in \mathcal{I}$ such that $J \subseteq S$, then any path in the complete hyper $T$-path from $x$ to $y$ uses nodes labeled by the elements of $N_{\Gamma}(I_{j}^{i}(u_{k}))$ if $x \in I_{j}^{i}(u_{k})$, elements of $N_{\Gamma}(I_{j}^{i}(u_{k-1}))$ if $u_{k} \in I_{j}^{i}(u_{k-1})$, $\dots$, and elements of $N_{\Gamma}(I_{j}^{i}(w_{g}))$ if $y \in I_{j}^{i}(w_{g})$. 
        \item \label{invalidTpath1}There does not exist a cycle using connections $o_1, e_2,$ and $o_3$ such that $o_1, o_3$ are odd, $e_2$ is even, $o_1 \cap o_3 \neq \emptyset$, and $W(o_1) \subsetneq W(e_2) \subsetneq W(o_3)$.
        \item \label{invalidTpath2} There does not exist a cycle using connections $o_1,e_1,o_2,$ and $e_2$ such that $o_1, o_2$ are odd, $e_1, e_2$ are even, $W(o_1) = W(e_1)$, and $W(o_2) = W(e_2)$.
        \item \label{invalidTpath3}If $\mathcal{G}_S$ contains a pair of boundary edges $e_1, e_2$ that are labeled with non-empty sets such that both share vertices with a diagonal $d$ but do not share any vertices with each other, and have $W(e_1) \subsetneq W(d) \subsetneq W(e_2)$, then there must be either no odd connections corresponding to $e_1$ and $e_2$ or there must be odd connections that correspond to both.
        \item \label{invalidTpath4} Suppose that $\mathcal{G}_{S}$ contains a pair of edges $e_1, e_2$ that share a vertex and that there exist other vertices $v_1,v_2$ in $\mathcal{G}_{S}$ such that $v_1$ is an endpoint of $e_1$ but not $e_2$ and $v_2$ is an endpoint of $e_2$ but not $e_1$.  If $v_1$ and $v_2$ are both endpoints of a diagonal labeled by $d$ where $W(e_1),W(e_2) \subseteq W(d)$ and $e_1,e_2$ are either both internal edges or both boundary edges incident to internal edges, then either both $e_1$ and $e_2$ have corresponding odd connections or neither have corresponding odd connections.
    \end{enumerate}
\end{defn}
Let $\alpha$ denote a complete hyper $T$-path. We continue to define its \emph{weight} as
\[ \wt(\alpha) := \left( \prod_{\textrm{odd connections  } c} \wt(c) \right) \left( \prod_{\textrm{even connections } c} Y_{W(c)} \right)^{-1}.  \]

 Using this more general definition of a complete hyper $T$-path, our previous expansion formula from \cite{rooted}  now extends to produce expansion formulas for all weakly rooted sets:
\[ Y_S = \sum_{\substack{\textrm{complete hyper} \\ T\textrm{-paths }\alpha \textrm{ for } S}} \wt(\alpha).\]

In our following examples, odd connections will be drawn with solid lines and even connections will be drawn with dashed lines. 

\begin{ex}
Let us return to Example~\ref{example:nonadmissibleMatching-Condition1} and consider the inadmissible matching shown in Figure~\ref{fig:nonadmissible-Condition1}. This inadmissible matching would correspond to the hyper $T$-path:
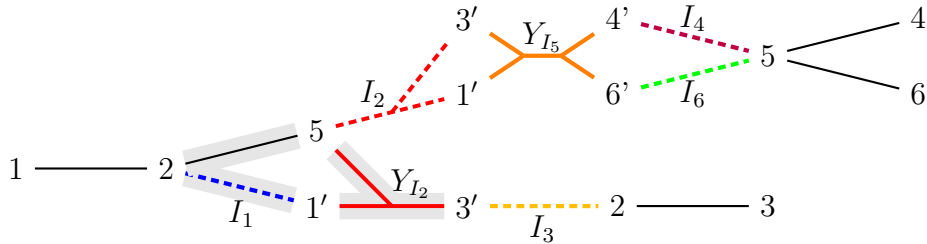
\begin{center}
    \begin{tikzpicture}
        \node (1) at (0,0) {$1$};
        \node (2) at (2,0) {$2$};
        \node (3) at (4,-0.5) {$1'$};
        \node (4) at (4,0.5) {$5$};
        \node (5) at (6,-0.5) {$3'$};
        \node (6) at (8,-0.5) {$2$};
        \node (7) at (10,-0.5) {$3$};
        \node (8) at (6,1) {$1'$};
        \node (9) at (6,2) {$3'$};
        \node (10) at (8,2) {4'};
        \node (11) at (8,1) {6'};
        \node (12) at (10,1.5) {$5$};
        \node (13) at (12,2) {4};
        \node (14) at (12,1) {$6$};

        \draw (1) to (2); 
        \draw[color=black!10!white,thick, line width = 10pt] (2) to (3); 
        \draw[thick, densely dashed, color=blue] (2) to (3); 
        \node at (3,-0.6) {${I_1}$};
        \draw[color=black!10!white,thick, line width = 10pt] (2) to (4); 
        \draw (2) to (4); 
        \draw[color=black!10!white,thick, line width = 10pt] (3) to (5); 
        \draw[color=black!10!white,thick, line width = 10pt] (4) to (5,-0.5); 
        \draw[thick,red] (3) to (5); 
        \draw[thick,red] (4) to (5,-0.5); 
        \node at (5.25,-0.2) {$Y_{I_2}$};
        \draw[thick, densely dashed, color = 7] (5) to (6); 
        \node at (7,-0.8) {${I_3}$};
        \draw (6) to (7);
        \draw[thick,red, densely dashed] (4) to (8); 
        \draw[thick,red, densely dashed] (5,0.75) to (9); 
        \node at (4.75,1) {${I_2}$};
        \draw (12) to (13); 
        \draw (12) to (14); 
        \draw[thick,densely dashed, purple] (10) to (12); 
        \node at (9,2) {${I_4}$};
        \draw[thick, densely dashed, green] (11) to (12); 
        \node at (9,1) {${I_6}$};
        \draw[thick,orange] (6.75,1.5) to (7.25,1.5);
        \draw[thick,orange] (9) to (6.75,1.5);
        \draw[thick,orange] (8) to (6.75,1.5);
        \draw[thick,orange] (7.25,1.5) to (10);
        \draw[thick,orange] (7.25,1.5) to (11);
        \node at (7,1.75) {$Y_{I_5}$};
    \end{tikzpicture}
\end{center}
 Notice that this hyper $T$-path contains a highlighted sequence of connections $(o_1,e_2,o_3)$ where $o_1$ is the $2 - 5$ connection, $e_2$ is the $2 - 1'$ connection, and $o_3$ is the odd connection between $5, 1',$ and $3'$. Notice that $o_1$ and $o_3$ share an endpoint and that $W(o_1) = \emptyset$, $W(e_2) = I_1$, and $W(o_3) = I_2$. Because $W(o_1) \subsetneq W(e_2) \subsetneq W(o_3)$, this hyper $T$-path violates Condition~(\ref{invalidTpath1}).
\end{ex}

\begin{ex}
Next, let us consider the graph and maximal nested collection from Example~\ref{example:nonadmissibleMatching-Condition2} and the snake graph $\mathcal{G}_{25}$ shown in Figure~\ref{fig:nonadmissible-Condition2}. This figure displays a non-admissible matching that corresponds to the following (invalid) hyper $T$-path:

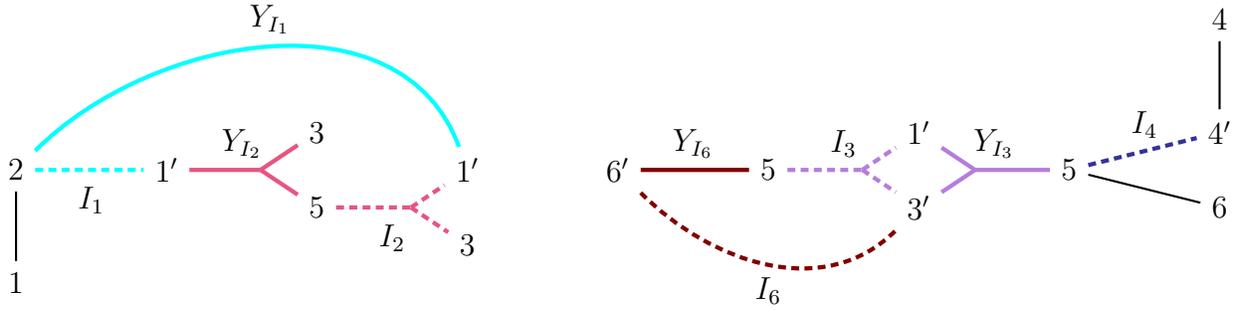
\begin{center}
    \begin{tikzpicture}
        \node (1) at (2,-1.5) {$1$};
        \node (2) at (2,0) {$2$};
        \node (3) at (4,0) {$1'$};
        \node (4) at (6,0.5) {$3$};
        \node (5) at (6,-0.5) {$5$};
        \node (6) at (8,0) {$1'$};
        \node (7) at (8,-1) {$3$};

        \node (8) at (10,0) {$6'$};
        \node (9) at (12,0) {$5$};
        \node (10) at (14,0.5) {$1'$};
        \node (11) at (14,-0.5) {$3'$};
        \node (12) at (16,0) {$5$};
        \node (13) at (18,-0.5) {$6$};
        \node (14) at (18,0.5) {$4'$};
        \node (15) at (18,2) {$4$};

        \draw (1) to (2);
        \draw[thick, densely dashed, color=1] (2) to (3);
        \node at (3,-0.4) {${I_1}$};
        \draw[thick, color=1,out = 45, in = 110] (2) to node[midway,above]{\textcolor{black}{$Y_{I_1}$}} (6);
        \draw[thick, color=2] (3) to (5.25,0);
        \draw[thick, color=2] (5.25,0) to (4);
        \draw[thick, color=2] (5.25,0) to (5);
        \node at (5,0.35) {$Y_{I_{2}}$};
        \draw[thick, densely dashed, color = 2] (5) to (7.25,-0.5);
        \draw[thick, densely dashed, color = 2] (7.25,-0.5) to (6);
        \draw[thick, densely dashed, color = 2] (7.25,-0.5) to (7);
        \node at (7,-0.9) {${I_2}$};

        \draw[thick, color = 6] (8) to (9);
        \node at (11,0.35) {$Y_{I_6}$};
        \draw[thick, densely dashed, color = 6,out=-45,in=-135] (8) to node[midway,below]{\textcolor{black}{${I_6}$}} (11);
        \draw[thick, densely dashed, color = 3] (9) to (13.25,0);
        \draw[thick, densely dashed, color = 3] (13.25,0) to (10);
        \draw[thick, densely dashed, color = 3] (13.25,0) to (11);
        \node at (13,0.35) {${I_3}$};
        \draw[thick, color = 3] (12) to (14.75,0);
        \draw[thick, color = 3] (14.75,0) to (10);
        \draw[thick, color = 3] (14.75,0) to (11);
        \node at (15,0.35) {$Y_{I_3}$};
        \draw (12) to (13);
        \draw[thick, densely dashed, color = 4] (12) to (14);
        \node at (17,0.6) {${I_4}$};
        \draw (14) to (15);
    \end{tikzpicture}
\end{center}
Note that this diagram is disconnected and also violates Condition~(\ref{invalidTpath3}), which requires that either both or neither of the pair of edges labeled by $e_1$ and $e_2$ in Figure~\ref{fig:nonadmissible-Condition2} must correspond to odd connections. In this example, only $e_1$ corresponds to an odd connection.
\end{ex}

\begin{ex}
Let us return to considering the weakly rooted set $S = \{ 3, 4, 5 ,6, 8, 0 \}$ and maximal nested collection $\mathcal{I}$ from Figure~\ref{fig:cluster}. We previously exhibited the set of admissible matchings of $\mathcal{G}_{S}$ in Example~\ref{ex:main-thm} and an example of a non-admissible matchings that violates Condition~(\ref{admissibilityCondition3}) of Definition~\ref{def:AdmissibleMatching} in Example~\ref{example:nonadmissibleMatching-Condition3}.  The corresponding hyper $T$-path for this non-admissible matching is shown below. Odd connections are solid and even connections are dashed.

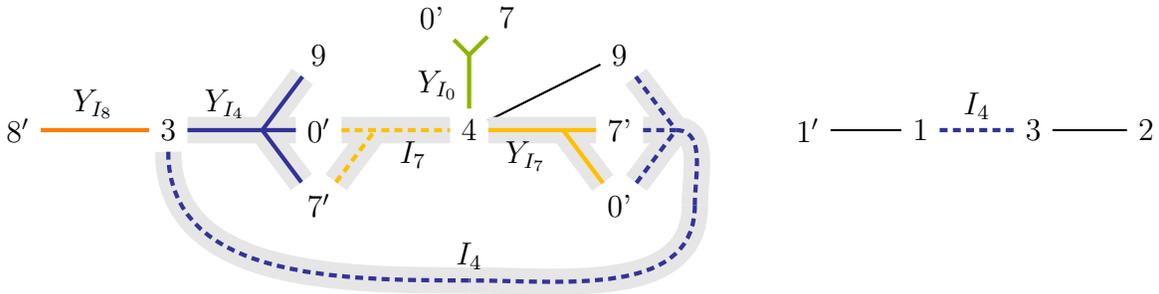
\begin{center}
\begin{tikzpicture}
    \node (1) at (-2,0) {$8'$};
    \node (2) at (0,0) {$3$};
    \node (3) at (2,0) {$0'$};
    \node (4) at (2,1) {$9$};
    \node (5) at (2,-1) {$7'$};
    \node (6) at (4,0) {$4$};
    \node (7) at (3.5,1.5) {0'};
    \node (8) at (4.5,1.5) {7};
    \node (9) at (6,1) {9};
    \node (10) at (6,0) {7'};
    \node (11) at (6,-1) {0'};
    \draw[color=8,thick] (1) to (2);
    \node[above] at (-1,0){$Y_{I_8}$};
    \draw[color=black!10!white,thick, line width = 10pt] (2) to (1.25,0); 
    \draw[color=black!10!white,thick, line width = 10pt] (1.25,0) to (3); 
    \draw[color=black!10!white,thick, line width = 10pt] (1.25,0) to (4); 
    \draw[color=black!10!white,thick, line width = 10pt] (1.25,0) to (5); 
    \draw[color=4,thick] (2) to (1.25,0);
    \draw[color=4,thick] (1.25,0) to (3);
    \draw[color=4,thick] (1.25,0) to (4);
    \draw[color=4,thick] (1.25,0) to (5);
    \node[above] at (0.75,0){$Y_{I_4}$};
    \draw[color=black!10!white,thick, line width = 10pt] (3) to (6); 
    \draw[color=black!10!white,thick, line width = 10pt] (5) to (2.75,0); 
    \draw[color=7, densely dashed,thick] (3) to (6);
    \draw[color=7, densely dashed,thick] (5) to (2.75,0);
    \node[below] at (3.25,0){${I_7}$};
    \draw[color=0,thick] (6) to (4,1);
    \draw[color=0,thick] (4,1) to (7);
    \draw[color=0,thick] (4,1) to (8);
    \node[left] at (4,0.6){$Y_{I_0}$};
    \draw (6) to (9);
    \draw[color=black!10!white,thick, line width = 10pt] (6) to (10); 
    \draw[color=black!10!white,thick, line width = 10pt] (11) to (5.25,0); 
    \draw[color=7,thick] (6) to (10);
    \draw[color=7,thick] (11) to (5.25,0);
    \node[below] at (4.75,0){$Y_{I_7}$};
    \draw[color=black!10!white,thick, line width = 10pt] (10) to (6.75,0); 
    \draw[color=black!10!white,thick, line width = 10pt,out=0,in=90] (6.75,0) to (7,-1); 
    \draw[color=black!10!white,thick, line width = 10pt,out=-90,in=180] (2) to (4,-2); 
    \draw[color=black!10!white,thick, line width = 10pt, out=-90, in =0]  (7,-1) to (4,-2); 
    \draw[color=black!10!white,thick, line width = 10pt] (9) to (6.75,0); 
    \draw[color=black!10!white,thick, line width = 10pt] (11) to (6.75,0); 
    \draw[color=4, densely dashed,thick] (10) to (6.75,0);
    \draw[color=4, densely dashed,out=0,in=90,thick] (6.75,0) to (7,-1);
    \draw[color=4, densely dashed,out=-90,in=180,thick] (2) to (4,-2);
    \draw[color=4, densely dashed, out=-90, in =0,thick]  (7,-1) to (4,-2);
    \draw[color=4, densely dashed,thick] (9) to (6.75,0);
    \draw[color=4, densely dashed,thick] (11) to (6.75,0);
    \node[above] at (4,-2){${I_4}$};

    \node (12) at (8.5,0) {$1'$};
    \node (13) at (10,0) {$1$};
    \node (14) at (11.5,0) {$3$};
    \node (15) at (13,0) {$2$};
    \draw (12) to (13);
    \draw (14) to (15);
    \draw[color=4,thick,densely dashed] (13) to (14);
    \node[above] at (10.75,0){$I_4$};
\end{tikzpicture}
\end{center}

Note that this diagram both is disconnected and violates Condition~(\ref{invalidTpath2}) of Definition~\ref{def:hyperTpath} due to the highlighted cycle, which uses connections $o_1, e_1, o_2, e_2$ such that $W(o_1) = W(e_1) = Y_{I_7}$ and $W(o_2) = W(e_2) = Y_{I_4}$. Hence, it is not a valid hyper $T$-path for $S$.

Next, we exhibit two hyper $T$-paths that do correspond to admissible matchings of $\mathcal{G}_{345680}$. We begin with the valid hyper $T$-path with weight $\frac{Y_{8}Y_{45670}Y_{560}^2}{Y_{12}Y_{2}Y_{5670}}$. 

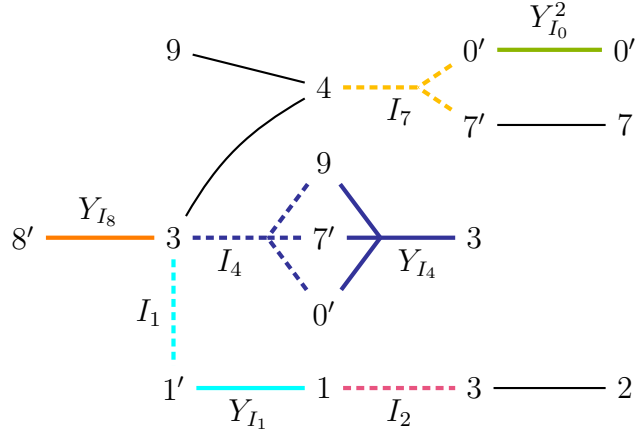
\begin{center}
\begin{tikzpicture}
    \node (0) at (12,0) {$2$};
    \node (1) at (10,0) {$3$};
    \node (2) at (8,0) {$1$};
    \node (3) at (6,0) {$1'$};
    \node (4) at (6,2) {$3$};
    \node (5) at (4,2) {$8'$};
    \node (6) at (8,2) {$7'$};
    \node (7) at (8,3) {$9$};
    \node (8) at (8,1) {$0'$};
    \node (9) at (10,2) {$3$};
    \node (10) at (8,4) {$4$};
    \node (11) at (6,4.5) {$9$};
    \node (12) at (10,3.5) {$7'$};
    \node (13) at (10,4.5) {$0'$};
    \node (14) at (12,3.5) {$7$};
    \node (15) at (12,4.5) {$0'$};

    \draw[color=8, thick] (5) to (4); 
    \node[above] at (5,2){$Y_{I_8}$};
    \draw[color=1, thick, densely dashed] (4) to (3);
    \node[left] at (6,1){${I_1}$};
    \draw[color=1, thick] (3) to (2);
    \node[below] at (7,0){$Y_{I_1}$};
    \draw[color=2, thick, densely dashed] (2) to (1);
    \node[below] at (9,0){${I_2}$};
    \draw (1) to (0);

    \draw[color=4, thick, densely dashed] (4) to (7.25,2);
    \node[below] at (6.75,2){${I_4}$};
    \draw[color=4, thick, densely dashed] (6) to (7.25,2);
    \draw[color=4, thick, densely dashed] (7) to (7.25,2);
    \draw[color=4, thick, densely dashed] (8) to (7.25,2);

    \node[below] at (9.25,2){$Y_{I_4}$};
    \draw[thick, color=4] (8.75,2) to (9);
    \draw[thick, color=4] (8.75,2) to (6);
    \draw[thick, color=4] (8.75,2) to (7);
    \draw[thick, color=4] (8.75,2) to (8);

    \draw[out = 60, in = 210] (4) to (10); 
    \draw (10) to (11);
    \draw[color=7, thick, densely dashed] (10) to (9.25,4); 
    \draw[color=7, thick, densely dashed] (12) to (9.25,4); 
    \draw[color=7, thick, densely dashed] (13) to (9.25,4); 
    \node[below] at (9,4){${I_7}$};
    \draw[color=0, thick] (13) to (15);
    \node[above] at (11,4.5){$Y_{I_0}^2$};
    \draw (12) to (14);
\end{tikzpicture}
\end{center}

Note that the set of boundary nodes of the hyper $T$-path with weight $Y_{8}Y_{45670}Y_{0}^2$ is larger than $S' = \{2, 7, 8', 9, 0' \}$: it also contains a vertex labeled by $3$ because there is a vertex labeled 3 with valence $0 \oplus 1$ in $\mathcal{G}_{345680}$. This is allowed in our more general definition for hyper $T$-paths for weakly rooted sets, but was not allowed in the definition of hyper $T$-paths for rooted clusters in \cite{rooted}. Additionally, we see that there is a vertex labeled 3 with  more than one adjacent even connection. In $\mathcal{G}_{345680}$, this vertex had valence $2 \oplus 0$ and therefore Condition~(\ref{TpathInternalNodeValence}) of our more general definition for hyper $T$-paths requires that it be adjacent to exactly two even connections. This would not have been allowed by the definition in \cite{rooted}.

Then we consider the valid hyper $T$-path with weight $\frac{Y_{123456780}Y_{45670}}{Y_{12}Y_{5670}}$. 

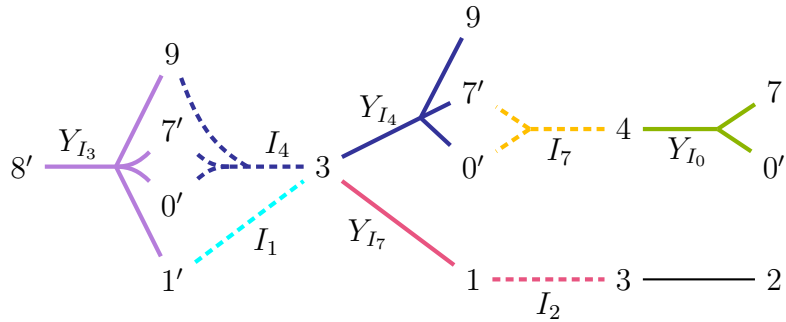
\begin{center}
\begin{tikzpicture}
        \node (0) at (0,0) {$8'$};
        \node (1) at (2,-0.5) {$0'$};
        \node (2) at (2,-1.5) {$1'$};
        \node (3) at (2,0.5) {$7'$};
        \node (4) at (2,1.5) {$9$};
        \node (5) at (4,0) {$3$};
        \node (6) at (6,-1.5) {$1$};
        \node (7) at (6,0) {$0'$};
        \node (8) at (6,1) {$7'$};
        \node (9) at (6,2) {$9$};
        \node (10) at (8,0.5) {$4$};
        \node (11) at (10,1) {$7$};
        \node (12) at (10,0) {$0'$};
        \node (13) at (8,-1.5) {$3$};
        \node (14) at (10,-1.5) {$2$};

        
        \draw[color=3,thick] (0) to (1.25,0); 
        \draw[color=3,thick,out=0,in=135] (1.25,0) to (1); 
        \draw[color=3,thick] (1.25,0) to (2); 
        \draw[color=3,thick,out=0,in=-135] (1.25,0) to (3); 
        \draw[color=3,thick] (1.25,0) to (4); 
        \node at (0.75,0.3) {$Y_{I_3}$}; 

        \draw[color=4,thick,densely dashed] (5) to (2.75,0); 
        \draw[color=4,thick,densely dashed,in=-70,out=150] (3,0) to (4); 
        \draw[color=4,thick,densely dashed,out=180,in=45] (2.75,0) to (1); 
        \draw[color=4,thick,densely dashed,out=180,in=-45] (2.75,0) to (3); 
        \node at (3.4,0.3) {${I_4}$}; 

        \draw[color=1,thick,densely dashed] (2) to (5); 
        \node at (3.25,-1) {${I_1}$};

        \draw[color=2,thick] (5) to (6); 
        \node at (4.6,-0.9) {$Y_{I_7}$}; 
        \draw[color=2,thick,densely dashed] (6) to (13); 
        \node at (7,-1.85) {${I_2}$}; 
        \draw (13) to (14); 

        \draw[color=4,thick] (5) to (8); 
        \draw[color=4,thick] (5.3, 0.65) to (9); 
        \draw[color=4,thick] (5.3, 0.65) to (7); 
        \node at (4.75,0.75) {$Y_{I_4}$};

        \draw[color=7,thick,densely dashed] (10) to (6.75,0.5); 
        \draw[color=7,thick,densely dashed] (6.75,0.5) to (8); 
        \draw[color=7,thick,densely dashed] (6.75,0.5) to (7); 
        \node at (7.15,0.2) {${I_7}$};

        \draw[color=0,thick] (10) to (9.25,0.5); 
        \draw[color=0,thick] (9.25,0.5) to (11); 
        \draw[color=0,thick] (9.25,0.5) to (12); 
        \node at (8.85,0.2) {$Y_{I_0}$}; 
\end{tikzpicture}
\end{center}

Notice that this hyper $T$-path also contains a node (labeled by 3) that has multiple adjacent even connections.
\end{ex}

\begin{ex}
Consider the graph and maximal nested collection from Example~\ref{example:nonadmissibleMatching-Condition4}. The inadmissible matching shown in Figure~\ref{fig:nonadmissible-Condition4} would correspond to the (invalid) hyper $T$-path shown below.

    
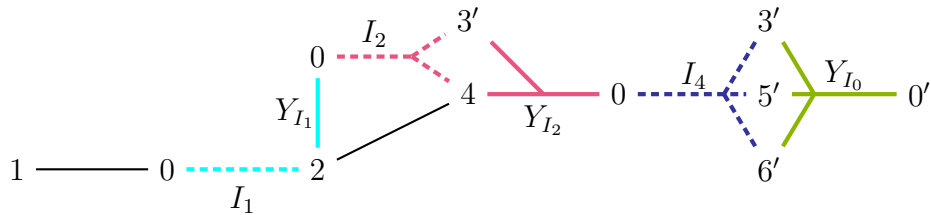
\begin{center}
    \begin{tikzpicture}
        \node (1) at (0,0) {$1$};
        \node (2) at (2,0) {$0$};
        \node (3) at (4,0) {$2$};
        \node (4) at (4,1.5) {$0$};
        \node (5) at (6,1) {$4$};
        \node (6) at (6,2) {$3'$};
        \node (7) at (8,1) {$0$};
        \node (8) at (10,1) {$5'$};
        \node (9) at (10,2) {$3'$};
        \node (10) at (10,0) {$6'$};
        \node (11) at (12,1) {$0'$};

        \draw (1) to (2); 
        \draw[thick, densely dashed, color=1] (2) to (3); 
        \node at (3,-0.4) {${I_1}$}; 
        \draw[thick, color=1] (3) to (4); 
        \node at (3.7,0.75) {$Y_{I_1}$};
        \draw (3) to (5); 
        \draw[thick, densely dashed,color=2] (4) to (5.25,1.5); 
        \draw[thick, densely dashed,color=2] (5.25,1.5) to (6); 
        \draw[thick, densely dashed,color=2] (5.25,1.5) to (5); 
        \node at (4.75,1.8) {${I_2}$};
        \draw[thick,color=2] (5) to (7); 
        \draw[thick,color=2] (6) to (7,1); 
        \node at (7,0.65) {$Y_{I_2}$};
        \draw[thick,color=4, densely dashed] (7) to (8); 
        \draw[thick,color=4, densely dashed] (9.4,1) to (9); 
        \draw[thick,color=4, densely dashed] (9.4,1) to (10); 
        \node at (9,1.25) {$I_4$};
        \draw[thick, color=0] (8) to (11); 
        \draw[thick, color=0] (9) to (10.6,1); 
        \draw[thick,color=0] (10) to (10.6,1); 
        \node at (11, 1.25) {$Y_{I_0}$};
    \end{tikzpicture}
    \end{center}
From Example~\ref{example:nonadmissibleMatching-Condition1}, we know that $\mathcal{G}_{0}$ contains internal edges $2 - 4$ and $2 - 3'$ such that an admissible matching of $\mathcal{G}_{0}$ must use either both or neither of these edges. Hence, a valid hyper $T$-path should contain odd connections that correspond to both or neither of these internal edges. The hyper $T$-path displayed above contains an odd connection that corresponds to the internal $2 - 4$ edge, but does not contain one that corresponds to the internal $2 - 3'$ edge, violating Condition~(\ref{invalidTpath4}).
\end{ex}

\section{Proof for Singleton Sets}
\label{sec:SingeltonProof}

This section and the following section consist of a proof of \Cref{thm:main}.

We first recall some relevant notation from Section~\ref{subsec:SingletonConstruction1} that will be used extensively in the next two sections:
\begin{itemize}
    \item We have $N_{\Gamma'}(v)=\{a_1,\dots,a_p\}$ with $a_i<_{\mathcal I} v$ when $1\leq i\leq r$ and $a_i>_{\mathcal I} v$ when $r< i\leq p$ for some $r$.
    \item For $1\leq i\leq p$, we define $c_i = \vert \{I \in \mathcal{I} : I_{a_i} \subseteq I \subsetneq I_v\} \vert$.
    \item The sets counted by $c_i$ are $I^i_j$ such that $I_{a_i} = I^i_{c_i} \subset I^i_{c_i-1} \subset \cdots \subset I^i_1$.
    \item We have $a_j^i := m_{\mathcal{I}}(I_j^i)$ for $1\leq i\leq p$ and $1\leq j\leq c_i$.  In other words, $I^i_j=I_{a^i_j}$.  Notice that 
    $a_{c_i}^i=a_i$, and $\mathcal{C}_v = \{a_1^1,\ldots,a_1^r\}$.
\end{itemize}

In addition, we will define the following: 
\begin{itemize}
    \item $I_{c_i+1}^i := \emptyset$ for $1\leq i\leq p$
    \item For all $1 \leq i \leq r$,
        \begin{align*}
        A_j^i := \begin{cases} \bigg( \prod_{w \in \mathcal{C}_{a_{j}^i} \backslash \{a_{j+1}^i\}} Y_{I_w} \bigg) \cdot Y_{I_{j+1}^i - [v,a_j^i]}^2 & j < c_i \\
        \prod_{w \in \mathcal{C}_{a_{c_i}^i}} Y_{I_w} & j = c_i \\
        \end{cases}
        \end{align*}
\end{itemize}

This section contains a proof of \Cref{thm:main} for singleton sets; recall all singleton sets are weakly rooted.  In Section~\ref{subsec:SingletonFormulas}, we give an explicit formula for variables $Y_v$. In fact, we give a more general formula for $Y_{S}$ where $S = \{v\} \cup I_{m_1}^1(v) \cup \cdots \cup I_{m_r}^r(v)$ for values $1 \leq m_i \leq c_i+1$ by applying the exchange relations for graph LP algebras to a specific mutation sequence; recall the graphs $\mathcal{G}_S$ were constructed in Section \ref{sec:Gluing}. In Section~\ref{subsec:ris1}, we prove Theorem~\ref{thm:main} for singleton sets $\{ v \}$ where $r = 1$ or $r=2$. In Section~\ref{subsec:risnot1}, we then prove Theorem~\ref{thm:main} for singleton sets $\{ v \}$ where $r > 2$.

\subsection{Formulas in $\mathcal{A}_\Gamma$}\label{subsec:SingletonFormulas}

When the \emph{rooted portion} $\bar{S}$ of a set $S$ that is weakly rooted with respect to a maximal nested collection $\mathcal{I}$ is a singleton, we can compute $Y_{\bar{S}}$ directly by exhibiting a mutation sequence from $\mathcal{I}$ to another maximal nested collection that contains $\bar{S}$.

In the following proposition, we will  consider sets of the form  $J_{\mathbf{m}}(v) = \{v\} \cup I_{m_1}^1 \cup \cdots \cup I_{m_r}^r$ where $1 \leq m_i \leq c_i+1$ for each $i$.  When the vertex $v$ is understood, we will denote this set simply as $J_{\mathbf{m}}$. For convenience, we will assume $a_1,\dots,a_r$ are labeled such that $m_i > 1$ for $1 \leq i \leq t$ and $m_i = 1$ for $t+1 \leq i \leq r$ for some $0 \leq t \leq r$.

\begin{prop}\label{prop:SingletonAdjoinSets}
For any choice of $1 \leq m_i \leq c_i+1$ for each $1 \leq i \leq r$, let $J_{\mathbf{m}} =\{v\} \cup I_{m_1}^1 \cup \cdots \cup I_{m_r}^r$. Then, \[
Y_{J_{\mathbf{m}}} = Y_{J_{\mathbf{m}} \backslash \{v\}} \bigg( \frac{Y_{I_v}}{Y_{I_1^1}\cdots Y_{I_1^r}} + \sum_{i=1}^r \sum_{j = 1}^{m_i-1} \frac{A_j^i}{Y_{I_j^i} Y_{I_{j+1}^i}} \bigg).
\]

\end{prop}

\begin{proof}

If $t = 0$, $J_{\mathbf{m}}=I_v \in \mathcal{I}$, so this case is trivial.  For $t > 0$ we exhibit a mutation sequence that results in a cluster $\mathcal{I}'$ containing $J_{\mathbf{m}}$.  Namely, this sequence (written left-to-right) is given by mutating at $I_1^1 I_2^1, \dots, I_{m_1-1}^1, I_1^2, I_2^2,\dots,I_{m_t - 2}^t,  I_{m_t-1}^t$.

For $0 \leq i \leq t$, let $W^i = I_{m_1}^1 \cup I_{m_2}^2 \cup \cdots \cup I_{m_{i-1}}^{i-1} \cup I_1^{i+1} \cup \cdots \cup I_1^r$.  Define ${\widetilde I^i_j}:=W^{i} \cup I_{j+1}^i \cup \{v\}$ for $0\leq j \leq m_i$.  One can check that  $\widetilde I_{0}^{i+1} = \widetilde I_{m_i-1}^i$.  Observe that each mutation in this sequence replaces $I_j^i$ with ${\widetilde I^i_j}$. Each $\widetilde I_j^i$ is a connected vertex subset because every connected component of $W^{i}$ is adjacent to $v$, as is $I_{j+1}^{i}$.

Using Lemma 4.11 of \cite{lam2016linear}, we obtain exchange relations of the form 
\begin{equation*}
Y_{\widetilde I^i_j}Y_{I^i_j}=Y_{\widetilde I^i_{j-1}}Y_{I^i_{j+1}} + Y_{W^{i}}\bigg(\prod_{w \in {\cal C}_{a_j^i} \backslash \{a_{j+1}^i\}} Y_{I_w}\bigg)Y_{I_{j+1}^i - [v,a_j^i]}^2 = Y_{\widetilde I^i_{j-1}}Y_{I^i_{j+1}} + Y_{W^{i}}A_j^i,
\end{equation*}
for $1 \leq i \leq t$ and $1 \leq j \leq m_i$.

The final mutation in our sequence produces the relation
$$Y_{\widetilde I^t_{m_t-1}}Y_{I^t_{m_t-1}}= Y_{\widetilde I^t_{m_t-2}}Y_{I^t_{m_t}} + Y_{W^{t}}A_{m_t-1}^t.$$
Because $Y_{\widetilde I^t_{m_t-1}}=Y_{J_{\bf m}}$, this can be equivalently written as
$$Y_{J_{\bf m}}= Y_{\widetilde I^t_{m_t-2}}\frac{Y_{I^t_{m_t}}}{Y_{I^t_{m_t-1}}} + \frac{Y_{W^{t}}}{Y_{I^t_{m_t-1}}}A_{m_t-1}^t.$$ 
We can then substitute the formula we have for $Y_{\widetilde I^t_{m_t-2}}$ from the previous mutation to obtain another formula for $Y_{J_{\bf m}}$, which contains $Y_{\widetilde I^t_{m_t-3}}$  if $m_t>2$ and $Y_{\widetilde I^{t-1}_{m_{t-1}-1}}$ if $m_t=2$. Via repeated recursive substitution using our formulas from previous mutations, we will eventually produce a formula for $Y_{J_{\bf m}}$ in terms the $A^i_j$'s and the variables in the original cluster.

Because every mutation has exactly one term containing $Y_{\widetilde I^i_j}$, we will eventually make the substitution $$Y_{\widetilde I_1^1} = Y_{I_v} \frac{Y_{I_2^1}}{Y_{I_1^1}} + \frac{Y_{W^1}}{Y_{I_1^1}}A^1_1$$ in exactly one place. Hence, our formula for $Y_{J_{\bf m}}$ will have exactly one term containing $Y_{I_{v}}$, which we can compute as
$$Y_{I_v}\cdot\frac{Y_{I_2^1}}{Y_{I_1^1}} \cdot \frac{Y_{I_3^1}}{Y_{I_2^1}}\cdot \frac{Y_{I_4^1}}{Y_{I_3^1}} \cdots \frac{Y_{I_{m_1}^1}}{Y_{I_{m_1-1}^1}}\cdot \frac{Y_{I_2^2}}{Y_{I_1^2}}\cdot \cdots \cdot \frac{Y_{I_{m_t}^t}}{Y_{I_{m_t-1}^t}} = \frac{Y_{J_{\mathbf{m}} \backslash \{v\}} Y_{I_v}}{Y_{I_1^1}\cdots Y_{I_{1}^t}}.$$

For each mutation, the term that does not contain a factor $Y_{\widetilde I^i_j}$ is divisible by $A^i_j$.  This means our formula for $Y_{J_{\bf m}}$ will have exactly one term containing $A^i_j$ for $1\leq i\leq t$ and $1\leq j\leq m_i$.  Tracking this term through further substitutions, we can see that it has the form
\[
\frac{Y_{W^{i}}A_j^i}{Y_{I_j^i}} \cdot \frac{Y_{I_{j+2}^i}}{Y_{I_{j+1}^i}}\cdot\frac{Y_{I_{j+3}^i}}{Y_{I_{j+2}^i}}\cdots \frac{Y_{I_{m_i}^i}}{Y_{I_{m_i-1}^i}}\cdot\frac{Y_{I_2^{i+1}}}{Y_{I_1^{i+1}}}\cdot\frac{Y_{I_3^{i+1}}}{Y_{I_2^{i+1}}}\cdots  \frac{Y_{I_{m_t}^t}}{Y_{I_{m_t-1}^t}} = \frac{ Y_{J_{\mathbf{m}} \backslash \{v\}} A_i^j }{Y_{I_j^i}Y_{I_{j+1}^i}},
\]
using the fact that $Y_{W^{i+1}}$ has factors $ Y_{I_1^{i+1}} Y_{I_1^{i+2}} \cdots Y_{I_1^r}$, which cancel with factors in the denominator. 

This accounts for all terms in this sequence of mutations, completing the proof.
\end{proof}

By setting all $m_i = c_i+1$, we obtain a formula for variables of the form $Y_{\{v\}}$.

\begin{cor}\label{cor:SingletonFormula}
Let $v \in V(\Gamma)$. Then \[
Y_{\{v\}} = \frac{Y_{I_v}}{Y_{I_1^1}\cdots Y_{I_1^r}} + \sum_{i = 1}^r \sum_{j = 1}^{c_i} \frac{A_j^i}{Y_{I_{j}^i} Y_{I_{j+1}^i}}
\]

where $I_{c_i+1}^i := \emptyset$. 
\end{cor}

It will be convenient to use the shorthand $Z(v) := \frac{Y_{I_v}}{\prod_{i \in \mathcal{C}_v} Y_{I_i}}$ for $v \in V(\Gamma)$. Using this notation, we can rewrite the statement of Corollary \ref{cor:SingletonFormula} as $Y_{\{v\}} = Z(v) + \sum_{i=1}^r \sum_{j=1}^{c_i} \frac{A_j^i}{Y_{I_j^i}Y_{I_{j+1}^i}}$.

\begin{ex}
Consider the nested collection from our running example in \Cref{fig:cluster}.  By applying Proposition \ref{prop:SingletonAdjoinSets} with $v=4$ and $m_1 = 2$, we can compute $Y_{4560}$ as

\begin{align*}
 Y_{\{4\}\cup I^1_2} &= Y_{I^1_2}\left(\frac{Y_{I_4}}{Y_{I^1_1}}+\frac{A^1_1}{Y_{I^1_1}Y_{I^1_2}}\right)\\
 Y_{\{4\}\cup I^1_2} &= Y_{I^1_2}\left(\frac{Y_{I_4}}{Y_{I^1_1}}+\frac{(\prod_{w\in\mathcal{C}_7\setminus\{0\}}Y_{I_w})Y^2_{I^1_2\setminus[4,7]}}{Y_{I^1_1}Y_{I^1_2}}\right)\\
 Y_{4560} &= Y_{560}\left(\frac{Y_{45670}}{Y_{5670}}+\frac{Y^2_0}{Y_{5670}Y_{560}}\right)\\
 Y_{4560} &= \frac{Y_{45670}Y_{560}+Y^2_0}{Y_{5670}}\\
\end{align*}

By Corollary \ref{cor:SingletonFormula}, we can write $Y_4$ as 
\begin{align*}
Y_4 &= \frac{Y_{I_4}}{Y_{I_1^1
}} + \frac{A^1_1}{Y_{I^1_1}Y_{I^1_2}} + \frac{A^1_2}{Y_{I^1_2}Y_{I^1_3}} + \frac{A^1_3}{Y_{I^1_3}Y_{\emptyset}}\\
&= \frac{Y_{I_4}}{Y_{I_1^1
}} + \frac{(\prod_{w\in\mathcal{C}_7\setminus\{0\}}Y_{I_w})Y^2_{I^1_2\setminus[4,7]}}{Y_{I^1_1}Y_{I^1_2}} + \frac{(\prod_{w\in\mathcal{C}_0\setminus\{5\}}Y_{I_w})Y^2_{I^1_3\setminus[4,0]}}{Y_{I^1_2}Y_{I^1_3}} + \frac{\prod_{w\in\mathcal{C}_5}Y_{I_w}}{Y_{I^1_3}Y_{\emptyset}}\\
&= \frac{Y_{45670}}{Y_{5670
}} + \frac{Y^2_0}{Y_{5670}Y_{560}} + \frac{Y^2_6}{Y_{560}Y_{56}} + \frac{Y_6}{Y_{56}}
\end{align*}
\end{ex}

\subsection{r = 1, 2}
\label{subsec:ris1}
We first prove our main theorem is true for snake graphs $\mathcal{G}_v$ when $r=1$ or $r=2$.
Following \Cref{sec:construction}, we have $A := \{\ell\in \Gamma'\setminus \Gamma\ |\ a_1\in [v,l]\}$, $B := \{k\in \Gamma'\setminus \Gamma\ |\  a_1\notin [v,k]\}$, and $b_{i,j}^{\ell,k}$ denotes the vertex other than $v$ in  $N_{\Gamma'}(I_j^i) \cap [\ell,k]$.

\begin{lemma}\label{lem:r1valence}
When $r=1$ or $r=2$, in $\cal G_v$: 
\begin{itemize}
\item The vertex $v_\opt$ has valence $1\oplus p-2$.
\item The vertex $a_i$ has valence $1\oplus \deg_{\Gamma'}(a_i) - 2$ for $1\leq i\leq r$.
\item The vertex $a_j^i$ incident to diagonal $I_{j+1}^i$ has valence $1\oplus \deg_{\Gamma'}(a_j^i) - 2$ for each $1\leq i\leq r$ and $1\leq j\leq c_i$.
\item The rest of the vertices are either labeled $v$, $a_j$ for some $r< j\leq p$, or $b_{i,j}^{\ell,k}$ for some $\ell\in A, k\in B, 1\leq i\leq r$, and $1\leq j\leq c_i$.  These vertices all have valence $1\oplus 0$.
\end{itemize}
\end{lemma}

\begin{proof}

Recall that the component snake graphs $H_{\ell,k}$ have the form shown in \Cref{fig:singleton-zigzag} or in \Cref{fig:singleton-zigzag-partial}.

First we consider the vertices in $\cal G_v$ with label $v$ except for $v_\opt$. Each vertex $v$ except for the vertex $v_{\text{opt}}$ is incident to edges with the same weights in each component snake graph.  We know that as we glue the component snake graphs to form  $\cal G_v$, these edges with the same weights will be glued into one hyperedge.  Thus, each corresponding vertex in $\mathcal{G}_v$ will have valence $1 \oplus 0$.

We then consider the vertex $v_{\opt}$, which has degree two in each of the component snake graph. One of these edges always has weight $Y_{I^1_2}$.  In the process of creating $\mathcal{G}_v$, these will get combined.  If $r=2$, we will also combine the edges $\{(v_{\opt},b_{i,1}^{\ell,k}):\ell\in A,k\in B, 2\leq i\leq r\}$ which will create a single hyperedge with label $I^2_1$.  For both $r=1$ and $r=2$, there will also be component snake graphs with edges from the set $\{(v_\opt,a_i):r < i\leq p\}$ which will not be identified for different $i$, so that $v_\opt$ is incident to $p$ edges in $\cal G_v$. Thus, we see that $v_\opt$ has valence $1\oplus (p-2)$ in $\mathcal{G}_v$. The vertices $a_i$ for $r < i \leq p$ are always incident to these edges as well as the edge $Y_{I_v}$, so they will have valence $1\oplus 0$.

Next, we look at what happens to a vertex $b_{1,j}^{\ell,k}$ from $H_{\ell,k}$.  First suppose $1<j<c_1$.  Then the vertex $b_{1,j}^{\ell,k}$  in $H_{\ell,k}$ will be incident to two different vertices $v$, a vertex $b_{1,j-1}^{\ell,k}$, and a vertex $b_{1,j+1}^{\ell,k}$.  Consider some other $H_{\ell',k'}$ where $b_{1,j}^{\ell',k'}=b_{1,j}^{\ell,k}$. This vertex will also have two edges incident to two different vertices $v$.  These will get  identified with the corresponding edges from $H_{\ell,k}$ when we create $\mathcal{G}_v$.  For all $\ell',k'$ with $b_{1,j}^{\ell',k'}=b_{1,j}^{\ell,k}$, we will have $b_{1,j+1}^{\ell',k'}=b_{1,j+1}^{\ell,k}$ because $N_{\Gamma'}(I^1_{j+1})\cap[\ell,k]=N_{\Gamma'}(I^1_{j+1})\cap[b_{1,j}^{\ell,k},v]=N_{\Gamma'}(I^1_{j+1})\cap[\ell',k']$.  This means that the edges $(b_{1,j}^{\ell,k}, b_{1,j+1}^{\ell,k})$ and $(b_{1,j}^{\ell',k'}, b_{1,j+1}^{\ell',k'})$ will also get identified when constructing  $\mathcal{G}_v$.  Thus, any additional valence comes from edges $(b_{1,j}^{\ell',k'}, b_{1,j-1}^{\ell',k'})$.

If $b_{1,j}^{\ell,k}\in I^1_{j-1}$, then $b_{1,j}^{\ell,k}=a_{j-1}^1$. Since the path $[\ell',k']$ passes through $a_{j-1}^1$, it will pass through two neighbors of $a_{j-1}^1$: the unique neighbor in $[a_{j-1}^1,v]$ and another. This second neighbor, call it $w$, is either in $I_{j-1}^1$ or not. If $w \in I_{j-1}^1$, then $w$ is contained in a set $I_z$ for $z \in \mathcal{C}_{a_{j-1}^1}$ and the path $[\ell',k']$ will pass through an element of $N_{\Gamma'}(I_z) \backslash \{a_{j-1}^1\}$. Thus, the edge $(b_{1,j-1}^{\ell',k'},b_{1,j}^{\ell',k'}) = (b_{1,j-1}^{\ell',k'},a_{j-1}^1)$ will have weight $Y_{I_z}$. Otherwise, $w$ is not in $I_{j-1}^1$. In this case, it must be that $(b_{1,j-1}^{\ell',k'},b_{1,j}^{\ell',k'})=(w,a_{j-1}^1) \in E(\Gamma)$, so we weight this edge with 1. Notice that in $\mathcal{G}_v$, all edges with weight $Y_z$  will get combined for each $z$ and the edges of weight 1 will all remain separate edges.  This means that there will be $\deg_{\Gamma'}(a_{j-1}^1)-1$ edges adjacent to the vertex $a_{j-1}^1$ in $\mathcal{G}_v$ that came from edges $(b_{1,j}^{\ell',k'}, b_{1,j-1}^{\ell',k'})$ with $b_{1,j}^{\ell',k'}=a_{j-1}^1$.  This gives the vertex $a_{j-1}^1$ a valence of $1\oplus\deg_{\Gamma'}(a_{j-1}^1)-2$.

If $b_{1,j}^{\ell,k}\not\in I^1_{j-1}$, then $b_{1,j}^{\ell,k}\in N_{\Gamma'}(I^1_{j-1})$.  Notably, in this case $b_{1,j}^{\ell,k}\neq a_{j-1}^1$ and $b_{1,j}^{\ell,k} = b_{1,j-1}^{\ell,k}$. Since the vertex $b_{1,j+1}^{\ell,k}$ is also determined by $b_{1,j}^{\ell,k}$, we know this vertex is incident to the same edges in all component snake graphs where it appears, so $b_{1,j}^{\ell,k}$ will have valence $1\oplus0$ in $\mathcal{G}_v$.

Now consider the vertex $b_{1,c_1}^{\ell,k}$ in $H_{\ell,k}$.  If $c_1=1$ then $b_{1,c_1}^{\ell,k}$ will either be incident to a vertex $a_1$ and a vertex $a_i$ for some $r<i$ or it will be incident to a vertex $a_1$, a vertex $b_{i,1}^{\ell,k}$ for some $1<i\leq r$, and a vertex $v$.  Consider some other $H_{\ell',k'}$ where $b_{1,c_1}^{\ell',k'}=b_{1,c_1}^{\ell,k}$. This vertex will either have edges incident to a vertex $a_1$ and a vertex $a_{i'}$ for some $r<i'$ or it will have edges incident to a vertex $a_1$, a vertex $b_{i',1}^{\ell',k'}$ for some $1<i'\leq r$, and a vertex $v$.  Note that the edge from $b_{1,c_1}^{\ell,k}$ to $a_i$ or $b_{i,1}^{\ell,k}$ and the edge from $b_{1,c_1}^{\ell',k'}$ to $a_{i'}$ or $b_{i',1}^{\ell',k'}$ will both have weight $Y_{I_v}$.  Thus the edges extending from $b_{1,c_1}^{\ell',k'}$ in each direction will get identified with the corresponding edges from $H_{\ell,k}$ when we create $\mathcal{G}_v$. So $b_{1,c_1}^{\ell,k}$ will have valence $1\oplus 0$. Note that in this case $b_{1,c_1}^{\ell,k}\neq a^i_j$ for any $i,j$ as $b_{1,c_1}^{\ell,k}=b_{1,1}^{\ell,k}\not\in I_1^i$ for any $i$.  If $c_1\neq 1$, the vertex $b_{1,c_1}^{\ell,k}$ in $H_{\ell,k}$ will be incident to a vertex $v$, a vertex $a_1$, and a vertex $b_{1,c_1-1}^{\ell,k}$.  Consider some other $H_{\ell',k'}$ where $b_{1,c_1}^{\ell',k'}=b_{1,c_1}^{\ell,k}$. This vertex will have edges incident to a vertex $v$ and a vertex $a_1$ which will get identified with the corresponding edges from $H_{\ell,k}$ when we  create $\mathcal{G}_v$.  The logic from case of $b_{i,j}^{\ell,k}$ with $1<j<c_1$ applies here to obtain our conclusion.

Our next case is vertices $b_{1,1}^{\ell,k}$.  Consider the vertex $b_{1,1}^{\ell,k}$ in $H_{\ell,k}$.  If $c_1=1$ we are done by a previous case.  Otherwise, this vertex will be incident to two vertices $v$, a vertex $b_{1,2}^{\ell,k}$, and either a vertex $a_i$ for some $i>r$ or a vertex $b_{i,1}^{\ell,k}$ for some $1<i\leq r$.  Consider some other $H_{\ell',k'}$ where $b_{1,1}^{\ell',k'}=b_{1,1}^{\ell,k}$. This vertex will have edges incident to two vertices $v$, a vertex $b_{1,2}^{\ell',k'}$, and either a vertex $a_{i'}$ for some $i'>r$ or a vertex $b_{i',1}^{\ell',k'}$ for some $1<i'\leq r$.  The edge from $b_{1,1}^{\ell,k}$ to $a_i$ or $b_{i,1}^{\ell,k}$ and the edge from $b_{1,1}^{\ell',k'}$ to $a_{i'}$ or $b_{i',1}^{\ell',k'}$ will both have weight $Y_{I_v}$, so they will get identified when we create $\mathcal{G}_v$. The edges adjacent to $b_{1,1}^{\ell,k}$ and vertices $v$ are exactly the same as edges adjacent to $b_{1,1}^{\ell',k'}$, and vertices $v$ so they will also get identified when we create $\mathcal{G}_v$.  For all $\ell',k'$ with $b_{1,1}^{\ell',k'}=b_{1,1}^{\ell,k}$, we will have $b_{1,2}^{\ell',k'}=b_{1,2}^{\ell,k}$ because $N_{\Gamma'}(I^1_{2})\cap[\ell,k]=N_{\Gamma'}(I^1_{2})\cap[b_{1,1}^{\ell,k},v]=N_{\Gamma'}(I^1_{2})\cap[\ell',k']$.  This means that the edges $(b_{1,1}^{\ell,k}, b_{1,2}^{\ell,k})$ and $(b_{1,1}^{\ell',k'}, b_{1,2}^{\ell',k'})$ will also get identified when constructing  $\mathcal{G}_v$. Thus, the  valence of $b_{1,1}^{\ell,k}$ in $\mathcal{G}_v$ is $1\oplus 0$.  As earlier, $b_{1,1}^{\ell,k}\neq a^i_j$ for any $i,j$ as $b_{1,1}^{\ell,k}\not\in I_j^i$ for any $i,j$.

The argument for vertices $b_{i,j}^{\ell,k}$ for $1<i\leq r$ and $1\leq j\leq c_i$ is exactly the same as in the cases where $i=1$ detailed above. 

Finally, we analyze the vertex $a_i$ for $1\leq i\leq r$. In $\cal G_v$, the vertex $a_1$ is incident to an edge of the form $(a_1,v)$ which is the result of combining such an edge from every component snake graph. On the other hand, after gluing the edges $(a_1,b^{\ell,k}_{1,c_1})$, for all $\ell\in A$ and $k\in B$, the vertex $a_1$ will have $\deg_{\Gamma'}(a_1)$ more incident edges by the same logic as above for the $a_j^1$ vertices. Therefore $a_1$ has valence $1\oplus \deg_{\Gamma'}({a_1})-1$.  The same logic applies to $a_2$ if $r=2$.
\end{proof}

Combining Lemma \ref{lem:r1valence} and Condition~(\ref{admissibilityCondition4}) in Definition \ref{def:AdmissibleMatching} will give us a useful characterization of subsets of edges which must appear together in an admissible matching. We define $B_j^i = N_{\Gamma'}(I_j^i) \backslash \{v\}$.

\begin{prop}\label{prop:OldCond4Singleton_rsmall}
Let $P$ be an admissible matching of $\mathcal{G}_v$ where $r \leq 2$. Let $e_1$ and $e_2$ be two internal edges which both share vertices with diagonals $d_1$ and $d_2$. Then, $e_1$ and $e_2$ are either both in $P$ or neither are in $P$. 
\end{prop}

\begin{proof}
The possible pairs of diagonals which can share vertices with the same edges would either have labels $d_1 = I_1^1$ and $d_2 = I_1^2$ or $d_1 = I_j^i$ and $d_2 = I_{j+1}^i$ for $1\leq i \leq r$ and $1 \leq j < c_i$. For the first case, note that there is exactly one internal edge, weighted $Y_{I_v}$, which shares vertices with diagonals $I_1^1$ and $I_1^2$, and so there is nothing to check in this case.

Consider the set of all internal edges which share vertices with diagonals $I_{j+1}^i$ and $I_j^i$ for some $1 \leq i \leq r$ and $1 \leq j < c_i$.  From step (4) in Section \ref{subsec:SingletonConstruction1}, we see that there is a set of edges which are each incident to $a_j^i \in B_{j+1}^i$ and that the remaining edges do not share vertices with this set or with each other. Condition~(\ref{admissibilityCondition4}) of Definition \ref{def:AdmissibleMatching} implies that an admissible matching must include all of the former set of edges or none of them. By Lemma \ref{lem:r1valence}, if we include this entire set of edges, we cannot include any other edge incident to $a_j^i \in B_{j+1}^i$. 

Suppose for sake of contradiction that there exists an admissible matching, $P$, which includes a proper subset of the set of internal edges which share vertices with diagonals $I_{j+1}^i$ and $I_j^i$. Suppose first that $j = c_i-1$. We know there are vertices in $B_{j+1}^i$ which are not covered by these internal edges, so they must be covered by some other edges. There is a hyperedge $Y_{I_{c_i}^i}$ incident to all vertices $B_{j+1}^i$. However, including this hyperedge which violate the valence of the edges which are incident to internal edges in $P$, so we cannot include it. There is also a set of edges from $a_i$, with labels $Y_{I_z}$ for $z \in \mathcal{C}_{a_i}$. Condition~(\ref{admissibilityCondition4}) of Definition \ref{def:AdmissibleMatching} implies that an admissible matching must include all or none of this set, so we also cannot cover the vertices with these edges. Therefore, in this case, this configuration cannot be part of an admissible matching.

Now, let $j < c_i-1$. If we take a proper subset of the internal edges incident to diagonals $I_{j+1}^i$ and $I_j^i$, we again have vertices in $B_{j+1}^i$ which must be covered by other edges in $P$. As before, we cannot use either of the boundary edges incident to $B_{j+1}^i$ since this would violate the valence of vertices incident to this proper subset of internal edges. There is a second set of internal edges incident to $B_{j+1}^i$ in this case, namely the set which is incident to diagonals $I_{j+2}^i$ and $I_{j+1}^i$. It is possible that we can cover the necessary vertices in $B_{j+1}^i$ using these internal edges. However, since we cannot use all of these internal edges, we have now included a proper subset of these internal edges. Proceeding in this way, we would be required to use a proper subset of the internal edges which share vertices with diagonals $I_{c_i}^i$ and $I_{c_i-1}^i$, and by our previous paragraph this cannot be a subset of an admissible matching. 
\end{proof}

\begin{lemma}\label{lem:r1admissible}
Let $B^i_j:= N_{\Gamma'}(I^i_j)\setminus\{v\}$ and $\tilde{a}:=\{a_{r+1},\dots,a_p\}$. If $r=1$ the admissible matchings of $\mathcal{G}_v$ are in weight preserving bijection with perfect matchings of the snake graph $\tilde{\mathcal{G}}_v$ shown in \Cref{fig:req1_snakegraph_simplify}, and if $r=2$ the admissible matchings of $\mathcal{G}_v$ are in weight preserving bijection with perfect matchings of the snake graph $\tilde{\mathcal{G}}_v$ shown in \Cref{fig:req2_snakegraph_simplify}.
\end{lemma}

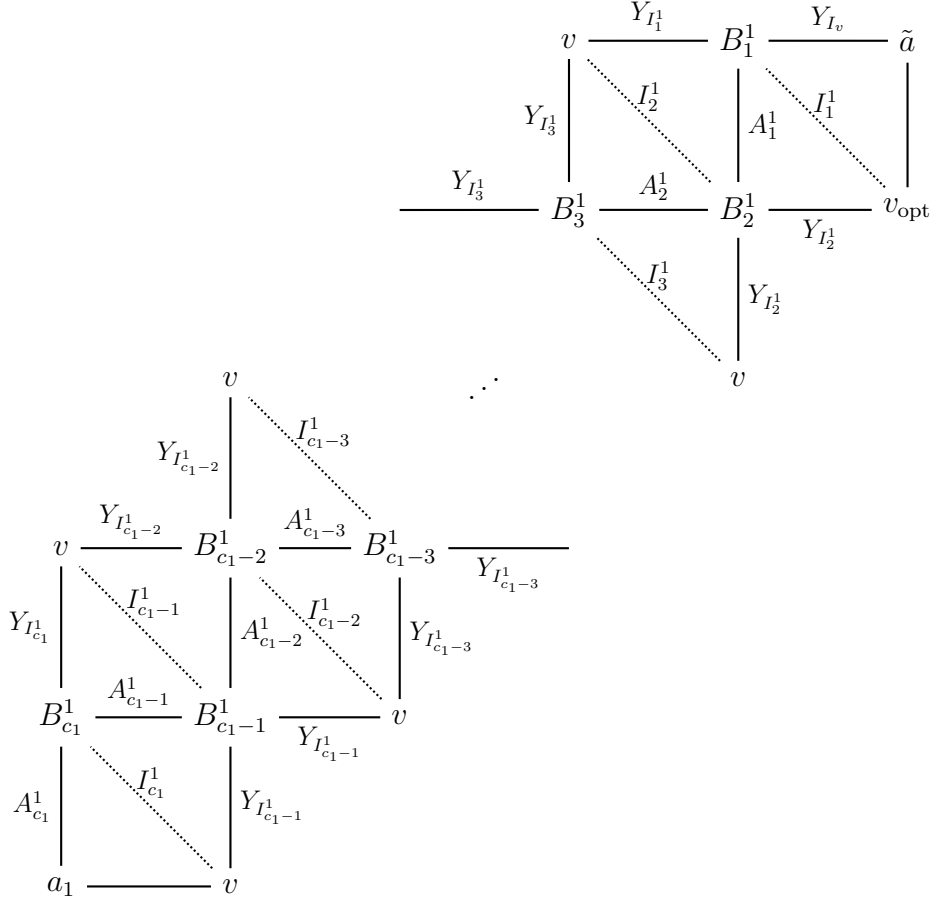
\begin{figure}[ht]
\begin{tikzpicture}[scale=2.25]
\node (11) at (1,0) {$v$};
\node (12) at (0,2) {$v$};
\node (13) at (2,1) {$v$};
\node (14) at (1,3) {$v$};
\node (15) at (4,3) {$v$};
\node (16) at (3,5) {$v$};
\node (17) at (5,4) {$v_{\opt}$};
\node (2) at (0,0) {$a_1$};
\node (a) at (0,1) {$B^1_{c_1}$};
\node (b) at (1,1) {$B^1_{c_1-1}$};
\node (c) at (1,2) {$B^1_{c_1-2}$};
\node (d) at (2,2) {$B^1_{c_1-3}$};
\node (i) at (3,4) {$B^1_3$};
\node (j) at (4,4) {$B^1_2$};
\node (r) at (4,5) {$B^1_1$};
\node (s) at (5,5) {$\tilde{a}$};
\node at (2.5,3) {$\reflectbox{$\ddots$}$};
\draw (11) to (2);
\draw (2) to node[scale=0.85,left]{$A^1_{c_1}$} (a);
\draw (a) to node[scale=0.85,above]{$A^1_{c_1-1}$} (b);
\draw (b) to node[scale=0.85,right]{$A^1_{c_1-2}$} (c);
\draw (c) to node[scale=0.85,above]{$A^1_{c_1-3}$} (d);
\draw (i) to node[scale=0.85,above]{$A^1_2$} (j);
\draw (r) to node[scale=0.85,right]{$A^1_1$} (j);
\draw (r) to node[scale=0.85,above]{$Y_{I_v}$} (s);
\draw (s) to (17);
\draw (a) to node[scale=0.85,left]{$Y_{I_{c_1}^1}$} (12);
\draw (b) to node[scale=0.85,right]{$Y_{I_{c_1-1}^1}$} (11);
\draw (b) to node[scale=0.85,below]{$Y_{I_{c_1-1}^1}$} (13);
\draw (c) to node[scale=0.85,above]{$Y_{I_{c_1-2}^1}$} (12);
\draw (c) to node[scale=0.85,left]{$Y_{I_{c_1-2}^1}$} (14);
\draw (d) to node[scale=0.85,right]{$Y_{I_{c_1-3}^1}$} (13);
\draw (d) to node[scale=0.85,below]{$Y_{I_{c_1-3}^1}$} (3,2);
\draw (i) to node[scale=0.85,above]{$Y_{I_3^1}$} (2,4);
\draw (i) to node[scale=0.85,left]{$Y_{I_3^1}$} (16);
\draw (j) to node[scale=0.85,right]{$Y_{I_2^1}$} (15);
\draw (j) to node[scale=0.85,below]{$Y_{I_2^1}$} (17);
\draw (r) to node[scale=0.85,above]{$Y_{I_1^1}$} (16);
\draw[dashed] (11) to node[scale=0.85,above]{$I_{c_1}^1$} (a);
\draw[dashed] (b) to node[scale=0.85,above, xshift = 6]{$I_{c_1-1}^1$} (12);
\draw[dashed] (13) to node[scale=0.85,above, xshift = 6]{$I_{c_1-2}^1$} (c);
\draw[dashed] (d) to node[scale=0.85,above,xshift = 6]{$I_{c_1-3}^1$} (14);
\draw[dashed] (15) to node[scale=0.85,above]{$I_3^1$} (i);
\draw[dashed] (j) to node[scale=0.85,above]{$I_2^1$} (16);
\draw[dashed] (17) to node[scale=0.85,above]{$I_1^1$} (r);
\end{tikzpicture}
\caption{The snake graph $\tilde{\cal G}_v$ when $r=1$.}
\label{fig:req1_snakegraph_simplify}
\end{figure}

\begin{figure}[ht]
\begin{tikzpicture}[scale=2.25]
\node (11) at (1,0) {$v$};
\node (12) at (0,2) {$v$};
\node (13) at (2,1) {$v$};
\node (14) at (1,3) {$v$};
\node (15) at (4,3) {$v_\opt$};
\node (16) at (3,5) {$v_\fix$};
\node (17) at (6,5) {$v$};
\node (2) at (0,0) {$a_1$};
\node (a) at (0,1) {$B_{c_1}^1$};
\node (b) at (1,1) {$B_{c_1-1}^1$};
\node (c) at (1,2) {$B_{c_1-2}^1$};
\node (d) at (2,2) {$B_{c_1-3}^1$};
\node (i) at (3,4) {$B_1^1$};
\node (j) at (4,4) {$B_1^2\cup\tilde{a}$};
\node (r) at (5,6) {$B_{c_2}^2$};
\node (s) at (6,6) {$a_2$};
\node at (2.5,3) {$\reflectbox{$\ddots$}$};
\node at (4.5,5) {$\reflectbox{$\ddots$}$};
\draw (11) to (2);
\draw (2) to node[scale=0.85,left]{$A_{c_1}^1$} (a);
\draw (a) to node[scale=0.85,above]{$A_{c_1-1}^1$} (b);
\draw (b) to node[scale=0.85,right]{$A_{c_1-2}^1$} (c);
\draw (c) to node[scale=0.85,above]{$A_{c_1-3}^1$}  (d);
\draw (i) to node[scale=0.85,above]{$Y_{I_v}$} (j);
\draw (r) to node[scale=0.85,above]{$A_{c_2}^2$} (s);
\draw (s) to (17);
\draw (a) to node[scale=0.85,left]{$Y_{I_{c_1}^1}$} (12);
\draw (b) to node[scale=0.85,right]{$Y_{I_{c_1-1}^1}$} (11);
\draw (b) to node[scale=0.85,below]{$Y_{I_{c_1-1}^1}$} (13);
\draw (c) to node[scale=0.85,above]{$Y_{I_{c_1-2}^1}$} (12);
\draw (c) to node[scale=0.85,left]{$Y_{I_{c_1-2}^1}$} (14);
\draw (d) to node[scale=0.85,right]{$Y_{I_{c_1-3}^1}$} (13);
\draw (d) to node[scale=0.85,below]{$Y_{I_{c_1-3}^1}$} (3,2);
\draw (i) to node[scale=0.85,above]{$Y_{I_1^1}$} (2,4);
\draw (i) to node[scale=0.85,left]{$Y_{I_1^1}$} (16);
\draw (j) to node[scale=0.85,right]{$Y_{I_1^2}$} (15);
\draw (j) to node[scale=0.85,below]{$Y_{I_1^2}$} (5,4);
\draw (r) to node[scale=0.85,above]{$Y_{I_{c_2}^2}$} (4,6);
\draw[dashed] (11) to node[scale=0.85,above]{$I_{c_1}^1$} (a);
\draw[dashed] (b) to node[scale=0.85,above, xshift = 6]{$I_{c_1-1}^1$} (12);
\draw[dashed] (13) to node[scale=0.85,above, xshift = 6]{$I_{c_1-2}^1$} (c);
\draw[dashed] (d) to node[scale=0.85,above,xshift = 6]{$I_{c_1-3}^1$} (14);
\draw[dashed] (15) to node[scale=0.85,above]{$I_1^1$} (i);
\draw[dashed] (j) to node[scale=0.85,above]{$I_1^2$} (16);
\draw[dashed] (17) to node[scale=0.85,above]{$I_{c_2}^2$} (r);
\end{tikzpicture}
\caption{The snake graph $\tilde{\cal G}_v$ when $r=2$.} 
\label{fig:req2_snakegraph_simplify}
\end{figure}

\begin{proof}
We begin discussing the structures of an admissible matching around each vertex in $\tilde{\cal{G}}_v$.

First of all, if the vertex $v$ adjacent to $a_1$ is not $v_\opt$, it is connected to an edge with weight $Y_{I^1_{c_1-1}}$ and an edge of weight 1.  An admissible matching of $\cal{G}_v$ will use exactly one of these edges, which corresponds to one of the two edges $(v,B^1_{c_1-1})$ and $(v,a_1)$ in $\tilde{\cal G}_v$. When $r=2$ the same is true for the vertex $v$ adjacent to $a_2$ as long as it is not $v_\fix$.  Similarly, every other vertex $v$ except for $v_\opt$ and $v_\fix$ is connected to two (hyper)edges with weights $Y_{I_{j+1}^i}$ and $Y_{I_{j-1}^i}$ respectively. An admissible matching of $\cal{G}_v$ will use exactly one of these edges, which corresponds to one of the two edges $(v,B^i_{j+1})$ and $(v,B^i_{j-1})$ in $\tilde{\cal G}_v$.

The vertex $v_\opt$ has higher valence, but it still has only two possibilities in $\cal{G}_v$. The vertices $a_j, j > r$ have degree 2 in $\cal{G}_v$ and if $r=2$ then the vertices in $B_1^2$ also have degree 2 in $\cal{G}_v$.  These vertices are each adjacent to the edge weighted $Y_{I_v}$ and an edge $(a_j,v_\opt)$ with weight 1 or the edge $(B_1^2,v_\opt)$ with weight $Y_1^2$. If the edge $Y_{I_v}$ is included in an admissible matching, then none of the edges $(a_j,v_\opt)$ or $(B_1^2,v_\opt)$ can be included in the matching. Conversely, if the edge $Y_{I_v}$ is not included in the matching, then we must use all the edges $(a_j,v_\opt)$ and edge $(B_1^2,v_\opt)$. This matches the configuration around $v_\opt$ in the ordinary snake graph $\tilde{\cal{G}_v}$.

If $r=1$, there is no vertex $v_\fix$.  If $r=2$, the vertex $v_\fix$ has valence $1\oplus 0$ and is connected to an edge of weight $Y_{I_1^1}$ and an edge of weight $Y_{I_2^2}$ respectively. An admissible matching of $\cal{G}_v$ will use exactly one of these edges, which corresponds to one of the two edges $(v,B^1_1)$ and $(v,B^2_2)$ in $\tilde{\cal G}_v$.

The vertex $a_1$ also has two possibilities in an admissible matching of $\cal G_v$. By Condition (\ref{admissibilityCondition4}) in Definition \ref{def:AdmissibleMatching}, if we include any of the edges between $a_1$ and vertices in $B_{c_1}^1)$ then we include all such edges.  Otherwise we instead must include $(a_1,v)$. This corresponds to the two choices we have for the vertex $a_1$ in $\tilde{\cal{G}_v}$.  When $r=2$, the same argument holds for the vertex $a_2$.

 We now consider the vertices $a_i$ for $r<i\leq p$. Each such $a_i$ is connected to $v_\opt$ with an edge of weight 1 and connected to all the vertices in $B_1^j$ for all $1\leq j\leq r$ by an edge of weight $Y_{I_v}$.  Since the vertices $a_i$ all have valence $1\oplus 0$, every admissible matching either contains the edge $Y_{I_v}$ or all the edges $a_i$ to $v_\opt$.  This matches the configuration around $\tilde{a}$ in $\tilde{\cal G}_v$ for $r=1$.  For $r=2$, if the edge $Y_{I_v}$ is used in an admissible matching on $\tilde{\cal G}_v$ the edge $B_1^2$ to $v_\opt$ must not be used and if the edge $Y_{I_v}$ is not used in an admissible matching on $\tilde{\cal G}_v$, then the edge $B_1^2$ to $v_\opt$ must also be used. So the edge $B_1^2$ to $v_\opt$ is used exactly when the edges $a_i$ to $v_\opt$ are used.  This matches the configuration around $\tilde{a}$ in $\tilde{\cal G}_v$ for $r=2$.

We next discuss what happens around vertices $b_{i,j}^{\ell,k}\in B_j^i$ for some $1\leq i\leq r$ and $1\leq j\leq c_i$.  First suppose $b_{i,j}^{\ell,k}\not\in B_{c_1}^1,B_1^1,B_1^2,B_{c_2}^2$.  Then the relevant part of $\cal G_v$ will look as follows, where the sets $B_{j-1}^i,B_j^i$, and $B_{j+1}^i$ represent (possibly) multiple vertices and the double edges mean that there are (possibly) multiple hyperedges between vertices in the sets. By abusing notation, we use $(B_{j-1}^i,B_j^i)$ and $(B_j^i,B_{j+1}^i)$ to denote the set of edges between vertices in two sets in $\cal G_v$. Note that the vertices $v,v_1,$ and $v_2$ in the picture are all labeled by the same vertex $v$ from $\Gamma$, but we label them here so that they are distinguishable on $\cal G_v$.

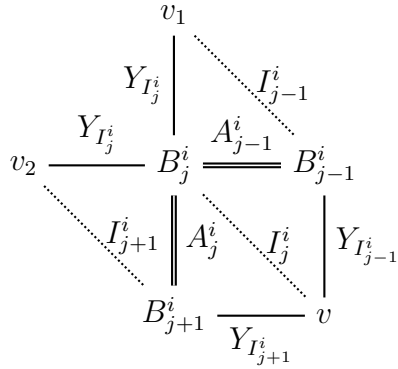
\begin{center}
\begin{tikzpicture}[scale=2]
\node(vopt) at (0,0){$v_2$};
\node(b12) at (1,0){$B_{j}^i$};
\node(b22) at (2,0){$B_{j-1}^i$};
\node(v) at (1,1){$v_1$};
\node(vfix) at (2,-1){$v$};
\node(b11) at (1,-1){$B_{j+1}^i$};
\draw (vopt) to node[above]{$Y_{I_j^i}$} (b12);
\draw [double]  (b12) to node[above]{$A_{j-1}^i$} (b22);
\draw(b12) to node[left]{$Y_{I_j^i}$} (v);
\draw [double] (b12) to node[right]{$A_j^i$} (b11);
\draw(vfix) to node[right]{$Y_{I_{j-1}^i}$} (b22);
\draw[dashed] (v) to node[right]{$I_{j-1}^i$} (b22);
\draw[dashed] (b12) to node[right]{$I_j^i$} (vfix);
\draw (vfix) to node[below]{$Y_{I_{j+1}^i}$} (b11);
\draw[dashed] (vopt) to node[right]{$I_{j+1}^i$}(b11);

\end{tikzpicture}
\end{center}

By Proposition \ref{prop:OldCond4Singleton_rsmall}, all edges in $(B_{j-1}^i,B_j^i)$ have to appear together in an admissible matching of $\cal G_v$, and similarly for $(B_j^i,B_{j+1}^i)$. We will show an admissible matching of $\cal G_v$ cannot contain any configuration near $B_j^i$ which would not correspond to a configuration near such a vertex in the ordinary snake graph $\tilde{\cal{G}_v}$.

\begin{enumerate}[(a)]
    \item Since $a^i_{j-1}$ is a neighbor of $I_j^i$ other than $v$, $a^i_{j-1}\in B_j^i$.  From the proof of \Cref{lem:r1valence}, there are $\deg_{\Gamma'}(a^i_{j-1})-1$ edges from $a^i_{j-1}$ to $B_{j-1}^i$. Thus by \Cref{lem:r1valence}, if we use the edges in $(B_j^i,B_{j-1}^i)$ in an admissible matching, we can't include the edges in $(B_j^i,B_{j+1}^i)$ the edge $(v_1,B_j^i)$, or the edge $(v_2,B_j^i)$, as this would exceed the valence of $a_{j-1}^i\in B_j^i$.

    \item The edges $(v_1,{B_j^i})$ and $(v_2,{B_j^i})$ cannot be used together in an admissible matching because $v$ has valence $1\oplus 0$ by \Cref{lem:r1valence} and adding in either edge adjacent to $v$ would cause a violation of Condition~(\ref{admissibilityCondition3}) of Definition \ref{def:AdmissibleMatching}.

    \item Suppose we use the edges $(B_j^i, B_{j+1}^i)$ as well as the edge $(v_1, B_j^i)$ in an admissible matching. Then by Condition~(\ref{admissibilityCondition3}) of Definition \ref{def:AdmissibleMatching}, we cannot use the edge $(v, B_{j-1}^i)$. This means we must use the edge $(v, B_{j+1}^i)$. By the same argument as in part (a), we know $a^i_{j}\in B_{j+1}^i$ and there are $\deg_{\Gamma'}(a^i_{j})-1$ edges from $a^i_{j}$ to $B_{j}^i$. Thus by \Cref{lem:r1valence}, if we use the edges in $(B_j^i, B_{j+1}^i)$ and the edge $(v, B_{j+1}^i)$ we will exceed the valence of $a_{j}^i\in B_{j+1}^i$.  Therefore it is impossible to use the edges $(B_j^i, B_{j+1}^i)$ as well as the edge $(v_1, B_j^i)$ in an admissible matching.


    \item Now suppose we use the edges $(B_j^i, B_{j+1}^i)$ and the edge $(v_2,B_j^i)$ in an admissible matching.
    \begin{enumerate}[(i)]
    \item Assume $j<c_i-2$.  In this case, we can extend our local picture as follows, where $v_3$ and $v_4$ are additional vertices labeled $v$.

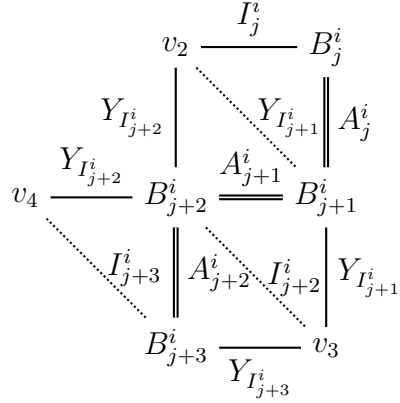
\begin{center}
\begin{tikzpicture}[scale=2]
\node(vopt) at (0,0){$v_2$};
\node(bb) at (1,0){$B_j^i$};
\node(b11) at (1,-1){$B_{j+1}^i$};
\node(Bj2i) at (0,-1){$B_{j+2}^i$};
\node(v3) at (1,-2){$v_3$};
\node(a) at (0,-2){$B_{j+3}^i$};
\node(b) at (-1,-1){$v_4$};
\draw[dashed] (vopt) to node[right]{$Y_{I_{j+1}^i}$}(b11);
\draw (Bj2i) to node[left]{$Y_{I_{j+2}^i}$} (vopt);
\draw(vopt) to node[above]{$I_j^i$}(bb);
\draw[double] (bb) to node[right]{$A_j^i$}(b11);
\draw[double] (Bj2i) to node[above]{$A_{j+1}^i$} (b11);
\draw[dashed] (Bj2i) to node[right]{$I_{j+2}^i$} (v3);
\draw (v3) to node[below]{$Y_{I_{j+3}^i}$} (a);
\draw[double] (Bj2i) to node[right]{$A_{j+2}^i$} (a);
\draw (b11) to node[right]{$Y_{I_{j+1}^i}$} (v3);
\draw (b) to node[above]{$Y_{I_{j+2}^i}$} (Bj2i);
\draw[dashed] (a) to node[right]{$I_{j+3}^i$} (b);
\end{tikzpicture}
\end{center}

    By the same argument as in part (a), we know $a^i_{j+1}\in B_{j+2}^i$. By \Cref{lem:r1valence}, $a^i_{j+1}$ must be covered in an admissible matching. We cannot use the edge $(v_2,B_{j+2}^i)$, as this would violate the valence of $v_2$ by \Cref{lem:r1valence}.  We also cannot use the edges $(B_{j+1}^i,B_{j+2}^i)$ as this would violate the valence of $a_{j}^i\in B_{j+1}^i$ by \Cref{lem:r1valence}.  If we use the edge $(v_4,B_{j+2}^i)$, we could not use the edge $(v_3,B_{j+3}^i)$ by Condition~(\ref{admissibilityCondition3}) of Definition \ref{def:AdmissibleMatching}.  So, by \Cref{lem:r1valence}, we would have to use $(v_3,B_{j+1}^i)$.  This would violate the valence of $a^i_j\in B_{j+1}^i$, so we can't use $(v_4,B_{j+2}^i)$.  This means we must use $(B_{j+2}^i,B_{j+3}^i)$ to cover $a^i_{j+1}$.  However we are then unable to cover $v_3$, as using the edge $(v_3,B_{j+1}^i)$ would violate the valence of $a^i_j\in B_{j+1}^i$ and using the edge $(v_3,B_{j+3}^i)$ would violate the valence of $a^i_{j-2}\in B_{j+3}^i$.  Thus we  have a contradiction and cannot use the edges $(B_j^i, B_{j+1}^i)$ and the edge $(v_2,B_j^i)$ together in an admissible matching.

    \item Now assume $j=c_i-2$.  In this case we can extend  our local picture as follows.
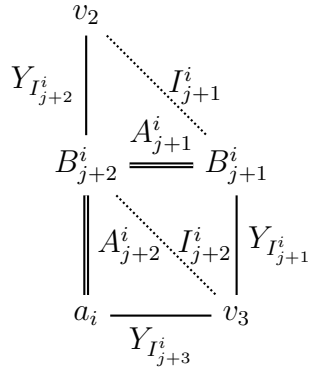
\begin{center}
\begin{tikzpicture}[scale=2]
\node(vopt) at (0,0){$v_2$};
\node(b11) at (1,-1){$B_{j+1}^i$};
\node(Bj2i) at (0,-1){$B_{j+2}^i$};
\node(v3) at (1,-2){$v_3$};
\node(a) at (0,-2){$a_i$};
\draw[dashed] (vopt) to node[right]{$I_{j+1}^i$}(b11);
\draw (Bj2i) to node[left]{$Y_{I_{j+2}^i}$} (vopt);
\draw[double] (Bj2i) to node[above]{$A_{j+1}^i$} (b11);
\draw[dashed] (Bj2i) to node[right]{$I_{j+2}^i$} (v3);
\draw (v3) to node[below]{$Y_{I_{j+3}^i}$} (a);
\draw[double] (Bj2i) to node[right]{$A_{j+2}^i$} (a);
\draw (b11) to node[right]{$Y_{I_{j+1}^i}$} (v3);
\end{tikzpicture}
\end{center}
    The vertex $a_i=a_{c_i}^i=a_{j+2}^i$ is in the same configuration with respect to $B_{j+2}^i$ and $v_3$ as it was in the $j<c_i-2$ case. So the same arguments as in the $j<c_i-2$ case show we cannot use any of the edges adjacent to $B_{j+2}^i$ in an admissible matching at the same  time as the edges $(B_j^i, B_{j+1}^i)$ and the edge $(v_2,B_j^i)$.  Therefore, we can't use the edges $(B_j^i, B_{j+1}^i)$ and the edge $(v_2,B_j^i)$ together in an admissible matching in this case either.

    \item Finally, if $j=c_i-1$, we can extend  our local picture as follows.
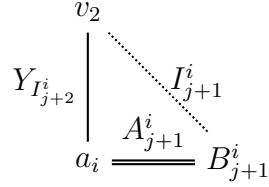
\begin{center}
\begin{tikzpicture}[scale=2]
\node(vopt) at (0,0){$v_2$};
\node(b11) at (1,-1){$B_{j+1}^i$};
\node(Bj2i) at (0,-1){$a_i$};
\draw[dashed] (vopt) to node[right]{$I_{j+1}^i$}(b11);
\draw (Bj2i) to node[left]{$Y_{I_{j+2}^i}$} (vopt);
\draw[double] (Bj2i) to node[above]{$A_{j+1}^i$} (b11);
\end{tikzpicture}
\end{center}
    The vertex $a_i=a_{c_i}^i=a_{j+1}^i$ is in the same configuration with respect to $B_{j+1}^i$ and $v_2$ as it was in the $j<c_i-2$ case. So the same arguments as in the $j<c_i-2$ case show we cannot use any of the edges adjacent to $a_i$ in an admissible matching at the same time as the edges $(B_j^i, B_{j+1}^i)$ and the edge $(v_2,B_j^i)$.  Therefore, we can't use the edges $(B_j^i, B_{j+1}^i)$ and the edge $(v_2,B_j^i)$ together in an admissible matching in this final case.
    \end{enumerate}
\end{enumerate}

Therefore, for an admissible matching $P$,
the only possibilities around $B_j^i$ are the following:
\begin{enumerate}[noitemsep]
    \item $P$ uses $(v_1,B_j^i)$, and no other edges connected to $B_j^i$.
    \item $P$ uses $(v_2,B_j^i)$, and no other edges connected to $B_j^i$.
    \item  $P$ uses all of $(B_j^i,B_{j+1}^i)$, and no other edges connected to $B_j^i$.
    \item $P$ uses all of $(B_j^i,B_{j-1}^i)$, and no other edges connected to $B_j^i$.
\end{enumerate}

Now suppose $b_{i,c_i}^{\ell,k}\in B_{c_i}^i$ for $1\leq i\leq r$.  Then the relevant part of $\cal G_v$ will look as follows.

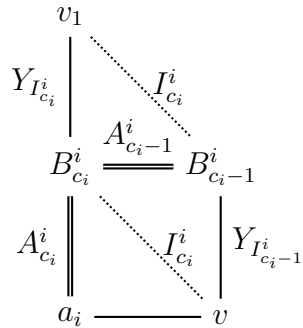
\begin{center}
\begin{tikzpicture}[scale=2]
\node(b12) at (1,0){$B_{c_i}^i$};
\node(b22) at (2,0){$B_{c_i-1}^i$};
\node(v) at (1,1){$v_1$};
\node(vfix) at (2,-1){$v$};
\node(b11) at (1,-1){$a_i$};
\draw [double]  (b12) to node[above]{$A_{c_i-1}^i$} (b22);
\draw(b12) to node[left]{$Y_{I_{c_i}^i}$} (v);
\draw [double] (b12) to node[left]{$A_{c_i}^i$} (b11);
\draw(vfix) to node[right]{$Y_{I_{c_i-1}^i}$} (b22);
\draw[dashed] (v) to node[right]{$I_{c_i}^i$} (b22);
\draw[dashed] (b12) to node[right]{$I_{c_i}^i$} (vfix);
\draw (vfix) to node[below]{} (b11);
\end{tikzpicture}
\end{center}

The arguments from parts (a), (b), and (c) from the $1<j<c_i$ case show that for an admissible matching $P$,
the only possibilities around $B_{c_i}^i$ are the following:
\begin{enumerate}[noitemsep]
    \item $P$ uses $(v^*,B_{c_i}^i)$, and no other edges connected to $B_{c_i}^i$.
    \item  $P$ uses all of $(B_{c_i}^i,a_i)$, and no other edges connected to $B_{c_i}^i$.
    \item $P$ uses all of $(B_{c_i}^i,B_{c_i-1}^i)$, and no other edges connected to $B_{c_i}^i$.
\end{enumerate}

Finally, suppose $b_{i,j}^{\ell,k}\in B_1^i$ for $1\leq i\leq r$.  Then the relevant part of $\cal G_v$ will look as follows, where the picture on the left is for $r=1$ and the picture on the right  is for $r=2$.

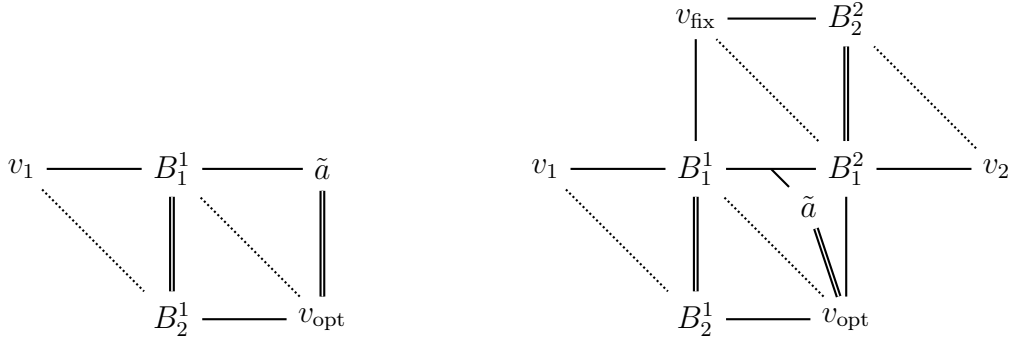
\begin{center}
\begin{tikzpicture}[scale=2]
\node(vopt) at (0,0){$v_\opt$};
\node(v1) at (-2,1){$v_1$};
\node(b11) at (-1,1){$B_1^1$};
\node(b21) at (-1,0){$B_2^1$};
\node(ai) at (0,1){$\tilde{a}$};
\draw (vopt) to node[above]{} (ai);
\draw (ai) to (b11);
\draw [double] (b21) to (b11);
\draw(vopt) to node[right]{} (b21);
\draw(v1) to node[right]{} (b11);
\draw[double] (vopt) to node[right]{} (ai);
\draw[dashed] (vopt) to node[right]{}(b11);
\draw[dashed] (v1) to node[right]{} (b21);
\end{tikzpicture}
\hspace{0.75in}
\begin{tikzpicture}[scale=2]
\node(vopt) at (0,0){$v_\opt$};
\node(b12) at (0,1){$B_1^2$};
\node(b22) at (0,2){$B_2^2$};
\node(v2) at (1,1){$v_2$};
\node(v1) at (-2,1){$v_1$};
\node(vfix) at (-1,2){$v_\fix$};
\node(b11) at (-1,1){$B_1^1$};
\node(b21) at (-1,0){$B_2^1$};
\node(ai) at (-0.25,0.75){$\tilde{a}$};
\draw (vopt) to node[above]{} (b12);
\draw [double]  (b12) -- (b22);
\draw(b12) to node[left]{} (v2);
\draw (b12) to (b11);
\draw(vfix) to node[right]{} (b22);
\draw [double] (b21) to (b11);
\draw(vopt) to node[right]{} (b21);
\draw(v1) to node[right]{} (b11);
\draw[dashed] (v2) to node[right]{} (b22);
\draw[dashed] (b12) to node[right]{} (vfix);
\draw (vfix) to node[below]{} (b11);
\draw[double] (vopt) to node[right]{} (ai);
\draw (-0.5,1) to node[right]{} (ai);
\draw[dashed] (vopt) to node[right]{}(b11);
\draw[dashed] (v1) to node[right]{} (b21);
\end{tikzpicture}
\end{center}

By Proposition \ref{prop:OldCond4Singleton_rsmall} of Definition \ref{def:AdmissibleMatching}, the edges in $(B_1^1,B_2^1)$ have to appear together in an admissible matching of $\cal G_v$.  By \Cref{lem:r1valence}, all vertices in $B_1^1$ have valence $1\oplus 0$.  This means that if $r=1$ an admissible matching of $\mathcal{G}_v$ must use either all of the $(B_1^1,B_2^1)$ edges or exactly one of the other edges adjacent to $B_1^1$.  Similarly, if $r=2$ then an admissible matching of $\mathcal{G}_v$ must use either all of the $(B_1^2,B_2^2)$ edges or exactly one of the other edges adjacent to $B_1^2$.

We have now seen that the admissible matchings of $\cal G_v$ are in bijection with perfect matchings of $\tilde{\cal G}_v$. We will next show that this bijection is weight preserving.  That is, we will show that each edge in $\tilde{\cal G}_v$ has the same weight as the corresponding set of edges in $\cal G_v$.

All boundary edges except for $(a_i,B^i_{c_i})$ correspond to either a hyperedge with weight $Y_{I^i_j}$, a set of edges with weight $1$ in $\cal G_v$, or a hyperedge with weight $Y_{I^i_j}$ and a set of edges with weight $1$ in $\cal G_v$.  Thus our claim is clear for these edges. If $r=2$, this is also clear for the edge $(B_1^1,B_1^2)$ because it  corresponds to an edge of weight $Y_{I_v}$ in $\cal G_v$.
 
For the edges $(a_i,B^i_{c_i})$, notice that after combining all component snake graphs, we get an hyperedge from $a_i$ to a subset of $B^i_{c_i}$ in $\cal G_v$ for each $I_w$ where $w\in \cal C_{a_i}$.  We also potentially get several edges of weight 1 from $a_i$ to vertices in $B^i_{c_i}\cap (\Gamma'_{>  a_i})^{\mathcal I}$.
Since $a_i=a^i_{c_1}$, the product of all these edges in $\cal G_v$ is $\prod_{w\in\cal C(a^i_{c_i})} Y_{I_w} = A_{c_i}^i$, which agrees with the weight of the edge $(a_i,B^i_{c_i})$ in $\tilde{\cal G}_v$.

We are now left to consider the edges $(B^i_j,B^i_{j+1})$ in $\tilde{\cal G}_v$.  The edge $(B^i_j,B^i_{j+1})$ corresponds to the edges $(b_{i,j}^{\ell,k},b_{i,j+1}^{\ell,k})$ from $\mathcal{G}_v$. Recall from \Cref{sec:construction}, in the component snake graphs, the edge $(b_{i,j}^{\ell,k},b_{i,j+1}^{\ell,k})$ can have three different possible weights.
\begin{enumerate}[noitemsep]
    \item When $b_{i,j}^{\ell,k}\in N_{\Gamma'}(a^i_j)$, the edge $(b_{i,j}^{\ell,k},b_{i,j+1}^{\ell,k})$ has weight $1$.
    \item When $b^{\ell,k}_{i,j} \in N_{\Gamma'} (I_w) $ 
    for some $w\in \mathcal{C}_{a^i_j} \setminus \{a^i_{j+1} \}$
    , the edge $(b_{i,j}^{\ell,k},b_{i,j+1}^{\ell,k})$ is weighted by $Y_{I_w}$.
    \item When $b^{\ell,k}_{i,j}\in N_{\Gamma'}(I_{i,j+1})$, $(b_{i,j}^{\ell,k},b_{i,j+1}^{\ell,k})$ is weighted by $Y_{C_{i,j}^{\ell,k}}$ where $C_{i,j}^{\ell,k}$ is the connected component of $I^i_{j+1}\setminus [v,a^i_j]$ which contains a vertex adjacent to $b^{\ell,k}_{i,j}$.
\end{enumerate}
In the second scenario, as we range over all possible $\ell\in A$, we get edges weighted by $Y_{I_w}$ for all $w\in \cal{C}_{a^i_j}\setminus{a^i_{j+1}}$ in $\cal G_v$. These edges will contribute to a weight of $\prod_{w\in \cal{C}_{a^i_j}\setminus{a^i_{j+1}}} Y_{I_w}$. 

In the third scenario, since each $C_{i,j}^{\ell,k}$ is a connected component of $I^i_{j+1}\setminus [v,a^i_j]$ which contains a vertex adjacent to $b^{\ell,k}_{i,j}$, we know that $\bigcup_{\ell\in A} C_{i,j}^{\ell,k}$ is union of all the connected components of $I^i_{j+1}\setminus [v,a^i_j]$ which is adjacent to some vertex in $B^i_j = N_{\Gamma'}(I^i_{j})\setminus \{v\}$. This must be the entire set of $I^i_{j+1}\setminus [v,a^i_j]$, therefore the product of weights of these edges will be
\[\prod_{\ell\in A} Y^2_{C_{i,j}^{\ell,k}} = Y^2_{I^i_{j+1}\setminus [v,a^i_j]}.\]

Finally, when all the edges of the form $(b_{i,j}^{\ell,k},b_{i,j+1}^{\ell,k})$ are used together in $\cal G_{v}$, their weight will be
\[\left(\prod_{w\in \cal{C}_{a^i_j}\setminus{a^i_{j+1}}} Y_{I_w}\right)Y^2_{I^i_{j+1}\setminus [v,a^i_j]}\]
which is exactly $A^i_j$.

Therefore, the admissible matchings of $\cal G_v$ is in weight preserving bijection with perfect matchings of $\tilde{\cal G}_v$ whose edge weights are defined as in \Cref{fig:req1_snakegraph_simplify}.
\end{proof}

Since the previous lemma compares our snake graph with a zigzag snake graph, we can use some a basic facts concerning the latter.  To state these, we first introduce some notation analogous to \cite{banaian2022generalization}. When considering a surface snake graph, we will say a boundary edge which is only incident to one diagonal and one tile is \emph{dominant} and a boundary edge which is incident to one diagonal but two tiles is \emph{non-dominant}. We provide an alternative definition for our snake graphs since the tile structure is not always clear.

\begin{defn}\label{def:Dominant}
Let $e$ be an edge of a snake graph $\mathcal{G}_v$ which only shares vertices with one diagonal. We say $e$ is \emph{dominant} if $e$ is of the form $(v,w)$ where $w \in N_\Gamma(v)$.  Otherwise, we say $e$ is \emph{non-dominant}.
\end{defn}

Recall that when we form snake graphs for larger sets $S$ we glue along snake graph boundary edges of a snake graph of the form  $(v,w)$ where $v$ and $w$ are neighbors in $\Gamma$. These edges are always dominant edges.

The following lemma is a general statement for all zigzag snake graphs.  

\begin{lemma}\label{lem:ZigZagSurfaceMatchings}
If $\mathcal{G}$ is a zigzag surface snake graph on $n$ tiles, then it has $n+1$ perfect matchings. In particular, it has $n-1$ perfect matchings such that each use exactly one internal edge and two perfect matchings that use complementary sets of boundary edges. Moreover, each of the boundary edge matchings uses exactly one dominant edge.

In particular, if $\mathcal{G}$ is a zigzag with tiles labeled by $\tau_{1},\ldots,\tau_{d}$ in order from southwest to northeast, with the internal edge between $\tau_{j}$ and $\tau_{j+1}$ labeled with $z_j$, the non-dominant edge on tile  $\tau_1$ (resp. tile $\tau_d$) labeled with $z_0$ (resp. $z_d$) and the dominant edge at on tile $\tau_1$ (resp. tile $\tau_d$) labeled with $u$ (resp. $v$), then \[
\chi(\mathcal{G}) = \frac{u}{x_1} + \frac{v}{x_d} + \sum_{i=1}^{d-1} \frac{z_i}{x_i x_{i+1}}
\]
where $x_i$ is the cluster variable associated to $\tau_i$.
\end{lemma}

\begin{proof}
The fact that $\mathcal{G}$ has $n+1$ perfect matchings is well-known (see, for example, Section 4 of \cite{propp}). The perfect matchings of $\mathcal{G}$ are explicitly described in the proof of Theorem 3.4 of \cite{canakci2017continuedfractions}, and the expansion for this particular zigzag snake graph follows from applying the definition of $\chi(\mathcal{G})$ given at the beginning of Section~\ref{subsec:main_theorem}. 
\end{proof}

\begin{lemma}\label{lem:req1zigzag}
Let $\cal P_v$ denote the set of admissible matchings of $\cal G_v$. When $r=1$, we have
\[
\frac{1}{\ell(\cal{G}_v)}\sum_{P \in \cal{P}_v} \wt(P) = \frac{Y_{I_v}}{Y_{I^1_1}} + \sum_{j=1}^{c_1} \frac{A^1_j}{Y_{I^1_j}Y_{I^1_{j+1}}}
\]
\end{lemma}
\begin{proof}

In \Cref{lem:r1valence}, we showed that $\cal P_v$ is in bijection with the set of ordinary perfect matchings in $\tilde {\cal G_v}$. Hence, the conclusion follows from  \Cref{fig:req1_snakegraph_simplify} and \Cref{lem:ZigZagSurfaceMatchings}.
\end{proof}

\begin{lemma}\label{lem:req2zigzag}
Let $\cal P_v$ denote the set of admissible matchings of $\cal G_v$. When $r=2$, we have
\[
\frac{1}{\ell(\cal{G}_v)}\sum_{P \in \cal{P}_v} \wt(P) = \frac{Y_{I_v}}{Y_{I^1_1}Y_{I^1_2}} + \sum_{j=1}^{c_1} \frac{A^1_j}{Y_{I^1_j}Y_{I^1_{j+1}}} + \sum_{j=1}^{c_2} \frac{A^2_j}{Y_{I^2_j}Y_{I^2_{j+1}}}
\]
\end{lemma}
\begin{proof}
In \Cref{lem:r1valence}, we showed that $\cal P_v$ is in bijection with the set of ordinary perfect matchings in $\tilde {\cal G_v}$. Hence, the conclusion follows from  \Cref{fig:req2_snakegraph_simplify} and \Cref{lem:ZigZagSurfaceMatchings}.
\end{proof}

\begin{proof}[Proof of Theorem \ref{thm:main} for $S = \{v\}$ when $r=1,2$]
Comparing \Cref{lem:req1zigzag} to \Cref{cor:SingletonFormula} shows that $\frac{1}{\ell(\mathcal{G}_v)} \sum_{P \in \mathcal{P}_v} \wt(P) = Y_v$ for any $v \in \Gamma$ with $r=1$. Comparing \Cref{lem:req2zigzag} to \Cref{cor:SingletonFormula} shows that $\frac{1}{\ell(\mathcal{G}_v)} \sum_{P \in \mathcal{P}_v} \wt(P) = Y_v$ for any $v \in \Gamma$ with $r=2$. 
\end{proof}

\subsection{r > 2}
\label{subsec:risnot1}

In this section we study the structure of admissible matchings on a singleton snake graph $\mathcal{G}_v$ such that $\vert N_{\Gamma}(v)\,\cap\,\Gamma^{\mathcal{I}}_{<v} \vert>2$. Unlike in Section \ref{subsec:ris1}, here we will have vertices with valence $a \oplus b$ where $a > 1$. We will first look at a family of intermediate graphs, which are the result of gluing some but not all component snake graphs. Let $\mathcal{H}_i$ denote the effect of gluing all component snake graphs $H_{\ell,k}$ where $a_1 \in [\ell,v]$ and $a_i \in [k,v]$, with $2 \leq i \leq r$. We call each $\mathcal{H}_i$ a \emph{branch snake graph}. The center and right graphs in Figure \ref{fig:restrictToBranch} are the two branch snake graphs for the snake graph $\mathcal{G}_3$ built in Example \ref{ex:G3}.

Recall that, for a vertex $v$ with $r \leq 2$, Lemma \ref{lem:r1valence} listed all possible valences of vertices of $\mathcal{G}_v$. When $r > 2$, this lemma is true for all vertices except $v_\fix$, which has valence $(r - 1) \oplus 0$. From this observation, we have that Proposition \ref{prop:OldCond4Singleton_rsmall} is also true for $r > 2$.

\begin{lemma}\label{lem:OldCond4Singleton_rlarge}
Let  $P$ be an admissible matching of $\mathcal{G}_v$ where $r > 2$. Let $e_1$ and $e_2$ be two internal edges which both share vertices with diagonals $d_1$ and $d_2$. Then, $e_1$ and $e_2$ are either both in $P$ or neither are in $P$.     
\end{lemma}

\begin{proof}
The only difference between this setting and a snake graph $\mathcal{G}_v$ where $r \leq 2$ is that there is a vertex, $v_\fix$, with valence $(r -1) \oplus 0$. There are no internal edges incident to $v_\fix$ and the same arguments as for the $r \leq 2$ will hold here. 
\end{proof}

Along the lines of  \Cref{lem:r1admissible}, we will show that admissible matchings of $\mathcal{H}_i$ are in weight-preserving bijection with the set of perfect matchings of a snake graph.

\begin{lemma}\label{lem:DecomposeAdimssibleBranch}
Let $\tilde{\mathcal{G}}_v^i$ be the snake graph of the same shape as that in Figure \ref{fig:req2_snakegraph_simplify} and with weights given by replacing all superscript ``2"s with ``$i"s$ and all ``$c_2$"s with ``$c_i$"s. The   set of admissible matchings of $\mathcal{H}_i$ is in weight-preserving bijection with the set of perfect matchings of $\tilde{\mathcal{G}}_{v}^i$.
\end{lemma}

\begin{proof}
For vertices other than $v_\fix$ and $v_\opt$, the same logic as Lemma \ref{lem:r1valence} holds for considering the valence of vertices in $\mathcal{H}_i$. The vertices $v_\fix$ and $v_\opt$ will have valence $1 \oplus 0$ in $\mathcal{H}_i$ since each is adjacent to the same edges in all the component snake graphs $H_{\ell,k}$ which glue together to form $\mathcal{H}_i$.
Then, we can apply similar reasoning as in the proof of Lemma \ref{lem:r1admissible} for the $r = 2$ case.

\end{proof}

We will now compare admissible matchings of the snake graph $\mathcal{G}_v$ to matchings of the branch snake graphs $\mathcal{H}_i$. It will not always be true that an admissible matching of $\mathcal{G}_v$ will decompose to admissible matchings of each $\mathcal{H}_i$. For example, in Figure \ref{fig:restrictToBranch}, we give an admissible matching of the snake graph $\mathcal{G}_3$ from Example \ref{ex:G3} and its restrictions to its two branch snake graphs, $\mathcal{H}_2$ and $\mathcal{H}_3$. We notice that the restriction to $\mathcal{H}_3$ is an admissible matching of $\mathcal{H}_3$ but the restriction to $\mathcal{H}_2$ is not a matching. 

\begin{figure}
    \centering
    \input{figures/branch_eg}
    \caption{From left-to-right, we have an admissible matching $P$ of the snake graph $\mathcal{G}_3$ from Example \ref{ex:G3}, the restriction of $P$ to $\mathcal{H}_2$, and the restriction of $P$ to $\mathcal{H}_3$, where we order the neighbors of $3$ as $a_1 = 8, a_2 = 2$ and  $a_3 = 4$.}
    \label{fig:restrictToBranch}
\end{figure}
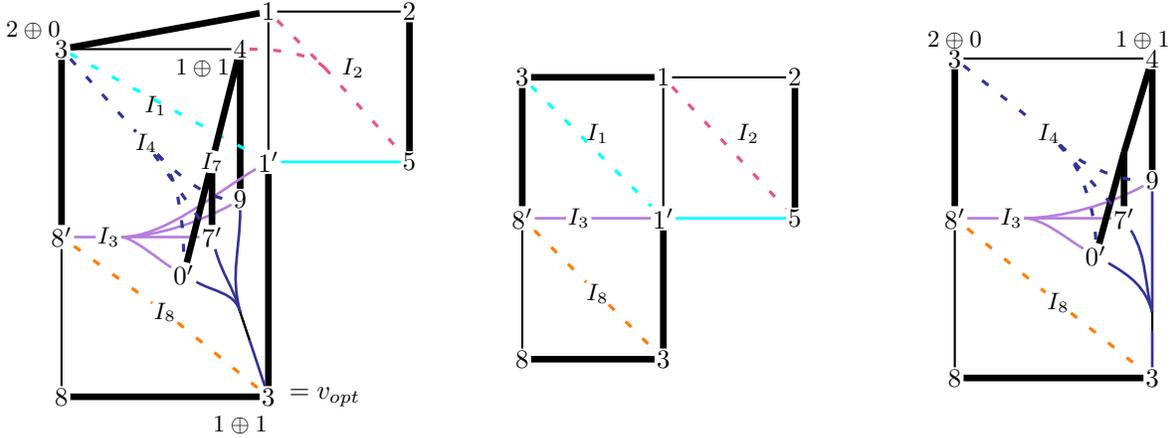
\begin{center}

\end{center}

The reason this fails to be a matching comes from the fact that vertices $v_\fix$ and $v_\opt$ have valence $1 \oplus 0$ in the branch snake graphs. When these restrictions are not matchings, they will have a common form, which we will call the \emph{distinguished non-matching of $\mathcal{H}_i$}.  This set of edges will include the edges weighted $Y_{I_1^i}$ and $Y_{I_2^1}$ incident to the vertex $v$ which lies on diagonal $I_1^1$ and similarly the edges $Y_{I_1^1}$ and $Y_{I_2^i}$ incident to the vertex $v$ on diagonal $I_1^i$, as described above.  Then, for all $2 \leq j \leq c_1$, it will include the boundary edge weighted $Y_{I_{j+1}^1}$ incident to the vertex $v$ on diagonal $I_j^1$ and for $2 \leq j \leq c_i$, it will include the boundary edge weighted $Y_{I_{j+1}^i}$ incident to the vertex $v$ on diagonal $I_j^i$. Finally, this matching includes both dominant edges. The set of edges highlighted in $\mathcal{H}_2$ in Figure \ref{fig:restrictToBranch} is an example of a distinguished non-matching.  The weight of this matching will be $Y_{I_1^1} Y_{I_2^1} \cdots Y_{I_{c_1}^1} Y_{I_1^i} Y_{I_2^i} \cdots Y_{I_{m_i+1}^i}$ so that its weight is equal to $\ell(\mathcal{H}_i)$.

In the following, let $\mathcal{P}_v$ denote the set of admissible matchings of $\mathcal{G}_v$. Let $\mathcal{P}_v^1$ denote the set of admissible matchings of $\mathcal{G}_v$ which restrict to an admissible matching for each $\mathcal{H}_i$. Let $\mathcal{P}_v^i$, $2 \leq i \leq r$, denote the set of admissible matchings which only restrict to an admissible matching of $\mathcal{H}_i$ and restrict to the distinguished non-matching for all $\mathcal{H}_j, j \neq i$. Since $r > 2$, it follows from the definitions that  $\mathcal{P}_v^i$ and $\mathcal{P}_v^j$ are disjoint for $i \neq j$.

\begin{lemma}\label{lem:HowDoGvRestrictToHi}
The sets $\mathcal{P}_v^1,\ldots,\mathcal{P}_v^r$ partition $\mathcal{P}_v$. 
\end{lemma}

\begin{proof}

We have noted that the vertices $v_\opt$ and $v_\fix$ are the only vertices which have different valences in $\mathcal{G}_v$ and $\mathcal{H}_i$. In $\mathcal{H}_i$, the vertex $v_{\fix}$ is incident to boundary edges weighted $Y_{I_1^1}$ and $Y_{I_2^i}$, the vertex $v_\opt$ is incident to boundary edges weighted $Y_{I_2^1}$ and $Y_{I_1^i}$, and both vertices have valence $1 \oplus 0$. Now, we turn to $\mathcal{G}_v$.
The vertex $v_\fix$ is incident to $r$ edges, with weights $Y_{I_1^1}, Y_{I_2^2}, Y_{I_2^3},\ldots, Y_{I_2^r}$ and has valence $(r-1) \oplus 0$. The vertex $v_\opt$ has valence $(1 \oplus \deg_{\Gamma'}(v)-2)$ and is incident to $\deg_{\Gamma'}(v)$ edges, where $r$ have weights $Y_{I_2^1},Y_{I_1^2},Y_{I_1^3},\ldots,Y_{I_1^r}$ and the remaining $\deg_{\Gamma'}(v)-r$ are of the form $(v_\opt,a_i)$ for $r+1 \leq i \leq \deg_{\Gamma'}(v)$.  

Note that every admissible matching of $\mathcal{G}_v$ must include all but one edge incident to $v_\fix$. We will show that $\mathcal{P}_v^1$ is the set of admissible matchings which does not contain the edge $Y_{I_1^1}$ incident to $v_\fix$ and $\mathcal{P}_v^i$ is the set of admissible matchings which does not use the edge $Y_{I_2^i}$ incident to $v_\fix$.

Suppose first that a matching $P$ does not contain the edge $Y_{I_1^1}$ incident to $v_\fix$. Then, it must contain all edges in the set $\{Y_{I_2^i}| 2 \leq i \leq r\}$ which are incident to $v_\fix$. We claim any such matching will restrict to an admissible matching on each $\mathcal{H}_i$. We can show this by showing that we will not use both the edge weighted $Y_{I_2^1}$ incident to $v_\opt$ and any edge $Y_{I_1^i}$ for $ 2 \leq i \leq r$ since this is the only other vertex whose valence is different in $\mathcal{G}_v$ and $\mathcal{H}_i$. 

Suppose for sake of contradiction that the edges from $P$ incident to $v_\opt$ consists of the edge $Y_{I_2^1}$  and some nonempty subset of the edges $\{Y_{I_1^i}\vert 2 \leq i \leq r\}$. Once the set contains at least one edge $Y_{I_1^i}$ incident to $v_\opt$, we know it cannot contain the edge $Y_{I_v}$ because all vertices incident to $I_v$ have valence $1 \oplus 0$.  Therefore, the matching must include all edges $(v_\opt,a_i)$ for $r+1 \leq i \leq \deg_{\Gamma'}(v)$. This means that there must also be at least one value $1 \leq i \leq r$ such that the edge $Y_{I_1^i}$ incident to $v_\opt$ is not in $P$.

Suppose without loss of generality that this edge is $Y_{I_1^2}$. As in Section \ref{subsec:ris1}, let $B_j^i = N_{\Gamma'}(I_j^i) \backslash \{v\} = \{b_{i,j}^{\ell,k} : \ell \in A, k \in B\}$ for $1 \leq j \leq c_1$.   Now, we claim that there is no way to find an edge in $P$ which is incident to any vertex in $B_1^2$. The local picture is as follows.

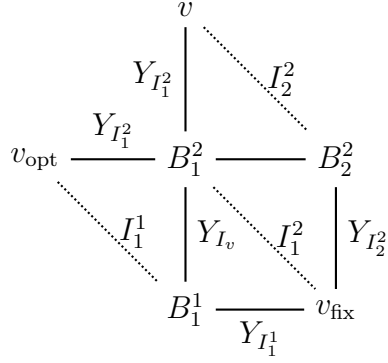
\begin{center}
\begin{tikzpicture}[scale=2]
\node(vopt) at (0,0){$v_\opt$};
\node(b12) at (1,0){$B_1^2$};
\node(b22) at (2,0){$B_2^2$};
\node(v) at (1,1){$v$};
\node(vfix) at (2,-1){$v_\fix$};
\node(b11) at (1,-1){$B_1^1$};
\draw (vopt) to node[above]{$Y_{I_1^2}$} (b12);
\draw(b12) -- (b22);
\draw(b12) to node[left]{$Y_{I_1^2}$} (v);
\draw(b12) to node[right]{$Y_{I_v}$} (b11);
\draw(vfix) to node[right]{$Y_{I_2^2}$} (b22);
\draw[dashed] (v) to node[right]{$I_2^2$} (b22);
\draw[dashed] (b12) to node[right]{$I_1^2$} (vfix);
\draw (vfix) to node[below]{$Y_{I_1^1}$} (b11);
\draw[dashed] (vopt) to node[right]{$I_1^1$}(b11);
\end{tikzpicture}
\end{center}

We already assume we cannot use the edge between $v_\opt$ and $B_1^2$, and we have shown we cannot use the edge $Y_{I_v}$. We claim we cannot use any of the edges between the sets $B_2^2$ and $B_1^2$. Note that we already use the edge between $v_\fix$ and $B_2^2$, and if we included one edge between $B_2^2$ and $B_1^2$, by Lemma \ref{lem:OldCond4Singleton_rlarge} we would need to include all such edges. We can immediately see that this configuration would violate the valence of at least one vertex in $B_2^2$ unless $B_2^2 = \{a_1^2\}$ and the valence of $a_1^2$ is $1 \oplus a$ for $a > 0$. However, from the proof of Lemma \ref{lem:r1valence} we know that $a_1^2$ is incident to $a+1$ distinct edges of the form $(a_1^2, B_1^2)$, so that including all of these edges plus the edge $(v_\fix, B_2^2)$ would again violate the valence of $a_1^2$.

 Finally, we cannot use the edge incident to $B_1^2$ which has weight $Y_{I_1^2}$ and is incident to a vertex $v \neq v_\opt$ because, since we use the edge $Y_{I_2^2}$ incident to $v_\fix$, this would violate Condition~(\ref{admissibilityCondition3}) in Definition \ref{def:AdmissibleMatching}. In conclusion, when we know that $P$ contains the edges incident to $v_\fix$ labeled $I_2^2,\ldots,I_r^2$, then $P$ will restrict to an admissible matching on all $\mathcal{H}_i$. 
 

 Next, suppose that the edge with weight $Y_{I_1^1}$ incident to $v_\fix$ is in the matching $P$. Then one edge $Y_{I^i_2}$ must not be in the matching; suppose without loss of generality this is $I_2^2$. Now, we claim that $P$ restricts to the distinguished non-matching on each $\mathcal{H}_i$ for $i \geq 3$ and $P$ restricts to an admissible matching of $\mathcal{H}_2$. This will be equivalent to showing that $P$ contains all edges incident to $v_\opt$ except the edge $Y_{I_1^2}$. In the following, we assume $c_1 \geq 3$; the claim can be checked directly for small values of $c_1$.

Suppose for sake of contradiction that one of these edges is not in $P$. Since the edge $Y_{I_1^1}$ incident to $v_\fix$ is in $P$, we know that we cannot use the edge $Y_{I_v}$. Thus, we know immediately that we must use all edges $(v_\opt,a_i)$ for $r+1 \leq i \leq \deg_{\Gamma'}(v)$. 

If an edge $Y_{I_1^i}$, $i > 2$, incident to $v_\opt$ is not in $P$, then with a very similar argument to the previous case where edges $Y_{I_2^2},\ldots,Y_{I_2^r}$ incident to $v_\fix$ were in $P$, we can show that there is no way to find an edge in $P$ which is incident to any vertex in $ B_1^3$. So all edges $Y_{I_1^i}$, $i > 2$ must be in $P$.

Finally, suppose that the edge $Y_{I_2^1}$ incident to $v_\opt$ is not in $P$. We claim that there cannot be an edge in $P$ which is incident to $B_2^1$. The local picture is as follows.

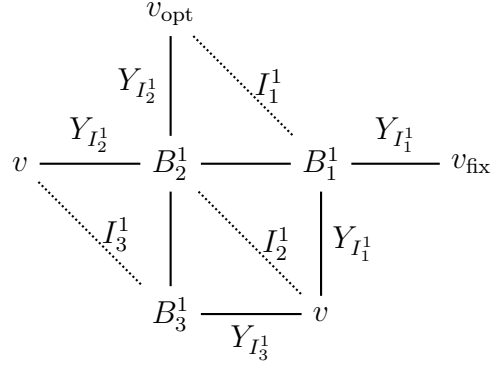
\begin{center}
\begin{tikzpicture}[scale=2]
\node(vopt) at (0,0){$v$};
\node(b12) at (1,0){$B_2^1$};
\node(b22) at (2,0){$B_1^1$};
\node(v) at (1,1){$v_\opt$};
\node(vfix) at (2,-1){$v$};
\node(b11) at (1,-1){$B_3^1$};
\node(vreallyfix) at (3,0){$v_\fix$};
\draw (vopt) to node[above]{$Y_{I_2^1}$} (b12);
\draw(b12) -- (b22);
\draw(b12) to node[left]{$Y_{I_2^1}$} (v);
\draw(b12) to (b11);
\draw(vfix) to node[right]{$Y_{I_1^1}$} (b22);
\draw[dashed] (v) to node[right]{$I_1^1$} (b22);
\draw[dashed] (b12) to node[right]{$I_2^1$} (vfix);
\draw (vfix) to node[below]{$Y_{I_3^1}$} (b11);
\draw[dashed] (vopt) to node[right]{$I_3^1$}(b11);
\draw (b22) to node[above]{$Y_{I_1^1}$} (vreallyfix);
\end{tikzpicture}
\end{center}

First, note that since the edge $Y_{I_1^1}$ incident to $v_\fix$ is in $P$, we must include the edge $Y_{I_1^3}$ incident to the vertex $v$ on diagonal $I_2^1$. The other edge incident to this vertex would create a valence violation at $B_1^1$ as all vertices in $B_1^1$ have valence  $1 \oplus 0$. Therefore, $P$ cannot contain the edge $Y_{I_2^1}$ which is incident to $B_2^1$ and the vertex $v$ on diagonal $I_3^1$ by Condition~(\ref{admissibilityCondition3}) of Definition \ref{def:AdmissibleMatching}.  We cannot use any edge between $B_2^1$ and $B_1^1$ because this would also create a valence violation for vertices in $B_1^1$. We also claim we cannot use any edges  between $B_3^1$ and $B_2^1$. As in the previous case, by Lemma \ref{lem:OldCond4Singleton_rlarge}, we would need to include all internal edges which share vertices with diagonals $I_3^1$ and $I_2^1$ and this would still create a valence violation since we also have included  the edge $Y_{I_3^1}$. Therefore, we have shown that when an admissible matching does not contain $Y_{I_1^i}$,then it will restrict to an admissible matching of $\mathcal{H}_i$ and the distinguished non-matching of $\mathcal{H}_j, j \neq i$.

\end{proof}

We include the following corollary to Lemma \ref{lem:HowDoGvRestrictToHi} which will be useful in Section \ref{sec:GlueProof}.
Recall from Proposition \ref{prop:ij-edge} that for $a_i \in N_{\Gamma'}(v)$, a unique boundary edge $(v,a_i)$  is guaranteed to exist in $\mathcal{G}_v$.

\begin{cor}\label{cor:WhichMatchingsAvoidDominant}
For each $a_i \in N_{\Gamma'}(v)$, there is a unique admissible matching of $\mathcal{G}_v$ which does not include the boundary edge $(v,a_i)$.    
\end{cor}

\begin{proof}
This is immediate for the case where $r=1$  and $r=2$ by Lemma \ref{lem:r1admissible}. Now, suppose $r > 2$. We first consider the case $1<i \leq r$. Then, a matching which does not include $(v,a_i)$ could only be in the sets $\mathcal{P}_v^1$ and $\mathcal{P}_v^i$. However, in the proof of Lemma \ref{lem:HowDoGvRestrictToHi}, we show that all matchings in $\mathcal{P}_v^1$ include the edge $(v,a_i)$. Combining Lemmas \ref{lem:ZigZagSurfaceMatchings} and \ref{lem:DecomposeAdimssibleBranch} shows that there is exactly one matching in $\mathcal{P}_v^i$ which includes the non-dominant boundary edges incident to diagonal $I_{c_i}^i$ instead and therefore does not include edge $(v,a_i)$.

If $i = 1$, then we can see that every matching in $\mathcal{P}_v^i$ for $i > 1$ contains this boundary edge. Thus, we will similarly get exactly one matching not containing $(v,a_1)$ in $\mathcal{P}_v^1$.

Now, if $r<i\leq p$ then in $\mathcal{G}_v$ this vertex is incident to edge $(a_i,v)$ and another edge with weight $Y_{I_v}$. We also know that this vertex has valence $1 \oplus 0$ so exactly one of these edges will be present in any admissible matching. From the proof of Lemma \ref{lem:HowDoGvRestrictToHi}, we see that there is a unique matching using the edge $Y_{I_v}$, so we are done. 
\end{proof}

The proof of Theorem \ref{thm:main} in the case of $S = \{v\}$ follows from Lemma \ref{lem:HowDoGvRestrictToHi}.

\begin{proof}[Proof of Theorem \ref{thm:main} for $S = \{v\}$]
By Lemma \ref{lem:HowDoGvRestrictToHi}, we can separate the sum in the definition of $\chi(\mathcal{G}_v)$ according to how the various admissible matchings of $\mathcal{G}_v$ restrict on the branch snake graphs. It remains to evaluate the weights of these sums. 

We first evaluate the possible weights of the matchings in $\mathcal{P}_v^1$. Repeated use of Condition~(\ref{admissibilityCondition3}) in Definition \ref{def:AdmissibleMatching} shows that, given a vertex $v$ incident to diagonal $I_j^i$ for $i>1$, the matching can only contain the boundary edge at this vertex with weight $Y_{I_{j+1}^i}$. This includes the dominant edge incident to each diagonal $I_{c_i}^i$. 

 None of the edges specified so far are incident to any diagonal $I_j^1$. There are in fact no restrictions on which edges in this subgraph can or cannot be included in the matching. Therefore, we can apply the reasoning of Lemma \ref{lem:r1admissible} to list all possible ways the matching can look for this section. Since the contributions for the edges outside this subgraph  cancel with the part of the denominator coming from diagonals outside this subgraph, the sum of the weights of such matchings again follows from Lemmas \ref{lem:r1admissible} and \ref{lem:ZigZagSurfaceMatchings}. Thus, we have 
\[
\frac{1}{\ell(\mathcal{G}_v)}\sum_{P \in \mathcal{P}_v^1} wt(P) = \frac{Y_{I_v}}{Y_{I_1^1}\cdots Y_{I_1^r}} + \sum_{j=1}^{c_1} \frac{A_j^1}{Y_{I_j^1}Y_{I_{j+1}^1}}.
 \]

Now, consider $\mathcal{P}_v^i$ for $2 \leq i \leq r$. Since any matching in $\mathcal{P}_v^i$ restricts to the distinguished non-matching on $\mathcal{H}_j$ for $j \geq 2, j \neq i$, we can simply consider the restriction to $\mathcal{H}_i$. In Lemma \ref{lem:HowDoGvRestrictToHi}, we have shown that any $P \in \mathcal{P}_v^i$ contains the edge $Y_{I_1^1}$ at $v_\fix$ and the edge $Y_{I_2^1}$ at $v_\opt$. Using Condition~(\ref{admissibilityCondition3}) of Definition \ref{def:AdmissibleMatching} as we have previously, we can see that we will have edges $Y_{I_3^1},\ldots,Y_{I_{c_1}^1}$ and the dominant edge incident to diagonal $I_{c_1}^1$. No other edges are forced, so we see that it is possible to find a $P$ which restricts to any admissible matching of $\mathcal{H}_i$ which includes the set of internal edges incident to diagonals $I_j^i$ and $I_{j+1}^i$ for $1 \leq j < c_i$ or the admissible matching which includes the set of non-dominant edges incident to diagonal $I_{c_i}^i$. By combining Lemmas \ref{lem:ZigZagSurfaceMatchings} and \ref{lem:DecomposeAdimssibleBranch}, these are all the possibilities. Thus, we have,

\[
\frac{1}{\ell(\mathcal{G}_v)}\sum_{P \in \mathcal{P}_v^i} wt(P) = \sum_{j=1}^{c_i} \frac{A_j^i}{Y_{I_j^i}Y_{I_{j+1}^i}}.
\]  

Thus, from Corollary \ref{cor:SingletonFormula}, we can conclude \[
\frac{1}{\ell(\mathcal{G}_v)} \sum_{P \in \mathcal{P}_v} wt(P) = \frac{1}{\ell(\mathcal{G}_v)} \sum_{P \in \mathcal{P}_v^1} wt(P) + \sum_{i=2}^r \sum_{P \in \mathcal{P}_v^i} wt(P) = Y_v.
\]

\end{proof}

\section{Proof for General Weakly Rooted Sets}
\label{sec:ProofGeneral}

Now, we turn to proving Theorem \ref{thm:main} for general weakly rooted sets. The main proof will be contained in Section \ref{sec:GlueProof} where we will induct on $\vert \bar{S} \vert$.  We complete the base case $\vert \bar{S} \vert = 1$ in Section \ref{subsec:SingeltonProofWeakSingleton}, building on the machinery we developed for when $\vert S \vert = 1$ in Section \ref{sec:SingeltonProof}. 

\subsection{Proof for Sets with Singleton Weakly Rooted Portions}\label{subsec:SingeltonProofWeakSingleton}

We provide here a proof of Theorem \ref{thm:main} for a weakly rooted set $S = J_{\mathbf{m}}$. We constructed snake graphs for two such sets in Example \ref{ex:weakly-rooted-graph} as intermediate steps.

Recall from the statement of Proposition \ref{prop:SingletonAdjoinSets}, for a vertex $v$ and a tuple $\mathbf{m} = (m_1,\ldots,m_r)$, that we set $J_{\mathbf{m}} = \{v\} \cup I_{m_1}^1(v) \cup \cdots \cup I_{m_r}^r(v)$. It is clear that $J_{\mathbf{m}}$ is weakly rooted and  $\overline{J_{\mathbf{m}}} = v$. Moreover, every set $S$ with $\overline{S} = \{v\}$ will be of the form $S = J_{\mathbf{m}}$.

First, we show that Proposition \ref{prop:OldCond4Singleton_rsmall} remains true in this setting.

\begin{lemma}\label{lem:OldCond4WeaklyRootedSingleton}
Let $S = J_{\mathbf{m}}$ be a weakly rooted set with $\bar{S} = \{v\}$ and let $P$ be an admissible matching of $\mathcal{G}_S$. Let $e_1$ and $e_2$ be two internal edges which both share vertices with diagonals $d_1$ and $d_2$. Then, $e_1$ and $e_2$ are either both in $P$ or neither are in $P$.
\end{lemma}

\begin{proof}
We can repeat the logic as in the proof of Proposition \ref{prop:OldCond4Singleton_rsmall} but we must take a bit more care in the first case for any $i$ such that $m_i \leq c_i$.  In order for edges between $B_{j}^i$ and $B_{j-1}^i$ to be internal, we must have $j<m_i$.  Therefore, let $m_i > 2$ and suppose we take a proper subset of the internal edges which share vertices with diagonals $I_{m_i-1}^i$ and $I_{m_i-2}^i$. A consequence of Condition (\ref{admissibilityCondition4}) and Lemma \ref{lem:r1valence} is that every vertex in $B_{m_i-1}^i$ is either not covered or is covered a maximal number of times. Since this is a proper subset, there must be vertices in $B_{m_i-1}^i$ which have not been covered at all.   We cannot cover these with the boundary edge $Y_{I_{m_i-1}^i}$ since this would violate the valence of the vertices which are already incident to these internal edges. If $m_i\neq c_i+1$, then the other set of boundary edges incident to $B_{m_i-1}^i$ consists of a set of edges incident to $a_{m_i-1}^i$ and other edges disjoint from these. This is because these edges were the internal edges which shared vertices with diagonals $I_{m_i-1}^i$ and $I_{m_i}^i$ in $\mathcal{G}_v$.

It is possible that we could cover the necessary vertices in $B_{m_i-1}^i$ using some of the edges between $B_{m_i}^i$ and $B_{m_i-1}^i$. However, in order to not violate the valences of any vertex in $B_{m_i-1}^i$, this would require a proper subset of these edges. The only other edge, $Y_{I_{m_i}^i}$, incident to the set of vertices $B_{m_i}^i$ is incident to all of these vertices. Condition~(\ref{admissibilityCondition4}) of Definition \ref{def:AdmissibleMatching} forces us to include either all or none the edges incident to $a_{m_i-1}^i$ in an admissible matching. By the same logic as in Lemma \ref{lem:r1valence}, an admissible matching that contains all of these edges cannot also contain any other edge incident to $a_{m_i-1}^i$. Because all vertices in $B_{m_i}^i$ other than $a_{m_i-1}^i$ have valence $1\oplus 0$, using edge $Y_{I_{m_i}^i}$ would violate the valence of at least one vertex in $B_{m_i}^i$. Therefore, we cannot include a proper subset of the internal edges which share vertices with diagonals $I_{m_i-1}^i$ and $I_{m_i-2}^i$.

As in the proof of Proposition \ref{prop:OldCond4Singleton_rsmall}, if we include proper subsets of the set of internal edges which share vertices with $I_{m_i-j}^i$ and $I_{m_i-j-1}^i$ with $1<j<m_i+1$, we would be forced back to this case.

\end{proof}

\begin{proof}[Proof of Theorem \ref{thm:main} when $\bar{S} = \{v\}$]
Let $S$ be a weakly rooted set such that $\bar{S} = \{v\}$.  Then  $S= J_{\mathbf{m}}$ for some ${\bf m}=(m_1,\dots,m_r)$. In Proposition \ref{prop:SingletonAdjoinSets}, we have an exact formula for $Y_S$, so all that remains is to show that this formula agrees with the weighted sum indexed by all admissible matchings of $\mathcal{G}_S$.

We reorder the neighbors of $v$ as $a_1,\ldots,a_r$ so that for some $0 \leq t \leq r$, we have $m_i > 1$ for $1 \leq i \leq t$ and $m_i = 1$ for $t+1 \leq i \leq r$. 

First, note that if $t = r$, then we can use the same reasoning as in Sections \ref{subsec:ris1} and \ref{subsec:risnot1} to understand the admissible matchings of $\mathcal{G}_S$. The only difference is that, when $m_i < c_i+1$, instead of a dominant edge $(v,a_i)$ we have an edge with weight $Y_{I_{m_i}^i}$ that shares vertices with diagonal $I_{m_i-1}^i$. For a given $i$, every admissible matching will include this edge except the boundary edge with weight $Y_{I_{m_i-2}^i}$ which shares vertices with the diagonal $I_{m_i-1}^i$. Such a matching must also include the edges whose product of weights is $A_{m_i-1}^i$, and no other matching will use these edges. Comparing these observations with the formula given in Proposition \ref{prop:SingletonAdjoinSets} shows the claim in this case.

The other extreme value of $t$ is when $t = 0$. In this case, $S = I_v$ and $\mathcal{G}_S$ is the single hyperedge labeled $I_v$ where all vertices have valence $1 \oplus 0$. The claim is immediate.

Now suppose $0 < t < r$. In particular, we know that $m_1 > 1$, which means the valence at $v_\opt$ will not change when we pass from $\mathcal{G}_v$ to $\mathcal{G}_S$. However, the valence of the vertex $v_\fix$ will decrease from $(r-1) \oplus 0$ to $(t-1) \oplus 0$ because we will remove the diagonals $I_i^1$ for all $t+1 \leq i \leq r$. For each such $i$, the only remaining edge with a label of the form $I_j^i$ will be the edge incident to $v_\opt$ with label $I_1^i$, and the only other edge the  vertices $N_{\Gamma'}(I_j^i) \backslash \{v\}$ are incident to is $Y_{I_v}$. Therefore, this edge now behaves just as the edges $(v,a_i)$ for $a_i >_{\mathcal{I}} v$. In Lemma \ref{lem:r1admissible}, we see that if $r = 1$ or $r = 2$, exactly one admissible matching uses the edge weighted $Y_{I_v}$. Now, suppose $r > 2$ and let $\mathcal{H}_2,\ldots,\mathcal{H}_r$ be the branch snake graphs from $\mathcal{G}_v$. The edge $Y_{I_v}$ is in each branch snake graph and by definition the distinguished non-matching does not include $Y_{I_v}$. Therefore, by Lemma \ref{lem:HowDoGvRestrictToHi}, any matching including this edge is in $\mathcal{P}_v^1$. In the proof of Theorem \ref{thm:main} for $r > 2$, we analyze the weights of matchings in $\mathcal{P}_v^1$ and we again see that there is only one admissible matching including $Y_{I_v}$. Therefore, all other matchings must include the edge $Y_{I_j^i}$, which matches what we expect from Proposition \ref{prop:SingletonAdjoinSets}. 
\end{proof}

\subsection{General Proof}\label{sec:GlueProof}

We turn to proving Theorem \ref{thm:main} for sets with larger rooted portions with the following steps.  First, in Lemma \ref{lem:YRecursionRootedSet} we show an identity  involving $Y_S, Y_T,$ and $Y_{S \cup T}$ for disjoint incompatible sets $S$ and $T$ on a tree. Our goal is then to show, for a weakly rooted set $S \cup J_{\mathbf{m}}$ with $S$ weakly rooted, and $S$ and $J_{\mathbf{m}}$ disjoint, the same relation holds for $\chi(S), \chi(J_{\mathbf{m}}),$ and $\chi(S\cup J_{\mathbf{m}})$.  In Lemma \ref{lem:DecomposeToAdmissible} we show that the admissible matchings of $\mathcal{G}_{S \cup J_{\mathbf{m}}}$ can always be decomposed to admissible matchings of $\mathcal{G}_{S}$ and $\mathcal{G}_{J_{\mathbf{m}}}$ such that at least one of these matchings used the glued edge. Then, in Lemma \ref{lemma:term_cancellation}, we study the set of matchings of a snake graph $\mathcal{G}_{S}$ which do not use a given boundary edge. These two results put together show that the sums of admissible matchings satisfies the same relation as for $Y$-variables in Lemma \ref{lem:YRecursionRootedSet}. The proof of Theorem \ref{thm:main} follows quickly from these results.

\begin{lemma}\label{lem:YRecursionRootedSet}
Let $\Gamma$ be a tree. Let $S,T \subset V(\Gamma)$ be disjoint, connected sets of vertices and suppose there exists $(w,v) \in E(\Gamma)$ such that $w \in S,  v \in T$.  Then, $Y_{S}Y_{T} = Y_{S\cup T} + Y_{S \backslash \{w\}}Y_{T \backslash \{v\}}$. 
\end{lemma}

\begin{proof}

Note that the sets $S \backslash \{w\}, T \backslash \{v\}, S, $ and $S \cup T$ with the given conditions are all compatible.  Suppose $\mathcal{I}$ is a cluster containing these sets; it is possible that $S \backslash \{w\}$ or $T \backslash \{v\}$ is disconnected, in which case $\mathcal{I}$ would contain all connected components. In this cluster, the mutation of the set $S$ produces the set $T$, and the expansion of $Y_{S}Y_{T}$ follows from Equation (5.2) of \cite{lam2016linear}.

\end{proof}

To prepare for further arguments involving admissible matchings, we show that Proposition \ref{prop:OldCond4Singleton_rsmall} will also hold for general weakly rooted sets. 

\begin{lemma}\label{lem:OldCond4General}
Let $S$ be a weakly rooted set and let $P$ be an admissible matching of $\mathcal{G}_S$. Let $e_1$ and $e_2$ be two internal edges which both share vertices with diagonals $d_1$ and $d_2$. Then, $e_1$ and $e_2$ are either both in $P$ or neither are in $P$. 
\end{lemma}

\begin{proof}
We construct $\mathcal{G}_S$ by gluing snake graphs of the form $\mathcal{G}_{J_{\mathbf{m}}}$. This process takes an edge of the form $(v,w)$ in two such snake graphs and glues them together to form a new internal edge. There will be no other internal edge which shares vertices with the same set of diagonals as this glued edge. Therefore, any pair of distinct internal edges as in the statement must have come from one of the snake graphs $\mathcal{G}_{J_{\mathbf{m}}}$. One can then use the same proof strategy as in Lemma \ref{lem:OldCond4WeaklyRootedSingleton} to show the result. 
\end{proof}

Given a weakly rooted set $S$,  let $\mathcal{P}_S$ be the set of admissible matchings of $\mathcal{G}_{S}$.  Given an edge $e$ in a snake graph $\mathcal{G}_S$, let $\mathcal{P}_S^e$ denote the set of admissible matchings of $\mathcal{G}_S$ that contain $e$, and  let $\mathcal{P}_S^{ne}:= \mathcal{P}\backslash \mathcal{P}_S^e$ denote the complement. 
For example, we could rephrase Corollary \ref{cor:WhichMatchingsAvoidDominant} as saying that, given a dominant boundary edge $e$ in $\mathcal{G}_v$, there is a unique matching in $\mathcal{P}_v^{ne}$.

\begin{lemma}\label{lem:AdmissibleEdges}
Consider disjoint, weakly rooted sets $S$ and $J_{\mathbf{m}}(v) = J_{\mathbf{m}}$ such that there exists $w \in \overline{S}$ with $(v,w) \in E(\Gamma)$ and $w >_{\mathcal{I}} v$. 
\begin{enumerate}[(1)]
    \item There does not exist an admissible matching of $\mathcal{G}_{S \cup J_{\mathbf{m}}}$ which contains all edges incident to $v$ which come from $\mathcal{G}_S$.
    \item There does not exist an admissible matching of $\mathcal{G}_{S \cup J_{\mathbf{m}}}$ which contains only edges from $\mathcal{G}_S$ incident to $v$ and only edges from $\mathcal{G}_{J_{\mathbf{m}}}$ incident to $w$.
    \item There does not exist an admissible matching of $\mathcal{G}_{S \cup J_{\mathbf{m}}}$ which does not include any edges incident to $v$ which come from $\mathcal{G}_S$ and includes all edges incident to $w$ but not $v$ which come from $\mathcal{G}_S$.
\end{enumerate}

The three statements above are also true if we switch $v$ and $w$ and/or if we switch $\mathcal{G}_S$ and $\mathcal{G}_{J_{\mathbf{m}}}$ in each statement.  (Note that this means (1) and (3) each give four statements and (2) gives two statements.)
\end{lemma}

\begin{proof}

There are three cases for the local configuration near the edge $(v,w)$ in $\mathcal{G}_{S \cup J_{\mathbf{m}}}$, based on whether $w$ covers $v$ or not, and if it does, whether the diagonal $I_v$ in the graph $\mathcal{G}_{S}$ is incident to a vertex $w_\opt$ or $w_\fix$.  Since we will frequently be referring to the set $N_\Gamma(v) \cap \Gamma^{\mathcal{I}}_{<v}$ in the following text and figures, we will introduce the shorthand $\mathcal{N}_v:= N_\Gamma(v) \cap \Gamma^{\mathcal{I}}_{<v}$.

\textbf{Case 1:} First, suppose $v \lessdot w$ and the diagonal $I_v$ in $\mathcal{G}_S$ is incident to $w_\opt$ so that $a_1(w) = v$. This means that we glue on the edge $(v_\opt,w)$ in $\mathcal{G}_{J_{\mathbf{m}}}$ and on the edge $(v,w_\opt)$ in $\mathcal{G}_S$. The local configuration is shown below.   Note that some of the edges adjacent to $v_{\text{opt}}$ from $\mathcal{G}_{J_{\mathbf{m}}}$ could share vertices with the edge $Y_{I_v}$; we draw our diagram simply to focus on which edges are adjacent to $v_{\text{opt}}$ and $w_{\text{opt}}$. Note also that there could be edges of the form $(v_{\text{opt}},u)$ or $(w_{\text{opt}},u)$ for  for $u>_{\mathcal{I}} v$ and $u>_{\mathcal{I}} w$ respectively which have weight 1.  Since in $\mathcal{G}_{S}$, the valence of $w_{\text{opt}}$ is $(1 \oplus \deg_{\Gamma'}(w)-2)$ and the valence of $v$ is $(1 \oplus \deg_{\Gamma'}(v) - 2)$ and in $\mathcal{G}_{J_{\mathbf{m}}}$, the valence of the vertex $w$ is $(1 \oplus 0)$ and the valence of $v_{\text{opt}}$ is $(1 \oplus \deg_{\Gamma'}(v) - 2)$, in the glued graphs these vertices have valences $(1 \oplus \deg_{\Gamma'}(w)-2)$ and $(1 \oplus 2\deg_{\Gamma'}(v) -4)$ respectively.

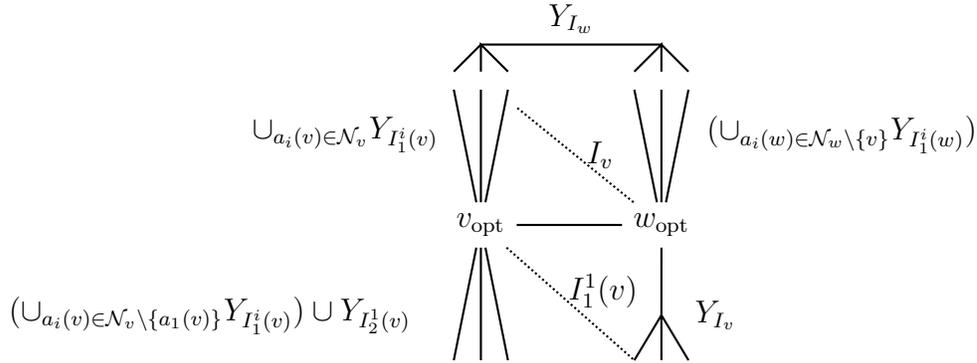
\begin{center}
\begin{tikzpicture}[scale = 1.2]
\node[white] at (5,-1){$(\cup_{a_i(v) \in \mathcal{N}_v \backslash \{a_1(v)\}} Y_{I_1^i(v)}) \cup Y_{I_2^1(v)}$};
\node(v) at (0,0) {$v_{\text{opt}}$};
\node(w) at (2,0) {$w_{\text{opt}}$}; 
\draw (v) -- (w);
\draw (v) -- (-.3,1.5);
\draw (v) -- (0,1.5);
\draw (v) -- (.3,1.5);
\node[] at (-1.5,1){$\cup_{a_i(v) \in \mathcal{N}_v} Y_{I_1^i(v)}$};
\draw (w) -- (1.7,1.5);
\draw (w) -- (2,1.5);
\draw (w) -- (2.3,1.5);
\node[] at (4,1){$(\cup_{a_i(w) \in \mathcal{N}_w \backslash \{v\}} Y_{I_1^i(w)})$};
\draw (w) -- (2,-1);
\draw (2,-1) -- (2.3,-1.5);
\draw (2,-1) -- (2,-1.5);
\draw (2,-1) --  (1.7,-1.5);
\node[] at (2.6,-1){$Y_{I_v}$};
\draw (v) -- (-.3,-1.5);
\draw (v) -- (0,-1.5);
\draw (v) -- (0.3,-1.5);
\node[] at (-3,-1){$(\cup_{a_i(v) \in \mathcal{N}_v \backslash \{a_1(v)\}} Y_{I_1^i(v)}) \cup Y_{I_2^1(v)}$};
\draw[dashed] (v) to node[right, yshift = 5pt, xshift = -5pt]{$I_1^1(v)$} (1.7,-1.5);
\draw[dashed] (w) to node[right]{$I_v$} (0.4,1.3);
\draw (-.3,1.7) -- (0,2) to node[above]{$Y_{I_w}$} (2,2) -- (2.3,1.7);
\draw (0,1.7) -- (0,2) -- (.3,1.7);
\draw (2,1.7) -- (2,2) -- (1.7,1.7);
\end{tikzpicture}
\end{center}

First, we will show statement (1) holds for this case.  If $P$ contains $(v,w)$, it cannot contain the edge $Y_{I_v}$ by Condition~(\ref{admissibilityCondition1}) from Definition \ref{def:AdmissibleMatching} so we see immediately it cannot contain all edges coming from $\mathcal{G}_{J_{\mathbf{m}}}$ which are incident to $w$. Since the valence of $w$ in $\mathcal{G}_{J_{\mathbf{m}}}$ is $1 \oplus 0$, the valence of $w$ in $\mathcal{G}_{S \cup J_{\mathbf{m}}}$ is the same as in $\mathcal{G}_S$. This valence does not allow an admissible matching to include every edge incident to $w_{\opt}$ from $\mathcal{G}_{J_{\mathbf{m}}}$, so we do not need to consider such a configuration in $\mathcal{G}_{S\cup J_{\mathbf{m}}}$ either.

Now, assume that $P$ contains all edges incident to $v$ which come from $\mathcal{G}_S$. Since this includes the edge $(v,w)$, we cannot include the edge weighted $Y_{I_w}$ by Condition~(\ref{admissibilityCondition1}) of Definition \ref{def:AdmissibleMatching}. This means the edges $(w_{\opt},u)$ for any $u \in N_{\Gamma'}(w) \cap (\Gamma')^{\mathcal{I}}_{>w}$ must  be in $P$ since these vertices are only incident to the edge $(w_{\opt},u)$ and the edge weighted $Y_{I_w}$. 
Therefore, there is some $a_i(w) \in \mathcal{N}_w$ such that the edge $Y_{I_1^i(w)}$ incident to $w_\opt$ is not in $P$.  Then,  we claim there is no way to find an edge which is incident to the vertices $N_\Gamma(I_1^i(w)) \backslash \{w\}$ and which  can be included to such a set of edges and preserve admissibility. We draw the general situation below, where  the set of vertices which attain the contradiction are marked with *. 

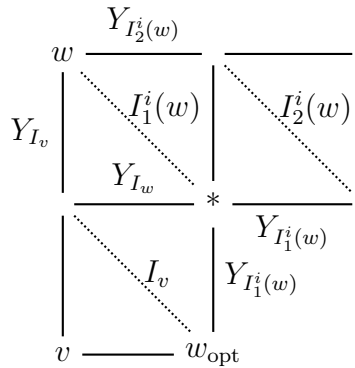
\begin{center}
\begin{tikzpicture}[scale = 2]
\node(5) at (2,1){};
\node(6) at (2,0){};
\node(3) at (0,1){$w$};
\node(4) at (1,1){};
\node(1) at (0,0){};
\node(2) at (1,0){*};
\node(v) at (0,-1){$v$};
\node(w) at (1,-1){$w_\opt$};
\draw (3) to node[above]{$Y_{I_2^i(w)}$} (4);
\draw (3) to node[left]{$Y_{I_v}$} (1);
\draw(4) -- (2);
\draw (1) -- (v);
\draw (1) to node[above]{$Y_{I_w}$} (2);
\draw(v) -- (w);
\draw (2) to node[right]{$Y_{I_1^i(w)}$}(w);
\draw(2) to node[below]{$Y_{I_1^i(w)}$} (6);
\draw (6) -- (5);
\draw (5) -- (4);
\draw[dashed] (1) to node[right]{$I_v$} (w);
\draw[dashed] (3) to node[right, yshift = 8pt, xshift = -8pt]{$I_1^i(w)$} (2);
\draw[dashed] (4) to node[right, yshift = 8pt, xshift = -8pt]{$I_2^i(w)$} (6);
\end{tikzpicture} 
\end{center}

Note that the edge $Y_{I_2^i(w)}$ must be included in $P$ since by valence requirements of the vertices in the set $N_{\Gamma'}(I_v) \backslash \{v\}$ and the assumption that we are including all edges incident to $v$, we cannot include the edge $Y_{I_v}$.  Thus,  by valence requirements and Condition~(\ref{admissibilityCondition3}) of Definition \ref{def:AdmissibleMatching}, none of the edges incident to * can be included in $P$. If $c_i = 1$, then the picture is simpler and the claim follows by similar logic.

Finally, we consider the situation where $P$ contains all edges incident to $v$ which comes from $\mathcal{G}_{J_{\mathbf{m}}}$. First, suppose $r > 2$, so that $\mathcal{G}_{J_{\mathbf{m}}}$ has a vertex $v_\fix$ with valence $(r - 1) \oplus 0$. From the proof of Lemma \ref{lem:HowDoGvRestrictToHi}, we see that if $J_{\mathbf{m}} = \{v\}$, the subset of $P$ consisting of edges incident to $v_\fix$ puts restrictions on which edges incident to $v_\opt$ can be in $P$, and the same arguments hold for a more general set $J_{\mathbf{m}}$. In order for such a configuration to satisfy the admissibilty criteria, $P$ must also include the edges with weights $Y_{I_{1}^2(v)},\ldots,Y_{I_{1}^r(v)}$ incident to $v_\fix$. However, we can reach a similar contradiction to the previous discussion now and show that we cannot add any edges incident to $N_{\Gamma'}(I_1^1(v))$ to this matching and preserve admissibility. 
When $r \leq 2$ or $\{v\} \subsetneq J_{\mathbf{m}}(v)$, we can quickly reach the same contradiction without appealing to any previous result.

We will now  prove statement (2) in this case.  If  the set of edges in $P$ incident to $v$ lie only in $\mathcal{G}_S$ and the set of edges in $P$ incident to $w$ lie only in $\mathcal{G}_{J_{\mathbf{m}}}$, $P$ could not be decomposed into admissible matchings of $\mathcal{G}_S$ and $\mathcal{G}_{J_{\mathbf{m}}}$ since, even if we add the edge $(v,w)$ to one of the decompositions, the other decomposition would leave either $v$ or $w$ uncovered, which would not respect the valences in $\mathcal{G}_S$ and $\mathcal{G}_{J_{\mathbf{m}}}$. 
By Condition~(\ref{admissibilityCondition3}) in Definition \ref{def:AdmissibleMatching}, we cannot use the edge labeled $I_1^1(v)$ incident to $v$ and the edge $Y_{I_v}$ incident to $w$. If we use the edge $Y_{I_v}$ incident to $w$ and a proper subset of the edges incident to $v$ lying in $\mathcal{G}_S$, then we cannot use the edge labeled $I_w$ since the endpoints of the edges incident to $v$ in $P$ will have valence $1 \oplus 0$. This would leave the vertices in $N_\Gamma(I_1^1(v))\backslash \{v\}$ without any incident edge in the matching. 
Therefore, there cannot be an admissible matching with the edges incident to $v$ lying only in $\mathcal{G}_S$ and the edges incident to $w$ lying only in $\mathcal{G}_{J_{\mathbf{m}}}$.

Now, suppose that we have an admissible matching such that the set of edges in $P$ incident to $v$ lie only in $\mathcal{G}_{J_{\mathbf{m}}}$ and the set of edges in $P$ incident to $w$ lie only in $\mathcal{G}_S$. Then, we again cannot include the edge $Y_{I_w}$ due since the endpoints of some of the endpoints of the edge are already covered. Moreover, by assumption we do not include the edge $Y_{I_1^1(v)}$,  so we cannot complete such a configuration to include any edges which cover $N_\Gamma(I_1^1(v)) \backslash \{v\}$. 

Finally, we turn to proving statement (3). First, we show that there does not exist an admissible matching of $\mathcal{G}_{S \cup \Jm}$ that does not include any edges incident to $v$ from $\mathcal{G}_{\Jm}$ but includes all edges from $\mathcal{G}_{\Jm}$ which are incident to $w$ but not to $v$. If we have $Y_{I_v} \in P$,  by Condition~(\ref{admissibilityCondition3}) we cannot have $Y_{I_1^1(v)}$ as well.  However, we must include some of the incident to $v$ which only lie in $\mathcal{G}_{S}$. By Condition (\ref{admissibilityCondition4}) if we include some of these, we must include all of them, forcing us to include $Y_{I_1^1(v)}$.

Now we will show there does not exist an admissible matching of $\mathcal{G}_{S \cup J_{\mathbf{m}}}$ which does not include any edges incident to $v$ which come from $\mathcal{G}_S$ and includes all edges incident to $w$ but not $v$ which come from $\mathcal{G}_S$. If we include all edges incident to $w$ which lie only in $\mathcal{G}_S$, we cannot include $Y_{I_w}$ since all vertices incident to this edge have valence $1 \oplus 0$ and now some are already covered. If $\vert \mathcal{N}_w \vert = 1$, then since we assumed we did not include any edges incident to $v$ which lie only in $\mathcal{G}_S$, there are no other edges incident to the vertices $N_{\Gamma'}(I_v) \backslash \{w\}$,  so we cannot complete this to an admissible matching. If $\vert \mathcal{N}_w \vert > 1$,  then there will be an edge $Y_{I_v}$ incident to these vertices,  but we cannot include it by Condition (\ref{admissibilityCondition3}).

Lastly, the last two statements from (3) follow from the same argument that we used to show we cannot have an admissible matching of $\mathcal{G}_{S \cup \Jm}$ which includes all edges incident to $v$ which come from $\mathcal{G}_{\Jm}$ or which includes all edges incident to $v$ which come from $\mathcal{G}_S$.

\textbf{Case 2:} Now, suppose $v \lessdot w$ and the diagonal $I_v$ in $\mathcal{G}_S$ is incident to $w_\fix$. The local configuration is again drawn below; note that $w_{fix}$ is adjacent to $m-1$ more diagonals for some $m \leq \vert \mathcal{C}_w \vert$ determined by which parts of $\mathcal{G}_w$ are removed in $\mathcal{G}_w^S$.  In $\mathcal{G}_S$, the valence of $w_\fix$ is $m \oplus 0$ and the valence of $v$ is $(1 \oplus \deg_{\Gamma'}(v)  - 2)$. The valences in $\mathcal{G}_{J_{\mathbf{m}}}$ are the same as in case 1. Therefore, in $\mathcal{G}_{S \cup J_{\mathbf{m}}}$, the valence of $w$ is $(m \oplus 0)$ and the valence of $v$ is $(1 \oplus 2\deg_{\Gamma'}(v) -4)$.

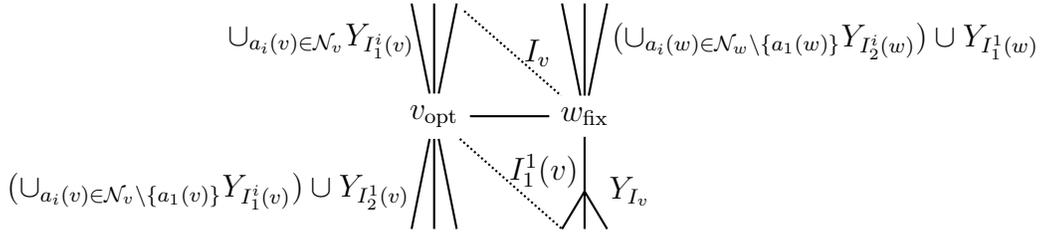
\begin{center}
\begin{tikzpicture}
\node(v) at (0,0) {$v_{\text{opt}}$};
\node(w) at (2,0) {$w_{\text{fix}}$}; 
\draw (v) -- (w) ;
\draw (v) -- (-.3,1.5);
\draw (v) -- (0,1.5);
\draw (v) -- (.3,1.5);
\node[] at (-1.5,1){$\cup_{a_i(v) \in \mathcal{N}_v} Y_{I_1^i(v)}$};
\draw (w) -- (1.7,1.5);
\draw (w) -- (2,1.5);
\draw (w) -- (2.3,1.5);
\node[] at (5.2,1){$(\cup_{a_i(w) \in \mathcal{N}_w \backslash \{a_1(w)\}} Y_{I^{i}_2(w)}) \cup Y_{I_1^{1}(w)}$};
\draw (w) -- (2,-1);
\draw (2,-1) -- (2.3,-1.5);
\draw (2,-1) -- (2,-1.5);
\draw (2,-1) --  (1.7,-1.5);
\node[] at (2.6,-1){$Y_{I_v}$};
\draw (v) -- (-.3,-1.5);
\draw (v) -- (0,-1.5);
\draw (v) -- (0.3,-1.5);
\node[] at (-3,-1){$(\cup_{a_i(v) \in \mathcal{N}_v \backslash \{a_1(v)\}} Y_{I_1^i(v)}) \cup Y_{I_2^1(v)}$};
\draw[dashed] (v) to node[right, xshift = -5pt, yshift = 5pt]{$I_1^1(v)$} (1.7,-1.5);
\draw[dashed] (0.4,1.4) to node[right]{$I_v$} (w);
\end{tikzpicture}
\end{center}

We will now prove statement (1) for this case. We can immediately rule out encountering an admissible matching of $\mathcal{G}_{S \cup J_{\mathbf{m}}}$ which includes all edges incident to $w$ coming from $\mathcal{G}_S$ or $\mathcal{G}_{J_{\mathbf{m}}}$ by valence considerations and Condition~(\ref{admissibilityCondition1}) of Definition \ref{def:AdmissibleMatching} respectively.

Suppose for sake of contradiction that there exists an admissible matching $P$ of $\mathcal{G}_{S \cup J_{\mathbf{m}}}$ containing all edges incident to $v$ which come from $\mathcal{G}_S$.  Since $P$ includes $(v_{opt},w)$, by Condition~(\ref{admissibilityCondition1}) of admissibility the edge $Y_{I_v}$ cannot be in $P$, so $P$ must contain all other edges $(v_\opt,u)$ for $u \in N_{\Gamma'}(v) \cap \Gamma^{\mathcal{I}}_{>v}$ which come from $\mathcal{G}_{J_{\mathbf{m}}}$. The only edges we do not need to include are  edges from the set $(\cup_{a_i(v) \in \mathcal{N}_v \backslash \{a_1(v)\}} Y_{I_1^i(v)}) \cup Y_{I_2^1(v)}$ which has size $r = \vert \mathcal{N}_v \vert$. If $r= 1$, then we have included $2\deg_{\Gamma'}(v) - 2$ edges incident to $v$, which violates the valence of $v$ in $\mathcal{G}_{S \cup J_{\mathbf{m}}}$.

 Now, assume $r > 1$.  Then there are at least two edges from the set of edges with weights $(\cup_{a_i(v) \in \mathcal{N}_v \backslash \{a_1(v)\}} Y_{I_1^i(v)}) \cup Y_{I_2^1(v)}$  from $\mathcal{G}_{J_{\mathbf{m}}}$ which cannot be in $P$.  At least one of these edges must have a label $I_1^i(w)$ for $a_i(w) \in \mathcal{N}_v, a_i(w) \neq a_1(w)$.  We first claim this means we cannot use the edge $Y_{I_1^1(v)}$ incident to $v_\fix$. 

The situation is drawn below. Suppose we do include this edge. Then,  $P$ must also include the edge $Y_{I_2^i(v)}$ incident to diagonal $I_1^i(v)$ since the vertices in the set $N_{\Gamma'}(I_1^1(v)) \backslash \{v\}$ have valence $1 \oplus 0$.  If $r > 2$, then by similar logic as in the proof of Lemma \ref{lem:HowDoGvRestrictToHi},  we will see that there is no way to complete this to an admissible matching and contain at least one edge incident to the vertices denoted * below. If $r = 2$, then we can use a simpler version of the logic from this proof.  

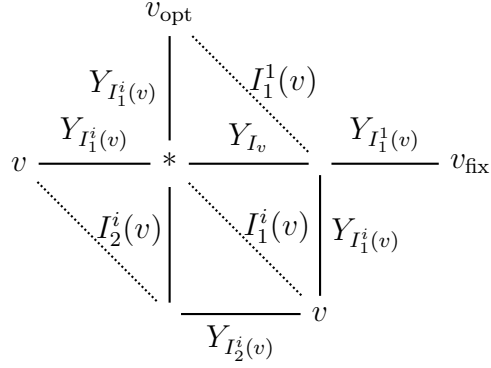
\begin{center}
\begin{tikzpicture}[scale=2]
\node(vopt) at (0,0){$v$};
\node(b12) at (1,0){*};
\node(b22) at (2,0){};
\node(v) at (1,1){$v_\opt$};
\node(vfix) at (2,-1){$v$};
\node(b11) at (1,-1){};
\node(vreallyfix) at (3,0){$v_\fix$};
\draw (vopt) to node[above]{$Y_{I_1^i(v)}$} (b12);
\draw(b12) to node[above]{$Y_{I_v}$} (b22);
\draw(b12) to node[left]{$Y_{I_1^i(v)}$} (v);
\draw(b12) to (b11);
\draw(vfix) to node[right]{$Y_{I_1^i(v)}$} (b22);
\draw[dashed] (v) to node[right, xshift = -5pt, yshift = 5pt]{$I_1^1(v)$} (b22);
\draw[dashed] (b12) to node[right, xshift = -5pt, yshift = 5pt]{$I_1^i(v)$} (vfix);
\draw (vfix) to node[below]{$Y_{I_2^i(v)}$} (b11);
\draw[dashed] (vopt) to node[right, xshift = -5pt, yshift = 5pt]{$I_2^i(v)$}(b11);
\draw (b22) to node[above]{$Y_{I_1^1(v)}$} (vreallyfix);
\end{tikzpicture}
\end{center}

Therefore,  $P$ must include all other edges incident to $v_\fix$.  In particular,  $P$ includes $I_2^i(v)$ for the designated $a_i(v) \in \mathcal{N}_v$ where $P$ does not include the edge $Y_{I_1^i(v)}$ incident to $v_\opt$. We claim that there is no way to complete this to an admissible matching such that there is at least one edge incident to any vertex in the set $N_\Gamma(I_1^i(v)) \backslash \{v\}$, in the position denoted by * below.  By assumption,  we do not include the edge $Y_{I_1^i(v)}$ incident to $v_\opt$, and the other edge labeled $Y_{I_1^i(v)}$ cannot be included by Condition~(\ref{admissibilityCondition3}) of Definition \ref{def:AdmissibleMatching}  since we include the edge $Y_{I_2^i(v)}$ incident to $v_\fix$.  We already showed we do not include $Y_{I_v}$.  We can show that the remaining set of internal edges cannot be included by an argument akin to those in the proofs of Lemma \ref{lem:r1admissible} and \ref{lem:HowDoGvRestrictToHi}.

\begin{center}
\begin{tikzpicture}[scale = 2]
\node(vopt) at (0,0){$v_\opt$};
\node(1) at (0,-1){};
\node(2) at (1,-1){};
\node(star) at (1,0){*};
\node(vfix) at (2,-1){$v_\fix$};
\node(3) at (2,0){};
\node(4)at (1,1){};
\node(5) at (2,1){};
\draw(vopt) -- (1);
\draw (vopt) to node[above]{$Y_{I_1^i(v)}$} (star);
\draw (2) to node[right]{$Y_{I_v}$} (star);
\draw(2) to node[below]{$Y_{I_1^1(v)}$} (vfix);
\draw (vfix) to node[right]{$Y_{I_2^i(v)}$} (3);
\draw(3) -- (5);
\draw(star) to node[left]{$Y_{I_1^i(v)}$} (4);
\draw(star) -- (3);
\draw[dashed] (vopt) to node[right, xshift = -5pt, yshift = 5pt]{$I_1^1(v)$} (2);
\draw[dashed] (star) to node[right, xshift = -5pt, yshift = 5pt]{$I_1^i(v)$} (vfix);
\draw[dashed] (4) to node[right, xshift = -5pt, yshift = 5pt]{$I_2^i(v)$} (3);
\end{tikzpicture} 
\end{center}

By the same logic as in case 1,  there cannot be an admissible matching of $\mathcal{G}_{S\cup J_{\mathbf{m}}}$ containing all edges incident to $v$ which come from $\mathcal{G}_{J_{\mathbf{m}}}$.  This finishes the proof of statement (1) for this case.

Statement (2) is true in this case by the same logic as in case 1. Statement (3) is true for cases ($v,\mathcal{G}_{\Jm}$), ($v,\mathcal{G}_{S}$), and ($w,\mathcal{G}_{\Jm}$) for the same logic as in case 1 as well; the case ($w,\mathcal{G}_{S}$) is impossible for some graphs $\mathcal{G}_{S \cup \Jm}$ as it is possible that $w = w_{\fix}$ has valence $r \oplus 0$ for $r > 1$. If this is not the case, then we again can use similar logic.

\textbf{Case 3:} Finally, suppose $v < w$ but $w$ does not cover $v$.  Then there is a unique vertex $u$ of $\Gamma$ such that $v \lessdot u$. The valence of $v_{\text{opt}}$ in the glued graph is again $(1 \oplus 2\deg_{\Gamma'}(v) - 4)$ and the valence of the vertex $w$ on the glued edge is $(1 \oplus 0)$.

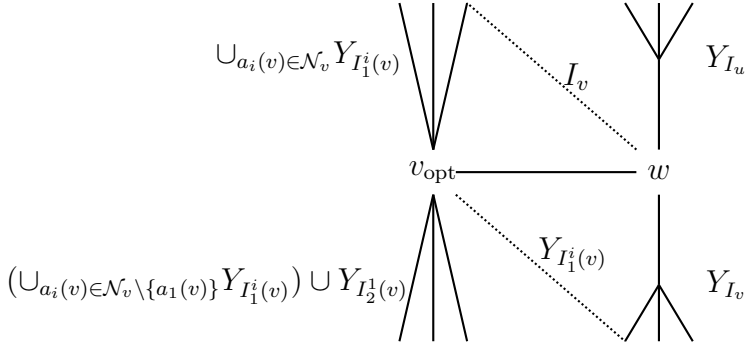
\begin{center}
\begin{tikzpicture}[scale = 1.5]
\node[white] at (4,-1){$(\cup_{a_i(v) \in \mathcal{N}_v \backslash \{a_1(v)\}} Y_{I_1^i(v)}) \cup Y_{I_2^1(v)}$};
\node[] at (0,0) {$v_{\text{opt}}$};
\node[] at (2,0) {$w$}; 
\draw (.2,0) -- (1.8,0);
\draw (0,.2) -- (-.3,1.5);
\draw (0,.2) -- (0,1.5);
\draw (0,.2) -- (.3,1.5);
\node[] at (-1.1,1){$\cup_{a_i(v) \in \mathcal{N}_v} Y_{I_1^i(v)}$};
\draw (2,.2) -- (2,1);
\draw (2,1) -- (1.7,1.5);
\draw (2,1) -- (2,1.5);
\draw (2,1) -- (2.3,1.5);
\node[] at (2.6,1){$Y_{I_{u}}$};
\draw[dashed] (1.8,0.2) to node[right]{$I_v$} (.3,1.5);
\draw (2,-.2) -- (2,-1);
\draw (2,-1) -- (2.3,-1.5);
\draw (2,-1) -- (2,-1.5);
\draw (2,-1) --  (1.7,-1.5);
\node[] at (2.6,-1){$Y_{I_v}$};
\draw (0,-.2) -- (-.3,-1.5);
\draw (0,-.2) -- (0,-1.5);
\draw (0,-.2) -- (0.3,-1.5);
\node[] at (-2,-1){$(\cup_{a_i(v) \in \mathcal{N}_v \backslash \{a_1(v)\}} Y_{I_1^i(v)}) \cup Y_{I_2^1(v)}$};
\draw[dashed] (0.2,-0.2) to node[right, xshift = -5pt, yshift = 5pt]{$Y_{I_1^i(v)}$} (1.7,-1.5);
\end{tikzpicture}
\end{center}

To prove statement (1), since the vertex $w$ has valence $1 \oplus 0$, we do not need to consider a matching using all edges incident to $w$ in $\mathcal{G}_S$.  Similarly, we do not need to consider a matching using all edges incident to $w$ in $\mathcal{G}_{J_{\mathbf{m}}}$.

If $P$ includes the edge $(v_\opt,w)$ as well as an edge from $\mathcal{G}_S$ with a weight from the set $\cup_{a_i(v) \in \mathcal{N}_v} Y_{I_1^i(v)}$,  Condition~(\ref{admissibilityCondition2}) of Definition \ref{def:AdmissibleMatching} would force us to include the edge $Y_{I_u}$ incident to $w$. This violates the valence of $w$, so we cannot have a matching which uses every edge incident to $v$ from $\mathcal{G}_S$. Showing that we will not have an admissible matching which uses all edges incident to $v$ coming from $\mathcal{G}_{J_{\mathbf{m}}}$ follows the same logic as in the previous cases as it only relies on properties of $\mathcal{G}_{J_{\mathbf{m}}}$.

Now we turn to  statement (2) and consider matchings such that the edges incident to $v$ come from one snake graph and the edges incident to $w$ come from the other snake graph. If we include the edge $Y_{I_u}$ incident to $w$ in an admissible  matching, by Condition~(\ref{admissibilityCondition2}) of Definition \ref{def:AdmissibleMatching} we must use all the edges from $\mathcal{G}_S$ labeled $I_1^i(v)$ and similarly if we include at least one such edge,  we must include the edge $Y_{I_u}$. Therefore,  it is impossible for there to be an admissible matching of $\mathcal{G}_{S \cup J_{\mathbf{m}}}$ such that the edges incident to $v$ only lie in $\mathcal{G}_S$ and the edges incident to $w$ only lie in $\mathcal{G}_{J_{\mathbf{m}}}$ or vice versa.

We can quickly show statement (3)  by combining admissibility Conditions (\ref{admissibilityCondition2}) and (\ref{admissibilityCondition4}) which imply that we have $Y_{I_u}$ in an admissible matching if and only if we have all edges which lie only in $\mathcal{G}_S$ and which are incident to $v$.
\end{proof}

\begin{lemma}\label{lem:DecomposeToAdmissible}
Consider disjoint, weakly rooted sets $S$ and $J_{\mathbf{m}}(v) = J_{\mathbf{m}}$ such that there exists $w \in \overline{S}$ with $(v,w) \in E(\Gamma)$ and $w >_{\mathcal{I}} v$. 
For any pair of matchings $(P_S,P_{J_{\mathbf{m}}}) \in \mathcal{P}_S \times \mathcal{P}_{J_{\mathbf{m}}}$ such that the edge $(v,w)$ is included in $P_S$ or $P_{J_{\mathbf{m}}}$, there is a matching $P \in \mathcal{P}_{S \cup J_{\mathbf{m}}}$ formed by taking the union of $P_S$ and $P_{J_{\mathbf{m}}}$ and deleting an instance of $(v,w)$.
Moreover, any admissible matching $P \in \mathcal{P}_{S \cup J_{\mathbf{m}}}$ has restrictions $P\vert_{\mathcal{G}_{S}}$ and $P\vert_{\mathcal{G}_{J_{\mathbf{m}}}}$ that are admissible matchings of $\mathcal{G}_{S}$ and $\mathcal{G}_{J_{\mathbf{m}}}$ respectively when $(v,w)\in P$, and can be completed to admissible matchings of $\mathcal{G}_{S}$ and $\mathcal{G}_{J_{\mathbf{m}}}$ respectively by adding the edge $(v,w)$ to exactly one of the restrictions when $(v,w)\not\in P$.

In particular, if $e = (v,w)$ simultaneously denotes this edge in $\mathcal{G}_S$ and $\mathcal{G}_{J_{\mathbf{m}}}$, there is a bijection between $\mathcal{P}_{S \cup J_{\mathbf{m}}}$ and $(\mathcal{P}_S \times \mathcal{P}_{J_{\mathbf{m}}}) \backslash (\mathcal{P}_S^{ne} \times \mathcal{P}_{J_{\mathbf{m}}}^{ne})$. 
\end{lemma}

\begin{proof}
Note there are figures in the proof of Lemma \ref{lem:AdmissibleEdges} which display the various possible local configurations of this gluing.

First we will show that for any pair of matchings $(P_S,P_{J_{\mathbf{m}}}) \in \mathcal{P}_S \times \mathcal{P}_{J_{\mathbf{m}}}$ such that the edge $(v,w)$ is included in $P_S$ or $P_{J_{\mathbf{m}}}$, there is a matching $P \in \mathcal{P}_{S \cup J_{\mathbf{m}}}$ formed by taking the union of $P_S$ and $P_{J_{\mathbf{m}}}$ and deleting an instance of $(v,w)$.
If we take a pair of $(P_S,P_{J_{\mathbf{m}}}) \in \mathcal{P}_S \times \mathcal{P}_{J_{\mathbf{m}}}$ such that $(v,w) \in P_S \cup P_{J_{\mathbf{m}}}$, we can see that all valence conditions are respected by the set $P := P_S \cup P_{J_{\mathbf{m}}} \backslash \{(v,w)\}$. We would violate Condition~(\ref{admissibilityCondition1}) of Definition \ref{def:AdmissibleMatching} if $P$ contained $(v,w)$ and the edge $Y_{I_v}$. However, if $(v,w) \in P$, $(v,w) \in P_S$ and $(v,w) \in P_{J_{\mathbf{m}}}$, and since $P_{J_{\mathbf{m}}}$ is admissible we already know that the edge $Y_{I_v}$ cannot also be in $P_{J_{\mathbf{m}}}$. We would break Condition~(\ref{admissibilityCondition3}) if $P_S$ contained the edge weighted $Y_{I_1^1(v)}$ and $P_{J_{\mathbf{m}}}$ contained the edge weighted $Y_{I_v}$. However, by Condition (\ref{admissibilityCondition4}) of Definition \ref{def:AdmissibleMatching}, including the edge $Y_{I_1^1(v)}$ in $P_S$ requires us to include all edges $Y_{I_i^1(v)}$ for $a_i(v) \in \mathcal{N}_v$, so by valence considerations the edge $(v,w_{\text{opt}})$ could not also be in $P_S$. In $\mathcal{G}_{J_{\mathbf{m}}}$, $w$ is valence $1 \oplus 0$ so $P_{J_{\mathbf{m}}}$ could also not include $(v_{\text{opt}},w)$, contradicting our assumption that the edge $(v,w)$ was in $P_S \cup P_{J_{\mathbf{m}}}$. To check that $P$ satisfies Conditions~(\ref{admissibilityCondition2}) and (\ref{admissibilityCondition4}), it suffices to know that $P_S$ and $P_{J_{\mathbf{m}}}$ already satisfy this condition. For example, this gluing does not create any new pair of edges as in the hypothesis of Condition~(\ref{admissibilityCondition4}). 

Now, consider $P \in \mathcal{P}_{S \cup \Jm}$ with $(v,w)\in P$ and take the restrictions $P_S:= P \vert_{\mathcal{G}_S}$ and $P_{\Jm}:= P\vert_{\mathcal{G}_{\Jm}}$.  We wish to show that if $(v,w) \in P$, then $P_S \in \mathcal{P}_S$ and $P_{\Jm} \in \mathcal{P}_{\Jm}$, and if $(v,w) \not\in P$, then adding $(v,w)$ to one of $P_S,P_{\Jm}$ gives an element of $\mathcal{P}_S$ and an element of $\mathcal{P}_{\Jm}$.  These restrictions would not be admissible matchings if either they include too many or two few edges incident to $v$ or $w$ or if they broke one of the conditions of Definition \ref{def:AdmissibleMatching}. 

We first address the admissibility conditions.  First,  suppose that $(v,w) \in P$,  so that $(v,w)$ is in both restrictions.  In $\mathcal{G}_{S \cup \Jm}$,  this edge is an internal edge while in $\mathcal{G}_S$ and $\mathcal{G}_{\Jm}$ this is a boundary edge.  The fact that $P$ satisfied Condition (\ref{admissibilityCondition1}) implies that the restrictions satisfy this condition as Condition (\ref{admissibilityCondition1}) is not about boundary edges.  Conditions (\ref{admissibilityCondition2}) and (\ref{admissibilityCondition3}) concern boundary edges, but since $(v,w)$ comes from an edge in $\Gamma$ so that $W((v,w)) = \emptyset$ and any diagonal $d$ has a nonempty label $W(d)$, there is nothing to check.  To check Condition (\ref{admissibilityCondition4}), notice that the edge $(v,w)$ in $\mathcal{G}_{J_{\mathbf{m}}}$ is always the edge $(v_\opt,w)$.  Therefore, the only other boundary edges that share a vertex with $(v_\opt,w)$ share vertex $v_\opt$.  However, vertex $w$ is not incident to a diagonal so none of these edges could  interact with $(v_\opt,w)$ as in Condition (\ref{admissibilityCondition4}).  Thus, Condition (\ref{admissibilityCondition4}) holds for $P\vert_{\mathcal{G}_{J_{\mathbf{m}}}}$.  The edge $(v,w)$ in $\mathcal{G}_{S}$ is not incident to any internal edges so Condition (\ref{admissibilityCondition4}) holds for this  restriction as well.  So in this case, the fact that $P$ is admissible implies that $P_S$ and $P_{\Jm}$ are both admissible.

Now, suppose that $(v,w) \notin P$.  We wish to show that we could add $(v,w)$ to exactly one of $P \vert_{\mathcal{G}_S}$ or $P\vert_{\mathcal{G}_{\Jm}}$ and have both sets of edges still satisfy Definition \ref{def:AdmissibleMatching}. Since $W((v,w)) = 1$ and all diagonals have nonempty label,  we automatically do not need to check admissibility Conditions (\ref{admissibilityCondition2}) and (\ref{admissibilityCondition3}).  In $\mathcal{G}_S$,  $(v,w)$ is a boundary edge which is either not incident to any internal edges, or is incident to $Y_{I_w}$.  Since $W(Y_{I_w})=I_w$ is not contained in the labels of any incident diagonals, we have nothing to check for admissibility Condition (\ref{admissibilityCondition4}) in $\mathcal{G}_S$. A similar argument holds for $\mathcal{G}_{\Jm}$. We would violate Condition (\ref{admissibilityCondition1}) if we added $(v,w)$ to $\mathcal{G}_{\Jm}$ and $Y_{I_v} \in P$,  or if $v$ is not covered by $w$, we added $(v,w)$ to $\mathcal{G}_S$, and $Y_{I_w} \in P$. If $v$ is not covered by $w$, the  the valence of $w$ in all three graphs is $1 \oplus 0$,  so we could not have both $Y_{I_v}$ and $Y_{I_w}$ in $P$.  Therefore,  in all cases we could add $(v,w)$ to one of the restrictions and preserve admissibility. 

Now, we compare the number of edges incident to $v$ and $w$ in the restrictions with the valence of these vertices in both graphs. In all of our cases,  the valence of $v$ or $w$ in each of the smaller graphs will be of the form $a \oplus b$ where the vertex is incident to $a+b+1$ edges.  Therefore,  we will have too many edges incident to a vertex in the restriction only if the restriction includes every edge incident to this vertex.  This occurs if $P$ contains all edges incident to $v$ or $w$ which comes from $\mathcal{G}_S$ or $\mathcal{G}_{\Jm}$. By part (1) of Lemma \ref{lem:AdmissibleEdges}, there cannot be such a matching $P \in \mathcal{P}_{S \cup \Jm}$.

In most of our cases,  the first component of the valence of $v$ or $w$ in the smaller graphs is 1, so in these cases,  we will not have enough edges incident to this vertex in the restriction if there are no edges incident to it. In order for this to occur,  we certainly have $(v,w) \notin P$ since otherwise $(v,w)$ would be in both restrictions.  Moreover,  if $(v,w) \notin P$, then we add $(v,w)$ to one of the restrictions.  Therefore, in order for this process to not include an edge incident to $v$ or $w$ in one of the restrictions, we must either have a situation as in statement (2) of Lemma \ref{lem:AdmissibleEdges}, where no matter which restriction we add $(v,w)$ to, a vertex is unmatched, or a situation as in statement (3) of this lemma, where we have an uncovered vertex in a restriction, but we cannot add $(v,w)$ without violating another vertex's valence. As the lemma shows, we will never encounter such a matching of $\mathcal{G}_{S \cup \Jm}$.

We have one remaining way in which a restriction can include too few edges incident to a vertex. Suppose that $v \lessdot_{\mathcal{I}} w$ and $v = a_1^i(w)$ with $i > 1$, so that the edge $(v,w)$ in $\mathcal{G}_S$ is incident to $w_{\fix}$. Then, the valence of $w$ on this edge in both $\mathcal{G}_S$ and $\mathcal{G}_{S \cup \Jm}$ is $r \oplus 0$ for some $r$ possibly larger than 1 where  $w$ is adjacent to $r+2$ edges in $\mathcal{G}_{S \cup \Jm}$ and $r+1$ edges in $\mathcal{G}_S$. We would not be able to decompose $P \in \mathcal{P}_{S \cup \Jm}$ into admissible matchings of $\mathcal{G}_{S}$ and $\mathcal{G}_{\Jm}$ if $P$ included $Y_{I_v}$, the one edge incident to $w$ which lies only in $\mathcal{G}_{\Jm}$, did not include $(v,w)$, and included all edges incident to $v$ but not $w$ which lie in $\mathcal{G}_S$. This would mean that $w$ would not have enough incident edges in $P \vert_{\mathcal{G}_S}$ but we could not add $(v,w)$ to this set without violating the valence of $v$. However,  as before,  we would now include the edge $Y_{I_v}$ incident to $w$ and the edge $Y_{I_1^1(v)}$ incident to $v$,  which violates Condition (\ref{admissibilityCondition3}) of Definition \ref{def:AdmissibleMatching}. This completes our proof, as we have shown that any $P \in \mathcal{P}_{S \cup \Jm}$ can be decomposed into matchings of $\mathcal{G}_{S}$ and $\mathcal{G}_{\Jm}$.

The last statement in the lemma follows directly from the first two statements.
\end{proof}

Now that we understand how the admissible matchings of a snake graph $\mathcal{G}_S$ compare with admissible matchings of the snake graphs that are glued together to make $\mathcal{G}_S$, we analyze the weighted sum of matchings which avoid the edges used in gluings. 

In what follows, if $T$ is a disconnected set whose connected components are weakly rooted,  we define the snake graph $\mathcal{G}_T$ to simply be the a disconnected graph with connected components $\mathcal{G}_{T_i}$ for the connected components $T_i$ of $T$.  Given a weakly rooted set $S$ and $v \in \overline{S}$, it follows from the definition that the connected components of $S \backslash \{v\}$ are all weakly rooted. Note that a set $I \in \mathcal{I}$ is trivially weakly rooted as its rooted portion is the empty set. Recall that $Z(v) = \frac{Y_{I_v}}{\prod_{u \in \mathcal{C}_v} Y_{I_u}}$.

\begin{lemma}\label{lemma:term_cancellation}
\begin{enumerate}
    \item Let $S$ be a weakly rooted set. Suppose there exists $w \in \overline{S}$ and $v \in N_\Gamma(w) \cap \Gamma^{\mathcal{I}}_{<w}$, $v \notin S$. Let $e$ be the boundary edge $(v,w)$ of $\mathcal{G}_S$. Then, 
    \[
    \frac{1}{\ell(\mathcal{G}_S)}\sum_{P \in \mathcal{P}_S^{ne}} \wt(P) = \frac{1}{Z(v)\ell(\mathcal{G}_{S \backslash \{w\}})} \sum_{P \in \mathcal{P}_{S \backslash \{w\}}} \wt(P)
    \]
    where, if $S = \{w\}$, we set $\mathcal{P}_{\emptyset} = \{\emptyset\}$ and $\wt(\emptyset) = 1$. 
    \item Let $w \in N_{\Gamma'}(v) \cap \Gamma^{\mathcal{I}}_{>v}$ and let $e$ be the boundary edge $(v,w)$ of $\mathcal{G}_{J_{\mathbf{m}}(v)}$. Then, \[
    \frac{1}{\ell(\mathcal{G}_{J_{\mathbf{m}}(v)})} \sum_{P \in \mathcal{P}_{J_{\mathbf{m}}}^{ne}} \wt(P) = Z(v) \cdot Y_{I^1_{m_1}} \cdots Y_{I^r_{m_r}}.
    \]
\end{enumerate}
\end{lemma}

\begin{proof}

\textbf{(1)} First, suppose that $\mathcal{G}_w$ is a subgraph of $\mathcal{G}_S$. We note in Corollary \ref{cor:WhichMatchingsAvoidDominant} that, for the snake graph $\mathcal{G}_w$, there is a unique admissible matching which does not include $e$. By following the set of edges this matching includes, as in the proof of Lemma \ref{lem:HowDoGvRestrictToHi}, we see that the weight of this matching divided by $\ell(\mathcal{G}_S)$ is $\frac{1}{Z(v)}$ and that this matching contains all edges of the form $(w,u)$ for $u \in N_\Gamma(w)$. Therefore, by Lemma \ref{lem:DecomposeToAdmissible}, this matching can be combined with any matchings of the graphs $\mathcal{G}_{S_i}$ for the connected components $S_i$ of $S \backslash \{w\}$ and the claim follows.

Now, more generally we could have that $\mathcal{G}_w$ is not a subgraph of $\mathcal{G}_S$, but for some $J_{\mathbf{m}}(w) = \{w\} \cup I^1_{m_1}(w) \cup \cdots \cup I^r_{m_r}(w)$, $1 \leq m_i \leq c_i(w)+1$ for all $1 \leq i \leq \vert N_{\Gamma'}(w)\cap \Gamma^{\mathcal{I}}_{<w}\vert$, $\mathcal{G}_{J_{\mathbf{m}}(w)}$ is a subgraph of $\mathcal{G}_S$. Since $I_w \nsubseteq S$, there is at least one $i$ such that $m_i > 1$.  Here, the same logic as in the simpler case holds, but the matching of $\mathcal{G}_{J_{\mathbf{m}}(w)}$ which does not include $e$ includes the boundary edges which have weight $Y_{I_{m_j}^j(w)}$. When $m_i < c_i+1$, so that this weight is not 1, we do not glue onto these boundary edges, so we can use Lemma \ref{lem:DecomposeToAdmissible} to again extend this admissible matching to one on the whole graph $\mathcal{G}_S$.

\textbf{(2)} The same logic for Corollary \ref{cor:WhichMatchingsAvoidDominant} holds for a snake graph of the form $\mathcal{G}_{J_{\mathbf{m}}}$. The claim follows from evaluating the weights of the unique matching which does not use the edge $(v,w)$. 

\end{proof}

\begin{proof}[Proof of Theorem \ref{thm:main} for general weakly rooted sets]

We will use strong induction on the size of $\vert \overline{S} \vert$.  The $\vert \overline{S} \vert = 1$ case was shown in Section \ref{subsec:SingeltonProofWeakSingleton},  so we assume $\vert \overline{S} \vert > 1$.  

Let $v$ be a leaf in the subgraph $\overline{S}$ such that $v \neq m_{\mathcal{I}}(S)$.
Let $S'$ be the connected component of $S\setminus\{v\}$ containing $m_{\mathcal{I}}(S)$. It follows that the set $S \setminus S'$ is of the form $J_{\mathbf{m}}(v)$.  Since $v \in \bar{S}$, we know that, in $\mathbf{m} = (m_1,\ldots,m_{r)}$, at least one $m_i > 1$.  

Notice that there exists $w \in \overline{S'}$ such that $(v,w) \in E(\Gamma)$. Since we assumed $v \neq m_{\mathcal{I}}(S)$, by the weakly rooted condition we know that $v<_{\mathcal{I}} w$.  Let $e$ simultaneously denote the edge $(v,w)$ in $\mathcal{G}_{S'}$ and $\mathcal{G}_{J_{\mathbf{m}}}$.  Then, by Lemma \ref{lem:DecomposeToAdmissible},  we can rewrite $\chi(\mathcal{G}_S)$ in the following way
\begin{align*}
\chi(\mathcal{G}_{S}) &= \chi(\mathcal{G}_{S' \cup J_{\mathbf{m}}}) \\&= \frac{1}{\ell(\mathcal{G}_{S' \cup J_{\mathbf{m}}})} \sum_{P \in \mathcal{P}_{S' \cup J_{\mathbf{m}}}} \wt(P) \\&= \frac{1}{\ell(\mathcal{G}_{S'})\ell(\mathcal{G}_{J_{\mathbf{m}}})} \bigg(\sum_{P \in \mathcal{P}_{S'}} \wt(P) \sum_{P \in \mathcal{P}_{J_{\mathbf{m}}}} \wt(P) - \sum_{P \in \mathcal{P}^{ne}_{S'}} \wt(P) \sum_{P \in \mathcal{P}^{ne}_{J_{\mathbf{m}}}} \wt(P) \bigg).
\end{align*}
By the inductive hypothesis, $Y_{S'} = \chi(\mathcal{G}_{S'})$ and $Y_{J_{\mathbf{m}}} = \chi(\mathcal{G}_{J_{\mathbf{m}}})$.  We can use Lemma \ref{lemma:term_cancellation} to reduce the second term, to obtain the following, 
\begin{align*}
\chi(\mathcal{G}_{S}) = Y_{S'} Y_{J_{\mathbf{m}}} - Y_{I_{m_1+1}^1(v)} \cdots Y_{I_{m_r+1}^r(v)}  \frac{1}{\ell(\mathcal{G}_{S' \backslash \{w\}})}\sum_{P \in \mathcal{P}_{S' \backslash \{w\}}} \wt(P) = Y_{S'} Y_{J_{\mathbf{m}}} -  Y_{J_{\mathbf{m}} \backslash \{v\}}Y_{S' \backslash \{w\}}
\end{align*}

where the last equality follows from invoking the inductive hypothesis on each connected component of $S' \backslash \{w\}$.  
Finally,  we can now use Lemma \ref{lem:YRecursionRootedSet} to conclude $\chi(\mathcal{G}_S) = Y_S$.
\end{proof}

\section{Counting Admissible Matchings}\label{sec:Counting}

We give a determinantal description of the number of admissible matchings of $\mathcal{G}_S$ for a weakly rooted set $S$.  For convenience, relabel $V(\Gamma)$ so that the vertices in $\overline{S}$ are numbered $1,\ldots,s$. For any $v \in \overline{S}$ with $\vert N_\Gamma(v)\cap\Gamma_{<v} \vert = r$, let $m_1,\ldots,m_r$ be minimal such that $\{v\} \cup I_{m_1}^1(v) \cup \cdots \cup I_{m_r}^r(v) \subseteq S$. Then, define $f(v) = 1 + (m_1 - 1) + \cdots + (m_r-1)$.
Let $N_S = (n_{i,j})$ be the $s \times s$ matrix defined by \[
n_{i,j} = \begin{cases} f(i) & i = j \\ -1 & (i,j) \in E(\Gamma) \\ 0 & \text{ otherwise } \end{cases}
\]

This matrix is related to the $\mathfrak{N}$ matrix from Proposition 4.5 in \cite{lam2016linear}.

\begin{thm}\label{thm:Counting}
The number of admissible matchings of $\mathcal{G}_S$ for a weakly rooted set $S$ is equal to $\det(N_S)$.
\end{thm}

\begin{proof}
We induct on $\vert \bar{S} \vert$, the size of the weakly rooted portion of $S$. First, suppose that $\bar{S} = \{v\}$, so that $S = J_{\mathbf{m}}(v)$. In our proof for Theorem \ref{thm:main} for this case, following the reasoning for $S = \{v\}$ in Lemma \ref{lem:DecomposeAdimssibleBranch}, we see that, for each $1 \leq i \leq r$, there are $m_i-2$ admissible matchings which use an internal edge which shares vertices with diagonals $I_j^i$ and $I_{j+1}^i$ and one matching which uses the edge $Y_{I_{m_i-2}^i}$ incident to the diagonal $I_{m_i-1}^i$, and these sets of matchings are all distinct. There is one other matching besides these families which uses the edge weighted $Y_{I_v}$. Therefore, there are $(m_1-1) + \cdots + (m_r-1) + 1$ admissible matchings of $\mathcal{G}_{J_{\mathbf{m}}}$, showing the claim in this case.

Now, we consider a set $S$ such that $\vert \bar{S} \vert > 1$. We assume we have shown the claim for all sets with a strictly smaller weakly rooted portion. Let $v$ be a leaf in  $\bar{S}$, and let $w$ be the unique vertex in $\bar{S}$ adjacent to $v$. As in the proof of Theorem \ref{thm:main} for the general case, let $S'$ be the connected component of $S \backslash \{v\}$ containing $m_{\mathcal{I}}(S)$ so that $S \backslash S' = J_{\mathbf{m}}(v)$ with at least one $m_i > 1$ in $\mathbf{m}$.  Then, since the $v$-th row of $N_S$ only has a $-1$ in column $w$ and $f(v)$ in column $v$ and similarly for the $v$-th column, we can compute
\begin{align*}
\det(N_S) &= f(v) \det(N_{S \backslash J_{\mathbf{m}}(v)}) - \det(N_{S \backslash (J_{\mathbf{m}}(v) \cup \{w\})})\\
&= \vert \mathcal{P}_{J_{\mathbf{m}}(v)} \vert \vert \mathcal{P}_{S \backslash J_{\mathbf{m}}(v)}\vert  - \vert \mathcal{P}_{S \backslash (J_{\mathbf{m}}(v) \cup \{w\})}\vert\\
\end{align*}
where the second equality follows from our inductive hypothesis. 

From Lemma \ref{lem:DecomposeToAdmissible}, we know that every admissible matching of $\mathcal{G}_S$ can be decomposed into an admissible matching of $\mathcal{G}_{S \backslash J_{\mathbf{m}(v)}}$ and of $\mathcal{G}_{J_{\mathbf{m}(v)}}$ where at least one matching uses the edge $e = (v,w)$. It follows from Lemma \ref{lemma:term_cancellation} that $\vert \mathcal{P}_{J_{\mathbf{m}}(v)}^{ne} \vert = 1$ and  $\vert \mathcal{P}^{ne}_{S \backslash J_{\mathbf{m}}(v)} \vert = \vert \mathcal{P}_{S \backslash (J_{\mathbf{m}}(v) \cup w)} \vert$. 
Therefore, we conclude that $\det(N_S) = \vert \mathcal{P}_S \vert$. 
\end{proof}

When $\Gamma$ is a path graph on $n$ vertices, the subalgebra of $\mathcal{A}_\Gamma$ spanned by variables of the form $Y_S, S \subseteq V(\Gamma)$ is a type $A_n$ cluster algebra by Corollary 6.2 in \cite{lam2016linear}. Then, up to reordering, the matrix $N_S$ has values $f(i)$ on the main diagonal and $-1$'s on the super and sub diagonal. Such a matrix is known as a \emph{continuant}.  The determinant of such a continuant with entries $f(1),\ldots,f(s)$ on the main diagonal is given by the \emph{negative continued fraction} $[[f(1),\ldots,f(s)]]$, defined by \[
f(1) - \frac{1}{f(2) - \frac{1}{\ddots - \frac{1}{f(s)}}}. 
\]

In this case, Theorem \ref{thm:Counting} recovers a previously known formula for the number of perfect matchings of a surface snake graph $G_{\gamma}$.  To illustrate this, we briefly discuss the case of a cluster algebra from a polygon; for more general cluster algebras, one can take a suitable lift of the surface. Suppose $\gamma$ is a diagonal in a polygon with vertices $v_0,\ldots,v_{n-1}$ and triangulation $T$. Let $\gamma = (v_i,v_j)$ where, without loss of generality $i < j$. For all $i < k < j$, let $f_k$ denote the number of arcs from $T$ incident to $v_k$ which cross $\gamma$. Then, if $\mathcal{P}_\gamma$ denotes the set of perfect matchings of $\mathcal{G}_\gamma$, \[
\vert \mathcal{P}_\gamma \vert = [[f_{i+1},\ldots,f_{j-1}]].
\]

This result can be seen as combining the theory of Conway-Coxeter's frieze patterns \cite{CCFrieze} with the connection between frieze patterns and cluster algebras of type $A$ \cite{caldero2006cluster}. More details on the correspondence are given in \cite{morier2019farey}.

\section*{Acknowledgements}

The authors would like to thank Pavlo  Pylyavskyy for suggesting the original idea for this project and MRWAC 2020 for providing the setting for our initial collaboration.  We would also like to thank the Summer Research in Mathematics (SRiM) program hosted by SLMath (formerly MSRI) for their financial support and hospitality.

Esther Banaian was supported by Research Project 2 from the Independent Research Fund Denmark (grant no. 1026-00050B).  Elizabeth Kelley was supported by NSF Grant No. DMS-1937241.  Sylvester Zhang was partially supported by NSF Grants No. DMS-1949896 and No. DMS-1745638.

\newpage
\thispagestyle{fancy}
\fancyhf{}
\rhead{\footnotesize \thepage}
\renewcommand{\headrulewidth}{0pt}

\printbibliography

\end{document}

%% file: command.tex
\usepackage{ifthen}

\newcommand{\gtilecontinued}[1]{
\begin{tikzpicture}
            \node (i) at (0,0) {$\D$};
            \node (j) at (0,2) {$\A$};
            \node (jj) at (2,2) {$\B$};
            \node (Rr) at (2,0) {$\C$};
            \draw (i)--(j);
            \ifnum \I>0 {\draw[ultra thick](j)--node [above] {$\H$}(jj);}\else{\draw[](j)--(jj);}\fi
            
            \ifnum #1>0{
            \draw [ultra thick] (i)--node [below] {$\E$} (Rr);}\else{
            \draw (i)--node [below] {} (Rr);}\fi

            \draw (Rr)--node [right] {$\F$}(jj);
            \draw [dotted] (j)-- node [midway, fill=white] {$\G$}(Rr);
        \end{tikzpicture}
}

\newcommand{\tilecontinued}[2]{
\begin{tikzpicture}
            \node (i) at (0,0) {$\D$};
            \node (j) at (0,2) {$\A$};
            \node (jj) at (2,2) {$\B$};
            \node (Rr) at (2,0) {$\C$};

            \ifnum #1>0 {\draw[ultra thick](j)--node [above] {$\E$}(jj);}\else{\draw[](j)--(jj);}\fi
            
             \draw (Rr)--node [right] {$\F$}(jj);
            
           \ifnum #2>0{\draw [ultra thick] (i)--node [below] {$\G$} (Rr);}\else{
            \draw (i)--node [below] {} (Rr);}\fi
                 
            \draw (i)-- node [left] {$\H$} (j);
            
            \draw [dotted] (j)-- node [midway, fill=white] {$\I$}(Rr);
\end{tikzpicture}
}

\DeclareMathOperator{\wt}{wt}

\renewcommand{\phi}{\varphi}

\renewcommand{\bar}[1]{\overline{#1}}

\newcommand{\calC}{\mathcal{C}}

\newcommand{\C}{\mathbb{C}}

%% file: tikz-setup.tex
\usepackage{tikz}
\usetikzlibrary{backgrounds}
\usetikzlibrary{arrows}
\usetikzlibrary{shapes,shapes.geometric,shapes.misc}
\tikzstyle{tikzfig}=[baseline=-0.25em,scale=0.5]
\pgfkeys{/tikz/tikzit fill/.initial=0}
\pgfkeys{/tikz/tikzit draw/.initial=0}
\pgfkeys{/tikz/tikzit shape/.initial=0}
\pgfkeys{/tikz/tikzit category/.initial=0}
\pgfdeclarelayer{edgelayer}
\pgfdeclarelayer{nodelayer}
\pgfsetlayers{background,edgelayer,nodelayer,main}

\tikzstyle{none}=[inner sep=0mm]
\tikzstyle{Vert}=[fill=black, draw=black, shape=circle, minimum size=0.4em, inner sep=0.4pt, scale=0.6]
\tikzstyle{label-red}=[fill=white, draw=black, shape=circle, scale=0.4, color=red]
\tikzstyle{label small}=[fill=none, draw=none, shape=circle, scale=0.7, inner sep = 0.1]
\tikzstyle{bullet}=[fill=black, draw=black, shape=circle, scale=0.5]
\tikzstyle{G-vert}=[fill=white, draw=black, shape=circle, inner sep=0.7pt, minimum size=0.4em, scale=0.6]
\tikzstyle{label}=[inner sep=0.15mm, font={\footnotesize}]
\tikzstyle{vert}=[fill=black, draw=white, shape=circle, line width=0.25mm, tikzit shape=circle, inner sep=0.35mm]
\tikzstyle{nodes}=[fill=white, draw=none, shape=circle, inner sep=0.7pt, font={\scriptsize}, minimum size=11pt]
\tikzstyle{label-s}=[inner sep=0.1mm, font={\scriptsize}]

\tikzstyle{label-w}=[fill=white, draw=white, shape=circle, inner sep=0.07mm, font={\scriptsize}]

\tikzstyle{dashed}=[-, densely dotted]
\tikzstyle{virtual}=[-, double]
\tikzstyle{RED}=[-, draw=red]
\tikzstyle{CYAN}=[-, draw=cyan]
\tikzstyle{BLUE}=[-, draw=blue]
\tikzstyle{LGREEN}=[-, draw=green]
\tikzstyle{DGREEN}=[-, draw={rgb,255: red,0; green,128; blue,128}]
\tikzstyle{PURPLE}=[-, draw={rgb,255: red,128; green,0; blue,128}]
\tikzstyle{ORANGE}=[-, draw={rgb,255: red,255; green,128; blue,0}]
\tikzstyle{MAGENTA}=[-, draw=magenta]
\tikzstyle{BLUE(dashed)}=[-, draw=blue, densely dotted]
\tikzstyle{GREEN}=[-, draw={rgb,255: red,0; green,208; blue,6}]
\tikzstyle{MAGENTA(dashed)}=[-, draw=magenta, densely dotted]
\tikzstyle{ORANGE(dashed)}=[-, draw={rgb,255: red,255; green,128; blue,0}, densely dotted]
\tikzstyle{RED(dashed)}=[-, draw={rgb,255: red,191; green,0; blue,64}, densely dotted]
\tikzstyle{PURPLE(dashed)}=[-, draw={rgb,255: red,128; green,0; blue,128}, densely dotted]
\tikzstyle{GREEN(dashed)}=[-, draw={rgb,255: red,0; green,208; blue,6}, densely dotted]
\tikzstyle{shade}=[-, opacity=0.4, draw={rgb,255: red,128; green,128; blue,128}, line width=6.5, fill=none, line cap=round,rounded corners]
\tikzstyle{hyper}=[-, fill={rgb,255: red,48; green,255; blue,214}, fill opacity=0.6, draw={rgb,255: red,18; green,229; blue,85}, tikzit fill={rgb,255: red,48; green,255; blue,214}]
\tikzstyle{hyper1}=[-, fill={rgb,255: red,230; green,138; blue,9}, draw={rgb,255: red,255; green,128; blue,0}, fill opacity=0.5]
\tikzstyle{hyper2}=[-, fill={rgb,255: red,128; green,179; blue,255}, draw={rgb,255: red,46; green,87; blue,115}, fill opacity=0.55]
\tikzstyle{new edge style 0}=[-, fill={rgb,255: red,245; green,255; blue,39}, draw={rgb,255: red,168; green,170; blue,22}, fill opacity=0.6]
\tikzstyle{blue}=[-, draw={rgb,255: red,49; green,149; blue,255}, line width=1.5pt]
\tikzstyle{red}=[-, draw=red, line width=1.5pt]
\tikzstyle{blue thick}=[-, tikzit draw=blue, draw=blue, line width=1.7pt]
\tikzstyle{thick}=[-, ultra thick]
\tikzstyle{dotted}=[-, densely dotted]
\tikzstyle{e1}=[-, color=1, line width=1]
\tikzstyle{e2}=[-, color=2, line width=1]
\tikzstyle{e3}=[-, color=3, line width=1]
\tikzstyle{e4}=[-, color=4, line width=1]
\tikzstyle{e5}=[-, color=5, line width=1]
\tikzstyle{e6}=[-, color=6, line width=1]
\tikzstyle{e7}=[-, color=7, line width=1]
\tikzstyle{e8}=[-, color=8, line width=1]
\tikzstyle{e9}=[-, color=9, line width=1]
\tikzstyle{e0}=[-, color=0, line width=1]
\tikzstyle{d1}=[-, color=1, line width=1.2, loosely dashed]
\tikzstyle{d2}=[-, color=2, line width=1.2, loosely dashed]
\tikzstyle{d3}=[-, color=3, line width=1.2, loosely dashed]
\tikzstyle{d4}=[-, color=4, line width=1.2, loosely dashed]
\tikzstyle{d5}=[-, color=5, line width=1.2, loosely dashed]
\tikzstyle{d6}=[-, color=6, line width=1.2, loosely dashed]
\tikzstyle{d7}=[-, color=7, line width=1.2, loosely dashed]
\tikzstyle{d8}=[-, color=8, line width=1.2, loosely dashed]
\tikzstyle{d9}=[-, color=9, line width=1.2, loosely dashed]
\tikzstyle{d0}=[-, color=0, line width=1.2, loosely dashed]

\tikzstyle{m1}=[-, line width=2.5]

\tikzset{%
every path/.append style={line width = 0.8 pt}
}

\definecolor{2}{rgb}{0.91, 0.33, 0.5}
\definecolor{1}{rgb}{0.0, 1.0, 1.0} 
\definecolor{8}{rgb}{1.0, 0.49, 0.0}
\definecolor{3}{rgb}{0.71, 0.49, 0.86}
\definecolor{6}{rgb}{0.5, 0, 0}
\definecolor{5}{rgb}{0.0, 0.5, 1.0}
\definecolor{0}{rgb}{0.55, 0.71, 0.0}
\definecolor{7}{rgb}{1.0, 0.75, 0.0}
\definecolor{4}{rgb}{0.2, 0.2, 0.6}
\definecolor{9}{rgb}{0.24, 0.17, 0.12}

%% file: figures/ex1.tex
\begin{tikzpicture}[scale = 0.5]
	\begin{pgfonlayer}{nodelayer}
		\node [style=label-s] (0) at (-8, 2) {$1$};
		\node [style=label-s] (1) at (-5, 2) {$2$};
		\node [style=label-s] (2) at (-2, 2) {$3$};
		\node [style=label-s] (3) at (2, 2) {$1$};
		\node [style=label-s] (4) at (5, 2) {$2$};
		\node [style=label-s] (5) at (8, 2) {$3$};
		\node [style=label-s] (6) at (-8, -2) {$1$};
		\node [style=label-s] (7) at (-5, -2) {$2$};
		\node [style=label-s] (8) at (-2, -2) {$3$};
		\node [style=label-s] (9) at (2, -2) {$1$};
		\node [style=label-s] (10) at (5, -2) {$2$};
		\node [style=label-s] (11) at (8, -2) {$3$};
		\node [style=none] (12) at (-8, 2.5) {};
		\node [style=none] (13) at (-8, 1.5) {};
		\node [style=none] (14) at (-8, 3) {};
		\node [style=none] (15) at (-8, 1) {};
		\node [style=none] (16) at (-5, 3) {};
		\node [style=none] (17) at (-5, 1) {};
		\node [style=none] (18) at (-8, -1) {};
		\node [style=none] (19) at (-8, -3) {};
		\node [style=none] (20) at (-2, -1) {};
		\node [style=none] (21) at (-2, -3) {};
		\node [style=none] (22) at (-5, 2.75) {};
		\node [style=none] (23) at (-5, 1.25) {};
		\node [style=none] (24) at (-2, 2.75) {};
		\node [style=none] (25) at (-2, 1.25) {};
		\node [style=none] (26) at (2, 2.5) {};
		\node [style=none] (27) at (2, 1.5) {};
		\node [style=none] (28) at (8, 2.5) {};
		\node [style=none] (29) at (8, 1.5) {};
		\node [style=none] (30) at (2, -1.5) {};
		\node [style=none] (31) at (2, -2.5) {};
		\node [style=none] (32) at (8, -1.5) {};
		\node [style=none] (33) at (8, -2.5) {};
		\node [style=none] (34) at (-8, -1.5) {};
		\node [style=none] (35) at (-8, -2.5) {};
		\node [style=none] (36) at (-2, -1.5) {};
		\node [style=none] (37) at (-2, -2.5) {};
		\node [style=none] (38) at (2, 3) {};
		\node [style=none] (39) at (2, 1) {};
		\node [style=none] (40) at (5, 3) {};
		\node [style=none] (41) at (5, 1) {};
		\node [style=none] (42) at (2, -1) {};
		\node [style=none] (43) at (2, -3) {};
		\node [style=none] (44) at (8, -1) {};
		\node [style=none] (45) at (8, -3) {};
		\node [style=none] (46) at (-6.5, -3) {};
		\node [style=none] (47) at (-3.5, -3) {};
		\node [style=none] (48) at (-5.5, -2) {};
		\node [style=none] (49) at (-4.5, -2) {};
	\end{pgfonlayer}
	\begin{pgfonlayer}{edgelayer}
		\draw (0) to (1);
		\draw (1) to (2);
		\draw (3) to (4);
		\draw (4) to (5);
		\draw (6) to (7);
		\draw (7) to (8);
		\draw (9) to (10);
		\draw (10) to (11);
		\draw [bend left=90, looseness=1.75] (12.center) to (13.center);
		\draw [bend right=90, looseness=1.75] (12.center) to (13.center);
		\draw (14.center) to (16.center);
		\draw [bend left=90, looseness=1.75] (16.center) to (17.center);
		\draw (17.center) to (15.center);
		\draw [bend right=270, looseness=1.75] (15.center) to (14.center);
		\draw (18.center) to (20.center);
		\draw [bend left=90, looseness=1.75] (20.center) to (21.center);
		\draw [bend right=270, looseness=1.75] (19.center) to (18.center);
		\draw (22.center) to (24.center);
		\draw [bend left=90, looseness=1.75] (24.center) to (25.center);
		\draw (25.center) to (23.center);
		\draw [bend right=270, looseness=1.75] (23.center) to (22.center);
		\draw [bend left=90, looseness=1.75] (26.center) to (27.center);
		\draw [bend right=90, looseness=1.75] (26.center) to (27.center);
		\draw [bend left=90, looseness=1.75] (28.center) to (29.center);
		\draw [bend right=90, looseness=1.75] (28.center) to (29.center);
		\draw [bend left=90, looseness=1.75] (30.center) to (31.center);
		\draw [bend right=90, looseness=1.75] (30.center) to (31.center);
		\draw [bend left=90, looseness=1.75] (32.center) to (33.center);
		\draw [bend right=90, looseness=1.75] (32.center) to (33.center);
		\draw [bend left=90, looseness=1.75] (34.center) to (35.center);
		\draw [bend right=90, looseness=1.75] (34.center) to (35.center);
		\draw [bend left=90, looseness=1.75] (36.center) to (37.center);
		\draw [bend right=90, looseness=1.75] (36.center) to (37.center);
		\draw (38.center) to (40.center);
		\draw [bend left=90, looseness=1.75] (40.center) to (41.center);
		\draw (41.center) to (39.center);
		\draw [bend right=270, looseness=1.75] (39.center) to (38.center);
		\draw (42.center) to (44.center);
		\draw [bend left=90, looseness=1.75] (44.center) to (45.center);
		\draw (45.center) to (43.center);
		\draw [bend right=270, looseness=1.75] (43.center) to (42.center);
		\draw (19.center) to (46.center);
		\draw (47.center) to (21.center);
		\draw [bend left=90, looseness=2.00] (48.center) to (49.center);
		\draw [bend right=45] (46.center) to (48.center);
		\draw [bend right=45, looseness=1.25] (49.center) to (47.center);
	\end{pgfonlayer}
\end{tikzpicture}

%% file: figures/graph_fig.tex
\begin{tikzpicture}[scale=0.7]
	\begin{pgfonlayer}{nodelayer}
		\node [style=label] (0) at (-9, 0) {$1$};
		\node [style=label] (1) at (-7, 0) {$2$};
		\node [style=label] (2) at (-5, 0) {$3$};
		\node [style=label] (3) at (-3, 0) {$4$};
		\node [style=label] (4) at (-1, 0) {$5$};
		\node [style=label] (5) at (1, 0) {$6$};
		\node [style=label] (6) at (3, 0) {$7$};
		\node [style=label] (7) at (-5, -1.75) {$8$};
		\node [style=label] (8) at (-3, -3) {$9$};
		\node [style=label] (9) at (-1, -1.75) {$0$};
		\node [style=none] (12) at (-5, -1.25) {};
		\node [style=none] (13) at (-5, -2.5) {};
		\node [style=none] (14) at (-7, 0.5) {};
		\node [style=none] (15) at (-7, -0.5) {};
		\node [style=none] (16) at (1, 0.5) {};
		\node [style=none] (17) at (1, -0.5) {};
		\node [style=none] (18) at (-8, 1.75) {};
		\node [style=none] (19) at (-8, -3.5) {};
		\node [style=none] (20) at (2.25, 1.75) {};
		\node [style=none] (21) at (2.25, -3.5) {};
		\node [style=none] (22) at (-1, 0.75) {};
		\node [style=none] (23) at (-1, -0.75) {};
		\node [style=none] (24) at (1, 0.75) {};
		\node [style=none] (25) at (1, -0.75) {};
		\node [style=none] (26) at (-9, 0.75) {};
		\node [style=none] (27) at (-9, -0.75) {};
		\node [style=none] (28) at (-7, 0.75) {};
		\node [style=none] (29) at (-7, -0.75) {};
		\node [style=none] (30) at (-1, 1.25) {};
		\node [style=none] (31) at (-1, -2.5) {};
		\node [style=none] (32) at (2.25, 1.25) {};
		\node [style=none] (33) at (2.25, -1.5) {};
		\node [style=none] (34) at (-0.25, -2) {};
		\node [style=none] (35) at (0.25, -1.5) {};
		\node [style=none] (36) at (-1.75, 1.5) {};
		\node [style=none] (37) at (-1.75, -2.75) {};
		\node [style=none] (38) at (2.25, 1.5) {};
		\node [style=none] (39) at (2.25, -1.75) {};
		\node [style=none] (40) at (0.25, -2) {};
		\node [style=none] (41) at (0.5, -1.75) {};
		\node [style=none] (42) at (-8, 2) {};
		\node [style=none] (43) at (-8, -3.75) {};
		\node [style=none] (44) at (2.25, 2) {};
		\node [style=none] (45) at (2.25, -3.75) {};
		\node [style=none] (46) at (-3.5, -3) {};
		\node [style=none] (47) at (-2.5, -3) {};
		\node [style=none] (48) at (-4, -3.5) {};
		\node [style=none] (49) at (-2, -3.5) {};
		\node [style=none] (50) at (-2, -1.5) {};
		\node [style=none] (51) at (-2, 0) {};
		\node [style=none] (52) at (-1, 1) {};
		\node [style=none] (53) at (1.5, 1) {};
		\node [style=none] (54) at (-0.5, -1.5) {};
		\node [style=none] (55) at (1.5, -1) {};
		\node [style=none] (56) at (-0.5, -1.5) {};
		\node [style=none] (57) at (0, -1) {};
	\end{pgfonlayer}
	\begin{pgfonlayer}{edgelayer}
		\draw (1) to (0);
		\draw (1) to (2);
		\draw (2) to (3);
		\draw (3) to (4);
		\draw (4) to (5);
		\draw (5) to (6);
		\draw (2) to (7);
		\draw (3) to (8);
		\draw (4) to (9);
		\draw [style=e8, bend left=90, looseness=1.75] (12.center) to (13.center);
		\draw [style=e8, bend right=90, looseness=1.75] (12.center) to (13.center);
		\draw [style=e2, bend left=90, looseness=1.75] (14.center) to (15.center);
		\draw [style=e2, bend right=90, looseness=1.75] (14.center) to (15.center);
		\draw [style=e6, bend left=90, looseness=1.75] (16.center) to (17.center);
		\draw [style=e6, bend right=90, looseness=1.75] (16.center) to (17.center);
		\draw [style=e3] (49.center)
			 to (21.center)
			 to [bend right=90, looseness=1.75] (20.center)
			 to (18.center)
			 to [bend right=90, looseness=1.75] (19.center)
			 to (48.center)
			 to [in=-90, out=0, looseness=1.25] (46.center)
			 to [bend left=90, looseness=1.50] (47.center)
			 to [in=-180, out=-90, looseness=1.25] cycle;
		\draw [style=e5] (22.center)
			 to [bend right=90, looseness=1.75] (23.center)
			 to (25.center)
			 to [bend right=90, looseness=1.75] (24.center)
			 to cycle;
		\draw [style=e1] (26.center) to (28.center);
		\draw [style=e1] (27.center) to (29.center);
		\draw [style=e1, bend right=90, looseness=1.75] (29.center) to (28.center);
		\draw [style=e1, bend right=90, looseness=1.75] (26.center) to (27.center);
		\draw [style=e7] (34.center)
			 to [in=0, out=-90, looseness=1.25] (31.center)
			 to [in=-180, out=-180, looseness=1.50] (30.center)
			 to (32.center)
			 to [bend left=90, looseness=1.75] (33.center)
			 to (35.center)
			 to [bend right=45, looseness=1.25] cycle;
		\draw [style=e4] (40.center)
			 to [in=0, out=-90, looseness=1.25] (37.center)
			 to [in=-180, out=-180, looseness=1.50] (36.center)
			 to (38.center)
			 to [bend left=90, looseness=1.75] (39.center)
			 to (41.center)
			 to [bend right=45, looseness=1.25] cycle;
		\draw [style=e9] (43.center)
			 to (45.center)
			 to [bend right=90, looseness=1.75] (44.center)
			 to (42.center)
			 to [bend right=90, looseness=1.75] cycle;
		\draw [style=e0] (52.center)
			 to [bend right=45, looseness=1.25] (51.center)
			 to (50.center)
			 to [bend right=90, looseness=1.75] (54.center)
			 to (56.center)
			 to [bend left=45, looseness=1.50] (57.center)
			 to (55.center)
			 to [bend right=90, looseness=1.50] (53.center)
			 to cycle;
	\end{pgfonlayer}
\end{tikzpicture}

%% file: figures/gammaprime.tex
\begin{tikzpicture}[scale=0.6]
		\node (0) at (-9, 0) {$1$};
		\node (1) at (-7, 0) {$2$};
		\node (2) at (-5, 0) {$3$};
		\node (3) at (-3, 0) {$4$};
		\node (4) at (-1, 0) {$5$};
		\node (5) at (1, 0) {$6$};
		\node (6) at (3, 0) {$7$};
		\node (7) at (-5, -1.75) {$8$};
		\node (8) at (-3, -1.75) {$9$};
		\node (9) at (-1, -1.75) {$0$};
		\node (10) at (-11, 0) {$1'$};
		\node  (11) at (5, 0) {$7'$};
		\node (12) at (-5, -3.5) {$8'$};
		\node (13) at (-3, -3.5) {$9'$};
		\node (14) at (-1, -3.5) {$0'$};
		\draw (1) to (0);
		\draw (1) to (2);
		\draw (2) to (3);
		\draw (3) to (4);
		\draw (4) to (5);
		\draw (5) to (6);
		\draw (2) to (7);
		\draw (3) to (8);
		\draw (4) to (9);
		\draw (10) to (0);
		\draw (7) to (12);
		\draw (13) to (8);
		\draw (14) to (9);
		\draw (6) to (11);
\end{tikzpicture}

%% file: figures/G4Q.tex
\begin{tikzpicture}[scale = 0.6]
	\begin{pgfonlayer}{nodelayer}
		\node [style=label] (0) at (-7.5, 4) {$7$};
		\node [style=label] (1) at (-9.25, 3) {$7'$};
		\node [style=label] (2) at (-9.25, 1) {$3$};
		\node [style=label] (3) at (-7.5, 0) {$4$};
		\node [style=label] (4) at (-5.75, 1) {$5$};
		\node [style=label] (5) at (-5.75, 3) {$7$};
		\node [style=label] (6) at (-2.5, 4) {$7$};
		\node [style=label] (7) at (-4.25, 3) {$7'$};
		\node [style=label] (8) at (-4.25, 1) {$9$};
		\node [style=label] (9) at (-2.5, 0) {$4$};
		\node [style=label] (10) at (-0.75, 1) {$5$};
		\node [style=label] (11) at (-0.75, 3) {$7$};
		\node [style=label] (12) at (2.5, 4) {$0'$};
		\node [style=label] (13) at (0.75, 3) {$0'$};
		\node [style=label] (14) at (0.75, 1) {$3$};
		\node [style=label] (15) at (2.5, 0) {$4$};
		\node [style=label] (16) at (4.25, 1) {$5$};
		\node [style=label] (17) at (4.25, 3) {$0$};
		\node [style=label] (18) at (7.5, 4) {$0'$};
		\node [style=label] (19) at (5.75, 3) {$0'$};
		\node [style=label] (20) at (5.75, 1) {$9$};
		\node [style=label] (21) at (7.5, 0) {$4$};
		\node [style=label] (22) at (9.25, 1) {$5$};
		\node [style=label] (23) at (9.25, 3) {$0$};
		\node [style=none] (24) at (-7.5, 5) {$Q_{7',1'}=Q_{7',8'}$};
		\node [style=none] (25) at (-2.5, 5) {$Q_{7',9'}$};
		\node [style=none] (26) at (2.5, 5) {$Q_{0',1'}=Q_{0',8'}$};
		\node [style=none] (27) at (7.5, 5) {$Q_{0',9'}$};
	\end{pgfonlayer}
	\begin{pgfonlayer}{edgelayer}
		\draw (1) to (0);
		\draw (0) to (5);
		\draw (5) to (4);
		\draw (4) to (3);
		\draw (3) to (2);
		\draw (1) to (2);
		\draw (7) to (6);
		\draw (6) to (11);
		\draw (11) to (10);
		\draw (10) to (9);
		\draw (9) to (8);
		\draw (7) to (8);
		\draw (13) to (12);
		\draw (12) to (17);
		\draw (17) to (16);
		\draw (16) to (15);
		\draw (15) to (14);
		\draw (13) to (14);
		\draw (19) to (18);
		\draw (18) to (23);
		\draw (23) to (22);
		\draw (22) to (21);
		\draw (21) to (20);
		\draw (19) to (20);
	\end{pgfonlayer}
\end{tikzpicture}

%% file: figures/G3Q.tex
\begin{tikzpicture}[scale=0.6]
	\begin{pgfonlayer}{nodelayer}
		\node [style=label] (0) at (-7.5, 4) {$1'$};
		\node [style=label] (1) at (-9.25, 3) {$1$};
		\node [style=label] (2) at (-9.25, 1) {$2$};
		\node [style=label] (3) at (-7.5, 0) {$3$};
		\node [style=label] (4) at (-5.75, 1) {$8$};
		\node [style=label] (5) at (-5.75, 3) {$8'$};
		\node [style=label] (7) at (-3.75, 3.75) {$7'$};
		\node [style=label] (8) at (-4.25, 1.75) {$4$};
		\node [style=label] (9) at (-2.5, 0.25) {$3$};
		\node [style=label] (10) at (-0.75, 1.75) {$8$};
		\node [style=label] (11) at (-1.25, 3.75) {$8'$};
		\node [style=none] (24) at (-7.5, 5) {$Q_{8',1'}$};
		\node [style=none] (25) at (-2.5, 5) {$Q_{8',7'}$};
		\node [style=none] (26) at (2.5, 5) {$Q_{8',9'}$};
		\node [style=none] (27) at (7.5, 5) {$Q_{8',0'}$};
		\node [style=label] (28) at (1.25, 3.75) {$9$};
		\node [style=label] (29) at (0.75, 1.75) {$4$};
		\node [style=label] (30) at (2.5, 0.25) {$3$};
		\node [style=label] (31) at (4.25, 1.75) {$8$};
		\node [style=label] (32) at (3.75, 3.75) {$8'$};
		\node [style=label] (33) at (6.25, 3.75) {$0'$};
		\node [style=label] (34) at (5.75, 1.75) {$4$};
		\node [style=label] (35) at (7.5, 0.25) {$3$};
		\node [style=label] (36) at (9.25, 1.75) {$8$};
		\node [style=label] (37) at (8.75, 3.75) {$8'$};
	\end{pgfonlayer}
	\begin{pgfonlayer}{edgelayer}
		\draw (1) to (0);
		\draw (0) to (5);
		\draw (5) to (4);
		\draw (4) to (3);
		\draw (3) to (2);
		\draw (1) to (2);
		\draw (11) to (10);
		\draw (10) to (9);
		\draw (9) to (8);
		\draw (7) to (8);
		\draw (7) to (11);
		\draw (32) to (31);
		\draw (31) to (30);
		\draw (30) to (29);
		\draw (28) to (29);
		\draw (28) to (32);
		\draw (37) to (36);
		\draw (36) to (35);
		\draw (35) to (34);
		\draw (33) to (34);
		\draw (33) to (37);
	\end{pgfonlayer}
\end{tikzpicture}

%% file: figures/G4T.tex
\begin{tikzpicture}[scale=0.65]
	\begin{pgfonlayer}{nodelayer}
		\node [style=label] (0) at (-8.25, 4) {$7$};
		\node [style=label] (1) at (-10, 3) {$7'$};
		\node [style=label] (2) at (-10, 1) {$3$};
		\node [style=label] (3) at (-8.25, 0) {$4$};
		\node [style=label] (4) at (-6.5, 1) {$5$};
		\node [style=label] (5) at (-6.5, 3) {$7$};
		\node [style=label] (6) at (-2.75, 4) {$7$};
		\node [style=label] (7) at (-4.5, 3) {$7'$};
		\node [style=label] (8) at (-4.5, 1) {$9$};
		\node [style=label] (9) at (-2.75, 0) {$4$};
		\node [style=label] (10) at (-1, 1) {$5$};
		\node [style=label] (11) at (-1, 3) {$7$};
		\node [style=label] (12) at (2.75, 4) {$0'$};
		\node [style=label] (13) at (1, 3) {$0'$};
		\node [style=label] (14) at (1, 1) {$3$};
		\node [style=label] (15) at (2.75, 0) {$4$};
		\node [style=label] (16) at (4.5, 1) {$5$};
		\node [style=label] (17) at (4.5, 3) {$0$};
		\node [style=label] (18) at (8.25, 4) {$0'$};
		\node [style=label] (19) at (6.5, 3) {$0'$};
		\node [style=label] (20) at (6.5, 1) {$9$};
		\node [style=label] (21) at (8.25, 0) {$4$};
		\node [style=label] (22) at (10, 1) {$5$};
		\node [style=label] (23) at (10, 3) {$0$};
		\node [style=label] (24) at (-8.25, 5) {$T_{7',1'}=T_{7',8'}$};
		\node [style=label] (25) at (-2.75, 5) {$T_{7',9'}$};
		\node [style=label] (26) at (2.75, 5) {$T_{0',1'}=T_{0',8'}$};
		\node [style=label] (27) at (8.25, 5) {$T_{0',9'}$};
		\node [style=label-s] (28) at (-9.25, 3.75) {$1$};
		\node [style=label-s, xshift = 2, yshift = 2] (29) at (-7.25, 3.75) {$Y_6^2$};
		\node [style=label-s] (30) at (-10.5, 2) {$Y_{I_4}$};
		\node [style=label-s] (31) at (-9.25, 0.25) {$1$};
		\node [style=label-s] (32) at (-7.25, 0.25) {$1$};
		\node [style=label-s] (33) at (-6, 2) {$Y_{I_6}$};
		\node [style=label-s] (34) at (-8, 3) {$I_0$};
		\node [style=label-s] (35) at (-9.5, 1.5) {$I_7$};
		\node [style=label-s] (36) at (-7.75, 1.5) {$I_5$};
		\node [style=label-s] (37) at (-3.75, 3.75) {$1$};
		\node [style=label-s, xshift = 2, yshift = 2] (38) at (-1.75, 3.75) {$Y_6^2$};
		\node [style=label-s] (39) at (-5, 2) {$Y_{I_4}$};
		\node [style=label-s] (40) at (-3.75, 0.25) {$1$};
		\node [style=label-s] (41) at (-1.75, 0.25) {$1$};
		\node [style=label-s] (42) at (-0.5, 2) {$Y_{I_6}$};
		\node [style=label-s] (43) at (-2.5, 3) {$I_0$};
		\node [style=label-s] (44) at (-4, 1.5) {$I_7$};
		\node [style=label-s] (45) at (-2.25, 1.5) {$I_5$};
		\node [style=label-s] (46) at (1.5, 3.75) {$Y_{0}^2$};
		\node [style=label-s] (47) at (3.75, 3.75) {$1$};
		\node [style=label-s] (48) at (0.5, 2) {$Y_{I_4}$};
		\node [style=label-s] (49) at (1.75, 0.25) {$1$};
		\node [style=label-s] (50) at (3.75, 0.25) {$1$};
		\node [style=label-s] (51) at (4.75, 2) {$1$};
		\node [style=label-s] (52) at (3, 3) {$I_0$};
		\node [style=label-s] (53) at (1.5, 1.5) {$I_7$};
		\node [style=label-s] (54) at (3.25, 1.5) {$I_5$};
		\node [style=label-s] (55) at (7, 3.75) {$Y_{0}^2$};
		\node [style=label-s] (56) at (9.25, 3.75) {$1$};
		\node [style=label-s] (57) at (6, 2) {$Y_{I_4}$};
		\node [style=label-s] (58) at (7.25, 0.25) {$1$};
		\node [style=label-s] (59) at (9.25, 0.25) {$1$};
		\node [style=label-s] (60) at (10.25, 2) {$1$};
		\node [style=label-s] (61) at (8.5, 3) {$I_0$};
		\node [style=label-s] (62) at (7, 1.5) {$I_7$};
		\node [style=label-s] (63) at (8.75, 1.5) {$I_5$};
	\end{pgfonlayer}
	\begin{pgfonlayer}{edgelayer}
		\draw (1) to (0);
		\draw (0) to (5);
		\draw [style=e6] (5) to (4);
		\draw (4) to (3);
		\draw (3) to (2);
		\draw [style=e4] (1) to (2);
		\draw (7) to (6);
		\draw (6) to (11);
		\draw [style=e6] (11) to (10);
		\draw (10) to (9);
		\draw (9) to (8);
		\draw [style=e4] (7) to (8);
		\draw (13) to (12);
		\draw (12) to (17);
		\draw (17) to (16);
		\draw (16) to (15);
		\draw (15) to (14);
		\draw [style=e4] (13) to (14);
		\draw (19) to (18);
		\draw (18) to (23);
		\draw (23) to (22);
		\draw (22) to (21);
		\draw (21) to (20);
		\draw [style=e4] (19) to (20);
		\draw [style=d7] (1) to (3);
		\draw [style=d7] (7) to (9);
		\draw [style=d7] (13) to (15);
		\draw [style=d7] (19) to (21);
		\draw [style=d0] (0) to (3);
		\draw [style=d0] (6) to (9);
		\draw [style=d0] (12) to (15);
		\draw [style=d0] (18) to (21);
		\draw [style=d5] (23) to (21);
		\draw [style=d5] (17) to (15);
		\draw [style=d5] (11) to (9);
		\draw [style=d5] (5) to (3);
	\end{pgfonlayer}
\end{tikzpicture}

%% file: figures/G3T.tex
\begin{tikzpicture}[scale=0.65]
	\begin{pgfonlayer}{nodelayer}
		\node [style=label] (0) at (-8.25, 4) {$1'$};
		\node [style=label] (1) at (-10, 3) {$1$};
		\node [style=label] (2) at (-10, 1) {$2$};
		\node [style=label] (3) at (-8.25, 0) {$3$};
		\node [style=label] (4) at (-6.5, 1) {$8$};
		\node [style=label] (5) at (-6.5, 3) {$8'$};
		\node [style=label] (7) at (-4, 3.75) {$7'$};
		\node [style=label] (8) at (-4.5, 1.75) {$4$};
		\node [style=label] (9) at (-2.75, 0.25) {$3$};
		\node [style=label] (10) at (-1, 1.75) {$8$};
		\node [style=label] (11) at (-1.5, 3.75) {$8'$};
		\node [style=none] (24) at (-8.25, 5) {$T_{8',1'}$};
		\node [style=none] (25) at (-2.75, 5) {$T_{8',7'}$};
		\node [style=none] (26) at (2.75, 5) {$T_{8',9'}$};
		\node [style=none] (27) at (8.25, 5) {$T_{8',0'}$};
		\node [style=label] (28) at (1.5, 3.75) {$9$};
		\node [style=label] (29) at (1, 1.75) {$4$};
		\node [style=label] (30) at (2.75, 0.25) {$3$};
		\node [style=label] (31) at (4.5, 1.75) {$8$};
		\node [style=label] (32) at (4, 3.75) {$8'$};
		\node [style=label] (33) at (7, 3.75) {$0'$};
		\node [style=label] (34) at (6.5, 1.75) {$4$};
		\node [style=label] (35) at (8.25, 0.25) {$3$};
		\node [style=label] (36) at (10, 1.75) {$8$};
		\node [style=label] (37) at (9.5, 3.75) {$8'$};
		\node [style=label-s] (38) at (-9.25, 3.75) {$1$};
		\node [style=label-s, xshift = 2, yshift = 0] (39) at (-7.25, 3.75) {$Y_{I_3}$};
		\node [style=label-s] (40) at (-10.25, 2) {$1$};
		\node [style=label-s] (41) at (-9.25, 1.25) {$I_3$};
		\node [style=label-s, xshift=-1] (42) at (-8.5, 2.75) {$I_1$};
		\node [style=label-s] (43) at (-7.5, 2) {$I_8$};
		\node [style=label-s] (44) at (-9.25, 0.25) {$1$};
		\node [style=label-s] (45) at (-7.25, 0.25) {$1$};
		\node [style=label-s] (46) at (-6.25, 2) {$1$};
		\node [style=label-s] (47) at (-4.75, 2.75) {$Y_{I_7}$};
		\node [style=label-s, yshift = 2] (48) at (-2.75, 4) {$Y_{I_3}$};
		\node [style=label-s] (49) at (-1, 2.75) {$1$};
		\node [style=label-s] (50) at (-3.25, 2.75) {$I_4$};
		\node [style=label-s] (51) at (-2.25, 2.75) {$I_8$};
		\node [style=label-s] (52) at (-3.75, 0.75) {$1$};
		\node [style=label-s] (53) at (-1.75, 0.75) {$1$};
		\node [style=label-s] (54) at (1, 2.75) {$1$};
		\node [style=label-s, yshift = 2] (55) at (2.75, 4) {$Y_{I_3}$};
		\node [style=label-s] (56) at (4.5, 2.75) {$1$};
		\node [style=label-s] (57) at (2.25, 2.75) {$I_4$};
		\node [style=label-s] (58) at (3.25, 2.75) {$I_8$};
		\node [style=label-s] (59) at (1.75, 0.75) {$1$};
		\node [style=label-s] (60) at (3.75, 0.75) {$1$};
		\node [style=label-s] (61) at (6.25, 2.75) {$Y_{I_7}$};
		\node [style=label-s, yshift = 2] (62) at (8.25, 4) {$Y_{I_3}$};
		\node [style=label-s] (63) at (10, 2.75) {$1$};
		\node [style=label-s] (64) at (7.75, 2.75) {$I_4$};
		\node [style=label-s] (65) at (8.75, 2.75) {$I_8$};
		\node [style=label-s] (66) at (7.25, 0.75) {$1$};
		\node [style=label-s] (67) at (9.25, 0.75) {$1$};
	\end{pgfonlayer}
	\begin{pgfonlayer}{edgelayer}
		\draw (1) to (0);
		\draw [style=e3] (0) to (5);
		\draw (5) to (4);
		\draw (4) to (3);
		\draw (3) to (2);
		\draw (1) to (2);
		\draw (11) to (10);
		\draw (10) to (9);
		\draw (9) to (8);
		\draw [style=e7] (7) to (8);
		\draw [style=e3] (7) to (11);
		\draw (32) to (31);
		\draw (31) to (30);
		\draw (30) to (29);
		\draw (28) to (29);
		\draw [style=e3] (28) to (32);
		\draw (37) to (36);
		\draw (36) to (35);
		\draw (35) to (34);
		\draw [style=e7] (33) to (34);
		\draw [style=e3] (33) to (37);
		\draw [style=d3] (1) to (3);
		\draw [style=d1] (0) to (3);
		\draw [style=d8] (3) to (5);
		\draw [style=d4] (7) to (9);
		\draw [style=d8] (9) to (11);
		\draw [style=d4] (28) to (30);
		\draw [style=d8] (30) to (32);
		\draw [style=d4] (33) to (35);
		\draw [style=d8] (35) to (37);
	\end{pgfonlayer}
\end{tikzpicture}

%% file: figures/G4H.tex
\begin{tikzpicture}[scale=0.75]
	\begin{pgfonlayer}{nodelayer}
		\node [style=label] (0) at (0.75, 4) {$4$};
		\node [style=label] (1) at (2.75, 4) {$7'$};
		\node [style=label] (2) at (2.75, 2) {$7$};
		\node [style=label] (3) at (0.75, 2) {$7$};
		\node [style=label] (4) at (0.75, 0) {$5$};
		\node [style=label] (5) at (2.75, 0) {$4$};
		\node [style=label] (6) at (4.75, 2) {$4$};
		\node [style=label] (7) at (4.75, 4) {$3$};
		\node [style=label] (8) at (6.25, 4) {$4$};
		\node [style=label] (9) at (8.25, 4) {$7'$};
		\node [style=label] (10) at (8.25, 2) {$7$};
		\node [style=label] (11) at (6.25, 2) {$7$};
		\node [style=label] (12) at (6.25, 0) {$5$};
		\node [style=label] (13) at (8.25, 0) {$4$};
		\node [style=label] (14) at (10.25, 2) {$4$};
		\node [style=label] (15) at (10.25, 4) {$3$};
		\node [style=label] (16) at (11.75, 4) {$4$};
		\node [style=label] (17) at (13.75, 4) {$0'$};
		\node [style=label] (18) at (13.75, 2) {$0'$};
		\node [style=label] (19) at (11.75, 2) {$0$};
		\node [style=label] (20) at (11.75, 0) {$5$};
		\node [style=label] (21) at (13.75, 0) {$4$};
		\node [style=label] (22) at (15.75, 2) {$4$};
		\node [style=label] (23) at (15.75, 4) {$3$};
		\node [style=label] (24) at (17.25, 4) {$4$};
		\node [style=label] (25) at (19.25, 4) {$0'$};
		\node [style=label] (26) at (19.25, 2) {$0'$};
		\node [style=label] (27) at (17.25, 2) {$0$};
		\node [style=label] (28) at (17.25, 0) {$5$};
		\node [style=label] (29) at (19.25, 0) {$4$};
		\node [style=label] (30) at (21.25, 2) {$4$};
		\node [style=label] (31) at (21.25, 4) {$9$};
		\node [style=label-s, yshift=-3] (32) at (1.75, 4.5) {$Y_{I_7}$};
		\node [style=label-s, yshift=-3] (33) at (3.75, 4.5) {$Y_{I_4}$};
		\node [style=label-s] (34) at (4, 3.25) {${I_7}$};
		\node [style=label-s] (35) at (2, 3.25) {${I_0}$};
		\node [style=label-s, xshift=2] (36) at (0.25, 3) {$Y_{I_5}$};
		\node [style=label-s] (37) at (2, 1.25) {${I_5}$};
		\node [style=label-s, xshift=2] (38) at (0.25, 1) {$Y_{I_6}$};
		\node [style=label-s] (39) at (3.25, 0.75) {$Y_{I_0}$};
		\node [style=label-s,yshift=2] (40) at (4, 1.5) {$Y_{I_0}$};
		\node [style=label-s, yshift=-3] (41) at (7.25, 4.5) {$Y_{I_7}$};
		\node [style=label-s, yshift=-3] (42) at (9.25, 4.5) {$Y_{I_4}$};
		\node [style=label-s] (43) at (9.5, 3.25) {${I_7}$};
		\node [style=label-s] (44) at (7.5, 3.25) {${I_0}$};
		\node [style=label-s, xshift=2] (45) at (5.75, 3) {$Y_{I_5}$};
		\node [style=label-s] (46) at (7.5, 1.25) {${I_5}$};
		\node [style=label-s] (47) at (5.75, 1) {$Y_{I_6}$};
		\node [style=label-s] (48) at (8.75, 0.75) {$Y_{I_0}$};
		\node [style=label-s,yshift=2] (49) at (9.5, 1.5) {$Y_{I_0}$};
		\node [style=label-s, yshift=-3] (50) at (12.75, 4.5) {$Y_{I_7}$};
		\node [style=label-s, yshift=-3] (51) at (14.75, 4.5) {$Y_{I_4}$};
		\node [style=label-s] (52) at (15, 3.25) {${I_7}$};
		\node [style=label-s] (53) at (13, 3.25) {${I_0}$};
		\node [style=label-s,xshift=2] (54) at (11.25, 3) {$Y_{I_5}$};
		\node [style=label-s] (55) at (13, 1.25) {${I_5}$};
		\node [style=label-s] (56) at (14.25, 0.75) {$Y_{I_0}$};
		\node [style=label-s,yshift=2] (57) at (15, 1.5) {$Y_{I_0}$};
		\node [style=label-s, yshift=-3] (58) at (18.25, 4.5) {$Y_{I_7}$};
		\node [style=label-s, yshift=-3] (59) at (20.25, 4.5) {$Y_{I_4}$};
		\node [style=label-s] (60) at (20.5, 3.25) {${I_7}$};
		\node [style=label-s] (61) at (18.5, 3.25) {${I_0}$};
		\node [style=label-s,xshift=2] (62) at (16.75, 3) {$Y_{I_5}$};
		\node [style=label-s] (63) at (18.5, 1.25) {${I_5}$};
		\node [style=label-s] (64) at (19.75, 0.75) {$Y_{I_0}$};
		\node [style=label-s,yshift=2] (65) at (20.5, 1.5) {$Y_{I_0}$};
		\node [style=label] (66) at (2.75, 5) {$H_{7',1'}=H_{7',8'}$};
		\node [style=label] (67) at (8.25, 5) {$H_{7',9'}$};
		\node [style=label] (68) at (13.75, 5) {$H_{0',1'}=H_{0',8'}$};
		\node [style=label] (69) at (19.25, 5) {$H_{0',9'}$};
		\node [style=label-s, yshift=-2.5] (70) at (1.75, 2.5) {$Y_6^2$};
		\node [style=label-s, yshift=-2.5] (71) at (7.25, 2.5) {$Y_6^2$};
	\end{pgfonlayer}
	\begin{pgfonlayer}{edgelayer}
		\draw [style=e7] (0) to (1);
		\draw [style=e4] (1) to (7);
		\draw (7) to (6);
		\draw [style=e0] (6) to (2);
		\draw (2) to (1);
		\draw [style=e5] (3) to (0);
		\draw (3) to (2);
		\draw [style=e0] (2) to (5);
		\draw (5) to (4);
		\draw [style=e6] (4) to (3);
		\draw [style=e7] (8) to (9);
		\draw [style=e4] (9) to (15);
		\draw (15) to (14);
		\draw [style=e0] (14) to (10);
		\draw (10) to (9);
		\draw [style=e5] (11) to (8);
		\draw (11) to (10);
		\draw [style=e0] (10) to (13);
		\draw (13) to (12);
		\draw [style=e6] (12) to (11);
		\draw [style=e7] (16) to (17);
		\draw [style=e4] (17) to (23);
		\draw (23) to (22);
		\draw [style=e0] (22) to (18);
		\draw (18) to (17);
		\draw [style=e5] (19) to (16);
		\draw (19) to (18);
		\draw [style=e0] (18) to (21);
		\draw (21) to (20);
		\draw (20) to (19);
		\draw [style=e7] (24) to (25);
		\draw [style=e4] (25) to (31);
		\draw (31) to (30);
		\draw [style=e0] (30) to (26);
		\draw (26) to (25);
		\draw [style=e5] (27) to (24);
		\draw (27) to (26);
		\draw [style=e0] (26) to (29);
		\draw (29) to (28);
		\draw (28) to (27);
		\draw [style=d0] (0) to (2);
		\draw [style=d5] (3) to (5);
		\draw [style=d7] (1) to (6);
		\draw [style=d8] (13) to (11);
		\draw [style=d0] (8) to (10);
		\draw [style=d7] (9) to (14);
		\draw [style=d8] (19) to (21);
		\draw [style=d0] (16) to (18);
		\draw [style=d7] (17) to (22);
		\draw [style=d8] (27) to (29);
		\draw [style=d0] (24) to (26);
		\draw [style=d7] (25) to (30);
	\end{pgfonlayer}
\end{tikzpicture}

%% file: figures/G3H.tex
\begin{tikzpicture}[scale=0.8]
	\begin{pgfonlayer}{nodelayer}
		\node [style=label] (0) at (-8.25, 2) {$3$};
		\node [style=label] (1) at (-8.25, 0) {$8'$};
		\node [style=label] (2) at (-6.25, 0) {$1'$};
		\node [style=label] (3) at (-6.25, 2) {$1$};
		\node [style=label] (4) at (-4.25, 2) {$2$};
		\node [style=label] (5) at (-4.25, 0) {$3$};
		\node [style=label] (6) at (-6.25, -2) {$3$};
		\node [style=label] (7) at (-8.25, -2) {$8$};
		\node [style=label-s] (8) at (-7.25, 2.3) {$Y_{I_2}$};
		\node [style=label-s, xshift = -2] (9) at (-8.5, 1) {$Y_{I_8}$};
		\node [style=label-s] (10) at (-7, 1.25) {$I_1$};
		\node [style=label-s] (11) at (-5, 1.25) {$I_2$};
		\node [style=label-s, yshift=-2] (12) at (-5.25, -0.25) {$Y_{I_1}$};
		\node [style=label-s,xshift=3] (13) at (-6, -1) {$Y_{I_1}$};
		\node [style=label-s] (14) at (-7.25, 0.25) {$Y_{I_3}$};
		\node [style=label-s] (15) at (-7, -0.75) {$I_8$};
		\node [style=label] (16) at (-2, 2) {$3$};
		\node [style=label] (17) at (0, 2) {$4$};
		\node [style=label] (18) at (0, -2) {$3$};
		\node [style=label] (19) at (-2, -2) {$8$};
		\node [style=label] (20) at (-2, 0) {$8'$};
		\node [style=label] (21) at (0, 0) {$7'$};
		\node [style=label] (22) at (2.25, 2) {$3$};
		\node [style=label] (23) at (4.25, 2) {$4$};
		\node [style=label] (24) at (4.25, -2) {$3$};
		\node [style=label] (25) at (2.25, -2) {$8$};
		\node [style=label] (26) at (2.25, 0) {$8'$};
		\node [style=label] (27) at (4.25, 0) {$9$};
		\node [style=label] (28) at (6.5, 2) {$3$};
		\node [style=label] (29) at (8.5, 2) {$4$};
		\node [style=label] (30) at (8.5, -2) {$3$};
		\node [style=label] (31) at (6.5, -2) {$8$};
		\node [style=label] (32) at (6.5, 0) {$8'$};
		\node [style=label] (33) at (8.5, 0) {$0'$};
		\node [style=label-s, xshift = -2] (34) at (-2.25, 1) {$Y_{I_8}$};
		\node [style=label-s] (35) at (-0.75, 1.25) {$I_4$};
		\node [style=label-s, xshift=3] (36) at (0.25, 1) {$Y_{I_7}$};
		\node [style=label-s] (37) at (-1, 0.25) {$Y_{I_3}$};
		\node [style=label-s] (38) at (-0.75, -0.75) {$I_8$};
		\node [style=label-s, xshift = 2] (39) at (0.25, -1) {$Y_{I_4}$};
		\node [style=label-s, xshift = -2] (40) at (2, 1) {$Y_{I_8}$};
		\node [style=label-s] (41) at (3.5, 1.25) {$I_4$};
		\node [style=label-s] (43) at (3.25, 0.25) {$Y_{I_3}$};
		\node [style=label-s] (44) at (3.5, -0.75) {$I_8$};
		\node [style=label-s, xshift = 2] (45) at (4.5, -1) {$Y_{I_4}$};
		\node [style=label-s, xshift = -2] (46) at (6.25, 1) {$Y_{I_8}$};
		\node [style=label-s] (47) at (7.75, 1.25) {$I_4$};
		\node [style=label-s, xshift=3] (48) at (8.75, 1) {$Y_{I_7}$};
		\node [style=label-s] (49) at (7.5, 0.25) {$Y_{I_3}$};
		\node [style=label-s] (50) at (7.75, -0.75) {$I_8$};
		\node [style=label-s, xshift = 2] (51) at (8.75, -1) {$Y_{I_4}$};
		\node [style=none] (52) at (-6.25, 3) {$H_{8',1'}$};
		\node [style=none] (53) at (-1, 3) {$H_{8',7'}$};
		\node [style=none] (54) at (3.25, 3) {$H_{8',9'}$};
		\node [style=none] (55) at (7.5, 3) {$H_{8',0'}$};
	\end{pgfonlayer}
	\begin{pgfonlayer}{edgelayer}
		\draw [style=e2] (0) to (3);
		\draw (3) to (4);
		\draw (4) to (5);
		\draw [style=e1] (5) to (2);
		\draw (2) to (3);
		\draw [style=e8] (0) to (1);
		\draw [style=e3] (1) to (2);
		\draw [style=e1] (2) to (6);
		\draw (6) to (7);
		\draw (7) to (1);
		\draw [style=d8] (1) to (6);
		\draw [style=d1] (2) to (0);
		\draw [style=d2] (3) to (5);
		\draw (16) to (17);
		\draw [style=e7] (21) to (17);
		\draw [style=e3] (21) to (20);
		\draw [style=e8] (20) to (16);
		\draw (20) to (19);
		\draw [in=180, out=0] (19) to (18);
		\draw (18) to (21);
		\draw (22) to (23);
		\draw (27) to (23);
		\draw [style=e3] (27) to (26);
		\draw [style=e8] (26) to (22);
		\draw (26) to (25);
		\draw [in=180, out=0] (25) to (24);
		\draw (24) to (27);
		\draw (28) to (29);
		\draw [style=e7] (33) to (29);
		\draw [style=e3] (33) to (32);
		\draw [style=e8] (32) to (28);
		\draw (32) to (31);
		\draw [in=180, out=0] (31) to (30);
		\draw (30) to (33);
		\draw [style=d4] (16) to (21);
		\draw [style=d8] (20) to (18);
		\draw [style=d4] (22) to (27);
		\draw [style=d4] (28) to (33);
		\draw [style=d8] (26) to (24);
		\draw [style=d8] (32) to (30);
	\end{pgfonlayer}
\end{tikzpicture}

%% file: figures/G4.tex
\begin{tikzpicture}[scale=0.65]
	\begin{pgfonlayer}{nodelayer}
		\node [style=label] (0) at (-7, 2) {$4$};
		\node [style=label] (1) at (-4, 2) {$7'$};
		\node [style=label] (2) at (-3, 1.5) {$0'$};
		\node [style=label] (3) at (-7, -1) {$7$};
		\node [style=label] (4) at (-6, -1.5) {$0$};
		\node [style=label] (5) at (-4, -1) {$7$};
		\node [style=label] (6) at (-3, -1.5) {$0'$};
		\node [style=label] (7) at (-7, -4) {$5$};
		\node [style=label] (8) at (-4, -4) {$4$};
		\node [style=label] (9) at (-1, -1) {$4$};
		\node [style=label] (10) at (-1, 2) {$3$};
		\node [style=label] (11) at (0, 1.5) {$9$};
		\node [style=none] (12) at (-6, 2) {};
		\node [style=none] (13) at (-2, 1.75) {};
		\node [style=none] (14) at (-1.75, -1) {};
		\node [style=none] (15) at (-7, 0.75) {};
		\node [style=none] (16) at (-4, -3) {};
		\node [style=label-s] (17) at (-5.25, -2.75) {$I_5$};
		\node [style=label-s] (18) at (-5.5, 0.75) {$I_0$};
		\node [style=label-s] (19) at (-2, 0) {$I_7$};
		\node [style=label-s] (20) at (-5.5, 2.35) {$Y_{I_7}$};
		\node [style=label-s] (21) at (-2.25, 2.25) {$Y_{I_4}$};
		\node [style=label-s] (22) at (-3, 0) {$Y_0^2$};
		\node [style=label-s] (23) at (-0.25, -1.5) {$1\oplus 1$};
		\node [style=label-s] (24) at (-2, -1.5) {$Y_{I_0}$};
		\node [style=label-s] (25) at (-3.5, -2.75) {$Y_{I_0}$};
		\node [style=label-s] (27) at (-7.5, 0.75) {$Y_{I_5}$};
		\node [style=label-s] (28) at (-7.5, -2.5) {$Y_{I_6}$};
		\node [style=label-s] (29) at (-7, -4.5) {$1\oplus 1$};
		\node [style=label-s] (30) at (-5.5, -0.6) {$Y_6^2$};
	\end{pgfonlayer}
	\begin{pgfonlayer}{edgelayer}
		\draw [style=e4, in=165, out=0, looseness=1.25] (1) to (13.center);
		\draw [style=e4, in=180, out=-15, looseness=1.25] (13.center) to (11);
		\draw [style=e4] (2) to (13.center);
		\draw [style=e4] (13.center) to (10);
		\draw [style=e0] (5) to (14.center);
		\draw [style=e0] (14.center) to (9);
		\draw [style=e0, in=45, out=180] (14.center) to (6);
		\draw [style=e7] (0) to (12.center);
		\draw [style=e7, in=180, out=0, looseness=1.25] (12.center) to (2);
		\draw [style=e7] (12.center) to (1);
		\draw (3) to (5);
		\draw (4) to (6);
		\draw [style=e5] (0) to (15.center);
		\draw [style=e5, in=105, out=-90, looseness=0.75] (15.center) to (4);
		\draw [style=e5] (15.center) to (3);
		\draw (7) to (8);
		\draw [style=e0] (8) to (16.center);
		\draw [style=e0] (16.center) to (5);
		\draw [style=e0, in=-150, out=90, looseness=0.75] (16.center) to (6);
		\draw (1) to (5);
		\draw (10) to (9);
		\draw (9) to (11);
		\draw [style=e6] (3) to (7);
		\draw (7) to (4);
		\draw [style=d5] (3) to (17);
		\draw [style=d5, in=-75, out=135, looseness=0.75] (17) to (4);
		\draw [style=d5] (17) to (8);
		\draw [style=d7] (1) to (19);
		\draw [style=d7, in=300, out=135] (19) to (2);
		\draw [style=d7] (19) to (9);
		\draw [style=d0] (0) to (18);
		\draw [style=d0, in=120, out=-45, looseness=0.75] (18) to (5);
		\draw [style=d0] (18) to (6);
		\draw (2) to (22);
		\draw (22) to (6);
	\end{pgfonlayer}
\end{tikzpicture}

%% file: figures/G3.tex
\begin{tikzpicture}[scale=0.56]
	\begin{pgfonlayer}{nodelayer}
		\node [style=label] (0) at (-7.25, 4.5) {$3$};
		\node [style=label] (1) at (-2.25, 4.5) {$4$};
		\node [style=label] (2) at (-7.25, 0.5) {$8'$};
		\node [style=label] (3) at (-2.25, 1.25) {$9$};
		\node [style=label] (4) at (-3, 0.5) {$7'$};
		\node [style=label] (5) at (-3.75, -0.25) {$0'$};
		\node [style=label] (6) at (-1.5, 2) {$1'$};
		\node [style=label] (7) at (-1.5, 5.75) {$1$};
		\node [style=label] (8) at (2.25, 5.75) {$2$};
		\node [style=label] (9) at (2.25, 2) {$3$};
		\node [style=label] (10) at (-7.25, -3.5) {$8$};
		\node [style=label] (11) at (-3.75, -3.5) {$3$};
		\node [style=none] (12) at (0.25, 4) {};
		\node [style=label-s] (13) at (-5.25, 2.5) {$I_4$};
		\node [style=label-s] (14) at (-3, 2) {$Y_{I_7}$};
		\node [style=none] (15) at (-3.75, -1.75) {};
		\node [style=label-s] (16) at (-5.75, 0.5) {$Y_{I_3}$};
		\node [style=label-s] (18) at (-5.5, -1.5) {$I_8$};
		\node [style=label-s] (21) at (-3.75, -2.25) {$Y_{I_4}$};
		\node [style=label-s] (22) at (-2.25, -1) {$Y_{I_1}$};
		\node [style=label-s] (23) at (0, 1.65) {$Y_{I_1}$};
		\node [style=label-s] (24) at (-5.5, 5.35) {$Y_{I_2}$};
		\node [style=label-s] (25) at (-7.75, 2.25) {$Y_{I_8}$};
		\node [style=label] (26) at (-2.75, -3.5) {$=v_{{opt}}$};
		\node [style=label-s] (27) at (-2.75, 5) {$1\oplus 1$};
		\node [style=label-s] (28) at (-5.25, 3.25) {$I_1$};
		\node [style=label-s] (29) at (0.75, 3.75) {$I_2$};
		\node [style=label-s] (30) at (-3.75, -4) {$1\oplus 1$};
		\node [style=label-s] (31) at (-7.5, 5) {$2\oplus 0$};
	\end{pgfonlayer}
	\begin{pgfonlayer}{edgelayer}
		\draw (0) to (1);
		\draw (1) to (3);
		\draw [style=e8] (0) to (2);
		\draw (7) to (8);
		\draw (8) to (9);
		\draw [style=e1] (9) to (6);
		\draw (6) to (7);
		\draw [style=e2] (0) to (7);
		\draw (2) to (10);
		\draw (10) to (11);
		\draw [style=e1] (11) to (6);
		\draw [style=e3] (2) to (16);
		\draw [style=e3, in=-150, out=0, looseness=0.75] (16) to (3);
		\draw [style=e3, in=-180, out=0] (16) to (4);
		\draw [style=e3, in=120, out=0] (16) to (5);
		\draw [style=e3, in=210, out=0, looseness=0.75] (16) to (6);
		\draw [style=e4, in=-120, out=90, looseness=0.50] (15.center) to (4);
		\draw [style=e4] (15.center) to (5);
		\draw [style=e7] (14) to (4);
		\draw [style=e7] (14) to (5);
		\draw [style=d4] (0) to (13);
		\draw [style=d4, in=-180, out=-45, looseness=0.75] (13) to (3);
		\draw [style=d4] (13) to (4);
		\draw [style=d4, in=90, out=-45, looseness=0.75] (13) to (5);
		\draw [style=d8] (2) to (18);
		\draw [style=d8] (18) to (11);
		\draw [style=d2, in=0, out=135, looseness=0.75] (12.center) to (1);
		\draw [style=e4, in=-120, out=90, looseness=0.50] (15.center) to (3);
		\draw (21) to (15.center);
		\draw [style=e4] (21) to (11);
		\draw (21) to (15.center);
		\draw [style=d1] (0) to (6);
		\draw [style=e7] (14) to (1);
		\draw [style=d2] (7) to (12.center);
		\draw [style=d2] (12.center) to (9);
	\end{pgfonlayer}
\end{tikzpicture}

%% file: figures/G3x.tex
\begin{tikzpicture}[scale=0.56]
	\begin{pgfonlayer}{nodelayer}
		\node [style=label] (0) at (-7.25, 4.5) {$3$};
		\node [style=label] (1) at (-2.25, 4.5) {$4$};
		\node [style=label] (2) at (-7.25, 0.5) {$8'$};
		\node [style=label] (3) at (-2.25, 1.25) {$9$};
		\node [style=label] (4) at (-3, 0.5) {$7'$};
		\node [style=label] (5) at (-3.75, -0.25) {$0'$};
		\node [style=label] (6) at (-1.5, 2) {$1'$};
		\node [style=label] (7) at (-1.5, 5.75) {$1$};
		\node [style=label] (8) at (2.25, 5.75) {$2$};
		\node [style=label] (9) at (2.25, 2) {$3$};
		\node [style=label] (11) at (-3.75, -3.5) {$3$};
		\node [style=none] (12) at (0.25, 4) {};
		\node [style=label-s] (13) at (-5.25, 2.5) {$I_4$};
		\node [style=label-s] (14) at (-3, 2) {$I_7$};
		\node [style=none] (15) at (-3.75, -1.75) {};
		\node [style=label-s] (16) at (-5.75, 0.5) {$Y_{I_3}$};
		\node [style=label-s] (18) at (-5.5, -1.5) {$I_8$};
		\node [style=label-s] (21) at (-3.75, -2.25) {$I_4$};
		\node [style=label-s] (22) at (-2.25, -1) {$Y_{I_1}$};
		\node [style=label-s] (23) at (0, 1.75) {$Y_{I_1}$};
		\node [style=label-s] (24) at (-5.5, 5.35) {$Y_{I_2}$};
		\node [style=label-s] (25) at (-7.75, 2.25) {$Y_{I_8}$};
		\node [style=label] (26) at (-2.75, -3.5) {$=v_{{opt}}$};
		\node [style=label-s] (27) at (-2.75, 5) {$1\oplus 1$};
		\node [style=label-s] (28) at (-5.25, 3.25) {$I_1$};
		\node [style=label-s] (29) at (0.75, 3.75) {$I_2$};
		\node [style=label-s] (30) at (-3.75, -4) {$1\oplus 1$};
		\node [style=label-s] (31) at (-7.5, 5) {$2\oplus 0$};
	\end{pgfonlayer}
	\begin{pgfonlayer}{edgelayer}
		\draw (0) to (1);
		\draw (1) to (3);
		\draw [style=e8] (0) to (2);
		\draw (7) to (8);
		\draw (8) to (9);
		\draw [style=e1] (9) to (6);
		\draw (6) to (7);
		\draw [style=e2] (0) to (7);
		\draw [style=e1] (11) to (6);
		\draw [style=e3] (2) to (16);
		\draw [style=e3, in=-150, out=0, looseness=0.75] (16) to (3);
		\draw [style=e3, in=-180, out=0] (16) to (4);
		\draw [style=e3, in=120, out=0] (16) to (5);
		\draw [style=e3, in=210, out=0, looseness=0.75] (16) to (6);
		\draw [style=e4, in=-120, out=90, looseness=0.50] (15.center) to (4);
		\draw [style=e4] (15.center) to (5);
		\draw [style=e7] (14) to (4);
		\draw [style=e7] (14) to (5);
		\draw [style=d4] (0) to (13);
		\draw [style=d4, in=-180, out=-45, looseness=0.75] (13) to (3);
		\draw [style=d4] (13) to (4);
		\draw [style=d4, in=90, out=-45, looseness=0.75] (13) to (5);
		\draw [style=d8] (2) to (18);
		\draw [style=d8] (18) to (11);
		\draw [style=d2, in=0, out=135, looseness=0.75] (12.center) to (1);
		\draw [style=e4, in=-120, out=90, looseness=0.50] (15.center) to (3);
		\draw (21) to (15.center);
		\draw [style=e4] (21) to (11);
		\draw (21) to (15.center);
		\draw [style=d1] (0) to (6);
		\draw [style=e7] (14) to (1);
		\draw [style=d2] (7) to (12.center);
		\draw [style=d2] (12.center) to (9);
	\end{pgfonlayer}
\end{tikzpicture}

%% file: figures/G38.tex
\begin{tikzpicture}[scale=0.56]
	\begin{pgfonlayer}{nodelayer}
		\node [style=label] (0) at (-7.25, 4.5) {$3$};
		\node [style=label] (1) at (-2.25, 4.5) {$4$};
		\node [style=label] (2) at (-7.25, 0.5) {$8'$};
		\node [style=label] (3) at (-2.25, 1.25) {$9$};
		\node [style=label] (4) at (-3, 0.5) {$7'$};
		\node [style=label] (5) at (-3.75, -0.25) {$0'$};
		\node [style=label] (6) at (-1.5, 2) {$1'$};
		\node [style=label] (7) at (-1.5, 5.75) {$1$};
		\node [style=label] (8) at (2.25, 5.75) {$2$};
		\node [style=label] (9) at (2.25, 2) {$3$};
		\node [style=label] (11) at (-3.75, -3.5) {$3$};
		\node [style=none] (12) at (0.25, 4) {};
		\node [style=label-s] (13) at (-5.25, 2.5) {$I_4$};
		\node [style=label-s] (14) at (-3, 2) {$I_7$};
		\node [style=none] (15) at (-3.75, -1.75) {};
		\node [style=label-s] (16) at (-5.75, 0.5) {$Y_{I_3}$};
		\node [style=label-s] (21) at (-3.75, -2.25) {$Y_{I_4}$};
		\node [style=label-s] (22) at (-2.25, -1) {$Y_{I_1}$};
		\node [style=label-s] (23) at (0, 1.75) {$Y_{I_1}$};
		\node [style=label-s] (24) at (-5.5, 5.35) {$Y_{I_2}$};
		\node [style=label-s] (25) at (-7.75, 2.25) {$Y_{I_8}$};
		\node [style=label] (26) at (-2.75, -3.5) {$=v_{{opt}}$};
		\node [style=label-s] (27) at (-2.75, 5) {$1\oplus 1$};
		\node [style=label-s] (28) at (-5.25, 3.25) {$I_1$};
		\node [style=label-s] (29) at (0.75, 3.75) {$I_2$};
		\node [style=label-s] (30) at (-3.75, -4) {$0\oplus 1$};
		\node [style=label-s] (31) at (-7.5, 5) {$2\oplus 0$};
	\end{pgfonlayer}
	\begin{pgfonlayer}{edgelayer}
		\draw (0) to (1);
		\draw (1) to (3);
		\draw [style=e8] (0) to (2);
		\draw (7) to (8);
		\draw (8) to (9);
		\draw [style=e1] (9) to (6);
		\draw (6) to (7);
		\draw [style=e2] (0) to (7);
		\draw [style=e1] (11) to (6);
		\draw [style=e3] (2) to (16);
		\draw [style=e3, in=-150, out=0, looseness=0.75] (16) to (3);
		\draw [style=e3, in=-180, out=0] (16) to (4);
		\draw [style=e3, in=120, out=0] (16) to (5);
		\draw [style=e3, in=210, out=0, looseness=0.75] (16) to (6);
		\draw [style=e4, in=-120, out=90, looseness=0.50] (15.center) to (4);
		\draw [style=e4] (15.center) to (5);
		\draw [style=e7] (14) to (4);
		\draw [style=e7] (14) to (5);
		\draw [style=d4] (0) to (13);
		\draw [style=d4, in=-180, out=-45, looseness=0.75] (13) to (3);
		\draw [style=d4] (13) to (4);
		\draw [style=d4, in=90, out=-45, looseness=0.75] (13) to (5);
		\draw [style=d2, in=0, out=135, looseness=0.75] (12.center) to (1);
		\draw [style=e4, in=-120, out=90, looseness=0.50] (15.center) to (3);
		\draw (21) to (15.center);
		\draw [style=e4] (21) to (11);
		\draw (21) to (15.center);
		\draw [style=d1] (0) to (6);
		\draw [style=e7] (14) to (1);
		\draw [style=d2] (7) to (12.center);
		\draw [style=d2] (12.center) to (9);
	\end{pgfonlayer}
\end{tikzpicture}

%% file: figures/GS.tex
\begin{tikzpicture}[scale=0.6]
	\begin{pgfonlayer}{nodelayer}
		\node [style=label] (0) at (-7, 4.5) {$3$};
		\node [style=label] (1) at (-2.25, 4.5) {$4$};
		\node [style=label] (2) at (-7, 0.5) {$8'$};
		\node [style=label] (3) at (-2.25, 1.25) {$9$};
		\node [style=label] (4) at (-3, 0.5) {$7'$};
		\node [style=label] (5) at (-3.75, -0.25) {$0'$};
		\node [style=label] (6) at (-1.5, 2) {$1'$};
		\node [style=label] (7) at (-1.5, 5.5) {$1$};
		\node [style=label] (8) at (2, 5.5) {$2$};
		\node [style=label] (9) at (2, 2) {$3$};
		\node [style=label] (11) at (-3.75, -3.5) {$3$};
		\node [style=none] (12) at (0, 4) {};
		\node [style=label-s] (13) at (-5, 2.5) {$I_4$};
		\node [style=label-s] (14) at (-3, 2) {$I_7$};
		\node [style=none] (15) at (-3.75, -1.75) {};
		\node [style=label-s] (16) at (-5.75, 0.5) {$Y_{I_3}$};
		\node [style=label-s] (21) at (-3.75, -2.25) {$Y_{I_4}$};
		\node [style=label-s] (22) at (-2.25, -1) {$Y_{I_1}$};
		\node [style=label-s] (23) at (0, 1.75) {$Y_{I_1}$};
		\node [style=label-s] (24) at (-4, 5) {$Y_{I_2}$};
		\node [style=label-s] (25) at (-7.5, 2.25) {$Y_{I_8}$};
		\node [style=label] (26) at (-2.75, -3.5) {$=v_{{opt}}$};
		\node [style=label-s] (27) at (-3.2, 4.15) {$1\oplus 2$};
		\node [style=label-s] (28) at (-4.25, 3.5) {$I_1$};
		\node [style=label-s] (29) at (0.75, 3.75) {$I_2$};
		\node [style=label-s] (30) at (-3.75, -4) {$0\oplus 1$};
		\node [style=label-s] (31) at (-8, 4.5) {$2\oplus 0$};
		\node [style=label] (38) at (-7.5, 8.5) {$7'$};
		\node [style=label] (39) at (-7, 9) {$0'$};
		\node [style=label] (40) at (-2.75, 8.5) {$7$};
		\node [style=label] (41) at (-2.25, 9) {$0'$};
		\node [style=label] (44) at (-6.5, 5) {$9$};
		\node [style=label-s] (45) at (-7, 6.75) {$Y_{I_4}$};
		\node [style=none] (46) at (-2.25, 6.5) {};
		\node [style=label-s] (47) at (-4.75, 6.75) {$I_7$};
		\node [style=label-s] (49) at (-5.25, 9) {$Y_0^2$};
		\node [style=label-s] (51) at (-1.75, 6.75) {$Y_{I_0}$};
	\end{pgfonlayer}
	\begin{pgfonlayer}{edgelayer}
		\draw (0) to (1);
		\draw (1) to (3);
		\draw [style=e8] (0) to (2);
		\draw (7) to (8);
		\draw (8) to (9);
		\draw [style=e1] (9) to (6);
		\draw (6) to (7);
		\draw [style=e1] (11) to (6);
		\draw [style=e3] (2) to (16);
		\draw [style=e3, in=-150, out=0, looseness=0.75] (16) to (3);
		\draw [style=e3, in=-180, out=0] (16) to (4);
		\draw [style=e3, in=120, out=0] (16) to (5);
		\draw [style=e3, in=210, out=0, looseness=0.75] (16) to (6);
		\draw [style=e4, in=-120, out=90, looseness=0.50] (15.center) to (4);
		\draw [style=e4] (15.center) to (5);
		\draw [style=e7] (14) to (4);
		\draw [style=e7] (14) to (5);
		\draw [style=d4] (0) to (13);
		\draw [style=d4, in=-180, out=-45, looseness=0.75] (13) to (3);
		\draw [style=d4] (13) to (4);
		\draw [style=d4, in=90, out=-45, looseness=0.75] (13) to (5);
		\draw [style=d2, in=0, out=135, looseness=0.75] (12.center) to (1);
		\draw [style=e4, in=-120, out=90, looseness=0.50] (15.center) to (3);
		\draw (21) to (15.center);
		\draw [style=e4] (21) to (11);
		\draw (21) to (15.center);
		\draw [style=d1] (0) to (6);
		\draw [style=e7] (14) to (1);
		\draw [style=d2] (7) to (12.center);
		\draw [style=d2] (12.center) to (9);
		\draw [style=e4, in=90, out=-60] (38) to (45);
		\draw [style=e4, in=120, out=-90] (45) to (44);
		\draw [style=e4] (39) to (45);
		\draw [style=e0, in=90, out=-60, looseness=0.75] (40) to (46.center);
		\draw [style=e0] (46.center) to (41);
		\draw (38) to (40);
		\draw [style=d7, in=135, out=-15, looseness=0.75] (38) to (47);
		\draw [style=d7] (47) to (39);
		\draw (39) to (49);
		\draw (49) to (41);
		\draw [style=e4] (45) to (0);
		\draw [style=e0] (46.center) to (1);
		\draw [style=d7] (47) to (1);
		\draw (44) to (1);
		\draw [style=e2] (0) to (24);
		\draw [style=e2] (24) to (7);
	\end{pgfonlayer}
\end{tikzpicture}

%% file: figures/fig6-1.tex
\begin{tikzpicture}[scale=0.7,baseline  = 0]
	\begin{pgfonlayer}{nodelayer}
		\node [style=label-s] (0) at (-4, 3) {$1$};
		\node [style=label-s] (1) at (-2, 3) {$4$};
		\node [style=label-s] (2) at (-4, 1) {$2$};
		\node [style=label-s] (3) at (-2, 1) {$5$};
		\node [style=label-s] (4) at (-4, -1) {$3$};
		\node [style=label-s] (5) at (-2, -1) {$6$};
		\node [style=none] (6) at (-4, -0.5) {};
		\node [style=none] (7) at (-4, -1.5) {};
		\node [style=none] (8) at (-4, 3.5) {};
		\node [style=none] (9) at (-4, 2.5) {};
		\node [style=none] (10) at (-2, -0.5) {};
		\node [style=none] (11) at (-2, -1.5) {};
		\node [style=none] (12) at (-2, 3.5) {};
		\node [style=none] (13) at (-2, 2.5) {};
		\node [style=none] (14) at (-4.75, 3) {};
		\node [style=none] (15) at (-3.25, 3) {};
		\node [style=none] (16) at (-4.75, -1) {};
		\node [style=none] (17) at (-3.25, -1) {};
		\node [style=none] (18) at (-5, 3) {};
		\node [style=none] (19) at (-1.25, 3) {};
		\node [style=none] (20) at (-5, -1) {};
		\node [style=none] (21) at (-1.25, -1) {};
	\end{pgfonlayer}
	\begin{pgfonlayer}{edgelayer}
		\draw (0) to (2);
		\draw (2) to (4);
		\draw (2) to (3);
		\draw (3) to (1);
		\draw (3) to (5);
		\draw [style=e3] (7.center)
			 to [bend right=90, looseness=1.75] (6.center)
			 to [bend left=270, looseness=1.75] cycle;
		\draw [style=e1] (9.center)
			 to [bend right=90, looseness=1.75] (8.center)
			 to [bend left=270, looseness=1.75] cycle;
		\draw [style=e6] (11.center)
			 to [bend right=90, looseness=1.75] (10.center)
			 to [bend left=270, looseness=1.75] cycle;
		\draw [style=e4] (13.center)
			 to [bend right=270, looseness=1.75] (12.center)
			 to [bend left=90, looseness=1.75] cycle;
		\draw [style=e2] (14.center) to (16.center);
		\draw [style=e2, bend right=90, looseness=1.75] (16.center) to (17.center);
		\draw [style=e2] (17.center) to (15.center);
		\draw [style=e2, bend right=90, looseness=1.75] (15.center) to (14.center);
		\draw [style=e5, bend left=90, looseness=1.25] (18.center) to (19.center);
		\draw [style=e5] (19.center) to (21.center);
		\draw [style=e5, bend left=90, looseness=1.25] (21.center) to (20.center);
		\draw [style=e5] (20.center) to (18.center);
	\end{pgfonlayer}
\end{tikzpicture}

%% file: figures/fig6-2.tex
\begin{tikzpicture}[scale=0.7, baseline = 0]
	\begin{pgfonlayer}{nodelayer}
		\node [style=label-s] (22) at (1.75, 1.75) {$1'$};
		\node [style=label-s] (23) at (2.25, 2.5) {$3'$};
		\node [style=label-s] (24) at (4.75, 2) {$4'$};
		\node [style=label-s] (25) at (5.25, 2.75) {$6'$};
		\node [style=label-s] (26) at (1.75, 5) {$5$};
		\node [style=label-s] (27) at (4.75, 5) {$4$};
		\node [style=label-s] (28) at (5.25, 5.75) {$6$};
		\node [style=label-s] (29) at (1.75, -0.75) {$2$};
		\node [style=label-s] (30) at (4.25, -1.25) {$3'$};
		\node [style=label-s] (31) at (4.75, -0.75) {$5$};
		\node [style=label-s] (32) at (1.75, -3.5) {$1$};
		\node [style=label-s] (33) at (4.25, -4) {$1'$};
		\node [style=label-s] (34) at (7, -1.25) {$3$};
		\node [style=label-s] (35) at (7, -4) {$2$};
		\node [style=none] (36) at (3.25, 2.25) {};
		\node [style=none] (37) at (3.75, 2.25) {};
		\node [style=none] (38) at (4.25, -2.75) {};
		\node [style=none] (39) at (1.75, 3.75) {};
		\node [style=none] (40) at (3.5, 0.75) {};
		\node [style=label-s] (41) at (0.75, -0.75) {$1\oplus 2$};
		\node [style=label-s] (42) at (5.5, -0.5) {$1\oplus 1$};
		\node [style=label-s] (43) at (3, -0.25) {$e_1$};
		\node [style=label-s] (44) at (3.25, -2) {$d$};
		\node [style=label-s] (45) at (4.75, -2.75) {$e_2$};
	\end{pgfonlayer}
	\begin{pgfonlayer}{edgelayer}
		\draw [style=m1] (26) to (28);
		\draw [style=m1, in=180, out=-30, looseness=0.75] (23) to (36.center);
		\draw [style=m1] (36.center) to (37.center);
		\draw [style=m1, bend right=15] (37.center) to (25);
		\draw [style=m1, in=0, out=150] (24) to (37.center);
		\draw [style=m1, bend right=15] (36.center) to (22);
		\draw [style=m1] (29) to (31);
		\draw [style=m1] (34) to (35);
		\draw [style=m1, in=-90, out=90, looseness=0.75] (38.center) to (31);
		\draw [style=m1] (30) to (38.center);
		\draw [style=m1] (38.center) to (33);
		\draw [style=m1] (26) to (27);
		\draw [style=e2] (26) to (39.center);
		\draw [style=e2, in=120, out=-90] (39.center) to (23);
		\draw [style=e2] (39.center) to (22);
		\draw [style=e4] (31) to (24);
		\draw [style=e6] (31) to (25);
		\draw [style=e1] (33) to (35);
		\draw [style=e1] (22) to (29);
		\draw [style=e3] (29) to (23);
		\draw [style=m1] (29) to (32);
		\draw (32) to (33);
		\draw [style=e3] (29) to (30);
		\draw [in=180, out=0] (30) to (34);
		\draw (27) to (24);
		\draw (28) to (25);
		\draw [style=d4] (26) to (24);
		\draw [style=d6] (26) to (25);
		\draw [style=d1] (29) to (33);
		\draw [style=d3] (30) to (35);
		\draw [style=d2] (23) to (40.center);
		\draw [style=d2, in=330, out=135] (40.center) to (22);
		\draw [style=d2] (40.center) to (31);
	\end{pgfonlayer}
\end{tikzpicture}

%% file: figures/fig7-1.tex
\begin{tikzpicture}[scale=0.7,baseline = 3.5em]
	\begin{pgfonlayer}{nodelayer}
		\node [style=label-s] (0) at (-4, 3) {$1$};
		\node [style=label-s] (1) at (-2, 3) {$4$};
		\node [style=label-s] (2) at (-4, 1) {$2$};
		\node [style=label-s] (3) at (-2, 1) {$5$};
		\node [style=label-s] (4) at (-4, -1) {$3$};
		\node [style=label-s] (5) at (-2, -1) {$6$};
		\node [style=none] (7) at (-4, -1.5) {};
		\node [style=none] (8) at (-3.75, 3.25) {};
		\node [style=none] (9) at (-4.25, 2.75) {};
		\node [style=none] (10) at (-2, -0.5) {};
		\node [style=none] (11) at (-2, -1.5) {};
		\node [style=none] (12) at (-2, 3.5) {};
		\node [style=none] (13) at (-2, 2.5) {};
		\node [style=none] (14) at (-4.75, 3) {};
		\node [style=none] (15) at (-3.25, 3) {};
		\node [style=none] (16) at (-4.75, -1) {};
		\node [style=none] (17) at (-3.25, -1) {};
		\node [style=none] (18) at (-5, 3) {};
		\node [style=none] (19) at (-1.25, 3) {};
		\node [style=none] (20) at (-5, -1) {};
		\node [style=none] (21) at (-1.25, -1) {};
		\node [style=none] (54) at (-4.5, 3) {};
		\node [style=none] (55) at (-3.5, 3) {};
		\node [style=none] (56) at (-4.5, 0.5) {};
		\node [style=none] (57) at (-3.5, 0.5) {};
	\end{pgfonlayer}
	\begin{pgfonlayer}{edgelayer}
		\draw (0) to (2);
		\draw (2) to (4);
		\draw (2) to (3);
		\draw (3) to (1);
		\draw (3) to (5);
		\draw [style=e1] (9.center)
			 to [bend right=90, looseness=1.75] (8.center)
			 to [bend left=270, looseness=1.75] cycle;
		\draw [style=e6] (11.center)
			 to [bend right=90, looseness=1.75] (10.center)
			 to [bend left=270, looseness=1.75] cycle;
		\draw [style=e4] (13.center)
			 to [bend right=270, looseness=1.75] (12.center)
			 to [bend left=90, looseness=1.75] cycle;
		\draw [style=e3] (14.center) to (16.center);
		\draw [style=e3, bend right=90, looseness=1.75] (16.center) to (17.center);
		\draw [style=e3] (17.center) to (15.center);
		\draw [style=e3, bend right=90, looseness=1.75] (15.center) to (14.center);
		\draw [style=e5, bend left=90, looseness=1.25] (18.center) to (19.center);
		\draw [style=e5] (19.center) to (21.center);
		\draw [style=e5, bend left=90, looseness=1.25] (21.center) to (20.center);
		\draw [style=e5] (20.center) to (18.center);
		\draw [style=e2] (54.center) to (56.center);
		\draw [style=e2, bend right=90, looseness=1.75] (56.center) to (57.center);
		\draw [style=e2] (57.center) to (55.center);
		\draw [style=e2, bend right=90, looseness=1.75] (55.center) to (54.center);
	\end{pgfonlayer}
\end{tikzpicture}%

%% file: figures/fig7-2.tex
\begin{tikzpicture}[scale = 0.7, baseline = 0]
	\begin{pgfonlayer}{nodelayer}
		\node [style=label-s] (0) at (-3, 3) {$5$};
		\node [style=label-s] (1) at (0, 3) {$4'$};
		\node [style=label-s] (2) at (3, 3) {$4$};
		\node [style=label-s] (3) at (0.5, 3.5) {$6'$};
		\node [style=label-s] (4) at (3.5, 3.5) {$6$};
		\node [style=label-s] (5) at (0, 0) {$3'$};
		\node [style=label-s] (6) at (0.5, 0.5) {$1'$};
		\node [style=label-s] (7) at (3, 0) {$5$};
		\node [style=label-s] (8) at (-3, 0) {$3$};
		\node [style=label-s] (9) at (-2.5, 0.5) {$1'$};
		\node [style=label-s] (10) at (-3, -3) {$2$};
		\node [style=label-s] (11) at (0, -3) {$5$};
		\node [style=label-s] (12) at (-0.5, -3.5) {$3$};
		\node [style=label-s] (13) at (0, -6) {$1'$};
		\node [style=label-s] (14) at (-3, -6) {$1$};
		\node [style=none] (15) at (1.5, 0) {};
		\node [style=none] (16) at (0, -4.75) {};
		\node [style=none] (17) at (0, -1.5) {};
		\node [style=none] (18) at (-3, 1.75) {};
		\node [style=none] (19) at (-1.5, -1.5) {};
		\node [style=none] (20) at (0.25, 2.25) {};
		\node [style=none] (21) at (0.25, 1.25) {};
		\node [style=none] (22) at (-1.25, 1.75) {};
		\node [style=label-s] (23) at (-3.5, 3.5) {$1\oplus 1$};
		\node [style=label-s] (24) at (-3.75, -3.25) {$1\oplus 2$};
		\node [style=label-s] (25) at (-2.5, -2) {$e_1$};
		\node [style=label-s] (26) at (-1.25, -1) {$d$};
		\node [style=label-s] (27) at (0.5, -1.5) {$e_2$};
		\node [style=label-s] (28) at (3.25, -0.5) {$2\oplus 1$};
	\end{pgfonlayer}
	\begin{pgfonlayer}{edgelayer}
		\draw [style=m1] (1) to (2);
		\draw [style=m1] (4) to (7);
		\draw [style=m1] (5) to (15.center);
		\draw [style=m1] (15.center) to (7);
		\draw [style=m1, bend right=15] (6) to (15.center);
		\draw [style=m1, bend left=15] (12) to (16.center);
		\draw [style=m1] (16.center) to (11);
		\draw [style=m1] (16.center) to (13);
		\draw [style=m1] (8) to (10);
		\draw [style=m1] (10) to (14);
		\draw [style=m1] (0) to (3);
		\draw [style=e4] (0) to (1);
		\draw [style=e3] (5) to (17.center);
		\draw [style=e3] (17.center) to (11);
		\draw [style=e3, bend left=15] (17.center) to (6);
		\draw [style=m1] (10) to (9);
		\draw [style=e2] (0) to (18.center);
		\draw [style=e2] (18.center) to (8);
		\draw [style=e2, in=120, out=-90, looseness=0.75] (18.center) to (9);
		\draw (8) to (5);
		\draw (3) to (4);
		\draw [style=e5, bend left=15] (1) to (20.center);
		\draw [style=e5] (20.center) to (21.center);
		\draw [style=e5, bend left=15] (21.center) to (5);
		\draw [style=e5, bend right=15] (21.center) to (6);
		\draw [style=e5, bend left=15] (20.center) to (3);
		\draw (2) to (7);
		\draw (10) to (12);
		\draw (11) to (10);
		\draw (14) to (13);
		\draw [style=d1] (10) to (13);
		\draw [style=d2] (8) to (19.center);
		\draw [style=d2, bend right=15] (9) to (19.center);
		\draw [style=d2] (19.center) to (11);
		\draw [style=d5] (22.center) to (0);
		\draw [style=d5, bend left=15] (22.center) to (5);
		\draw [style=d5] (22.center) to (6);
		\draw (9) to (6);
		\draw [style=d4] (7) to (1);
		\draw [style=d6] (3) to (7);
	\end{pgfonlayer}
\end{tikzpicture}

%% file: figures/BadMatching1.tex
\begin{tikzpicture}[scale=0.6]
	\begin{pgfonlayer}{nodelayer}
		\node [style=label] (0) at (-7, 4.5) {$3$};
		\node [style=label] (1) at (-2.25, 4.5) {$4$};
		\node [style=label] (2) at (-7, 0.5) {$8'$};
		\node [style=label] (3) at (-2.25, 1.25) {$9$};
		\node [style=label] (4) at (-3, 0.5) {$7'$};
		\node [style=label] (5) at (-3.75, -0.25) {$0'$};
		\node [style=label] (6) at (-1.5, 2) {$1'$};
		\node [style=label] (7) at (-1.5, 5.5) {$1$};
		\node [style=label] (8) at (2, 5.5) {$2$};
		\node [style=label] (9) at (2, 2) {$3$};
		\node [style=label] (11) at (-3.75, -3.5) {$3$};
		\node [style=none] (12) at (0, 4) {};
		\node [style=label-s,inner sep = 0.3 em] (13) at (-5, 2.5) {$d_2$};
		\node [style=label-s,inner sep = 0.3 em] (14) at (-3, 2) {$e_2$};
		\node [style=none] (15) at (-3.75, -1.75) {};
		\node [style=none] (16) at (-5.75, 0.5) {};
		\node [style=none] (21) at (-3.75, -2.25) {};
		\node [style=none] (24) at (-4, 5) {};
		\node [style=label-s] (27) at (-3.25, 4) {$1\oplus 2$};
		\node [style=label-s] (30) at (-3.75, -4) {$0\oplus 1$};
		\node [style=label-s] (31) at (-8, 4.5) {$2\oplus 0$};
		\node [style=label] (38) at (-7.5, 8.5) {$7'$};
		\node [style=label] (39) at (-7, 9) {$0'$};
		\node [style=label] (40) at (-2.75, 8.5) {$7$};
		\node [style=label] (41) at (-2.25, 9) {$0'$};
		\node [style=label] (44) at (-6.5, 5) {$9$};
		\node [style=label-s, inner sep = 0.3 em] (45) at (-7, 6.75) {$e_1$};
		\node [style=none] (46) at (-2.25, 6.5) {};
		\node [style=label-s, inner sep = 0.3 em] (47) at (-4.75, 6.75) {$d_1$};
		\node [style=none] (49) at (-5.25, 9) {};
	\end{pgfonlayer}
	\begin{pgfonlayer}{edgelayer}
		\draw (0) to (1);
		\draw [style=m1] (1) to (3);
		\draw [style=m1] (0) to (2);
		\draw (7) to (8);
		\draw [style=m1] (8) to (9);
		\draw [style=e1] (9) to (6);
		\draw [style=m1] (6) to (7);
		\draw [style=e1] (11) to (6);
		\draw [style=e3] (2) to (16.center);
		\draw [style=e3, in=-150, out=0, looseness=0.75] (16.center) to (3);
		\draw [style=e3, in=-180, out=0] (16.center) to (4);
		\draw [style=e3, in=120, out=0] (16.center) to (5);
		\draw [style=e3, in=210, out=0, looseness=0.75] (16.center) to (6);
		\draw [style=e4, in=-120, out=90, looseness=0.50] (15.center) to (4);
		\draw [style=e4] (15.center) to (5);
		\draw [style=m1] (14) to (4);
		\draw [style=m1] (14) to (5);
		\draw [style=d4] (0) to (13);
		\draw [style=d4, in=-180, out=-45, looseness=0.75] (13) to (3);
		\draw [style=d4] (13) to (4);
		\draw [style=d4, in=90, out=-45, looseness=0.75] (13) to (5);
		\draw [style=d2, in=0, out=135, looseness=0.75] (12.center) to (1);
		\draw [style=e4, in=-120, out=90, looseness=0.50] (15.center) to (3);
		\draw (21.center) to (15.center);
		\draw [style=e4] (21.center) to (11);
		\draw (21.center) to (15.center);
		\draw [style=d1] (0) to (6);
		\draw [style=m1] (14) to (1);
		\draw [style=d2] (7) to (12.center);
		\draw [style=d2] (12.center) to (9);
		\draw [style=m1, in=90, out=-60] (38) to (45);
		\draw [style=m1, in=120, out=-90] (45) to (44);
		\draw [style=m1] (39) to (45);
		\draw [style=m1, in=90, out=-60, looseness=0.75] (40) to (46.center);
		\draw [style=m1] (46.center) to (41);
		\draw (38) to (40);
		\draw [style=d7, in=135, out=-15, looseness=0.75] (38) to (47);
		\draw [style=d7] (47) to (39);
		\draw (39) to (49.center);
		\draw (49.center) to (41);
		\draw [style=m1] (45) to (0);
		\draw [style=m1] (46.center) to (1);
		\draw [style=d7] (47) to (1);
		\draw (44) to (1);
		\draw [style=e2] (0) to (24.center);
		\draw [style=e2] (24.center) to (7);
	\end{pgfonlayer}
\end{tikzpicture}

%% file: figures/BadMatching2.tex
\begin{tikzpicture}[scale=0.6,baseline = -1.2 em]
	\begin{pgfonlayer}{nodelayer}
		\node [style=label] (0) at (-10, 0) {$0$};
		\node [style=label] (1) at (-8, 0) {$1$};
		\node [style=label] (2) at (-6, 0) {$2$};
		\node [style=label] (3) at (-4, 0) {$4$};
		\node [style=label] (4) at (-2, 0) {$5$};
		\node [style=label] (5) at (-4, -2) {$6$};
		\node [style=label] (6) at (-6, -2) {$3$};
		\node [style=none] (7) at (-6, -1.5) {};
		\node [style=none] (8) at (-6, -2.5) {};
		\node [style=none] (9) at (-2, 0.5) {};
		\node [style=none] (10) at (-2, -0.5) {};
		\node [style=none] (11) at (-4, -1.5) {};
		\node [style=none] (12) at (-4, -2.5) {};
		\node [style=none] (13) at (-8, 0.5) {};
		\node [style=none] (14) at (-8, -0.5) {};
		\node [style=none] (15) at (-8.25, 1.5) {};
		\node [style=none] (16) at (-8.25, -3.5) {};
		\node [style=none] (17) at (-2.5, 1.5) {};
		\node [style=none] (18) at (-2.5, -3.5) {};
		\node [style=none] (19) at (-7.75, 1) {};
		\node [style=none] (20) at (-7.75, -3) {};
		\node [style=none] (21) at (-2.5, 1) {};
		\node [style=none] (22) at (-2.5, -3) {};
		\node [style=none] (31) at (-8, 0.75) {};
		\node [style=none] (32) at (-6.25, 0.75) {};
		\node [style=none] (33) at (-5.25, -0.25) {};
		\node [style=none] (34) at (-5.25, -2) {};
		\node [style=none] (35) at (-8, -0.75) {};
		\node [style=none] (36) at (-6.75, -2) {};
		\node [style=none] (37) at (-7.25, -0.75) {};
		\node [style=none] (38) at (-6.75, -1.25) {};
	\end{pgfonlayer}
	\begin{pgfonlayer}{edgelayer}
		\draw (0) to (1);
		\draw (1) to (2);
		\draw (6) to (2);
		\draw (2) to (3);
		\draw (3) to (4);
		\draw (5) to (3);
		\draw [style=e3] (8.center)
			 to [bend left=90, looseness=1.75] (7.center)
			 to [bend left=90, looseness=1.75] cycle;
		\draw [style=e5] (10.center)
			 to [bend left=90, looseness=1.75] (9.center)
			 to [bend left=90, looseness=1.75] cycle;
		\draw [style=e6] (12.center)
			 to [bend left=90, looseness=1.75] (11.center)
			 to [bend left=90, looseness=1.75] cycle;
		\draw [style=e1] (14.center)
			 to [bend left=90, looseness=1.75] (13.center)
			 to [bend left=90, looseness=1.75] cycle;
		\draw [style=e3] (17.center)
			 to (15.center)
			 to [bend left=270, looseness=1.75] (16.center)
			 to (18.center)
			 to [bend right=90, looseness=1.75] cycle;
		\draw [style=e4] (21.center)
			 to (19.center)
			 to [bend left=270, looseness=1.75] (20.center)
			 to (22.center)
			 to [bend right=90, looseness=1.75] cycle;
		\draw [style=e2] (33.center)
			 to [bend right=45, looseness=1.25] (32.center)
			 to (31.center)
			 to [bend right=90, looseness=1.75] (35.center)
			 to (37.center)
			 to [bend left=45, looseness=1.50] (38.center)
			 to (36.center)
			 to [bend right=90, looseness=1.50] (34.center)
			 to cycle;
	\end{pgfonlayer}
\end{tikzpicture}%
\hspace{-5em}
\input{ReduceOpacity}%
\begin{tikzpicture}[scale=0.6,baseline = 0]
	\begin{pgfonlayer}{nodelayer}
		\node [style=label-s] (39) at (3, 3) {$0$};
		\node [style=label-s] (40) at (6, 2.5) {$3'$};
		\node [style=label-s] (41) at (6.5, 3) {$5'$};
		\node [style=label-s] (42) at (7, 3.5) {$6'$};
		\node [style=label-s] (43) at (9.75, 3) {$0'$};
		\node [style=label-s] (44) at (3, 0) {$2$};
		\node [style=label-s] (45) at (6, 0) {$3'$};
		\node [style=label-s] (46) at (7, 1) {$4$};
		\node [style=label-s] (47) at (9.75, 0.5) {$0$};
		\node [style=label-s] (48) at (6, -3) {$0$};
		\node [style=label-s] (49) at (3, -3) {$1$};
		\node [style=none] (50) at (8.75, 0.5) {};
		\node [style=none] (51) at (8.75, 3) {};
		\node [style=none] (52) at (4.5, 3) {};
		\node [style=none] (53) at (6, -1.5) {};
		\node [style=label-s] (55) at (2.25, 0.5) {$1\oplus 1$};
		\node [style=label-s] (56) at (7.75, 1.25) {$1\oplus 1$};
		\node [style=none] (57) at (4.25, 1.75) {};
		\node [style=none] (59) at (8.25, 1.75) {};
		\node [style=label-w] (60) at (4.75, 0) {$e_2$};
		\node [style=label-w] (61) at (4.75, 0.5) {$e_1$};
		\node [style=label-w] (62) at (4.25, 1.75) {$d$};
	\end{pgfonlayer}
	\begin{pgfonlayer}{edgelayer}
		\draw [style=m1] (39) to (44);
		\draw [style=m1] (44) to (46);
		\draw [style=m1, in=180, out=0] (46) to (50.center);
		\draw [style=m1, in=0, out=-180] (50.center) to (45);
		\draw [style=m1] (50.center) to (47);
		\draw [style=m1] (51.center) to (43);
		\draw [style=m1, in=-180, out=-15, looseness=1.25] (42) to (51.center);
		\draw [style=m1] (51.center) to (41);
		\draw [style=m1, in=-180, out=0] (40) to (51.center);
		\draw [style=m1] (49) to (48);
		\draw (44) to (49);
		\draw [style=e4] (39) to (41);
		\draw [style=e4, in=180, out=0] (52.center) to (42);
		\draw [style=e4, in=180, out=0] (52.center) to (40);
		\draw [style=e2] (48) to (45);
		\draw [style=e2, in=-90, out=90] (53.center) to (46);
		\draw (44) to (45);
		\draw (43) to (47);
		\draw (40) to (45);
		\draw [style=e5] (41) to (46);
		\draw [style=e6] (42) to (46);
		\draw [style=d1] (44) to (48);
		\draw [style=d2, in=135, out=-45] (39) to (57.center);
		\draw [style=d2] (57.center) to (45);
		\draw [style=d2, in=150, out=-30, looseness=1.25] (57.center) to (46);
		\draw [style=d4] (41) to (59.center);
		\draw [style=d4] (59.center) to (47);
		\draw [style=d4, in=-45, out=150] (59.center) to (42);
		\draw [style=d4, in=150, out=-30] (40) to (59.center);
	\end{pgfonlayer}
\end{tikzpicture}

\input{ResetOpacity}

%% file: ReduceOpacity.tex
\tikzstyle{e1}=[-, color=1, line width=1, opacity = 0.4]
\tikzstyle{e2}=[-, color=2, line width=1, opacity = 0.4]
\tikzstyle{e3}=[-, color=3, line width=1, opacity = 0.4]
\tikzstyle{e4}=[-, color=4, line width=1, opacity = 0.4]
\tikzstyle{e5}=[-, color=5, line width=1, opacity = 0.4]
\tikzstyle{e6}=[-, color=6, line width=1, opacity = 0.4]
\tikzstyle{e7}=[-, color=7, line width=1, opacity = 0.4]
\tikzstyle{e8}=[-, color=8, line width=1, opacity = 0.4]
\tikzstyle{e9}=[-, color=9, line width=1, opacity = 0.4]
\tikzstyle{e0}=[-, color=0, line width=1, opacity = 0.4]
\tikzstyle{d1}=[-, color=1, line width=1.2, loosely dashed, opacity = 0.4]
\tikzstyle{d2}=[-, color=2, line width=1.2, loosely dashed, opacity = 0.4]
\tikzstyle{d3}=[-, color=3, line width=1.2, loosely dashed, opacity = 0.4]
\tikzstyle{d4}=[-, color=4, line width=1.2, loosely dashed, opacity = 0.4]
\tikzstyle{d5}=[-, color=5, line width=1.2, loosely dashed, opacity = 0.4]
\tikzstyle{d6}=[-, color=6, line width=1.2, loosely dashed, opacity = 0.4]
\tikzstyle{d7}=[-, color=7, line width=1.2, loosely dashed, opacity = 0.4]
\tikzstyle{d8}=[-, color=8, line width=1.2, loosely dashed, opacity = 0.4]
\tikzstyle{d9}=[-, color=9, line width=1.2, loosely dashed, opacity = 0.4]
\tikzstyle{d0}=[-, color=0, line width=1.2, loosely dashed, opacity = 0.4]

%% file: figures/matching1.tex
\begin{tikzpicture}[scale=0.5]
	\begin{pgfonlayer}{nodelayer}
		\node [style=label] (0) at (-7, 4.5) {$3$};
		\node [style=label] (1) at (-2.25, 4.5) {$4$};
		\node [style=label] (2) at (-7, 0.5) {$8'$};
		\node [style=label] (3) at (-2.25, 1.25) {$9$};
		\node [style=label] (4) at (-3, 0.5) {$7'$};
		\node [style=label] (5) at (-3.75, -0.25) {$0'$};
		\node [style=label] (6) at (-1.5, 2) {$1'$};
		\node [style=label] (7) at (-1.5, 5.5) {$1$};
		\node [style=label] (8) at (2, 5.5) {$2$};
		\node [style=label] (9) at (2, 2) {$3$};
		\node [style=label] (11) at (-3.75, -3.5) {$3$};
		\node [style=none] (12) at (0, 4) {};
		\node [style=label-s] (13) at (-5, 2.5) {${I_4}$};
		\node [style=none] (14) at (-3, 2) {};
		\node [style=none] (15) at (-3.75, -1.75) {};
		\node [style=none] (16) at (-5.75, 0.5) {};
		\node [style=label-s] (22) at (-2.25, -1) {$Y_{I_1}$};
		\node [style=label-s] (23) at (0.25, 1.5) {$Y_{I_1}$};
		\node [style=label-s] (25) at (-7.5, 2.25) {$Y_{I_8}$};
		\node [style=label-s] (27) at (-3.3, 4) {$1\oplus 2$};
		\node [style=label-s] (28) at (-4.5, 3.75) {${I_1}$};
		\node [style=label-s] (29) at (0.75, 3.75) {${I_2}$};
		\node [style=label-s] (30) at (-3.75, -4) {$0\oplus 1$};
		\node [style=label-s] (31) at (-8, 4.25) {$2\oplus 0$};
		\node [style=label] (38) at (-7.5, 8.5) {$7'$};
		\node [style=label] (39) at (-7, 9) {$0'$};
		\node [style=label] (40) at (-2.75, 8.5) {$7$};
		\node [style=label] (41) at (-2.25, 9) {$0'$};
		\node [style=label] (44) at (-6.5, 5) {$9$};
		\node [style=none] (45) at (-7, 6.75) {};
		\node [style=none] (46) at (-2.25, 6.5) {};
		\node [style=label-s] (47) at (-4.75, 6.75) {${I_7}$};
		\node [style=label-s] (51) at (-1.75, 6.75) {$Y_{I_0}$};
		\node [style=label-s] (52) at (-4.75, 9.5) {${Y_0^2}$};
		\node [style=label-s] (53) at (-5, 0) {$Y_{I_3}$};
		\node [style=label-s] (54) at (-4.25, 5.5) {$Y_{I_2}$};
		\node [style=label-s] (55) at (-4.25, -1.75) {$Y_{I_4}$};
		\node [style=label-s] (56) at (-7.5, 6.75) {$Y_{I_4}$};
		\node [style=label-s] (57) at (-3.5, 2.25) {$Y_{I_7}$};
		\node [style=none] (58) at (-4, -5) {$Y_{123456780}Y_2Y_{0}^2$};
	\end{pgfonlayer}
	\begin{pgfonlayer}{edgelayer}
		\draw [style=m1] (0) to (1);
		\draw [opacity=0.4] (1) to (3);
		\draw [style=e8] (0) to (2);
		\draw [opacity=0.4] (7) to (8);
		\draw [style=m1] (8) to (9);
		\draw [style=e1] (9) to (6);
		\draw [opacity=0.4] (6) to (7);
		\draw [style=e1] (11) to (6);
		\draw [style=m1] (2) to (16.center);
		\draw [style=m1, in=-150, out=0, looseness=0.75] (16.center) to (3);
		\draw [style=m1, in=-180, out=0] (16.center) to (4);
		\draw [style=m1, in=120, out=0] (16.center) to (5);
		\draw [style=m1, in=210, out=0, looseness=0.75] (16.center) to (6);
		\draw [style=e4, in=-120, out=90, looseness=0.50] (15.center) to (4);
		\draw [style=e4] (15.center) to (5);
		\draw [style=e7] (14.center) to (4);
		\draw [style=e7] (14.center) to (5);
		\draw [style=d4] (0) to (13);
		\draw [style=d4, in=-180, out=-45, looseness=0.75] (13) to (3);
		\draw [style=d4] (13) to (4);
		\draw [style=d4, in=90, out=-45, looseness=0.75] (13) to (5);
		\draw [style=d2, in=0, out=135, looseness=0.75] (12.center) to (1);
		\draw [style=e4, in=-120, out=90, looseness=0.50] (15.center) to (3);
		\draw [style=d1] (0) to (6);
		\draw [style=e7] (14.center) to (1);
		\draw [style=d2] (7) to (12.center);
		\draw [style=d2] (12.center) to (9);
		\draw [style=e4, in=90, out=-60] (38) to (45.center);
		\draw [style=e4, in=120, out=-90] (45.center) to (44);
		\draw [style=e4] (39) to (45.center);
		\draw [style=e0, in=90, out=-60, looseness=0.75] (40) to (46.center);
		\draw [style=e0] (46.center) to (41);
		\draw [style=m1] (38) to (40);
		\draw [style=d7, in=135, out=-15, looseness=0.75] (38) to (47);
		\draw [style=d7] (47) to (39);
		\draw [style=e4] (45.center) to (0);
		\draw [style=e0] (46.center) to (1);
		\draw [style=d7] (47) to (1);
		\draw [style=m1] (44) to (1);
		\draw [style=m1] (39) to (41);
		\draw [style=m1] (0) to (7);
		\draw [style=e4] (15.center) to (11);
	\end{pgfonlayer}
\end{tikzpicture}

%% file: figures/matching2.tex
\begin{tikzpicture}[scale=0.5]
	\begin{pgfonlayer}{nodelayer}
		\node [style=label] (0) at (-7, 4.5) {$3$};
		\node [style=label] (1) at (-2.25, 4.5) {$4$};
		\node [style=label] (2) at (-7, 0.5) {$8'$};
		\node [style=label] (3) at (-2.25, 1.25) {$9$};
		\node [style=label] (4) at (-3, 0.5) {$7'$};
		\node [style=label] (5) at (-3.75, -0.25) {$0'$};
		\node [style=label] (6) at (-1.5, 2) {$1'$};
		\node [style=label] (7) at (-1.5, 5.5) {$1$};
		\node [style=label] (8) at (2, 5.5) {$2$};
		\node [style=label] (9) at (2, 2) {$3$};
		\node [style=label] (11) at (-3.75, -3.5) {$3$};
		\node [style=none] (12) at (0, 4) {};
		\node [style=label-s] (13) at (-5, 2.5) {$I_4$};
		\node [style=none] (14) at (-3, 2) {};
		\node [style=none] (15) at (-3.75, -1.75) {};
		\node [style=none] (16) at (-5.75, 0.5) {};
		\node [style=label-s] (22) at (-2.25, -1) {$Y_{I_1}$};
		\node [style=label-s] (23) at (0.25, 1.5) {$Y_{I_1}$};
		\node [style=label-s] (25) at (-7.5, 2.25) {$Y_{I_8}$};
		\node [style=label-s] (27) at (-3.3, 4) {$1\oplus 2$};
		\node [style=label-s] (28) at (-4.5, 3.75) {$I_1$};
		\node [style=label-s] (29) at (0.75, 3.75) {$I_2$};
		\node [style=label-s] (30) at (-3.75, -4) {$0\oplus 1$};
		\node [style=label-s] (31) at (-8, 4.25) {$2\oplus 0$};
		\node [style=label] (38) at (-7.5, 8.5) {$7'$};
		\node [style=label] (39) at (-7, 9) {$0'$};
		\node [style=label] (40) at (-2.75, 8.5) {$7$};
		\node [style=label] (41) at (-2.25, 9) {$0'$};
		\node [style=label] (44) at (-6.5, 5) {$9$};
		\node [style=none] (45) at (-7, 6.75) {};
		\node [style=none] (46) at (-2.25, 6.5) {};
		\node [style=label-s] (47) at (-4.75, 6.75) {$I_7$};
		\node [style=label-s] (51) at (-1.75, 6.75) {$Y_{I_0}$};
		\node [style=label-s] (52) at (-4.75, 9.5) {$Y_0^2$};
		\node [style=label-s] (53) at (-5, 0) {$Y_{I_3}$};
		\node [style=label-s] (54) at (-4.25, 5.5) {$Y_{I_2}$};
		\node [style=label-s] (55) at (-4.25, -1.75) {$Y_{I_4}$};
		\node [style=label-s] (56) at (-7.5, 6.75) {$Y_{I_4}$};
		\node [style=label-s] (57) at (-3.5, 2.25) {$Y_{I_7}$};
		\node [style=none] (58) at (-4, -5) {};
		\node [style=none] (59) at (-4, -5) {$Y_8 Y_{45670} Y_0^2$};
	\end{pgfonlayer}
	\begin{pgfonlayer}{edgelayer}
		\draw [style=m1] (0) to (1);
		\draw [opacity=0.4] (1) to (3);
		\draw [style=m1] (0) to (2);
		\draw [opacity=0.4] (7) to (8);
		\draw [style=m1] (8) to (9);
		\draw [style=e1] (9) to (6);
		\draw [style=m1] (6) to (7);
		\draw [style=e1] (11) to (6);
		\draw [style=e3] (2) to (16.center);
		\draw [style=e3, in=-150, out=0, looseness=0.75] (16.center) to (3);
		\draw [style=e3, in=-180, out=0] (16.center) to (4);
		\draw [style=e3, in=120, out=0] (16.center) to (5);
		\draw [style=e3, in=210, out=0, looseness=0.75] (16.center) to (6);
		\draw [style=m1, in=-120, out=90, looseness=0.50] (15.center) to (4);
		\draw [style=m1] (15.center) to (5);
		\draw [style=e7] (14.center) to (4);
		\draw [style=e7] (14.center) to (5);
		\draw [style=d4] (0) to (13);
		\draw [style=d4, in=-180, out=-45, looseness=0.75] (13) to (3);
		\draw [style=d4] (13) to (4);
		\draw [style=d4, in=90, out=-45, looseness=0.75] (13) to (5);
		\draw [style=d2, in=0, out=135, looseness=0.75] (12.center) to (1);
		\draw [style=m1, in=-120, out=90, looseness=0.50] (15.center) to (3);
		\draw [style=d1] (0) to (6);
		\draw [style=e7] (14.center) to (1);
		\draw [style=d2] (7) to (12.center);
		\draw [style=d2] (12.center) to (9);
		\draw [style=e4, in=90, out=-60] (38) to (45.center);
		\draw [style=e4, in=120, out=-90] (45.center) to (44);
		\draw [style=e4] (39) to (45.center);
		\draw [style=e0, in=90, out=-60, looseness=0.75] (40) to (46.center);
		\draw [style=e0] (46.center) to (41);
		\draw [style=m1] (38) to (40);
		\draw [style=d7, in=135, out=-15, looseness=0.75] (38) to (47);
		\draw [style=d7] (47) to (39);
		\draw [style=e4] (45.center) to (0);
		\draw [style=e0] (46.center) to (1);
		\draw [style=d7] (47) to (1);
		\draw [style=m1] (44) to (1);
		\draw [style=m1] (39) to (41);
		\draw [opacity=0.4] (0) to (7);
		\draw [style=m1] (15.center) to (11);
	\end{pgfonlayer}
\end{tikzpicture}

%% file: figures/matching3.tex
\begin{tikzpicture}[scale=0.5]
	\begin{pgfonlayer}{nodelayer}
		\node [style=label] (0) at (-7, 4.5) {$3$};
		\node [style=label] (1) at (-2.25, 4.5) {$4$};
		\node [style=label] (2) at (-7, 0.5) {$8'$};
		\node [style=label] (3) at (-2.25, 1.25) {$9$};
		\node [style=label] (4) at (-3, 0.5) {$7'$};
		\node [style=label] (5) at (-3.75, -0.25) {$0'$};
		\node [style=label] (6) at (-1.5, 2) {$1'$};
		\node [style=label] (7) at (-1.5, 5.5) {$1$};
		\node [style=label] (8) at (2, 5.5) {$2$};
		\node [style=label] (9) at (2, 2) {$3$};
		\node [style=label] (11) at (-3.75, -3.5) {$3$};
		\node [style=none] (12) at (0, 4) {};
		\node [style=label-s] (13) at (-5, 2.5) {$I_4$};
		\node [style=none] (14) at (-3, 2) {};
		\node [style=none] (15) at (-3.75, -1.75) {};
		\node [style=none] (16) at (-5.75, 0.5) {};
		\node [style=label-s] (22) at (-2.25, -1) {$Y_{I_1}$};
		\node [style=label-s] (23) at (0.25, 1.5) {$Y_{I_1}$};
		\node [style=label-s] (25) at (-7.5, 2.25) {$Y_{I_8}$};
		\node [style=label-s] (27) at (-3.3, 4) {$1\oplus 2$};
		\node [style=label-s] (28) at (-4.5, 3.75) {$I_1$};
		\node [style=label-s] (29) at (0.75, 3.75) {$I_2$};
		\node [style=label-s] (30) at (-3.75, -4) {$0\oplus 1$};
		\node [style=label-s] (31) at (-8, 4.25) {$2\oplus 0$};
		\node [style=label] (38) at (-7.5, 8.5) {$7'$};
		\node [style=label] (39) at (-7, 9) {$0'$};
		\node [style=label] (40) at (-2.75, 8.5) {$7$};
		\node [style=label] (41) at (-2.25, 9) {$0'$};
		\node [style=label] (44) at (-6.5, 5) {$9$};
		\node [style=none] (45) at (-7, 6.75) {};
		\node [style=none] (46) at (-2.25, 6.5) {};
		\node [style=label-s] (47) at (-4.75, 6.75) {$I_7$};
		\node [style=label-s] (51) at (-1.75, 6.75) {$Y_{I_0}$};
		\node [style=label-s] (52) at (-4.75, 9.5) {$Y_0^2$};
		\node [style=label-s] (53) at (-5, 0) {$Y_{I_3}$};
		\node [style=label-s] (54) at (-4.25, 5.5) {$Y_{I_2}$};
		\node [style=label-s] (55) at (-4.25, -1.75) {$Y_{I_4}$};
		\node [style=label-s] (56) at (-7.5, 6.75) {$Y_{I_4}$};
		\node [style=label-s] (57) at (-3.5, 2.25) {$I_7$};
		\node [style=none] (58) at (-3.75, -5) {};
		\node [style=none] (59) at (-3.75, -5) {$Y_8Y_{45670}^2Y_{560}Y_{12}$};
	\end{pgfonlayer}
	\begin{pgfonlayer}{edgelayer}
		\draw [opacity=0.4] (0) to (1);
		\draw [opacity=0.4] (1) to (3);
		\draw [style=m1] (0) to (2);
		\draw [style=m1] (7) to (8);
		\draw [opacity=0.4] (8) to (9);
		\draw [style=m1] (9) to (6);
		\draw [opacity=0.4] (6) to (7);
		\draw [style=e1] (11) to (6);
		\draw [style=e3] (2) to (16.center);
		\draw [style=e3, in=-150, out=0, looseness=0.75] (16.center) to (3);
		\draw [style=e3, in=-180, out=0] (16.center) to (4);
		\draw [style=e3, in=120, out=0] (16.center) to (5);
		\draw [style=e3, in=210, out=0, looseness=0.75] (16.center) to (6);
		\draw [style=m1, in=-120, out=90, looseness=0.50] (15.center) to (4);
		\draw [style=m1] (15.center) to (5);
		\draw [style=e7] (14.center) to (4);
		\draw [style=e7] (14.center) to (5);
		\draw [style=d4] (0) to (13);
		\draw [style=d4, in=-180, out=-45, looseness=0.75] (13) to (3);
		\draw [style=d4] (13) to (4);
		\draw [style=d4, in=90, out=-45, looseness=0.75] (13) to (5);
		\draw [style=d2, in=0, out=135, looseness=0.75] (12.center) to (1);
		\draw [style=m1, in=-120, out=90, looseness=0.50] (15.center) to (3);
		\draw [style=d1] (0) to (6);
		\draw [style=e7] (14.center) to (1);
		\draw [style=d2] (7) to (12.center);
		\draw [style=d2] (12.center) to (9);
		\draw [style=m1, in=90, out=-60] (38) to (45.center);
		\draw [style=m1, in=120, out=-90] (45.center) to (44);
		\draw [style=m1] (39) to (45.center);
		\draw [style=m1, in=90, out=-60, looseness=0.75] (40) to (46.center);
		\draw [style=m1] (46.center) to (41);
		\draw [opacity=0.4] (38) to (40);
		\draw [style=d7, in=135, out=-15, looseness=0.75] (38) to (47);
		\draw [style=d7] (47) to (39);
		\draw [style=m1] (45.center) to (0);
		\draw [style=m1] (46.center) to (1);
		\draw [style=d7] (47) to (1);
		\draw [opacity=0.4] (44) to (1);
		\draw [opacity=0.4] (39) to (41);
		\draw [opacity=0.4] (0) to (7);
		\draw [style=m1] (15.center) to (11);
	\end{pgfonlayer}
\end{tikzpicture}

%% file: figures/matching4.tex
\begin{tikzpicture}[scale=0.5]
	\begin{pgfonlayer}{nodelayer}
		\node [style=label] (0) at (-7, 4.5) {$3$};
		\node [style=label] (1) at (-2.25, 4.5) {$4$};
		\node [style=label] (2) at (-7, 0.5) {$8'$};
		\node [style=label] (3) at (-2.25, 1.25) {$9$};
		\node [style=label] (4) at (-3, 0.5) {$7'$};
		\node [style=label] (5) at (-3.75, -0.25) {$0'$};
		\node [style=label] (6) at (-1.5, 2) {$1'$};
		\node [style=label] (7) at (-1.5, 5.5) {$1$};
		\node [style=label] (8) at (2, 5.5) {$2$};
		\node [style=label] (9) at (2, 2) {$3$};
		\node [style=label] (11) at (-3.75, -3.5) {$3$};
		\node [style=none] (12) at (0, 4) {};
		\node [style=label-s] (13) at (-5, 2.5) {$I_4$};
		\node [style=none] (14) at (-3, 2) {};
		\node [style=none] (15) at (-3.75, -1.75) {};
		\node [style=none] (16) at (-5.75, 0.5) {};
		\node [style=label-s] (22) at (-2.25, -1) {$Y_{I_1}$};
		\node [style=label-s] (23) at (0.25, 1.5) {$Y_{I_1}$};
		\node [style=label-s] (25) at (-7.5, 2.25) {$Y_{I_8}$};
		\node [style=label-s] (27) at (-3.3, 4) {$1\oplus 2$};
		\node [style=label-s] (28) at (-4.5, 3.75) {$I_1$};
		\node [style=label-s] (29) at (0.75, 3.75) {$I_2$};
		\node [style=label-s] (30) at (-3.75, -4) {$0\oplus 1$};
		\node [style=label-s] (31) at (-8, 4.25) {$2\oplus 0$};
		\node [style=label] (38) at (-7.5, 8.5) {$7'$};
		\node [style=label] (39) at (-7, 9) {$0'$};
		\node [style=label] (40) at (-2.75, 8.5) {$7$};
		\node [style=label] (41) at (-2.25, 9) {$0'$};
		\node [style=label] (44) at (-6.5, 5) {$9$};
		\node [style=none] (45) at (-7, 6.75) {};
		\node [style=none] (46) at (-2.25, 6.5) {};
		\node [style=label-s] (47) at (-4.75, 6.75) {$I_7$};
		\node [style=label-s] (51) at (-1.75, 6.75) {$Y_{I_0}$};
		\node [style=label-s] (52) at (-4.75, 9.5) {$Y_0^2$};
		\node [style=label-s] (53) at (-5, 0) {$Y_{I_3}$};
		\node [style=label-s] (54) at (-4.25, 5.5) {$Y_{I_2}$};
		\node [style=label-s] (55) at (-4.25, -1.75) {$Y_{I_4}$};
		\node [style=label-s] (56) at (-7.5, 6.75) {$Y_{I_4}$};
		\node [style=label-s] (57) at (-3.5, 2.25) {$I_7$};
		\node [style=none] (58) at (-3.75, -5) {};
		\node [style=none] (59) at (-3.75, -5) {$Y_8Y_{45670}Y_0^2Y_{12}$};
	\end{pgfonlayer}
	\begin{pgfonlayer}{edgelayer}
		\draw [style=m1] (0) to (1);
		\draw [opacity=0.4] (1) to (3);
		\draw [style=m1] (0) to (2);
		\draw [style=m1] (7) to (8);
		\draw [opacity=0.4] (8) to (9);
		\draw [style=m1] (9) to (6);
		\draw [opacity=0.4] (6) to (7);
		\draw [style=e1] (11) to (6);
		\draw [style=e3] (2) to (16.center);
		\draw [style=e3, in=-150, out=0, looseness=0.75] (16.center) to (3);
		\draw [style=e3, in=-180, out=0] (16.center) to (4);
		\draw [style=e3, in=120, out=0] (16.center) to (5);
		\draw [style=e3, in=210, out=0, looseness=0.75] (16.center) to (6);
		\draw [style=m1, in=-120, out=90, looseness=0.50] (15.center) to (4);
		\draw [style=m1] (15.center) to (5);
		\draw [style=e7] (14.center) to (4);
		\draw [style=e7] (14.center) to (5);
		\draw [style=d4] (0) to (13);
		\draw [style=d4, in=-180, out=-45, looseness=0.75] (13) to (3);
		\draw [style=d4] (13) to (4);
		\draw [style=d4, in=90, out=-45, looseness=0.75] (13) to (5);
		\draw [style=d2, in=0, out=135, looseness=0.75] (12.center) to (1);
		\draw [style=m1, in=-120, out=90, looseness=0.50] (15.center) to (3);
		\draw [style=d1] (0) to (6);
		\draw [style=e7] (14.center) to (1);
		\draw [style=d2] (7) to (12.center);
		\draw [style=d2] (12.center) to (9);
		\draw [style=e4, in=90, out=-60] (38) to (45.center);
		\draw [style=e4, in=120, out=-90] (45.center) to (44);
		\draw [style=e4] (39) to (45.center);
		\draw [style=e0, in=90, out=-60, looseness=0.75] (40) to (46.center);
		\draw [style=e0] (46.center) to (41);
		\draw [style=m1] (38) to (40);
		\draw [style=d7, in=135, out=-15, looseness=0.75] (38) to (47);
		\draw [style=d7] (47) to (39);
		\draw [style=e4] (45.center) to (0);
		\draw [style=e0] (46.center) to (1);
		\draw [style=d7] (47) to (1);
		\draw [style=m1] (44) to (1);
		\draw [style=m1] (39) to (41);
		\draw [opacity=0.4] (0) to (7);
		\draw [style=m1] (15.center) to (11);
	\end{pgfonlayer}
\end{tikzpicture}

%% file: figures/matching5.tex
\begin{tikzpicture}[scale=0.5]
	\begin{pgfonlayer}{nodelayer}
		\node [style=label] (0) at (-7, 4.5) {$3$};
		\node [style=label] (1) at (-2.25, 4.5) {$4$};
		\node [style=label] (2) at (-7, 0.5) {$8'$};
		\node [style=label] (3) at (-2.25, 1.25) {$9$};
		\node [style=label] (4) at (-3, 0.5) {$7'$};
		\node [style=label] (5) at (-3.75, -0.25) {$0'$};
		\node [style=label] (6) at (-1.5, 2) {$1'$};
		\node [style=label] (7) at (-1.5, 5.5) {$1$};
		\node [style=label] (8) at (2, 5.5) {$2$};
		\node [style=label] (9) at (2, 2) {$3$};
		\node [style=label] (11) at (-3.75, -3.5) {$3$};
		\node [style=none] (12) at (0, 4) {};
		\node [style=label-s] (13) at (-5, 2.5) {$I_4$};
		\node [style=none] (14) at (-3, 2) {};
		\node [style=none] (15) at (-3.75, -1.75) {};
		\node [style=none] (16) at (-5.75, 0.5) {};
		\node [style=label-s] (22) at (-2.25, -1) {$Y_{I_1}$};
		\node [style=label-s] (23) at (0.25, 1.5) {$Y_{I_1}$};
		\node [style=label-s] (25) at (-7.5, 2.25) {$Y_{I_8}$};
		\node [style=label-s] (27) at (-3.3, 4) {$1\oplus 2$};
		\node [style=label-s] (28) at (-4.5, 3.75) {${I_1}$};
		\node [style=label-s] (29) at (0.75, 3.75) {$I_2$};
		\node [style=label-s] (30) at (-3.75, -4) {$0\oplus 1$};
		\node [style=label-s] (31) at (-8, 4.25) {$2\oplus 0$};
		\node [style=label] (38) at (-7.5, 8.5) {$7'$};
		\node [style=label] (39) at (-7, 9) {$0'$};
		\node [style=label] (40) at (-2.75, 8.5) {$7$};
		\node [style=label] (41) at (-2.25, 9) {$0'$};
		\node [style=label] (44) at (-6.5, 5) {$9$};
		\node [style=none] (45) at (-7, 6.75) {};
		\node [style=none] (46) at (-2.25, 6.5) {};
		\node [style=label-s] (47) at (-4.75, 6.75) {$I_7$};
		\node [style=label-s] (51) at (-1.75, 6.75) {$Y_{I_0}$};
		\node [style=label-s] (52) at (-4.75, 9.5) {$Y_0^2$};
		\node [style=label-s] (53) at (-5, 0) {$Y_{I_3}$};
		\node [style=label-s] (54) at (-4.25, 5.5) {$Y_{I_2}$};
		\node [style=label-s] (55) at (-4.25, -1.75) {$Y_{I_4}$};
		\node [style=label-s] (56) at (-7.5, 6.75) {$Y_{I_4}$};
		\node [style=label-s] (57) at (-3.5, 2.25) {$Y_{I_7}$};
		\node [style=none] (58) at (-3.75, -5) {$Y_2Y_{5670}Y_8Y_0^2Y_{12}$};
	\end{pgfonlayer}
	\begin{pgfonlayer}{edgelayer}
		\draw [opacity=0.4] (0) to (1);
		\draw [opacity=0.4] (1) to (3);
		\draw [style=m1] (0) to (2);
		\draw [opacity=0.4] (7) to (8);
		\draw [style=m1] (8) to (9);
		\draw [style=e1] (9) to (6);
		\draw [opacity=0.4] (6) to (7);
		\draw [style=m1] (11) to (6);
		\draw [style=e3] (2) to (16.center);
		\draw [style=e3, in=-150, out=0, looseness=0.75] (16.center) to (3);
		\draw [style=e3, in=-180, out=0] (16.center) to (4);
		\draw [style=e3, in=120, out=0] (16.center) to (5);
		\draw [style=e3, in=210, out=0, looseness=0.75] (16.center) to (6);
		\draw [style=e4, in=-120, out=90, looseness=0.50] (15.center) to (4);
		\draw [style=e4] (15.center) to (5);
		\draw [style=m1] (14.center) to (4);
		\draw [style=m1] (14.center) to (5);
		\draw [style=d4] (0) to (13);
		\draw [style=d4, in=-180, out=-45, looseness=0.75] (13) to (3);
		\draw [style=d4] (13) to (4);
		\draw [style=d4, in=90, out=-45, looseness=0.75] (13) to (5);
		\draw [style=d2, in=0, out=135, looseness=0.75] (12.center) to (1);
		\draw [style=e4, in=-120, out=90, looseness=0.50] (15.center) to (3);
		\draw [style=d1] (0) to (6);
		\draw [style=m1] (14.center) to (1);
		\draw [style=d2] (7) to (12.center);
		\draw [style=d2] (12.center) to (9);
		\draw [style=e4, in=90, out=-60] (38) to (45.center);
		\draw [style=e4, in=120, out=-90] (45.center) to (44);
		\draw [style=e4] (39) to (45.center);
		\draw [style=e0, in=90, out=-60, looseness=0.75] (40) to (46.center);
		\draw [style=e0] (46.center) to (41);
		\draw [style=m1] (38) to (40);
		\draw [style=d7, in=135, out=-15, looseness=0.75] (38) to (47);
		\draw [style=d7] (47) to (39);
		\draw [style=e4] (45.center) to (0);
		\draw [style=e0] (46.center) to (1);
		\draw [style=d7] (47) to (1);
		\draw [style=m1] (44) to (1);
		\draw [style=m1] (39) to (41);
		\draw [style=m1] (0) to (7);
		\draw [style=e4] (15.center) to (11);
	\end{pgfonlayer}
\end{tikzpicture}

%% file: figures/matching6.tex
\begin{tikzpicture}[scale=0.5]
	\begin{pgfonlayer}{nodelayer}
		\node [style=label] (0) at (-7, 4.5) {$3$};
		\node [style=label] (1) at (-2.25, 4.5) {$4$};
		\node [style=label] (2) at (-7, 0.5) {$8'$};
		\node [style=label] (3) at (-2.25, 1.25) {$9$};
		\node [style=label] (4) at (-3, 0.5) {$7'$};
		\node [style=label] (5) at (-3.75, -0.25) {$0'$};
		\node [style=label] (6) at (-1.5, 2) {$1'$};
		\node [style=label] (7) at (-1.5, 5.5) {$1$};
		\node [style=label] (8) at (2, 5.5) {$2$};
		\node [style=label] (9) at (2, 2) {$3$};
		\node [style=label] (11) at (-3.75, -3.5) {$3$};
		\node [style=none] (12) at (0, 4) {};
		\node [style=label-s] (13) at (-5, 2.5) {$I_4$};
		\node [style=none] (14) at (-3, 2) {};
		\node [style=none] (15) at (-3.75, -1.75) {};
		\node [style=none] (16) at (-5.75, 0.5) {};
		\node [style=label-s] (22) at (-2.25, -1) {$Y_{I_1}$};
		\node [style=label-s] (23) at (0.25, 1.5) {$Y_{I_1}$};
		\node [style=label-s] (25) at (-7.5, 2.25) {$Y_{I_8}$};
		\node [style=label-s] (27) at (-3.3, 4) {$1\oplus 2$};
		\node [style=label-s] (28) at (-4.5, 3.75) {$I_1$};
		\node [style=label-s] (29) at (0.75, 3.75) {$I_2$};
		\node [style=label-s] (30) at (-3.75, -4) {$0\oplus 1$};
		\node [style=label-s] (31) at (-8, 4) {$2\oplus 0$};
		\node [style=label] (38) at (-7.5, 8.5) {$7'$};
		\node [style=label] (39) at (-7, 9) {$0'$};
		\node [style=label] (40) at (-2.75, 8.5) {$7$};
		\node [style=label] (41) at (-2.25, 9) {$0'$};
		\node [style=label] (44) at (-6.5, 5) {$9$};
		\node [style=none] (45) at (-7, 6.75) {};
		\node [style=none] (46) at (-2.25, 6.5) {};
		\node [style=label-s] (47) at (-4.75, 6.75) {$I_7$};
		\node [style=label-s] (51) at (-1.75, 6.75) {$Y_{I_0}$};
		\node [style=label-s] (52) at (-4.75, 9.5) {$Y_0^2$};
		\node [style=label-s] (53) at (-5, 0) {$Y_{I_3}$};
		\node [style=label-s] (54) at (-4.25, 5.5) {$I_2$};
		\node [style=label-s] (55) at (-4.25, -1.75) {$Y_{I_4}$};
		\node [style=label-s] (56) at (-7.5, 6.75) {$Y_{I_4}$};
		\node [style=label-s] (57) at (-3.5, 2.25) {$Y_{I_7}$};
		\node [style=none] (58) at (-3.75, -5) {};
		\node [style=none] (59) at (-3.75, -5) {$Y_{123456780}Y_{2}Y_{45670}Y_{560}$};
	\end{pgfonlayer}
	\begin{pgfonlayer}{edgelayer}
		\draw [opacity=0.4] (0) to (1);
		\draw [opacity=0.4] (1) to (3);
		\draw [style=e8] (0) to (2);
		\draw [opacity=0.4] (7) to (8);
		\draw [style=m1] (8) to (9);
		\draw [style=e1] (9) to (6);
		\draw [opacity=0.4] (6) to (7);
		\draw [style=e1] (11) to (6);
		\draw [style=m1] (2) to (16.center);
		\draw [style=m1, in=-150, out=0, looseness=0.75] (16.center) to (3);
		\draw [style=m1, in=-180, out=0] (16.center) to (4);
		\draw [style=m1, in=120, out=0] (16.center) to (5);
		\draw [style=m1, in=210, out=0, looseness=0.75] (16.center) to (6);
		\draw [style=e4, in=-120, out=90, looseness=0.50] (15.center) to (4);
		\draw [style=e4] (15.center) to (5);
		\draw [style=e7] (14.center) to (4);
		\draw [style=e7] (14.center) to (5);
		\draw [style=d4] (0) to (13);
		\draw [style=d4, in=-180, out=-45, looseness=0.75] (13) to (3);
		\draw [style=d4] (13) to (4);
		\draw [style=d4, in=90, out=-45, looseness=0.75] (13) to (5);
		\draw [style=d2, in=0, out=135, looseness=0.75] (12.center) to (1);
		\draw [style=e4, in=-120, out=90, looseness=0.50] (15.center) to (3);
		\draw [style=d1] (0) to (6);
		\draw [style=e7] (14.center) to (1);
		\draw [style=d2] (7) to (12.center);
		\draw [style=d2] (12.center) to (9);
		\draw [style=m1, in=90, out=-60] (38) to (45.center);
		\draw [style=m1, in=120, out=-90] (45.center) to (44);
		\draw [style=m1] (39) to (45.center);
		\draw [style=m1, in=90, out=-60, looseness=0.75] (40) to (46.center);
		\draw [style=m1] (46.center) to (41);
		\draw [opacity=0.4] (38) to (40);
		\draw [style=d7, in=135, out=-15, looseness=0.75] (38) to (47);
		\draw [style=d7] (47) to (39);
		\draw [style=m1] (45.center) to (0);
		\draw [style=m1] (46.center) to (1);
		\draw [style=d7] (47) to (1);
		\draw [opacity=0.4] (44) to (1);
		\draw [opacity=0.4] (39) to (41);
		\draw [style=m1] (0) to (7);
		\draw [style=e4] (15.center) to (11);
	\end{pgfonlayer}
\end{tikzpicture}

%% file: figures/matching7.tex
\begin{tikzpicture}[scale=0.5]
	\begin{pgfonlayer}{nodelayer}
		\node [style=label] (0) at (-7, 4.5) {$3$};
		\node [style=label] (1) at (-2.25, 4.5) {$4$};
		\node [style=label] (2) at (-7, 0.5) {$8'$};
		\node [style=label] (3) at (-2.25, 1.25) {$9$};
		\node [style=label] (4) at (-3, 0.5) {$7'$};
		\node [style=label] (5) at (-3.75, -0.25) {$0'$};
		\node [style=label] (6) at (-1.5, 2) {$1'$};
		\node [style=label] (7) at (-1.5, 5.5) {$1$};
		\node [style=label] (8) at (2, 5.5) {$2$};
		\node [style=label] (9) at (2, 2) {$3$};
		\node [style=label] (11) at (-3.75, -3.5) {$3$};
		\node [style=none] (12) at (0, 4) {};
		\node [style=label-s] (13) at (-5, 2.5) {$I_4$};
		\node [style=none] (14) at (-3, 2) {};
		\node [style=none] (15) at (-3.75, -1.75) {};
		\node [style=none] (16) at (-5.75, 0.5) {};
		\node [style=label-s] (22) at (-2.25, -1) {$Y_{I_1}$};
		\node [style=label-s] (23) at (0.25, 1.5) {$Y_{I_1}$};
		\node [style=label-s] (25) at (-7.5, 2.25) {$Y_{I_8}$};
		\node [style=label-s] (27) at (-3.3, 4) {$1\oplus 2$};
		\node [style=label-s] (28) at (-4.5, 3.75) {$I_1$};
		\node [style=label-s] (29) at (0.75, 3.75) {$I_2$};
		\node [style=label-s] (30) at (-3.75, -4) {$0\oplus 1$};
		\node [style=label-s] (31) at (-8, 4.25) {$2\oplus 0$};
		\node [style=label] (38) at (-7.5, 8.5) {$7'$};
		\node [style=label] (39) at (-7, 9) {$0'$};
		\node [style=label] (40) at (-2.75, 8.5) {$7$};
		\node [style=label] (41) at (-2.25, 9) {$0'$};
		\node [style=label] (44) at (-6.5, 5) {$9$};
		\node [style=none] (45) at (-7, 6.75) {};
		\node [style=none] (46) at (-2.25, 6.5) {};
		\node [style=label-s] (47) at (-4.75, 6.75) {$I_7$};
		\node [style=label-s] (51) at (-1.75, 6.75) {$Y_{I_0}$};
		\node [style=label-s] (52) at (-4.75, 9.5) {$Y_0^2$};
		\node [style=label-s] (53) at (-5, 0) {$Y_{I_3}$};
		\node [style=label-s] (54) at (-4.25, 5.5) {$Y_{I_2}$};
		\node [style=label-s] (55) at (-4.25, -1.75) {$Y_{I_4}$};
		\node [style=label-s] (56) at (-7.5, 6.75) {$Y_{I_4}$};
		\node [style=label-s] (57) at (-3.5, 2.25) {$Y_{I_7}$};
		\node [style=none] (58) at (-3.75, -5) {$Y_8Y_{45670}^2Y_{560}$};
	\end{pgfonlayer}
	\begin{pgfonlayer}{edgelayer}
		\draw [opacity=0.4] (0) to (1);
		\draw [opacity=0.4] (1) to (3);
		\draw [style=m1] (0) to (2);
		\draw [opacity=0.4] (7) to (8);
		\draw [style=m1] (8) to (9);
		\draw [style=e1] (9) to (6);
		\draw [style=m1] (6) to (7);
		\draw [style=e1] (11) to (6);
		\draw [style=e3] (2) to (16.center);
		\draw [style=e3, in=-150, out=0, looseness=0.75] (16.center) to (3);
		\draw [style=e3, in=-180, out=0] (16.center) to (4);
		\draw [style=e3, in=120, out=0] (16.center) to (5);
		\draw [style=e3, in=210, out=0, looseness=0.75] (16.center) to (6);
		\draw [style=m1, in=-120, out=90, looseness=0.50] (15.center) to (4);
		\draw [style=m1] (15.center) to (5);
		\draw [style=e7] (14.center) to (4);
		\draw [style=e7] (14.center) to (5);
		\draw [style=d4] (0) to (13);
		\draw [style=d4, in=-180, out=-45, looseness=0.75] (13) to (3);
		\draw [style=d4] (13) to (4);
		\draw [style=d4, in=90, out=-45, looseness=0.75] (13) to (5);
		\draw [style=d2, in=0, out=135, looseness=0.75] (12.center) to (1);
		\draw [style=m1, in=-120, out=90, looseness=0.50] (15.center) to (3);
		\draw [style=d1] (0) to (6);
		\draw [style=e7] (14.center) to (1);
		\draw [style=d2] (7) to (12.center);
		\draw [style=d2] (12.center) to (9);
		\draw [style=m1, in=90, out=-60] (38) to (45.center);
		\draw [style=m1, in=120, out=-90] (45.center) to (44);
		\draw [style=m1] (39) to (45.center);
		\draw [style=m1, in=90, out=-60, looseness=0.75] (40) to (46.center);
		\draw [style=m1] (46.center) to (41);
		\draw [opacity=0.4] (38) to (40);
		\draw [style=d7, in=135, out=-15, looseness=0.75] (38) to (47);
		\draw [style=d7] (47) to (39);
		\draw [style=m1] (45.center) to (0);
		\draw [style=m1] (46.center) to (1);
		\draw [style=d7] (47) to (1);
		\draw [opacity=0.4] (44) to (1);
		\draw [opacity=0.4] (39) to (41);
		\draw [opacity=0.4] (0) to (7);
		\draw [style=m1] (15.center) to (11);
	\end{pgfonlayer}
\end{tikzpicture}

%% file: ResetOpacity.tex
\tikzstyle{e1}=[-, color=1, line width=1]
\tikzstyle{e2}=[-, color=2, line width=1]
\tikzstyle{e3}=[-, color=3, line width=1]
\tikzstyle{e4}=[-, color=4, line width=1]
\tikzstyle{e5}=[-, color=5, line width=1]
\tikzstyle{e6}=[-, color=6, line width=1]
\tikzstyle{e7}=[-, color=7, line width=1]
\tikzstyle{e8}=[-, color=8, line width=1]
\tikzstyle{e9}=[-, color=9, line width=1]
\tikzstyle{e0}=[-, color=0, line width=1]
\tikzstyle{d1}=[-, color=1, line width=1.2, loosely dashed]
\tikzstyle{d2}=[-, color=2, line width=1.2, loosely dashed]
\tikzstyle{d3}=[-, color=3, line width=1.2, loosely dashed]
\tikzstyle{d4}=[-, color=4, line width=1.2, loosely dashed]
\tikzstyle{d5}=[-, color=5, line width=1.2, loosely dashed]
\tikzstyle{d6}=[-, color=6, line width=1.2, loosely dashed]
\tikzstyle{d7}=[-, color=7, line width=1.2, loosely dashed]
\tikzstyle{d8}=[-, color=8, line width=1.2, loosely dashed]
\tikzstyle{d9}=[-, color=9, line width=1.2, loosely dashed]
\tikzstyle{d0}=[-, color=0, line width=1.2, loosely dashed]
\tikzstyle{none}=[inner sep=0mm]

%% file: figures/branch_eg.tex
\begin{tikzpicture}[scale=0.5]
	\begin{pgfonlayer}{nodelayer}
		\node [style=label] (0) at (-6.75, 5) {$3$};
		\node [style=label] (1) at (-2, 5) {$4$};
		\node [style=label] (2) at (-6.75, 0) {$8'$};
		\node [style=label] (3) at (-2, 1) {$9$};
		\node [style=label] (4) at (-2.75, 0) {$7'$};
		\node [style=label] (5) at (-3.5, -1) {$0'$};
		\node [style=label] (6) at (-1.25, 2) {$1'$};
		\node [style=label] (7) at (-1.25, 6) {$1$};
		\node [style=label] (8) at (2.5, 6) {$2$};
		\node [style=label] (9) at (2.5, 2) {$5$};
		\node [style=label] (10) at (-6.75, -4.25) {$8$};
		\node [style=label] (11) at (-1.25, -4.25) {$3$};
		\node [style=none] (12) at (0.25, 4.5) {};
		\node [style=label-s] (13) at (-4.5, 2.5) {$I_4$};
		\node [style=label-s] (14) at (-2.75, 2) {$I_7$};
		\node [style=none] (15) at (-2, -2) {};
		\node [style=label-s] (16) at (-5.5, 0) {$I_3$};
		\node [style=label-s] (18) at (-4, -2) {$I_8$};
		\node [style=none] (21) at (-1.75, -2.75) {};
		\node [style=label] (26) at (0.25, -4.25) {$=v_{{opt}}$};
		\node [style=label-s] (27) at (-3, 4.5) {$1\oplus 1$};
		\node [style=label-s] (28) at (-4.25, 3.5) {$I_1$};
		\node [style=label-s] (29) at (1, 4.5) {$I_2$};
		\node [style=label-s] (30) at (-2, -5) {$1\oplus 1$};
		\node [style=label-s] (31) at (-7.5, 5.5) {$2\oplus 0$};
		\node [style=label] (32) at (5.5, 4.25) {$3$};
		\node [style=label] (34) at (5.5, 0.5) {$8'$};
		\node [style=label] (38) at (9.25, 0.5) {$1'$};
		\node [style=label] (39) at (9.25, 4.25) {$1$};
		\node [style=label] (40) at (12.75, 4.25) {$2$};
		\node [style=label] (41) at (12.75, 0.5) {$5$};
		\node [style=label] (42) at (5.5, -3.25) {$8$};
		\node [style=label] (43) at (9.25, -3.25) {$3$};
		\node [style=label-s] (48) at (7, 0.5) {$I_3$};
		\node [style=label-s] (49) at (7.5, -1.5) {$I_8$};
		\node [style=label-s] (57) at (7.5, 2.75) {$I_1$};
		\node [style=label-s] (58) at (11.5, 2.75) {$I_2$};
		\node [style=label] (61) at (17, 4.75) {$3$};
		\node [style=label] (62) at (22.25, 4.75) {$4$};
		\node [style=label] (63) at (17, 0.5) {$8'$};
		\node [style=label] (64) at (22.25, 1.5) {$9$};
		\node [style=label] (65) at (21.5, 0.5) {$7'$};
		\node [style=label] (66) at (20.75, -0.5) {$0'$};
		\node [style=label] (71) at (17, -3.75) {$8$};
		\node [style=label] (72) at (22.25, -3.75) {$3$};
		\node [style=label-s] (74) at (19.5, 2.75) {$I_4$};
		\node [style=none] (75) at (21.5, 2.25) {};
		\node [style=none] (76) at (22.25, -2) {};
		\node [style=label-s] (77) at (18.5, 0.5) {$I_3$};
		\node [style=label-s] (78) at (19.75, -1.75) {$I_8$};
		\node [style=none] (79) at (22.25, -2.5) {};
		\node [style=label-s] (85) at (22, 5.25) {$1\oplus 1$};
		\node [style=label-s] (89) at (17, 5.25) {$2\oplus 0$};
	\end{pgfonlayer}
	\begin{pgfonlayer}{edgelayer}
		\draw (0) to (1);
		\draw [style=m1] (1) to (3);
		\draw [style=m1] (0) to (2);
		\draw (7) to (8);
		\draw [style=m1] (8) to (9);
		\draw [style=e1] (9) to (6);
		\draw (6) to (7);
		\draw [style=m1] (0) to (7);
		\draw (2) to (10);
		\draw [style=m1] (10) to (11);
		\draw [style=m1] (11) to (6);
		\draw [style=e3] (2) to (16);
		\draw [style=e3, in=-150, out=0, looseness=0.75] (16) to (3);
		\draw [style=e3, in=-180, out=0] (16) to (4);
		\draw [style=e3, in=150, out=0, looseness=0.50] (16) to (5);
		\draw [style=e3, in=210, out=0, looseness=0.75] (16) to (6);
		\draw [style=e4, in=-75, out=105] (15.center) to (4);
		\draw [style=e4, in=-30, out=105] (15.center) to (5);
		\draw [style=m1] (14) to (4);
		\draw [style=m1] (14) to (5);
		\draw [style=d4] (0) to (13);
		\draw [style=d4, in=-180, out=-45, looseness=0.75] (13) to (3);
		\draw [style=d4] (13) to (4);
		\draw [style=d4, in=90, out=-45, looseness=0.75] (13) to (5);
		\draw [style=d8] (2) to (18);
		\draw [style=d8] (18) to (11);
		\draw [style=d2, in=0, out=135, looseness=0.75] (12.center) to (1);
		\draw [style=e4, in=270, out=105] (15.center) to (3);
		\draw (21.center) to (15.center);
		\draw [style=e4] (21.center) to (11);
		\draw (21.center) to (15.center);
		\draw [style=d1] (0) to (6);
		\draw [style=m1] (14) to (1);
		\draw [style=d2] (7) to (12.center);
		\draw [style=d2] (12.center) to (9);
		\draw [style=m1] (32) to (34);
		\draw (39) to (40);
		\draw [style=m1] (40) to (41);
		\draw [style=e1] (41) to (38);
		\draw (38) to (39);
		\draw [style=m1] (32) to (39);
		\draw (34) to (42);
		\draw [style=m1] (42) to (43);
		\draw [style=m1] (43) to (38);
		\draw [style=e3] (34) to (48);
		\draw [style=e3] (48) to (38);
		\draw [style=d8] (34) to (49);
		\draw [style=d8] (49) to (43);
		\draw [style=d1] (32) to (38);
		\draw (61) to (62);
		\draw [style=m1] (62) to (64);
		\draw [style=m1] (61) to (63);
		\draw (63) to (71);
		\draw [style=m1] (71) to (72);
		\draw [style=e3] (63) to (77);
		\draw [style=e3, in=-150, out=0, looseness=0.75] (77) to (64);
		\draw [style=e3, in=-180, out=0] (77) to (65);
		\draw [style=e3, in=150, out=0, looseness=0.75] (77) to (66);
		\draw [style=e4, in=-60, out=105] (76.center) to (65);
		\draw [style=e4, in=-30, out=105] (76.center) to (66);
		\draw [style=m1] (75.center) to (65);
		\draw [style=m1] (75.center) to (66);
		\draw [style=d4] (61) to (74);
		\draw [style=d4, in=-180, out=-45, looseness=0.75] (74) to (64);
		\draw [style=d4] (74) to (65);
		\draw [style=d4, in=90, out=-45, looseness=0.75] (74) to (66);
		\draw [style=d8] (63) to (78);
		\draw [style=d8] (78) to (72);
		\draw [style=e4] (76.center) to (64);
		\draw (79.center) to (76.center);
		\draw [style=e4] (79.center) to (72);
		\draw (79.center) to (76.center);
		\draw [style=m1] (75.center) to (62);
		\draw [style=d2] (39) to (41);
	\end{pgfonlayer}
\end{tikzpicture}